\documentclass[12pt,openany]{book}
\usepackage{amsmath, amsthm, amssymb}
\usepackage{mdframed}
\usepackage{lipsum}
\usepackage{graphicx}
\usepackage{algorithm,algpseudocode} 
\usepackage{dsfont}
\usepackage[shortlabels]{enumitem}  
\usepackage{todonotes}

\newmdtheoremenv{thm}{Theorem}[section]
\newtheorem{example}[thm]{Example}

\theoremstyle{definition}
\newmdtheoremenv{defi}[thm]{Definition}
\newmdtheoremenv{lemma}[thm]{Lemma}
\newmdtheoremenv{prop}[thm]{Proposition}
\newmdtheoremenv{cor}[thm]{Corollary}
\newmdtheoremenv{prob}[thm]{Problem}
\newmdtheoremenv{ass}[thm]{Assumption}

\newtheorem{remark}[thm]{Remark}

\usepackage{amsthm,amssymb,amsmath}
\theoremstyle{definition}
\newmdtheoremenv{exercise}{Exercise}[section]
\newmdtheoremenv{outlook}{Outlook}

\usepackage{titlesec}
\titleformat{\chapter}[display]
{\normalfont\bfseries\filcenter}
{\LARGE\thechapter}
{1ex}
{\titlerule[2pt]
\vspace{1ex}%
\LARGE}
[\vspace{1ex}%
{\titlerule[2pt]}]

\newcommand{\cA}{\mathcal{A}}
\newcommand{\cB}{\mathcal{B}}

\newcommand{\cD}{\mathcal{D}}
\newcommand{\cE}{\mathcal{E}}
\newcommand{\cF}{\mathcal{F}}
\newcommand{\cG}{\mathcal{G}}

\newcommand{\cL}{\mathcal{L}}
\newcommand{\cM}{\mathcal{M}}

\newcommand{\cO}{\mathcal{O}}
\newcommand{\cP}{\mathcal{P}}

\newcommand{\cU}{\mathcal{U}}

\newcommand{\cX}{\mathcal{X}}
\newcommand{\cY}{\mathcal{Y}}
\newcommand{\cZ}{\mathcal{Z}}

\newcommand{\E}{\mathbb{E}}

\newcommand{\N}{\mathbb{N}}

\newcommand{\R}{\mathbb{R}}

\newcommand{\V}{\mathbb{V}}

\newcommand{\PP}{\mathbb{P}}

\def\P{\mathbb{P}}

\newcommand{\dd}{{\mathrm{d}}}

\usepackage{mathtools}
\DeclarePairedDelimiter\ceil{\lceil}{\rceil}

\usepackage[onehalfspacing]{setspace}
\usepackage{cancel}
\usepackage[margin=2cm]{geometry}
\usepackage{hyperref}
\definecolor{refcol}{rgb}{0.1,0,0.6}
\hypersetup{colorlinks}
\hypersetup{linkcolor=refcol, citecolor=refcol}
\usepackage{cleveref}

\usepackage{fancyhdr}

\newcommand{\HRule}{\rule{\linewidth}{0.5mm}} 
\title{Lecture notes: Stochastic gradient methods}
\begin{document}

\begin{center}
\HRule \\[0.4cm]
{ \huge \bfseries Introduction to stochastic gradient methods}\\[0.4cm] 
\HRule \\[1.5cm]
\begin{minipage}{0.9\textwidth}
\begin{flushleft} \large
Simon {Weissmann} \newline
\normalsize University of Mannheim, Institute of Mathematics, 68138 Mannheim, Germany\newline
\small \texttt{simon.weissmann@uni-mannheim.de}
\end{flushleft}

\end{minipage}

\vspace{2cm}

\begin{minipage}{0.7\textwidth}
\emph{Abstract}\\
These lecture notes provide an introduction to first-order optimization methods with a particular emphasis on stochastic gradient methods. We begin with deterministic gradient based methods for unconstrained optimization and study their convergence under standard assumptions such as smoothness, convexity, strong convexity, and the Polyak--{\L}ojasiewicz condition. We then turn to stochastic approximation and stochastic gradient descent, motivated by empirical and expected risk minimization in machine learning. The main focus is on convergence theory: we discuss almost sure convergence and convergence in expectation, derive classical convergence rates, and present selected advanced topics, including almost sure convergence rates and variance reduction methods.
\end{minipage}\\[0.5cm]
\end{center}

\newpage

\section*{Preface}
These lecture notes were originally prepared for the course \textit{Optimization in Machine Learning} offered at the University of Mannheim during the spring semester of 2023 and the fall semester of 2024. The present version has been revised and expanded, with a particular emphasis on stochastic gradient methods and their convergence theory. Several of the results presented here are based on standard textbooks, monographs, and research articles; references are provided throughout the text. The notes are intended for graduate students with a basic background in analysis, linear algebra, probability theory, and optimization. Their purpose is not to provide a complete account of the
vast literature on stochastic optimization, but rather to give a self-contained introduction to central convergence arguments for deterministic and stochastic first-order methods.

I would like to thank Marc Schäfer and Tilman Aach for carefully proofreading earlier versions of this manuscript. I am also grateful to Felix Benning, Leif Döring, Sebastian Kassing, Sara Klein, and Ashia Wilson for helpful discussions and suggestions on related topics. Comments, corrections, and suggestions for improvement are very welcome and may be sent to the author by email. \medskip 

\newpage

\tableofcontents

\newpage
\chapter{Introduction}\label{ch:Intro}
The aim of these lecture notes is to provide an introduction to stochastic gradient methods from
the perspective of continuous optimization, complemented by the probabilistic tools needed for
their convergence analysis. While the motivating applications come from machine learning, the
main focus lies on the mathematical analysis of first-order optimization algorithms. We first study
deterministic gradient descent methods under standard structural assumptions on the objective
function. Building on this foundation, we then introduce stochastic gradient descent and analyze
different notions of convergence, including convergence in expectation and almost sure convergence.\medskip

Before turning to optimization methods, we briefly recall the role of optimization within \textit{machine
learning}. In many learning problems, one first chooses a class of candidate models and then trains
a model from this class by minimizing a suitable objective function. These lecture notes focus on
this \textit{training task}. Questions concerning \textit{approximation} and \textit{generalization} are equally important
in machine learning, but they will only serve as motivation here and will not be studied in detail. The field of machine learning is often divided into the following three classes:
\begin{enumerate}
\item \textit{Supervised learning:} Given pairs of \textit{input} and \textit{output} variables, the aim is to learn a relation between inputs and outputs. Typical examples are \textit{classification}, where the output takes values in a discrete set, and \textit{regression},
where the output is continuous.
\item \textit{Unsupervised learning:} Here the data set consists of input variables without specified labels. Typical tasks are to separate the data set into groups (\textit{clustering}) or to learn a probability distribution describing the data set (\textit{density estimation}).
\item \textit{Reinforcement learning:} In this setting, an \textit{agent} interacts with an environment and aims to learn a strategy of actions that maximizes a cumulative \textit{reward}. In contrast to the supervised setting, the optimal behavior is typically learned through trial and error rather than from a fixed labeled data set.
\end{enumerate}

These lecture notes focus mainly on optimization methods motivated by supervised learning. To introduce the corresponding learning task, we use the following notation:
\begin{itemize}
    \item Input set: $\cZ \subset \R^{d_z}$ 
    (e.g.~images, observation points, or features).
    \item Output/target set: $\cY\subset\R^{d_y}$ 
    (e.g.~a label in $\{0,1\}$ or a function value in $\R^{d_y}$).
    \item Training data: 
    $S=\{(z^{(1)},y^{(1)}),\dots,(z^{(m)},y^{(m)})\}$ with
    $(z^{(i)},y^{(i)})\in\cZ\times \cY$, $i=1,\dots,m$.
\end{itemize}
The goal in supervised learning is to approximate the unknown model
\begin{equation*}
    \varphi:\cZ\to\cY,\quad z\mapsto \varphi(z)=y\,.
\end{equation*}

The \textit{task of the learner} is to construct a prediction or approximation $g:\cZ\to\cY$, which is actually a task of function approximation. Typically, the learner aims to find a (finite dimensional) parameter $\theta\in\Theta$ and compute the approximation $g_\theta:\cZ\to\cY$. Here, we will call both $\cG$ and $\Theta$ the \textit{learning class} of possible candidates for $g\in\cG$ or $\theta\in\Theta$ respectively. 

In the following, we fix an underlying probability space $(\Omega,\cF,\mathbb P)$. The \textit{data model} is usually described by the \textit{training data set} as a family of random variables. For example, one may assume that inputs $Z^{(i)}\sim\mu_Z$ are generated independently and then passed through the unknown "true" model $\varphi$:
\[ Y^{(i)} := \varphi(Z^{(i)}) + \xi^{(i)}\,,\quad i=1,\dots,m\,,\]
where $(\xi^{(i)})_{i=1}^m$ denotes possible noise. Hence, we will model the input and output as a jointly varying random variable $(Z,Y)\sim\mu_{(Z,Y)}$ with unknown joint distribution $\mu_{(Z,Y)}$. Later, we will usually assume that we have access to an independent and identically distributed (i.i.d.) sample 
$\{(Z^{(i)},Y^{(i)})\}_{i=1,\dots,m}\,, m\in\N\,, \ \text{with}\ (Z^{(1)},Y^{(1)})\sim\mu_{(Z,Y)}\,.$ To measure the quality of a prediction, let $f:\cG\times \cZ \times \cY\to\R$ be a measurable loss function and let $(Z,Y)\sim\mu_{(Z,Y)}$ with $\E[|f(g,Z,Y)|]<+\infty$ for all $g\in\cG$. 
    \begin{enumerate}
        \item[(i)] We define the \textit{expected risk} $F:\cG\to\R$ by
        \[F(g):=\E_{\mu_{(Z,Y)}}[f(g,Z,Y)]:=\int_{\cZ\times\cY} f(g,z,y)\, \mu_{(Z,Y)}(d(z,y)),\quad g\in\cG\,. \]
        \item[(ii)] Let $(Z^{(1)},Y^{(1)}),\dots,(Z^{(m)},Y^{(m)})$ be i.i.d.~random variables with $(Z^{(1)},Y^{(1)})\sim\mu_{(Z,Y)}$. We define the \textit{empirical risk} $F_m:\cG\to\R$ by
        \[F_m(g):= \frac1m\sum_{i=1}^m f(g,Z^{(i)},Y^{(i)}),\quad g\in\cG\,.\]
    \end{enumerate}
    We call $f$ the \textit{loss function}, and the optimization problems
    \[\min_{g\in\cG}\ F(g) \quad (\min_{g\in\cG}\ F_m(g))\]
    are called the \textit{expected (empirical) risk minimization problem}.

\begin{example}[Classification problem]
    Let $\cY = \{0,1\}\subset \R$ and assume that the training data is generated without noise as
    $Y^{(i)} = \varphi(Z^{(i)})\,,\ i=1,\dots,m\,.$
    In a classification problem one usually aims to classify between a finite number of classes, here for simplicity between two classes. The first class corresponds to $0$ and the second one to $1$. Hence, the true model $\varphi$ maps to the set $\{0,1\}$ and states whether the input belongs to $0$ or $1$. 

    A common choice for a loss function is an indicator function over the predicted classification through $g\in\cG$. We receive a penalty $1$ if the prediction is wrong, whereas $0$ if the prediction is correct. Mathematically written, the loss function can be defined by
    \[f(g,z,y) := \mathds{1}_{\{g(z)\neq y\}} = \mathds{1}_{\{g(z)\neq \varphi(z)\}}\]
    and the expected risk corresponds to the probability of giving a wrong prediction
    \[F(g) = \E_{\mu_{(Z,Y)}}[\mathds{1}_{\{g(Z)\neq Y\}}] = \mathbb P_{\mu_{(Z,Y)}}(g(Z)\neq Y)\,.\]
    When having access to a training data set $\{(Z^{(i)},Y^{(i)})\}_{i=1}^m$ the empirical risk counts the relative number of failures in predicting the correct label
    \[F_m(g) = \frac1m\sum_{i=1}^m \mathds{1}_{\{g(Z^{(i)})\neq Y^{(i)}\}} = \frac{|\{i\in\{1,\dots,m\}\mid g(Z^{(i)})\neq Y^{(i)}\}|}{m}\,.\]
\end{example}

\begin{example}[Regression problem]
    A second classical example of supervised learning tasks are regression problems. As an example we consider the task to approximate a function $\varphi:\R^{d_z}\to\R^{d_y}$ based on the loss function 
    \( f(g,z,y):= \frac12\|g(z)-y\|_{\R^{d_y}}^2\,.\)
    The corresponding expected risk measures the expected squared distance of the prediction $g$
    \[F(g)=\E_{\mu_{(Z,Y)}}[\frac12\|g(Z)-Y\|_{\R^{d_y}}^2]\,,\]
    and the empirical risk computes the averaged squared distance of the prediction 
    \[ F_m(g) = \frac1m\sum_{i=1}^m\frac12\|g(Z^{(i)})-Y^{(i)}\|_{\R^{d_y}}^2\]
    within a training data set $\{(Z^{(i)},Y^{(i)})\}_{i=1}^m$.
\end{example}

Machine learning involves several mathematical questions. In these notes, we distinguish three
central aspects, but only the second one will be studied.
\begin{enumerate}
\item \textit{Approximation (expressive power):} In this research area, the concern revolves around the choice of the function class used to approximate the true model. The question is whether a suitable function class even exists for approximating the true model and how large one should choose the class to approximate the true model up to a certain accuracy.
\item \textit{Training/learning task:} Once we have decided on a specific class of functions in which we aim to approximate the underlying true model, the next step is to find the best possible representation within this chosen class $\cG$ or $\Theta$. For example, assuming a parametrized representation $g_\theta:\R^{d_z}\to\R^{d_y}$ of the function approximation and given a training data set $\{(z^{(i)},y^{(i)})\}_{i=1}^N$ constructed through the true model, i.e.~$y^{(i)} = \varphi(z^{(i)})$, $i=1,\dots,N$, we aim to find $\theta\in\Theta$ such that we obtain the best possible approximation $y^{(i)}\approx g_\theta(z^{(i)})$. As described above, the task in learning of the model is then to solve an optimization problem of the form
\[\min_{\theta\in\Theta}\ F_N(\theta,\{(z^{(i)},y^{(i)})\}_{i=1}^N),\]
where $F_N:\Theta\times \left(\times_{i=1}^N (\R^{d_z}\times \R^{d_y})\right)\to\R$ is a suitable objective function. As we have seen in the example of regression, a typical example of objective function is
\[F_N(\theta,\{z^{(i)},y^{(i)}\}_{i=1}^N) = \frac{1}{N} \sum_{i=1}^N\|g_\theta(z^{(i)}) - y^{(i)}\|^2 + \mathcal R(\theta), \]
where $\mathcal R:\Theta\to\R$ is a regularization function in order to avoid the so-called \textit{overfitting}. 
\item \textit{Generalization:} Once we have learned or trained a parameter $\theta_\ast$ that approximates the true model on the training data, the natural question is how well this
approximation generalizes to unseen data. In other words, one seeks to evaluate the quality of
\(
g_{\theta_\ast}(z) \approx \varphi(z)
\)
for data points $(z,\varphi(z))$ that have not been used during training.
\end{enumerate}

The questions of approximation and generalization are equally important, but they are not the focus of these notes. Here, we concentrate on the training problem: given an objective function, typically an empirical or expected risk, how can it be minimized efficiently by first-order methods?

\paragraph{Outline of this course:} 
In Chapter~\ref{ch:unOpt}, we introduce deterministic first-order methods for unconstrained
optimization, with a particular emphasis on gradient descent. We discuss convergence under
different assumptions on the objective function, including smoothness, convexity, strong
convexity, and the Polyak--{\L}ojasiewicz condition. We also briefly discuss sub-gradient
methods for non-smooth convex functions. Chapter~\ref{ch:accmethods} is devoted to
momentum and acceleration, including Polyak's heavy ball method and Nesterov's accelerated
gradient method. In Chapter~\ref{ch:sgd}, we turn to stochastic approximation and stochastic
gradient descent. We derive SGD from empirical and expected risk minimization and study almost sure convergence
and convergence in expectation. We also discuss convergence rates, lower bounds,
almost sure convergence rates, and variance reduction methods.

\chapter{Unconstrained optimization methods}\label{ch:unOpt}
In this chapter, we provide a brief overview of basic concepts of unconstrained optimization. The material presented here is based on standard references such as~\cite{DPB15,N2018,Lan2021}. Fundamental definitions and properties from convex analysis and optimality conditions are collected in \Cref{app:convex} and \Cref{app:op-condition}. Building upon this foundation, we will explore a series of gradient based methods to iteratively solve optimization problems. We analyze the convergence of the resulting schemes under different assumptions on the objective
function, such as differentiability, smoothness, convexity, and strong convexity. Consider the following problem formulation, which will be the central focus of this lecture.
\begin{prob}\label{prob:opti}
Let $f:X\to\R$ be a (continuous) function with domain $X\subset \R^d$. For which value(s) $x\in X$ is the function evaluation $f(x)$ minimal? 
\end{prob}
We will refer to Problem~\ref{prob:opti} as an \textit{optimization} or \textit{minimization problem} and write shortly
\begin{equation}\label{eq:opt_problem}
\min_{x\in X}\ f(x)\, .
\end{equation}
In this lecture course, the corresponding function $f$ is called \textit{objective function}. 
\begin{remark}
\begin{itemize}
\item Our special focus will lie in optimization problems without constraints, in which we assume $X=\R^d$. For the case $X\subsetneq \R^d$ we refer to \eqref{eq:opt_problem} as optimization problem under the constraint $x\in X$, where we also write
\begin{equation*}
\min_{x\in \R^d}\ f(x),\quad \text{s.t.}\quad x\in X, 
\end{equation*}
in order to highlight the constraint. Note that "s.t." stands for "subject to".
\item For simplicity, we will consider minimization problems, since each maximization problem can be equivalently rewritten as minimization problem: 
\[\max_{x\in X}\ f(x) \quad \widehat{=} \quad \min_{x\in X}\ -f(x)\, .\]
\item This lecture focuses on so-called continuous optimization problems, where the feasible set is infinite (uncountable). In case the feasible set is countable or finite, the optimization problem is called discrete.
\item Typically, the objective function $f$ is nonlinear, such that \eqref{eq:opt_problem} is often called nonlinear program.
\end{itemize}
\end{remark}

In the following, we introduce the notion of local and global solutions for \eqref{eq:opt_problem}.
\begin{defi}
Let $f:X\to\R$ for some $X\subset \R^d$. The point $x_\ast\in X$ is called 
\begin{itemize}
\item[a)] \textit{local minimum} of $f$ over $X$, if there exists $\varepsilon>0$ such that $f(x_\ast)\le f(x)$ for all $x\in X$ with $\|x-x_\ast\|<\varepsilon$. If it even holds true that $f(x_\ast)<f(x)$ for all $x\in X\setminus \{x_\ast\} $ with $\|x-x_\ast\|<\varepsilon$, we call $x_\ast$ \textit{strict local minimum}.
\item[b)] \textit{global minimum} of $f$ over $X$, if $f(x_\ast)\le f(x)$ for all $x\in X$. If it even holds true that $f(x_\ast)<f(x)$ for all $x\in X\setminus \{x_\ast\}$, we call $x_\ast$ \textit{strict global minimum}.
\end{itemize}
\end{defi}
\begin{remark}
A local minimum $x_\ast\in X$ minimizes the objective function $f$ only in a local neighborhood $\mathcal B_{\varepsilon}(x_\ast) = \{x\in \R^d\mid \|x-x_\ast\|<\varepsilon\}$, whereas a global minimum minimizes the objective function over the entire domain / feasible set $X$. This definition also covers constrained optimization problems.
Furthermore, every global minimum is also a local minimum. The converse does not hold. 
\end{remark}

\section{Optimization methods based on descent directions}
It is often challenging, and in many applications impossible, to compute solutions of the problem $\min_{x\in\R^d}\ f(x)$ analytically. Instead, iterative methods can be used to solve the minimization task numerically. The focus will be on so-called \textit{descent methods} based on descent directions.

\begin{defi}\label{defi:descentdirection}
Let $f:\R^d\to\R$ be the objective function. We call a vector $d\in\R^d$ \textit{descent direction} of $f$ in $x\in\R^d$, if there exists $\bar \alpha>0$ such that
\[f(x+\alpha d) < f(x) \quad \text{for all}\ \alpha\in(0,\bar \alpha].\]
\end{defi}

Our aim is to construct an iterative scheme $x_0,x_1,x_2,\dots$ initialized with $x_0\in\R^d$, such that $f(x_{k+1})<f(x_k)$ for $k\ge0$. The descent direction will be the key to construct this scheme. The following lemma describes a sufficient condition for a descent direction.

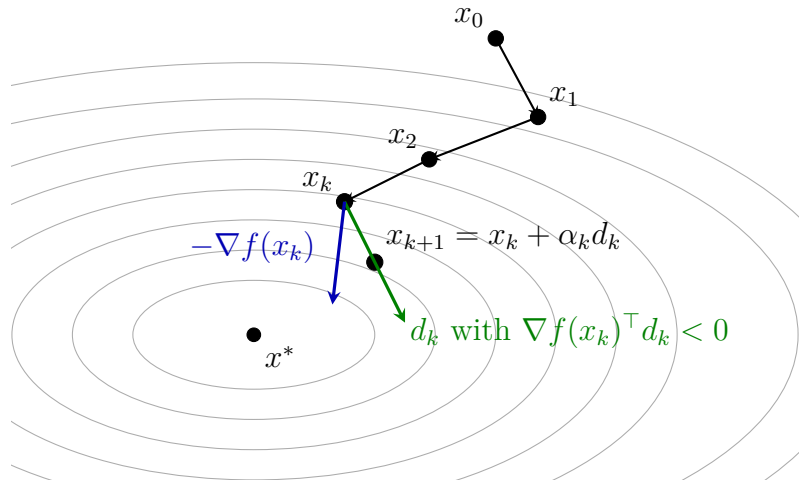
\begin{figure}[!htb]
\centering
\vspace{-2.0cm}
\begin{tikzpicture}[scale=1.6, >=stealth]
\clip (-2.0,-1.2) rectangle (4.5,4.0);
\foreach \a/\b in {1.0/0.45,1.5/0.70,2.0/0.95,2.5/1.20,3.0/1.45,3.5/1.70, 4.5/1.95, 5.0/2.25}
{
    \draw[gray!65] (0,0) ellipse ({\a} and {\b});
}
\filldraw[black] (0,0) circle (1.6pt);
\node[below right] at (0,0) {$x^\ast$};
\coordinate (x0) at (2.0,2.45);
\coordinate (x1) at (2.35,1.80);
\coordinate (x2) at (1.45,1.45);

\filldraw[black] (x0) circle (1.8pt) node[above left] {$x_0$};
\filldraw[black] (x1) circle (1.8pt) node[above right] {$x_1$};
\filldraw[black] (x2) circle (1.8pt) node[above left] {$x_2$};

\draw[->, thick] (x0) -- (x1);
\draw[->, thick] (x1) -- (x2);

\coordinate (xk) at (0.75,1.1);
\coordinate (xkp) at (1.0,0.6);
\coordinate (xkp_vec) at (1.25,0.1);

\draw[->, thick] (x2) -- (xk);

\filldraw[black] (xk) circle (1.9pt);
\node[above left] at (xk) {$x_k$};

\filldraw[black] (xkp) circle (1.9pt);
\node[above right] at (xkp) {$x_{k+1}=x_k+\alpha_k d_k$};

\draw[->, very thick, green!50!black] (xk) -- (xkp_vec);
\node[green!50!black, above right] at (1.2,-0.2) {$d_k$ with $\nabla f(x_k)^\top d_k<0$};

\draw[->, very thick, blue!70!black] (xk) -- (0.65,0.25);
\node[blue!70!black, below left] at (0.6,0.95) {$-\nabla f(x_k)$};

\end{tikzpicture}
\caption{Illustration of descent methods.}
\label{fig:descent-method}
\end{figure}

\begin{lemma}\label{lem:descent_condition}
Let $f:\R^d\to\R$ be continuously differentiable in $x\in\R^d$. Then the condition 
\begin{equation}\label{eq:descent_condition}
\nabla f(x)^\top d <0
\end{equation}
is sufficient for $d\in\R^d$ being a descent direction of $f$ in $x$.
\end{lemma}
\begin{proof}
We fix $x\in\R^d$ and $d\in\R^d$ satisfying \eqref{eq:descent_condition} and define $\varphi(\alpha) = f(x+\alpha d)$. By Taylor expansion it follows
\[\varphi(\alpha) = \varphi(0) + \alpha \varphi'(0) + o(\alpha). \]
Note that $\varphi(0) = f(x)$ and $\varphi'(0) = \nabla f(x)^\top d$ such that
\[ \frac{\varphi(\alpha)- \varphi(0)}{\alpha} = \underbrace{\nabla f(x)^\top d}_{<0}+ \underbrace{\frac{o(\alpha)}{\alpha}}_{\to 0,\ \alpha\to 0}.\]
This implies that there exists $\bar\alpha>0$ such that 
\[\frac{\varphi(\alpha)-\varphi(0)}{\alpha} < 0 \]
for all $\alpha\in(0,\bar\alpha]$.
\end{proof}

The condition \eqref{eq:descent_condition} is not a necessary condition for $d\in\R^d$ being a descent direction. See Exercise~\ref{exc:desc_cond} for more details.

\begin{exercise}\label{exc:desc_cond}
Let $x_\ast\in\R^d$ be a strict local maximum of $f:\R^d\to\R$. Prove that every $d\in\R^d\setminus\{0\}$ is a descent direction of $f$ in $x_\ast$. 
\end{exercise}

\begin{example}
The following two choices are classical examples of descent directions.
\begin{itemize}
\item Given $x\in\R^d$, the choice $d=-\nabla f(x)$ is a descent direction of $f$ in $x$. This direction is also called \textit{steepest descent direction}.
\item Given $x\in\R^d$ and positive definite matrix $M\in\R^{d\times d}$, the choice $d=-M\nabla f(x)$ is a descent direction of $f$ in $x$. We also call it \textit{preconditioned gradient based descent direction}.
\end{itemize}
\end{example}

The resulting iterative descent method is formulated in the following algorithm.

\begin{algorithm}[htb!]
\begin{algorithmic}[1]
\State \textbf{Input:} \begin{itemize}
 \item objective function $f:\R^d\to\R$
 \item initial $x_0\in\R^d$
 \end{itemize}
 \State set $k=0$
\While{"convergence/stopping criterion not met"}
	\State find a descent direction $d_k\in\R^d$ of $f$ in $x_k$
	\State determine a step size $\alpha_k>0$ such that $f(x_k+\alpha_k d_k)< f(x_k)$
	\State set $x_{k+1} = x_k +\alpha_k d_k$, $k \mapsto k+1$
\EndWhile
\end{algorithmic}
 \caption{Descent method}\label{alg:DM}
\end{algorithm}

\begin{remark}
\begin{itemize}
\item The "convergence/stopping criterion" is of practical relevance. We will suppress this criterion in our theoretical analysis and study the long-time behavior of various types of algorithms as the number of iterations tends to infinity.
\item The values $\alpha_k>0$ are called \textit{step sizes}. In the research area of machine learning, they are often called \textit{learning rates}. In general, the choice of the step size rate can be fixed, adaptive or based on line search. We will give more details below. 
\end{itemize}
\end{remark}

\subsection{Examples of descent directions.}
We consider several examples of descent directions of the unified form
\begin{equation}
d_k = -D_k\nabla f(x_k),
\end{equation}
where $D_k\in\R^{d\times d}$ is a positive definite matrix.

\begin{defi}
Iterative schemes of the form
\begin{equation}
x_{k+1} = x_k - \alpha_k D_k \nabla f(x_k),\quad x_0\in\R^d,\ k\ge0,
\end{equation}
with $\alpha_k>0$ and $D_k\in\R^{d\times d}$ positive definite are called \textit{gradient methods}.
\end{defi}

\begin{remark}
The particular choice $d_k=-D_k\nabla f(x_k)$ describes a descent direction due to
\[\nabla f(x_k)^\top d_k = -\nabla f(x_k)^\top D_k \nabla f(x_k) < 0. \]
\end{remark}

We consider the following examples of gradient methods.
\begin{enumerate}
\item[a)] \textit{Method of steepest descent:} This scheme is described by the simplified choice $D_k= {\mathrm{Id}}$, which means that
\begin{equation}\label{eq:GD_method}
x_{k+1} = x_k - \alpha_k \nabla f(x_k)\,,
\end{equation}
and is also known under the name \textit{gradient descent method} (GD). The name "steepest descent" can be motivated by the normalized descent direction 
\[d_k = -\frac{\nabla f(x_k)}{\|\nabla f(x_k)\|}.\]
More details will follow later.
\item[b)] \textit{Newton method:} Consider the quadratic approximation (second-order Taylor approximation) of the objective function $f$. Let $f:\R^d\to\R$ be twice continuously differentiable, $x_k\in\R^d$ be the current iteration and approximate
\[f(x) \approx f_q(x) := f(x_k) + \nabla f(x_k)^\top (x-x_k) + \frac12 (x-x_k)^\top \nabla^2f(x_k)(x-x_k). \]
In order to find a stationary point $x_\ast\in\R^d$ of $f_q$ one can solve
\[\nabla f_q(x) = \nabla f(x_k) + \nabla^2 f(x_k) (x-x_k) = 0 \]
which yields
\begin{equation}\label{eq:Newtoniteration}
x_\ast = x_k-(\nabla^2f(x_k))^{-1}\nabla f(x_k)
\end{equation}
assuming that $\nabla^2 f(x_k)$ is positive definite. This iteration corresponds to the \textit{Newton iteration}. The Newton method is a gradient method with the particular choice $D_k = (\nabla^2 f(x_k))^{-1}$ provided that $\nabla^2 f(x_k)$ is regular, i.e.~
\[x_{k+1} = x_k - \alpha_k (\nabla^2f(x_k))^{-1}\nabla f(x_k).\]
It is a second-order gradient method, since  second-order derivatives are used to formulate the iterative scheme. However,  our focus in this lecture will be on first-order gradient methods.
\item[c)] \textit{Quasi-Newton method:} If the Hessian $\nabla^2 f(x_k)$ is not invertible, or if solving the corresponding linear system \eqref{eq:Newtoniteration} is computationally expensive, one often replaces the Hessian or its inverse by an approximation. This leads to the class of \textit{quasi-Newton methods}.
\end{enumerate}

\subsection{Selection of the step size / learning rate.}
After we have introduced the descent direction, we will now consider the choice of step size / learning rate $\alpha_k$.
\begin{enumerate}
\item[a)] \textit{Constant step size:} The most straightforward choice of step size is the constant step size $\alpha_k=\bar \alpha$ for all $k\ge0$ and some $\bar \alpha>0$ sufficiently small. However, convergence to a stationary point might be slow for too small $\bar \alpha$. In contrast, if $\bar\alpha$ is chosen too large, the resulting scheme might diverge.  We will describe situations for which we can derive specific upper bounds on $\bar \alpha$ to ensure the convergence of different gradient methods.
\item[b)] \textit{Diminishing step size:} Another popular choice, in particular for stochastic gradient schemes, is a diminishing sequence of step sizes $\alpha_k\to 0$ for $k\to\infty$. Again, for too large choices of $\alpha_k$ the resulting scheme violates the monotonic descent along the iteration. Furthermore, if $\alpha_k$ decreases too fast, the steps may become too small to make significant
progress, even if the iterate is still far away from a stationary point. A popular condition for choosing $\alpha_k$, which will regularly appear in this lecture, is 
\[ \sum_{k=1}^\infty \alpha_k = \infty\quad \text{and}\quad \sum_{k=1}^\infty \alpha_k^2<\infty.\]
\item[c)] \textit{Adaptive step size rules:} In many gradient based methods, the step size is often not chosen according to a
pre-specified deterministic schedule, but is adapted based on information collected during the
iteration. Two important examples are AdaGrad-type step sizes \cite{JMLR:v12:duchi11a,tieleman2012rmsprop,zeiler2012adadeltaadaptivelearningrate} and Polyak step sizes \cite{hazan2022revisitingpolyakstepsize}. AdaGrad methods scale the step size by the accumulated size of past gradients. In a simple scalar
form, one may choose
\[
\alpha_k
=
\frac{\alpha}{\sqrt{\delta+\sum_{i=0}^{k}\|\nabla f(x_i)\|^2}},
\]
where $\alpha>0$ and $\delta>0$ are fixed parameters. Hence, the effective step size becomes
smaller in regions where large gradients have been observed. More refined variants use coordinate-wise scaling and are widely used in machine learning.

Polyak step sizes use information about the current objective value gap. If the optimal value
$f^\ast=\inf_{x\in\R^d}\,f(x)$ is known, one can choose, for example
\[
\alpha_k
=
\frac{f(x_k)-f^\ast}{\|\nabla f(x_k)\|^2}.
\]
This choice is motivated by the idea of taking a step that would minimize a local first-order model
of the objective. However, in practice, the value $f^\ast$ is usually unknown and must be estimated
or replaced by a lower bound. Both AdaGrad and Polyak step sizes can lead to strong practical performance and have their own advanced
convergence theory. A detailed analysis is beyond the scope of these notes.
\item[d)] \textit{Step size rule based on line search:} Ideally, one would like to choose $\alpha_k$ such that the decrease along the chosen descent direction is maximized. One possible choice is
\[\alpha_k \in \arg\min_{\alpha\in[0,s]}\ f(x_k+ \alpha d_k)\]
for some pre-specified $s>0$. In most scenarios, however, this one-dimensional minimization problem cannot be solved exactly, and approximate line-search procedures are used instead. One of the most popular line-search strategies is the Armijo rule. Here, the step size is successively decreased until a sufficient decrease condition is satisfied. The Armijo rule is given by the condition
\begin{equation}\label{eq:armijo_condition}
f(x+\alpha d) \le f(x) + \sigma\alpha \nabla f(x)^\top d,
\end{equation}
where $x$ denotes the current iteration, $d$ is the chosen descent direction, and $\sigma\in(0,1)$.
If $d\in\R^d$ satisfies $\nabla f(x)^\top d<0$ (e.g.~$d=-\nabla f(x)$), then condition \eqref{eq:armijo_condition} guarantees a decrease of the objective function. Furthermore, there exists some $\alpha>0$ satisfying condition \eqref{eq:armijo_condition}. In order to choose a suitable step size, we can apply a so-called \textit{backtracking} line search. Given a certain initial step size $\alpha^{(0)} = s_0>0$ we will reduce the step size sequentially until condition \eqref{eq:armijo_condition} is satisfied. The Armijo step size rule is summarized in Algorithm~\ref{alg:armijo}.
\end{enumerate}

\begin{figure}[!htb]
  \centering \includegraphics[width=0.5\textwidth]{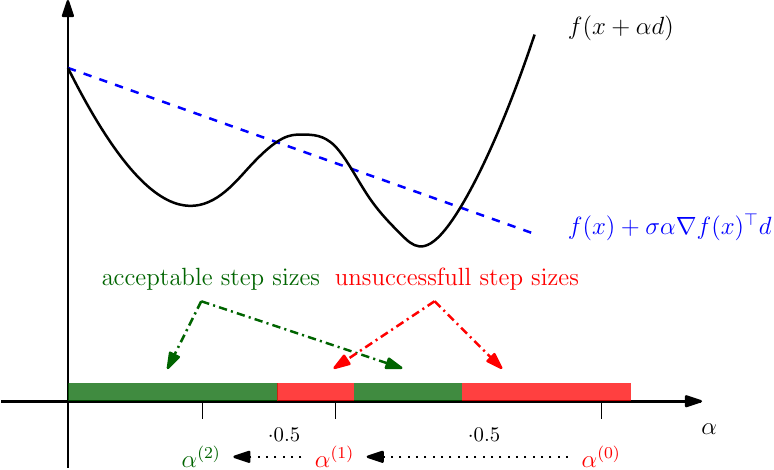}
 \caption{Illustration of the Armijo step size rule based on Figure~1.2.7 in \cite{DPB15}. In the green area condition \eqref{eq:armijo_condition} is satisfied such that $\alpha$ is accepted.} \label{fig:armijo}
\end{figure}

\begin{algorithm}[htb!]
\begin{algorithmic}[1]
\State \textbf{Input:} \begin{itemize}
 \item current iteration $x_k$ and descent direction $d_k$
 \item parameter $\sigma\in(0,1)$, $\rho\in(0,1)$
 \item initial step size $s_0>0$
 \end{itemize}
 \State set $\ell=0$, $\alpha^{(0)} = s_0$
\While{$f(x_k + \alpha^{(\ell)} d_k)> f(x_k) + \sigma \alpha^{(\ell)} \nabla f(x_k)^\top d_k $}
	\State set $\alpha^{(\ell+1)} = \rho\cdot \alpha^{(\ell)}$
	\State set $\ell\mapsto \ell+1$
\EndWhile
\State set $\alpha_k=\alpha^{(\ell)}$
\end{algorithmic}
 \caption{Armijo step size rule}\label{alg:armijo}
\end{algorithm}

We refer interested readers to \cite{DPB15} for more information on alternative step size rules. Throughout this course, we focus on constant and diminishing step sizes.

\subsection{Discussion about convergence behavior.}
Our aim in this course is to analyze the convergence of various optimization algorithms. The first question is what convergence behavior do we expect? In an optimal scenario, we wish that the optimization scheme should converge from any initial state to a global minimum of the objective function. Unfortunately, typically this scenario is way too optimistic.

We consider the gradient descent scheme \eqref{eq:GD_method}, where the state in each iteration moves in the direction of steepest descent. This means the iteration always moves downhill independent of the global structure of the objective function. On the one hand, the iteration is attracted by local minima but on the other hand gets stuck in any stationary point. Without specific structure of the objective function such as convexity, we can only hope for convergence to stationary points. In general, it is not clear whether there exists an accumulation or even a limit point of the sequence $(x_k)_{k\in\N}$ constructed by the gradient descent method \eqref{eq:GD_method}.

\begin{figure}[!htb]
  \centering \includegraphics[width=0.7\textwidth]{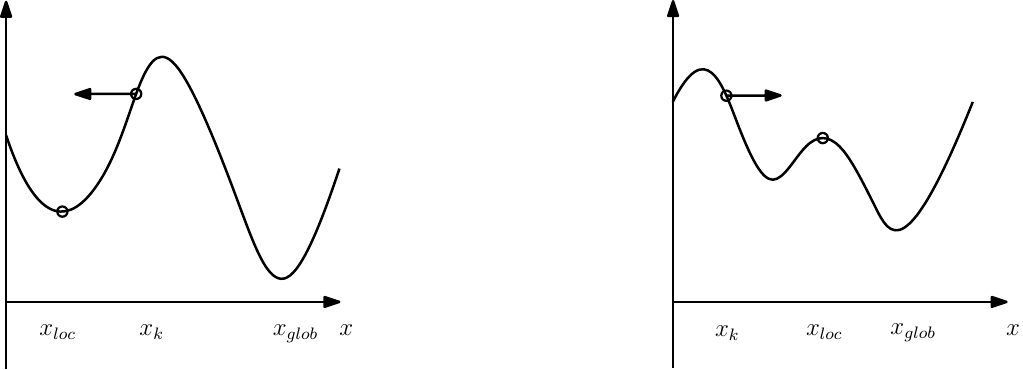}
 \caption{Illustration of possible terminations of the gradient descent method.} \label{fig:GD_stuck}
\end{figure}

However, not all stationary points are equally problematic. In particular, one can distinguish
between local minima and saddle points. A stationary point $x_\ast$ is called a strict saddle point
if the Hessian $\nabla^2 f(x_\ast)$ has at least one negative eigenvalue. Intuitively, this means that
there exists a direction of negative curvature along which the objective function decreases. Although gradient descent can in principle get stuck at saddle points, this behavior is unstable
for strict saddle points. Under suitable regularity assumptions and for sufficiently small constant step sizes, randomly initialized gradient descent avoids strict saddle points with probability one~\cite{lee_first-order_2019}. This is an
asymptotic statement: it says that the set of initializations from which gradient descent converges to a strict saddle point has Lebesgue measure zero, but it does not quantify the time needed to escape a neighborhood of a saddle point.

\section{Convergence analysis of gradient descent}
Before turning to the convergence analysis of gradient descent, we first motivate its interpretation as a \textit{steepest descent method}. By Taylor expansion, we have seen in Lemma~\ref{lem:descent_condition} that $\nabla f(x)^\top d<0$ characterizes the strength of the descent direction $d\in \R^d$. Let $x\in\R^d$ be fixed and choose a normalized descent direction $d\in\R^d$ such that 
\begin{equation}\label{eq:steepest_desc} 
\min_{d\in\R^d}\ \nabla f(x)^\top d,\quad \text{s.t.} \|d\|=1.
\end{equation}
We consider normalized directions because the normalization determines only the direction of movement, whereas the length of the step is later controlled by the step size $\alpha_k$. By the Cauchy-Schwarz inequality, we first observe that for $\|d\|=1$ 
\[0\le |\nabla f(x)^\top d| \le \|\nabla f(x)\| \|d\| = \|\nabla f(x)\|\, .\]
Then it also holds that
\[\nabla f(x)^\top d\ge -|\nabla f(x)^\top d|\ge -\|\nabla f(x)\|\,.\]
Since the choice $d_\ast=-\frac{\nabla f(x)}{\|\nabla f(x)\|}$ leads to $\nabla f(x)^\top d = -\|\nabla f(x)\|$, $d_\ast$ is a solution of \eqref{eq:steepest_desc}. Thus, the negative gradient $-\nabla f(x)$ is the direction of steepest descent among all normalized directions in the sense that it minimizes the directional derivative $\nabla f(x)^\top d$. 

\subsection{Descent estimates via smoothness}
In order to characterize the limit points of the steepest descent scheme, we will need more assumptions about our underlying objective function $f$. One important property considered in this lecture course is \textit{smoothness}. This property guarantees a descent of the objective function along the trajectories of the gradient descent method as long as the step size is sufficiently small.

\begin{defi}
A differentiable function $f:\R^d\to\R$ is called $L$-smooth for some $L>0$ if its gradient $\nabla f$ is $L$-Lipschitz continuous, i.e.
\[
\|\nabla f(x)-\nabla f(y)\|\leq L\|x-y\|,
\quad \text{for all}\ x,y\in\R^d\,.
\]
\end{defi}

If $f$ is twice continuously differentiable, then a sufficient condition for $L$-smoothness is given by a bounded Hessian $\sup_{x\in\R^d}\|\nabla^2 f(x)\|\le L$. Assuming that our objective function $f$ is $L$-smooth allows us to apply the following descent Lemma.
\begin{lemma}[Descent lemma]\label{lem:descentlemma}
Let $f:\R^d\to\R$ be $L$-smooth. Then, for all $x,y\in\R^d$,
\[f(x+y) \le f(x) + y^\top \nabla f(x) + \frac{L}{2}\|y\|^2. \]
\end{lemma}

\begin{proof}
We define $\phi(t) = f(x+ty)$ and apply chain rule in order to derive
\[\phi'(t) = y^\top \nabla f(x+ty),\quad t\in[0,1]. \]
By the fundamental theorem of calculus it follows
\begin{align*}
f(x+y) - f(x) = \phi(1)- \phi(0) &= \int_0^1 \phi'(t)\, {\mathrm d}t = \int_0^1 y^\top \nabla f(x+ty)\,{\mathrm d}t\\
&=\int_0^1 y^\top \nabla f(x)\,{\mathrm d}t + \int_0^1y^\top (\nabla f(x+ty)-\nabla f(x))\,{\mathrm d}t\\
&\le y^\top \nabla f(x) + \int_0^1\|y\|\cdot \|\nabla f(x+ty)-\nabla f(x)\|\,{\mathrm d}t\\
&\le y^\top \nabla f(x) + \|y\|\int_0^1 Lt\cdot \|y\|\,{\mathrm d}t\\
&= y^\top \nabla f(x) + \frac{L}{2}\|y\|^2,
\end{align*}
where we have applied Cauchy-Schwarz followed by $L$-smoothness. 
\end{proof}

From the descent lemma we obtain for $L$-smooth functions $f$ the upper bound
\begin{equation}\label{eq:upperbound_smooth} 
f(y) \le f(x) + \nabla f(x)^\top (y-x) + \frac{L}{2}\|x-y\|^2, \quad \text{for all}\ x,y\in\R^d.
\end{equation}
In comparison, for $\mu$-strongly convex functions $f$ we have a lower bound
\begin{equation}\label{eq:upperbound_strongconvex} 
f(y) \ge f(x) + \nabla f(x)^\top (y-x) + \frac{\mu}{2}\|x-y\|^2, \quad \text{for all}\ x,y\in\R^d.
\end{equation}

\begin{figure}[!htb]
  \centering \includegraphics[width=0.7\textwidth]{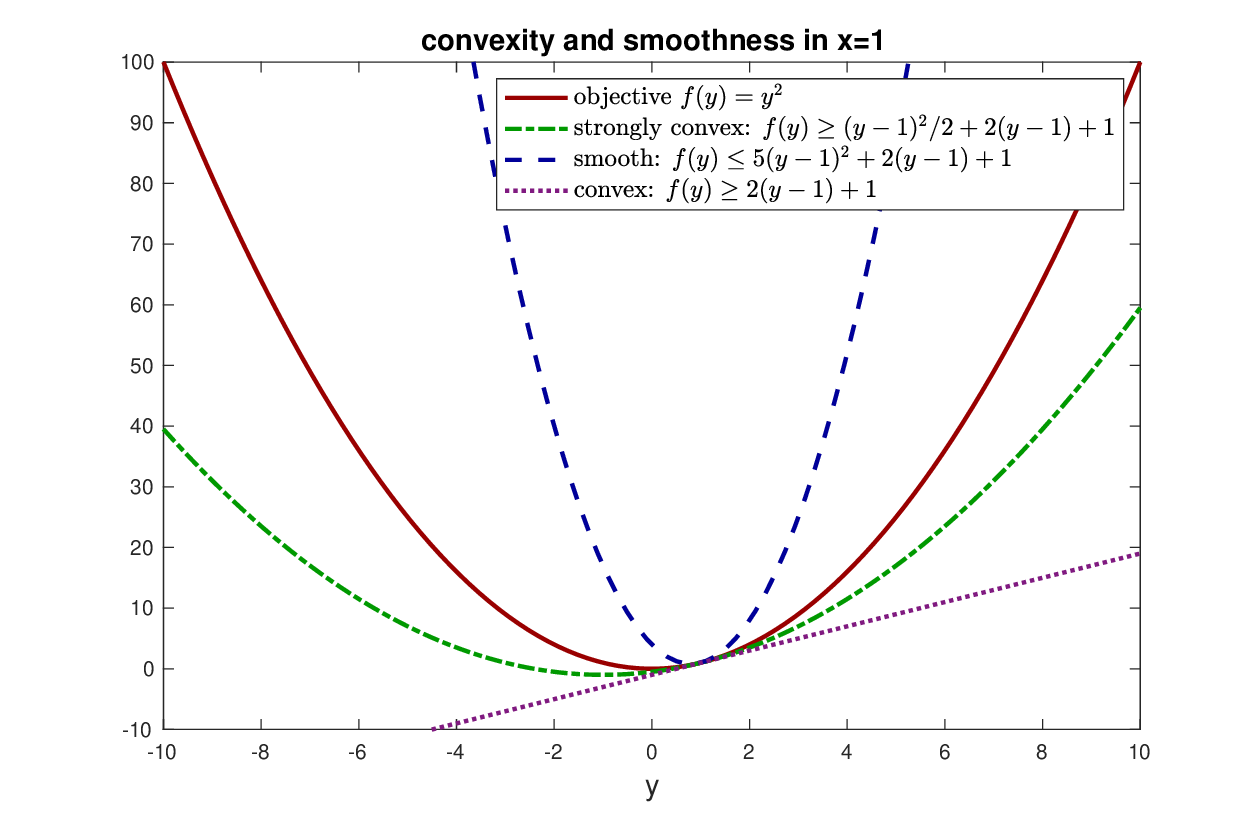}
 \caption{Illustration of $L$-smoothness and $\mu$-strong convexity. We consider the function $f(x) = x^2$ such that $f$ is $L$-smooth and $\mu$-strongly convex with $1:=\mu < f''(x) = 2 < L =: 10$ for all $x\in\R$. This plot illustrates convexity and smoothness of $f$ in $x=1$ using the upper and lower bound \eqref{eq:upperbound_smooth} and \eqref{eq:upperbound_strongconvex} on $f(y)$, $y\in\R$.}  \label{fig:smooth_strongconvex}
\end{figure} 

Combining both estimates, we obtain for $L$-smooth and $\mu$-strongly convex functions
\[\frac{\mu}{2}\|x-y\|^2 \le f(y)-f(x) - \nabla f(x)^\top (y-x) \le \frac{L}{2}\|x-y\|^2. \]
This motivates the definition of the Bregman divergence.
\begin{defi}
Let $f:\R^d\to\R$ be continuously differentiable. The Bregman divergence $D_f^{(B)}:\R^d\times\R^d\to\R$ is defined by
\[D_f^{(B)}(x,y) = f(x)-f(y)-\nabla f(y)^\top (x-y),\quad x,y\in\R^d. \]
\end{defi}
If we assume $L$-smoothness and convexity of $f$ (without $\mu$-strong convexity),  we are able to obtain an additional bound on the differences of the gradient evaluation.

\begin{lemma}\label{lem:convex_and_smooth}
Let $f:\R^d\to\R$ be convex and $L$-smooth, then
\[\frac{1}{2L}\|\nabla f(x)-\nabla f(y)\|^2\le D_f^{(B)}(x,y) \]
is satisfied for all $x,y\in\R^d$.
\end{lemma}
\begin{proof}
    We first verify that for arbitrary $L$-smooth $\tilde f$ and any $y\in\R^d$ we have
    \[ \min_{x\in\R^d}\ \tilde f(x) \le \tilde f(y) -\frac{1}{2L} \|\nabla \tilde f(y)\|^2\,.\]
    From \eqref{eq:upperbound_smooth} we have that 
    \[ \min_{x\in\R^d} \tilde f(x) \le \min_{x\in\R^d} \{\tilde f(y) + \nabla \tilde f(y)^\top (x-y) + \frac{L}2 \|y-x\|^2\} = \tilde f(x_0) - \nabla \tilde f(y)^\top y + \min_{x\in\R^d}\,\varphi(x)\]
    for $\varphi(x) := \nabla \tilde f(y)^\top x + \frac{L}{2} \|x-y\|^2$. Minimizing $\varphi(\cdot)$ yields the optimum $x_\ast = y-\frac{1}{L}\nabla \tilde f(y)$ and, consequently, we have
    \[ \min_{x\in\R^d} \tilde f(x) \le \tilde f(y) -\frac{1}{L}\|\nabla \tilde f(y)\|^2 + \frac{L}2 \|\frac{1}{L}\nabla \tilde f(y)\|^2 = \tilde f(y) - \frac{1}{2L} \|\nabla \tilde f(y)\|^2\,.\]

    We are now in the position to verify the assertion. Define $\tilde f(x) := f(x)-f(y) - \nabla f(y)^\top (x-y)$, $x\in\R^d$ for fixed $y\in\R^d$. Then we have that $\nabla \tilde f(x) = \nabla f(x) -\nabla f(y)$ meaning that $\tilde f$ is still $L$-smooth and satisfies 
    \[0 \le \min_{z\in\R^d}\tilde f(z) \le \tilde f(x) -\frac{1}{2L} \|\nabla \tilde f(x)\|^2 \] 
    where we have used that $\tilde f(x)\ge 0$ due to convexity of $f$. Reordering the inequality yields
    \[ \frac{1}{2L} \|\nabla f(x)-\nabla f(y)\|^2 \le \tilde f(x) = D_f^{(B)} (x,y) \]
    
\end{proof}
In general, for a convex function, it only holds true that
\[D_f^{(B)} (y,x) = f(y)-f(x)-\nabla f(x)^\top(y-x)\ge 0.\]
With the additional assumption of $L$-smoothness, we obtain a stronger characterization
\[f(y) \ge f(x) + \nabla f(x)^\top (y-x) + \frac{1}{2L}\|\nabla f(x)-\nabla f(y)\|^2. \]
In the next subsection, we will begin the convergence analysis of the gradient descent method in a general non-convex setting.

\subsection{Convergence for non-convex objective functions}
We start the convergence analysis of the gradient descent scheme in the most general setting, where the objective function is assumed to be continuously differentiable without any other assumption.  This and the following Subsection~\ref{ssec:nonconvexsmooth} are largely based on \cite[Section~1.2.2]{DPB15}, adapted to the notation and scope of these notes. As discussed above, we cannot expect to prove convergence to a global minimum in this scenario. However, assuming that there exists an accumulation point of the sequence generated by gradient descent, we are able to characterize this point as a stationary point. 

\begin{thm}[GD with Armijo rule]\label{thm:armijo}
Let $f:\R^d\to\R$ be continuously differentiable and $(x_k)_{k\in\N}$ be generated by
\[x_{k+1} = x_k + \alpha_k d_k,\quad d_k = -\frac{\nabla f(x_k)}{\|\nabla f(x_k)\|},\]
where $\alpha_k>0$ is chosen by the Armijo step size rule Algorithm~\ref{alg:armijo}. Then every accumulation point $\bar x\in\R^d$ of the sequence $(x_k)_{k\in\N}$ is a stationary point of $f$, i.e.~$\nabla f(\bar x)=0$.
\end{thm}
\begin{proof}
For full details, we refer the interested reader to Appendix~\ref{app:omittedproofs}. The argumentation
is based on the proof in~\cite{DPB15}.
\end{proof}

If we already know that the iterations of gradient descent converge to some limit, we can show that this limit point needs to be a stationary point. This requires additional assumptions about the sequence of step sizes. 
\begin{thm}
Let $f:\R^d\to\R$ be continuously differentiable and $(x_k)_{k\in\N}$ be defined by
\[ x_{k+1} = x_k - \alpha_k \nabla f(x_k), \quad x_0\in\R^d,\]
with step size $\alpha_k>0$ such that $\sum_{k=1}^\infty \alpha_k=\infty$ (e.g.~constant or diminishing). Suppose that $(x_k)_{k\in\N}$ converges to some $x_\ast\in\R^d$. Then $x_\ast$ is a stationary point of $f$, i.e.~$\nabla f(x_\ast)= 0$.
\end{thm}
\begin{proof}
    Let $(x_k)_{k\in\N}$ converge to some $x_\ast\in\R^d$, then by continuity of $\nabla f$ also $(\nabla f(x_k))_{k\in\N}$ converges to $\nabla f(x_\ast)$. Let $\varepsilon>0$ be arbitrary. By Cauchy-Schwarz inequality, we have
    \begin{align*}
        \langle \nabla f(x_j),\nabla f(x_i)\rangle  &= \|\nabla f(x_\ast)\|^2 + \langle \nabla f(x_i)-\nabla f(x_\ast),\nabla f(x_\ast)\rangle + \langle \nabla f(x_i),\nabla f(x_j)-\nabla f(x_\ast)\rangle\\
        &\ge \|\nabla f(x_\ast)\|^2-\|\nabla f(x_i)-\nabla f(x_\ast)\| \|\nabla f(x_\ast)\| - \|\nabla f(x_i)\| \|\nabla f(x_j)-\nabla f(x_\ast)\|\,.
    \end{align*}
    Due to convergence of $(\nabla f(x_k))_{k\in\N}$ there exists $n_\varepsilon\in\N$ such that for all $i,j\ge n_\varepsilon$ it holds 
    \[ \langle \nabla f(x_j),\nabla f(x_i)\rangle \ge \|\nabla f(x_\ast)\|^2 - 2\varepsilon \|\nabla f(x_\ast)\| - \varepsilon^2=:p(\varepsilon)\,, \]
    where we have used that 
    \[ \|\nabla f(x_i)\| \le \|\nabla f(x_i)-\nabla f(x_\ast)\| + \|\nabla f(x_\ast)\| \le \varepsilon + \|\nabla f(x_\ast)\|\,. \]
    For arbitrary $n,m\ge n_\varepsilon$, this results in 
    \begin{align*}
        \|x_n-x_m\|^2 = \|\sum_{k=n}^{m-1} \alpha_k \nabla f(x_k)\|^2 &= \sum_{i,j=n}^{m-1}\alpha_i\alpha_j \langle \nabla f(x_i),\nabla f(x_j)\rangle \\
        &\ge \Big(\sum_{k=n}^{m-1} \alpha_k\Big)^2 p(\varepsilon)\,.
    \end{align*}
    Taking the limit over $m$ results in 
    \[ \infty > \|x_n-x_\ast\|^2 \ge \Big(\lim_{m\to\infty} \sum_{k=n}^m \alpha_k\Big)^2 p(\varepsilon)\,.\]
    By $\lim_{m\to\infty} \sum_{k=n}^m \alpha_k=\infty$, we necessarily have $p(\varepsilon)\le 0$. Since $\varepsilon>0$ is arbitrary, we have
    \[0\le \|\nabla f(x_\ast)\|^2 = \lim_{\varepsilon\to0} p(\varepsilon) \le 0\,.\]
\end{proof}

We conclude this section with a property of gradient descent that guarantees that the iteration gets captured by isolated local minima. This means that once the iteration moves into a sufficiently small neighborhood around an isolated local minimum, it will remain within this neighborhood and even converge to the corresponding local minimum. 

\begin{thm}[Capture Theorem]
Let $f:\R^d\to \R$ be continuously differentiable and the sequence $(x_k)_{k\in\N}$ be generated by gradient descent
\[x_{k+1} = x_k - \alpha_k \nabla f(x_k), \]
with bounded step size $\alpha_k\le s$, $s>0$. Moreover, assume that the sequence $(f(x_k))_{k\in\N}$ is decreasing and every accumulation point of $(x_k)_{k\in\N}$ is a stationary point. Let $x_\ast\in\R^d$ be an isolated local minimum of $f$, i.e.~there exists an open neighborhood $U\subset \R^d$ of $x_\ast$ such that $x_\ast$ is the only stationary point within the set $U$. Then there exists an open set $S\subset \R^d$ with $x_\ast\in S$ and the following property: If $x_{\bar k}\in S$ for some $\bar k\in\N$, then $x_k\in S$ for all $k\ge \bar k$ and in particular the sequence $(x_k)_{k\in\N}$ converges to $x_\ast$.
\end{thm}

\begin{proof}
Since $x_\ast\in\R^d$ is assumed to be an isolated local minimum,  we can find some $r>0$ such that within a closed Ball $\bar{\mathcal B}_{r}(x_\ast):= \{x\in\R^d\mid \|x-x_\ast\|\le r\}$ with radius $r$ around $x_\ast$ it holds
\[f(x_\ast) < f(x),\quad x\in\bar{\mathcal B_{r}}(x_\ast)\]
while there exists no stationary point $x'\in \bar{\mathcal B}_{r}(x_\ast)\setminus \{x_\ast\}$. It directly follows that 
\[\min_{x\in \bar{\mathcal B_{r}}(x_\ast)\setminus \mathcal B_{t}(x_\ast)} f(x) > f(x_\ast)\]
for any open ball $\mathcal B_t(x_\ast):=\{x\in\R^d\mid \|x-x_\ast\|<t\}$ with radius $t\in(0,r]$ and we define the function
\[\Phi(t) = \min_{x\in \bar{\mathcal B_{r}}(x_\ast)\setminus \mathcal B_{t}(x_\ast)} f(x) - f(x_\ast),\]
which is decreasing for decreasing $t$. Since the gradient $\nabla f$ is assumed to be continuous and $\nabla f(x_\ast)=0$, for a fixed but arbitrary $\varepsilon>0$ we can find $q\in(0,\varepsilon]$ such that
\begin{equation}\label{eq:capturethm1}
\|x-x_\ast\| + s \|\nabla f(x)\| < \varepsilon \quad \text{for all}\ x\in \mathcal B_q(x_\ast).
\end{equation}
We will use \eqref{eq:capturethm1} in order to prove that the open set $S$ defined by
\[S:=\{x\in\R^d\mid \|x-x_\ast\|<\varepsilon,\ f(x)<f(x_\ast) + \Phi(q)\} \]
captures the iteration of gradient descent, i.e.~if $x_k\in S$ then it also holds true that $x_{k+1}\in S$. Let us suppose $x_k\in S$, then for $\tau = \|x_k-x_\ast\|$ we have that $\Phi(\tau) \le f(x_k)-f(x_\ast)<\Phi(q)$ by definition of $S$. Due to decreasing behavior of $\Phi$ we obtain $\|x_k-x_\ast\|=\tau<q$ and therefore $x_k\in\mathcal B_q(x_\ast)$. By \eqref{eq:capturethm1} we obtain 
\[\|x_k-x_\ast\|+s\|\nabla f(x_k)\|<\varepsilon. \]
Since the step size satisfies $\alpha_k\le s$, applying triangular inequality we can imply that
\[\|x_{k+1}-x_\ast\|\le \|x_k-x_\ast\| + \alpha_k \|\nabla f(x_k)\| \le \|x_k-x_\ast\| + s\|\nabla f(x_k)\| < \varepsilon. \]
By assumption $(f(x_k))_{k\in\N}$ is decreasing such that
\[f(x_{k+1}) - f(x_\ast) < f(x_k) - f(x_\ast) < \Phi(q), \]
where we have used $x_k\in S$, and therefore,  we obtain $x_{k+1}\in S$ as well.  Once we find $\bar k\in \N$ such that $x_{\bar k}\in S$, we can imply $x_k\in S$ for all $k\ge \bar k$. It is left to argue that in this case $\lim_{k\to\infty} x_k = x_\ast$. Consider the closure $\bar S$ of $S$ which is a compact set. Then there exists at least one accumulation point $\bar x$ of $(x_k)_{k\in\N}$, which by assumption is a stationary point.  By construction $\bar S\subset \mathcal B_r(x_\ast)$ such that $x_\ast\in\bar S$ is the unique stationary point of $f$ in $\bar S$ implying that $\lim_{k\to\infty} x_k = x_\ast$. 
\end{proof}

\subsection{Convergence for non-convex and smooth objective functions}\label{ssec:nonconvexsmooth}
We continue with the convergence analysis of gradient descent for $L$-smooth objective functions $f$, but without additional assumptions such as convexity. The following theorem then characterizes accumulation points of the iterates of the gradient descent scheme with fixed but sufficiently small step size. 

\begin{thm}[GD with constant step size]\label{thm:GD_conststep}
Let $f:\R^d\to\R$ be $L$-smooth and $(x_k)_{k\in\N}$ be generated by
\[x_{k+1} = x_k - \bar\alpha \nabla f(x_k),\]
where $\bar\alpha\in(0,\frac{2}{L})$. Then every accumulation point $\bar x\in\R^d$ of the sequence $(x_k)_{k\in\N}$ is a stationary point of $f$, i.e.~$\nabla f(\bar x)=0$.
\end{thm}

\begin{proof}
Since $\bar\alpha\in(0,\frac2L)$ is fixed, there exists $\varepsilon\in(0,\frac{2}{L+1}]$ such that $\bar\alpha\in[\varepsilon,\frac{2-\varepsilon}{L}]$. 
Since $f$ is assumed to be $L$-smooth, we apply the descent Lemma~\ref{lem:descentlemma} (with the choice $y\equiv -\bar\alpha\nabla f(x_k)$ and $x\equiv x_k$),
\begin{equation}\label{eq:lem:descentlemma1}
f(x_k-\bar\alpha\nabla f(x_k))-f(x_k) \le (-\bar\alpha\nabla f(x_k))^\top \nabla f(x_k) + \frac{L}2\|\bar\alpha \nabla f(x_k)\|^2 = \bar\alpha \|\nabla f(x_k)\|^2 (\frac{\bar\alpha L}{2}-1).
\end{equation} 
Due to $\bar\alpha\le \frac{2-\varepsilon}{L}$ it holds
\[ \frac{\bar\alpha L}{2}-1 \le -\frac{\varepsilon}2<0.\]
The inequality \eqref{eq:lem:descentlemma1} can be reformulated to obtain
\begin{equation}\label{eq:lem:descentlemma2}
f(x_k) - f(x_k-\bar\alpha\nabla f(x_k)) \ge \frac{\varepsilon}2 \bar\alpha \|\nabla f(x_k)\|^2 \ge \frac{\varepsilon^2}2 \|\nabla f(x_k)\|^2. 
\end{equation}
Similarly to the proof of Theorem~\ref{thm:armijo}, we first assume that $(x_{k_n})_{n\in\N}$ is a sub-sequence with limit point $\bar x\in\R^d$ and $\nabla f(\bar x)\neq 0$. With \eqref{eq:lem:descentlemma2} we can imply that the sequence $(f(x_k))_{k\in\N}$ is monotonically decreasing, and therefore, converges to $f(\bar x)$ using continuity of $f$. In particular, we have
\[\lim_{k\to\infty} (f(x_{k+1})-f(x_k)) = 0. \]
However, with \eqref{eq:lem:descentlemma2}, it follows that
\[\lim_{k\to\infty} \frac{\varepsilon^2}{2}\|\nabla f(x_k)\|^2 = 0, \]
which is in contradiction to $\nabla f(\bar x) \neq 0$.
\end{proof}

A similar convergence result for gradient descent with diminishing step size choice can be formulated. In order to characterize accumulation points as stationary points, we need to force $\alpha_k$ to vanish but not too fast. 
\begin{thm}[GD with diminishing step size]\label{thm:GD_diminishingstep}
Let $f:\R^d\to\R$ be $L$-smooth and $(x_k)_{k\in\N}$ be generated by gradient descent
\[x_{k+1} = x_k - \alpha_k \nabla f(x_k),\]
with 
\[\lim_{k\to\infty} \alpha_k = 0 \quad \text{and}\quad \sum_{k=0}^\infty \alpha_k = \infty\,. \]
Then the sequence $(f(x_k))_{k\in\N}$ satisfies either
\[\lim_{k\to\infty} f(x_k) = -\infty\quad \text{or}\quad \lim_{k\to\infty} \nabla f(x_k) = 0. \]
Moreover, every accumulation point $\bar x\in\R^d$ of the sequence $(x_k)_{k\in\N}$ is a stationary point of $f$, i.e.~$\nabla f(\bar x)=0$.
\end{thm}

\begin{proof}
Similarly to the proof of Theorem~\ref{thm:GD_conststep} we first apply Lemma~\ref{lem:descentlemma} to derive
\begin{align*}
f(x_{k+1}) = f(x_k-\alpha_k\nabla f(x_k)) \le f(x_k) - (\alpha_k - \frac{L\alpha_k^2}{2}) \|\nabla f(x_k)\|^2 = f(x_k) - \alpha_k (1 - \frac{L\alpha_k}{2}) \|\nabla f(x_k)\|^2.
\end{align*}
Since we assumed $\lim_{k\to\infty} \alpha_k = 0$, there exists $k_0\ge0$ such that
\begin{equation} \label{eq:iterative_descent}
f(x_{k+1}) \le f(x_k) - \alpha_k c \|\nabla f(x_k)\|^2 
\end{equation}
for some $c>0$ and all $k\ge k_0$. Hence, the sequence $(f(x_k))_{k\ge k_0}$ is decreasing, and either $\lim_{k\to\infty} f(x_k) = -\infty$ or $\lim_{k\to\infty} f(x_k) = M$ for some $M<\infty$.
Suppose that we are in the case where $\lim_{k\to\infty} f(x_k) = M$. From \eqref{eq:iterative_descent} it follows 
\begin{equation*} 
\sum_{k=k_0}^K \alpha_k c \|\nabla f(x_k)\|^2 \le \sum_{k=k_0}^K \left\{f(x_k) - f(x_{k+1})\right\}
\end{equation*}
 for all $K>k_0$. Using that the right-hand side is a telescoping sum, we obtain  
 \[\sum_{k=k_0}^K \left\{f(x_k) - f(x_{k+1})\right\} = f(x_{k_0}) - f(x_K)\]
 and for $K\to\infty$
 \begin{equation}\label{eq:graddegeneration}
 c\sum_{k=k_0}^\infty \alpha_k \|\nabla f(x_k)\|^2 \le f(x_{k_0}) - M < \infty\,. 
 \end{equation}
 Since $\sum_{k=k_0}^\infty \alpha_k = \infty$,  there cannot exist any $\varepsilon>0$ such that $\|\nabla f(x_k)\|^2>\varepsilon$ for all $k\ge \hat k\ge 0$. However, this only means that 
 \[\liminf_{k\to\infty} \|\nabla f(x_k)\| = 0. \]
In order to prove $\lim_{k\to\infty} \|\nabla f(x_k)\| = 0$ we will use \eqref{eq:graddegeneration} to prove that $\limsup_{k\to\infty} \|\nabla f(x_k)\| = 0$. Suppose that $\limsup_{k\to\infty} \| \nabla f(x_k)\|\ge \varepsilon$ for some $\varepsilon>0$ and consider two sub-sequences $(m_j)_{j\in\N}$, $(n_j)_{j\in\N}$, $n_j,m_j\in\N$, with $m_j <n_j<m_{j+1}$ such that 
 \[\frac{\varepsilon}{3}< \|\nabla f(x_k)\|, \quad \text{for}\ m_j\le k <n_j \]
and 
\[\|\nabla f(x_k)\| \le \frac{\varepsilon}{3}, \quad \text{for}\ n_j\le k<m_{j+1}. \]
Moreover, let $\bar j\in\N$ be sufficiently large such that
\[ \sum_{k=m_{\bar j}}^\infty \alpha_k c \|\nabla f(x_k)\|^2 \le \frac{\varepsilon^2}{9L}. \]
Using $L$-smoothness for $j\ge \bar j$ and $m_j\le m  \le n_{j}-1$ it holds true that
\begin{align*}
\|\nabla f(x_{n_j})-\nabla f(x_m)\| &\le \sum_{k=m}^{n_j-1} \|\nabla f(x_{k+1})-\nabla f(x_k)\|\le L\sum_{k=m}^{n_j-1} \|x_{k+1}-x_k\|\\ &= \frac{3\varepsilon}{3\varepsilon} L \sum_{k=m}^{n_j-1} \alpha_k\|\nabla f(x_k)\|\le L\frac{3}{\varepsilon} \sum_{k=m}^{n_j-1}\alpha_k \|\nabla f(x_k)\|^2\le L\frac{3}{\varepsilon} \frac{\varepsilon^2}{9L} = \frac{\varepsilon}{3},
\end{align*}
where we have used that $\|\nabla f(x_k)\|> \frac{\varepsilon}{3}$ for $m_j \le k \le n_j-1$.  This implies that
\[\|\nabla f(x_m)\| \le \|\nabla f(x_{n_j})\| + \|\nabla f(x_{n_j})-\nabla f(x_m)\| \le \|\nabla f(x_{n_j})\| + \frac{\varepsilon}{3} \le \frac{2\varepsilon}{3}  \]
and therefore $\|\nabla f(x_m)\|\le \frac{2\varepsilon}{3}$ for all $m\ge m_{\bar j}$. This is in contradiction to $\limsup_{k\to\infty} \|\nabla f(x_k)\|\ge \varepsilon$ and we proved
\[\limsup_{k\to\infty} \|\nabla f(x_k)\| = \liminf_{k\to\infty} \|\nabla f(x_k)\| = \lim_{k\to\infty} \|\nabla f(x_k)\| = 0. \]
Finally, let $\bar x\in\R^d$ be an accumulating point of $(x_k)_{k\in\N}$. Since $(f(x_k))_{k\ge k_0}$ is decreasing, it follows by continuity that
\[\lim_{k\to\infty} f(x_k) = f(\bar x)<\infty \]
and then also
\[\nabla f(\bar x) = \lim_{k\to\infty} \nabla f(x_k) = 0. \]
\end{proof}

\begin{remark}
Under the conditions of Theorem~\ref{thm:GD_diminishingstep} there is no guarantee of decrease in the objective function $f(x_k)$ along the initial iterations $k\le k_0$, which we have only verified for $k_0$ sufficiently large.  The decrease can be forced under the additional assumption $\alpha_k < \frac{2}{L}$, which was also used in Theorem~\ref{thm:GD_conststep} for an upper bound on the constant step size $\bar \alpha$.
\end{remark}

Although in the previous theorems we have characterized accumulation points of the gradient descent method, we are not able to characterize the existence of limit points or even the convergence rate. However, under the additional assumption of lower bounded objective functions, we can quantify the speed of degeneration of the gradients along the iterations of gradient descent. The following results are a direct consequence of the proofs of Theorem~\ref{thm:GD_conststep} and Theorem~\ref{thm:GD_diminishingstep}

\begin{cor}\label{cor:gradient_convergence}
Let $f:\R^d\to\R$ be $L$-smooth and $(x_k)_{k\in\N}$ be generated by
\[x_{k+1} = x_k - \alpha_k \nabla f(x_K),\quad \alpha_k>0.\]
We define $g_K^\ast := \min_{k\in\N,\ 0\le k \le K} \|\nabla f(x_k)\|^2$ and $\bar g_K := \frac{1}{K+1}\sum_{k=0}^K\|\nabla f(x_k)\|^2$ for $K\in\N$.
\begin{enumerate}
\item[(i)] For the choice of a constant step size $\bar\alpha \in(0,\frac{2}L)$ it holds $g_K^\ast,\bar g_K \in \mathcal O(\frac{1}{K})$.
\item[(ii)] Consider a diminishing step size $\alpha_k = \frac{C}{\sqrt{k+1}},\ k\ge0,$ for some $C>0$. Suppose that $\lim_{k\to\infty} f(x_k) = M \in(-\infty,\infty)$, then it holds $g_K^\ast \in \mathcal O(\frac{1}{\sqrt{K}})$.
\end{enumerate}
\end{cor}
\begin{proof}
    For the first claim, let $\bar\alpha \in[\varepsilon,\frac{2-\varepsilon}L]$ with $\varepsilon<\frac{2}{L+1}$. Similarly as in the proof of Theorem~\ref{thm:GD_conststep} we have that
    \[ \frac{\varepsilon^2}{2}\|\nabla f(x_k)\|^2 \le f(x_{k})-f(x_{k+1})\,.\]
    Taking a sum from $k=0,\dots, K$ for $K\in\N$, we deduce that
    \[ \frac{\varepsilon^2}{2}\sum_{k=0}^K \|\nabla f(x_k)\|^2 \le f(x_0)-f(x_{K+1})\le f(x_0)-f^\ast\,.\]
    This implies the first assertion
    \[g_K^\ast := \min_{k\in\N, 0\le k \le K} \|\nabla f(x_k)\|^2 \le \bar g_K = \frac1{K+1} \sum_{k=0}^{K} \|\nabla f(x_k)\|^2 \le \frac{2(f(x_0)-f^\ast)}{\varepsilon^2(K+1)}\,.\]

    For the second claim, we assume that the sequence of step sizes $\alpha_k$ is decreasing such that $\alpha_k\ge \alpha_K$ for all $0\le k\le K$. Similarly as in the proof of Theorem~\ref{thm:GD_diminishingstep}, we start with \eqref{eq:graddegeneration} and for simplicity we suppose that $k_0 = 0$. Hence, we have
    \[\alpha_K \min_{k\in\N,\ 0\le k\le K}\ \|\nabla f(x_k)\|^2 \le \min_{k\in\N,\ 0\le k\le K}\ \alpha_k\|\nabla f(x_k)\|^2 \le \frac{1}{K+1}\sum_{k=0}^K \alpha_k\|\nabla f(x_k)\|^2 \le \frac{f(x_0)-M}{c(K+1)}\,.\]
    Setting $\alpha_k = \frac{C}{\sqrt{k+1}}$ for some $C>0$ implies that
    \[g_K^\ast \le \frac{(f(x_0)-M)\sqrt{K+1}}{C\,c\,(K+1)}\in\cO(\frac{1}{\sqrt{K}})\,.\]
\end{proof}

\subsection{Convergence for convex and smooth objective functions}\label{sec:GD_convex}
In the following, we study the convergence behavior of gradient descent under the additional assumption of (strong) convex objective functions. While in the previous section we have quantified possible accumulation points, we will now consider the description of convergence through some error function.  Let $(x_k)_{k\in\N}$ be the sequence generated through some optimization scheme for a objective function $f$.  We consider an error function $e:\R^d\to\R$ with the property $e(x)\ge0$ for all $x\in\R^d$ and,  $e(x_\ast) = 0$ for some $x_\ast\in\R^d$, e.g.~$x_\ast \in\arg\min_{x\in\R} f(x)$ assuming it exists.  Typical examples include
\[e(x) = \|x-x_\ast\| \quad \text{or}\quad e(x) = |f(x)-f(x_\ast)|. \]
We define the following type of convergence behavior.

\begin{defi}
We say that the sequence of errors $(e(x_k))_{k\in\N}$ converges
linearly, if there exists $c\in(0,1)$ such that \[e(x_{k+1}) \le c e(x_k) \quad \text{for all}\ k\in\N\,.\]
\end{defi}
Since the focus of this lecture course lies in first-order methods, we do not expect faster convergence than linear such as super-linear or quadratic convergence.  

We will now study the convergence of gradient descent for general convex and $L$-smooth objective functions.  Since for general convex functions there is no guarantee for existence of a unique global minimum, we only expect convergence of the error function $e(x_k) = f(x_k)-f(x_\ast)$ for some global minimum $x_\ast\in\R^d$. The convergence is slower than linear, sometimes also referred to sub-linear convergence. 
\begin{thm}[GD for convex and smooth objective function]\label{thm:GD_convex}
Let $f:\R^d\to\R$ be convex and $L$-smooth, and let $(x_k)_{k\in\N}$ be generated by 
\[x_{k+1} = x_k - \bar\alpha \nabla f(x_k) \]
with 
$\bar\alpha\le \frac{1}{L}$. Moreover, we assume that the set of all global minima of $f$ is non-empty. Then the sequence $(x_k)_{k\in\N}$ converges in the sense that
\[e(x_k):=f(x_k)-f^\ast \le \frac{c}{k},\quad k\ge 1\]
for some constant $c>0$ and $f^\ast = \min_{x\in\R^d} f(x)$. 
\end{thm}
\begin{proof}
We again apply the descent Lemma~\ref{lem:descentlemma} (with $t=1$, $y=(x_{k+1}-x_k)$ and $x=x_k$) to derive
\[f(x_{k+1})\le f(x_k) + \nabla f(x)^\top (x_{k+1}-x_k) + \frac{L}{2} \|x_{k+1}-x_k\|^2. \]
With $x_{k+1}-x_k=-\bar\alpha \nabla f(x_k)$ we obtain
\begin{align*}
f(x_{k+1}) &\le f(x_k) - \bar\alpha \|\nabla f(x_k)\|^2 + \frac{L}{2}\bar\alpha^2 \|\nabla f(x_k)\|^2\\
&= f(x_k) +(\frac{L}{2}\bar\alpha - 1)\bar \alpha \|\nabla f(x_k)\|^2.
\end{align*}
Since $\bar\alpha\le \frac{1}{L}$, we have $(\frac{L}{2}\bar\alpha - 1)\le -\frac{1}{2}<0$ and therefore, the sequence $(f(x_k))_{k\in\N}$ is decreasing. Now, let $x_\ast\in\R^d$ be some global minimum of $f$ such that due to convexity it holds true that
\[f(x_k) + (x_\ast-x_k)^\top \nabla f(x_k) \le f(x_\ast). \]
We plug this in into the above inequality to imply
\begin{align*}
f(x_{k+1}) &\le f(x_k) + \bar \alpha(\frac{L}{2}\bar\alpha - 1) \|\nabla f(x_k)\|^2\\
&\le f(x_\ast) - \frac{\bar\alpha}{\bar\alpha} (x_\ast-x_k)^\top \nabla f(x_k) + \bar \alpha(\frac{L}{2}\bar\alpha - 1) \|\nabla f(x_k)\|^2\\
&=f(x_\ast) + \frac{1}{\bar\alpha}\left\{\frac{1}{2}\|x_\ast-x_k\|^2 + \frac{\bar\alpha^2}{2}\|\nabla f(x_k)\|^2 - \frac{1}{2}\|(x_\ast-x_k)+ \bar\alpha\nabla f(x_k)\|^2 \right\}\\&\quad + \bar \alpha(\frac{L}{2}\bar\alpha - 1) \|\nabla f(x_k)\|^2,
\end{align*}
where we have used $-a^\top b = \frac12 \|a\|^2 + \frac12\|b\|^2 - \frac12\|a+b\|^2$ for $a,b\in\R^d$. Rearranging the rhs leads to
\begin{align*}
f(x_{k+1}) &\le f(x_\ast) + \frac{1}{2\bar \alpha} \left(\|x_\ast-x_k\|^2 - \|x_\ast - x_{k+1}\|^2\right) + \bar\alpha(\frac{L}{2}\bar\alpha-\frac12) \|\nabla f(x_k)\|^2 \\ &\le f(x_\ast) + \frac{1}{2\bar \alpha} \left(\|x_\ast-x_k\|^2 - \|x_\ast - x_{k+1}\|^2\right),
\end{align*}
where we have used again $\bar\alpha\le\frac{1}{L}$. Summing over $k=0,\dots,K$ gives
\[\sum_{k=0}^K \{f(x_{k+1}) - f(x_\ast)\} \le \frac{1}{2\bar\alpha}\sum_{k=0}^{K} \{\|x_\ast-x_k\|^2 - \|x_\ast-x_{k+1}\|^2\}\le \frac{1}{2\bar\alpha} \{\|x_\ast-x_0\|^2-\|x_\ast-x_{K+1}\|^2\},\]
where we have applied a telescoping sum.
Since $(f(x_k))_{k\in\mathbb N}$ is decreasing, we have
\[\sum_{k=0}^K f(x_{k+1})\ge (K+1) f(x_{K+1})\,.\]
Therefore, the assertion follows with
\[f(x_{K+1})-f(x_\ast) \le \frac{1}{K+1} \frac{1}{2\bar\alpha} \|x_\ast-x_0\|^2 =: \frac{c}{K+1}. \]
\end{proof}

\begin{remark}[Larger constant step sizes]
The condition $\bar \alpha\leq \frac{1}{L}$ in Theorem~\ref{thm:GD_convex} is convenient for the proof, but it is not the largest admissible range for convergence of gradient descent. In fact, the conclusion remains true for every constant step size
\[
0<\bar \alpha<\frac{2}{L}.
\]
However, the proof is slightly more involved. The key step in the proof is
\[
f(x_{k+1})-f(x_\ast)
\leq
\frac{1}{2\bar \alpha}
\left(
\|x_k-x_\ast\|^2-\|x_{k+1}-x_\ast\|^2
\right)
+
\frac{L\bar \alpha-1}{2\bar \alpha}
\|x_{k+1}-x_k\|^2\,.
\]
For $\bar \alpha\leq 1/L$, the last term is non-positive and can simply be discarded. For
$\bar \alpha>1/L$, this term is positive, but it can still be controlled because the descent estimate gives
\[
f(x_k)-f(x_{k+1})
\geq
\left(\frac{1}{\bar \alpha}-\frac{L}{2}\right)
\|x_{k+1}-x_k\|^2\,.
\]
The coefficient on the right-hand side is positive when $\bar \alpha<2/L$.
Consequently, one still obtains the last-iterate rate
\[
f(x_k)-f(x_\ast)=\mathcal{O}\left(\frac{1}{k}\right),
\]
but with a constant that deteriorates as $\bar \alpha$ approaches $2/L$. 
\end{remark}

\subsection{Convergence for strongly convex and smooth objective functions}\label{sec:GD_strongconvex}
Under the additional assumption that $f$ is $\mu$-strongly convex for some $\mu>0$, gradient
descent converges linearly for sufficiently small constant step sizes. In this case, convergence can
be shown not only for the objective values, but also for the iterates themselves.
\begin{thm}[GD for strongly convex and smooth objective function]\label{thm:GD_strongconvex}
Let $f:\R^d\to\R$ be $\mu$-strongly convex and $L$-smooth, and let $(x_k)_{k\in\N}$ be generated by 
\[x_{k+1} = x_k - \bar\alpha \nabla f(x_k) \]
with 
$\bar\alpha\le \frac{1}{L}$. Then the sequence $(x_k)_{k\in\N}$ converges linearly in the sense that there exists $c\in(0,1)$ such that
\[e(x_k):=\|x_k-x_\ast\| \le c^k\|x_0-x_\ast\|,\quad k\in\N\]
where $x_\ast\in\R^d$ is the unique global minimum of $f$ with $f(x_\ast) = \min_{x\in\R^d} f(x)$. 
\end{thm}
\begin{proof}
Let $x_\ast\in\R^d$ be the unique global minimum of $f$ with $\nabla f(x_\ast)=0$. Since $f$ is assumed to be $\mu$-strongly convex,  by Definition~\ref{def:strong_convexity} it holds true that 
\[\frac{\mu}{2}\|x_{k+1}-x_\ast\|^2 = \nabla f(x_\ast)^\top (x_{k+1}-x_\ast) + \frac{\mu}{2}\|x_{k+1}-x_\ast\|^2 \le f(x_{k+1})-f(x_\ast).\]
On the other hand, the proof of Theorem~\ref{thm:GD_convex} yields
\[f(x_{k+1})-f(x_\ast) \le \frac{1}{2\bar\alpha} \{ \|x_k-x_\ast\|^2-\|x_{k+1}-x_\ast\|^2\}\,. \]
Combining the two estimates gives
\[(\frac{\mu}{2} + \frac{1}{2\bar\alpha}) \|x_{k+1}-x_\ast\|^2 \le \frac{1}{2\bar\alpha} \|x_k-x_\ast\|^2.\]
Iterating this inequality gives
\[
\|x_k-x_\ast\|
\leq
\left(\frac{1}{\sqrt{1+\mu\bar\alpha}}\right)^k
\|x_0-x_\ast\|\,.
\]
\end{proof}
\begin{remark}
For the choice $\bar\alpha = \frac1L$ the upper bound of gradient descent is given by
\[\|x_k-x_\ast\|\le \left(\sqrt{\frac{L}{L+\mu}}\right)^k \|x_0-x_\ast\|. \]
Thus, the convergence speed is governed by the ratio between the smoothness constant $L$ and
the strong convexity constant $\mu$. The quantity
$
\kappa:=\frac{L}{\mu}
$
is called the condition number of the problem. A large condition number corresponds to an
ill-conditioned problem and leads to slower convergence. The rate of convergence can be described as:
\[c= \sqrt{\frac{1}{1+\frac{\mu}L}} = \sqrt{\frac{\kappa}{\kappa+1}}\,, \]
which decreases for decreasing $L$ and increasing $\mu$.
\end{remark}

The proof of the previous Theorem~\ref{thm:GD_strongconvex} for convergence under strong convexity builds directly upon the proof of Theorem~\ref{thm:GD_convex} under convexity.  However, we can even improve the rate of convergence if we do not go the direct way from convex to strongly convex. Therefore, we consider the following improved convergence result \cite[Theorem~2.1.15]{N2018}.

\begin{thm}\label{thm:GD_strongconvex2}
Assume that the same conditions as in Theorem~\ref{thm:GD_strongconvex} are satisfied, and additionally that $\bar\alpha\le \min(\frac{2}{\mu+L},\frac{\mu+L}{2\mu L})$. Then the sequence $(x_k)_{k\in\N}$ generated by gradient descent
\[x_{k+1} = x_k - \bar\alpha \nabla f(x_k) \]
 converges linearly in the sense that 
\[e(x_k):=\|x_k-x_\ast\| \le \Big(1-2\bar\alpha \frac{\mu L}{\mu+L}\Big)^{\frac{k}{2}}\|x_0-x_\ast\|,\quad k\in\N\]
where $x_\ast\in\R^d$ is the unique global minimum of $f$ with $f(x_\ast) = \min_{x\in\R^d} f(x)$. 
\end{thm}

\begin{proof}
Let $x_\ast\in\R^d$ be the unique global minimizer of $f$, i.e.~$\nabla f(x_\ast) = 0$.For $k\in\N$ and $x_k\in\R^d$ it holds true that
\begin{align*}
\|x_{k+1}-x_\ast\|^2 = \|x_{k}-\bar \alpha \nabla f(x_k)-x_\ast\|^2 = \|x_k-x_\ast\|^2 - 2\langle x_k-x_\ast,\bar\alpha \nabla f(x_k)\rangle + \bar\alpha^2\|\nabla f(x_k)\|^2.
\end{align*}  
We will make use of the inequality
\[\langle \nabla f(x)-\nabla f(y),x-y\rangle \ge \frac{\mu L}{\mu+L} \|x-y\|^2 + \frac{1}{\mu+L}\|\nabla f(x)-\nabla f(y)\|^2\]
for any $x,y\in\R^d$, which is left as exercise in Lemma~\ref{lem:auxiliary_bound}.  It then follows that
\[\langle \nabla f(x_k), x_k-x_\ast\rangle \ge \frac{\mu L}{\mu+L} \|x_k-x_\ast\|^2 + \frac{1}{\mu+L} \|\nabla f(x_k)\|^2, \]
since $\nabla f(x_\ast) = 0$. We obtain the bound
\begin{align*}
\|x_{k+1}-x_\ast \|^2 &\le \|x_k-x_\ast\|^2 - 2\bar\alpha (\frac{\mu L}{\mu+L} \|x_k - x_\ast\|^2 + \frac{1}{\mu+L} \|\nabla f(x_k)\|^2) + \bar\alpha^2 \|\nabla f(x_k) \|^2\\
&=(1-2\bar\alpha \frac{\mu L}{\mu+L})\|x_k-x_\ast\|^2 + \bar\alpha (\bar\alpha-\frac{2}{\mu+L})\|\nabla f(x_k)\|^2\\
&\le (1-2\bar\alpha \frac{\mu L}{\mu+L})\|x_k-x_\ast\|^2,
\end{align*}
where we have used that $(\bar\alpha-\frac{2}{\mu+L})\le 0$.  Since we have assumed that $\bar\alpha \le \frac{\mu+L}{2\mu L}$, we finally obtain linear convergence with $c=\sqrt{1-2\bar\alpha \frac{\mu L}{\mu+L}}\in(0,1)$ in the sense that
\[\|x_k-x_\ast\| \le c^k \|x_0-x_\ast\|. \]
\end{proof}
In the proof of Theorem~\ref{thm:GD_strongconvex2} we have used the following auxiliary bound for smooth and strongly convex objective functions. The proof is left as an exercise, and for more details, we refer to \cite[Theorem~2.1.12]{N2018}.
\begin{lemma}\label{lem:auxiliary_bound}
Let $f:\R^d\to\R$ be $L$-smooth and $\mu$-strongly convex. Then it holds true that
\[\langle \nabla f(x)-\nabla f(y), x-y\rangle \ge \mu\|x-y\|^2 \]
and 
\[\langle \nabla f(x)-\nabla f(y),x-y\rangle \ge \frac{\mu L}{\mu+L} \|x-y\|^2 + \frac{1}{\mu+L} \|\nabla f(x)- \nabla f(y)\|^2. \]
\end{lemma}

\begin{remark}\label{rem:opt_step}
The motivation behind the alternative convergence result for gradient descent under $\mu$-strong convexity and $L$-smoothness is the following. The optimal convergence rate is in the sense of minimizing $c^2(\bar\alpha)=(1-2\bar\alpha \frac{\mu L}{\mu+L}) \in(0,1)$ is obtained for the choice $\bar\alpha = \frac{2}{\mu+L}$. The upper bound of the error of gradient descent is then given by
\[\|x_k-x_\ast\| \le c^{k} \|x_0-x_\ast\| \]
with 
\[c = \sqrt{\frac{(\mu+L)^2}{(\mu+L)^2} - \frac{4\mu L}{(\mu+L)^2}} = \frac{L-\mu}{L+\mu} = \frac{\kappa-1}{\kappa+1}, \]
where $\kappa:= \frac{L}{\mu}$ denotes the ratio between smoothness and strong convexity. We sometimes also refer to $\kappa$ as the condition number of $f$. Usually, we have $L\ge\mu$,  such that $\kappa\ge1$. Finally, we can rewrite $c$ through
\[c = \frac{\kappa-1}{\kappa+1} = 1-\frac{2}{\kappa+1}, \]
which again decreases for decreasing $L$ and increasing $\mu$. As will turn out in Example~\ref{ex:quadratic}, this rate matches the optimal convergence rate of gradient descent on quadratic objective functions and hence cannot be further improved. 
\end{remark}

\subsection{Convergence under PL-condition and smooth objective functions}
In the previous subsection, strong convexity allowed us to prove linear convergence of gradient descent. However, strong convexity is often too restrictive, especially in non-convex optimization problems. The key property used for linear convergence is not convexity itself, but the fact that
the gradient norm controls the objective value gap. This motivates the Polyak--{\L}ojasiewicz condition:
\begin{equation}\label{eq:PL}
\|\nabla f(x)\|^2 \ge 2r (f(x)-f^\ast)
\end{equation}
for some $r\in(0,L)$ and all $x\in\R^d$ with $f^\ast = \min_{x\in\R^d} f(x)>-\infty$. This condition can be used to prove linear convergence of gradient descent (see, for instance, \cite{10.1007/978-3-319-46128-1_50}) and can be seen as a natural relaxation of strong convexity: instead of requiring a
global quadratic lower bound on the function itself, it only requires that the gradient norm controls
the distance of the function value to the optimum. It is important to note that the PL condition does not imply convexity. This makes the condition particularly useful in
non-convex settings, where strong convexity is too restrictive but where one still wants to rule out
flat non-optimal regions and obtain global convergence rates. 

\begin{prop}\label{prop:strongconvex_PL}
Let $f:\R^d\to\R$ be differentiable and $\mu$-strongly convex, and assume that $f$ has a
minimizer $x_\ast$. Then $f$ satisfies the PL condition with parameter $\mu$.
\end{prop}

\begin{proof}
By strong convexity,
\[
f(y)\geq f(x)+\nabla f(x)^\top(y-x)+\frac{\mu}{2}\|y-x\|^2
\]
for all $x,y\in\R^d$. Minimizing the right-hand side with respect to $y$ gives
\[
f^\ast
\geq
f(x)-\frac{1}{2\mu}\|\nabla f(x)\|^2\,.
\]
Rearranging yields
\[
\frac12\|\nabla f(x)\|^2\geq \mu(f(x)-f^\ast)\,.
\]
\end{proof}
As mentioned above, the gradient descent scheme converges linearly under the PL condition.
\begin{thm}[GD under the PL condition]\label{thm:GDPL}
Let $f:\R^d\to\R$ be $L$-smooth and assume that $f$ satisfies the PL condition with parameter
$\mu>0$. Let $(x_k)_{k\in\N}$ be generated by gradient descent
\[
x_{k+1}=x_k-\bar\alpha\nabla f(x_k)
\]
with $0<\bar\alpha\leq 1/L$. Then
\[
f(x_k)-f^\ast
\leq
(1-\bar\alpha\mu)^k\big(f(x_0)-f^\ast\big).
\]
In particular, gradient descent converges linearly in objective value.
\end{thm}
\begin{proof}
By the descent condition \eqref{eq:descent_condition},
\[
f(x_{k+1})
\leq
f(x_k)-\bar\alpha\left(1-\frac{L\bar\alpha}{2}\right)\|\nabla f(x_k)\|^2.
\]
Since $\bar\alpha\leq 1/L$, we have
\[
1-\frac{L\bar\alpha}{2}\geq \frac12.
\]
Thus,
\[
f(x_{k+1})-f^\ast
\leq
f(x_k)-f^\ast-\frac{\bar\alpha}{2}\|\nabla f(x_k)\|^2.
\]
Using the PL condition yields
\[
f(x_{k+1})-f^\ast
\leq
(1-\bar\alpha\mu)(f(x_k)-f^\ast).
\]
Iterating this inequality proves the claim.
\end{proof}

The PL condition above is the strongest case in a broader family of gradient domination inequalities of the form
\[
\|\nabla f(x)\|\geq 2r(f(x)-f^\ast)^\beta,
\qquad \beta\in[\tfrac12,1].
\]
The PL condition corresponds to $\beta=\frac12$. The endpoint $\beta=1$ gives the weaker
condition
\[
\|\nabla f(x)\|\geq 2r(f(x)-f^\ast).
\]
In contrast to the PL case, this condition generally leads only to sublinear convergence rates for
gradient descent. It can be viewed as a limiting case of a (global) {\L}ojasiewicz-type gradient inequality.
\begin{thm}\label{thm:GDwPL}
Let $f:\R^d\to\R$ be $L$-smooth and satisfies the "weak" PL condition 
\begin{equation}\label{eq:wPL}
\|\nabla f(x)\| \ge 2r (f(x)-f^\ast)
\end{equation}
for some $r\in(0,L)$ and all $x\in\R^d$ with $f^\ast = \min_{x\in\R^d} f(x)>-\infty$. Then the sequence $(x_k)_{k\in\N}$ generated by gradient descent
\[x_{k+1} = x_k - \bar\alpha \nabla f(x_k) \]
with $\bar\alpha = \frac1L$ converges with
\[ e(x_k):= f(x_k)-f^\ast\le \frac{L}{2r^2(k+1)}.\]
\end{thm}
\begin{proof}
    First, under smoothness and \eqref{eq:wPL} it is straightforward to show that $f$ remains bounded. By the same argument as in the proof of Lemma~\ref{lem:convex_and_smooth} we have 
    \[ f^\ast \le f(x) -\frac1{2L} \|\nabla f(x)\|^2\]
    for all $x\in\R^d$. Combined with \eqref{eq:wPL}, this implies that
    \[f(x)-f^\ast \ge \frac{1}{2L}\|\nabla f(x)\|^2 \ge \frac{2r^2}{L} (f(x)-f^\ast)^2 \]
    and hence,
    \[f(x)-f^\ast\le \frac{L}{2r^2}\,.\]
    In the second step, we again apply the descent condition \eqref{eq:descent_condition} to deduce that
    \[ e(x_{k+1}) \le e(x_k) - \frac{1}{2L} \|\nabla f(x_k)\|^2 \le (1-\frac{2r^2}{L}e(x_k))e(x_k)\,.\]
    With $q:=\frac{2r^2}{L}$ we have that that $e(x_0)\in[0,1/q]$ and $(e(x_k))_{k\in\N_0}$ satisfies 
    \begin{equation}\label{eq:recursioncompact} 0\le e(x_{k+1}) \le (1-qe(x_k))(e(x_k))\,.
    \end{equation}
    It is a direct consequence that $(e(x_{k}))_{k\in\N_0}$ is decreasing, since $e(x_k)\le q$. Without loss of generality, suppose that $e(x_{k})>0$ for all $k\in\N_0$. Then, from \eqref{eq:recursioncompact} it follows that
    \[\frac{1}{e(x_{k+1})} \ge \frac{1}{e(x_k)} + q \frac{e(x_k)}{e(x_{k+1})} \ge \frac{1}{e(x_k)} + q\,.\]
    Finally, by a telescoping sum argument, it holds
    \[ \frac{1}{e(x_k)} - \frac{1}{e(x_0)} \ge nq\,, \]
    for all $k\ge1$, implying that
    \[ e(x_k) \le \frac{1}{q(k+1)} = \frac{L}{2r^2(k+1)}\,.\]
\end{proof}

\section{Sub-gradient descent method}
As the last class of objective functions in this chapter, we consider convex functions that are not
necessarily differentiable. Since the gradient may not exist, the gradient descent scheme has to be
reformulated. This leads to the \textit{sub-gradient descent method}. For further details, we refer the interested reader to~\cite[Section~3.1]{Lan2021}. 

\subsection{Sub-differential}
We begin with the definition of \textit{sub-gradients} and the corresponding \textit{sub-differential}.

\begin{defi}
Let $f:\R^d\to\R$ and $x\in\R^d$. A vector $g_x\in\R^d$ is called a \textit{sub-gradient} of $f$ at $x$ if
\begin{equation}\label{eq:def_subgrad}
f(y)\geq f(x)+g_x^\top(y-x)
\qquad \text{for all } y\in\R^d.
\end{equation}
The set of all sub-gradients of $f$ at $x$ is called the \textit{sub-differential} of $f$ at $x$ and is denoted by
$\partial f(x)$.
\end{defi}
Condition \eqref{eq:def_subgrad} is closely related to convexity. Indeed, if $f$ is differentiable and convex, then
\[f(y) \ge f(x) + \nabla f(x)^\top (y-x)\,. \]
Thus, for differentiable convex functions, the gradient $\nabla f(x)$ is a sub-gradient of $f$ at $x$.
There is a close connection between sub-gradients and convexity as also stated in the next proposition.
\begin{prop}
Let
$f:\R^d\to\R$.
\begin{enumerate}
\item If $\partial f(x)\neq \emptyset$ for all $x\in \R^d$, then $f$ is convex. 
\item If $f$ is convex, then $\partial f(x)\neq \emptyset$ for all $x\in\R^d$.
\end{enumerate}
\end{prop}
\begin{proof}
We will only prove the first assertion; for the second one we refer to Proposition~2.5 in \cite{Lan2021}. Suppose that $\partial f(x)\neq\emptyset$ and let $z_\lambda = \lambda x+(1-\lambda) y\in \R^d$ for $\lambda\in(0,1)$ and $x,y\in \R^d$. Moreover, consider an arbitrary sub-gradient $g_{z_\lambda}\in\partial f(z_\lambda)$, then by definition of the sub-gradient it holds true that
\begin{align*}
f(y) \ge f(z_\lambda) + g_{z_\lambda}^\top (y-z_\lambda)=f(\lambda x+(1-\lambda)y)+ \lambda g_{z_\lambda}^\top (y-x)
\end{align*}
and similarly
\begin{align*}
f(x) \ge f(z_\lambda) + g_{z_\lambda}^\top (x-z_\lambda)=f(\lambda x+(1-\lambda)y)+ (1-\lambda) g_{z_\lambda}^\top (x-y)-
\end{align*}
We combine both inequalities to obtain
\begin{align*} 
(1-\lambda) f(y) + \lambda f(x) &\ge (1-\lambda) f(\lambda x+(1-\lambda)y)+(1-\lambda) \lambda g_{z_\lambda}^\top (y-x)\\&\quad + \lambda f(\lambda x+ (1-\lambda)y) + \lambda (1-\lambda) g_{z_\lambda}^\top (x-y)\\ &= f((1-\lambda) y + \lambda x) 
\end{align*}
which proves convexity of $f$. 
\end{proof}

\begin{figure}[!htb]
  \centering \hspace{-3cm}\includegraphics[width=0.5\textwidth]{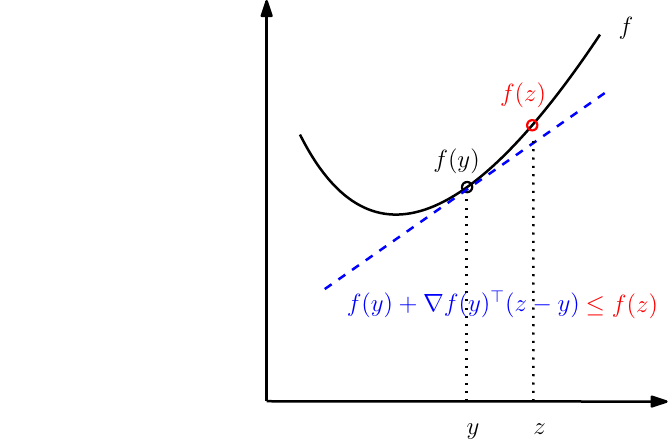}~~ \includegraphics[width=0.5\textwidth]{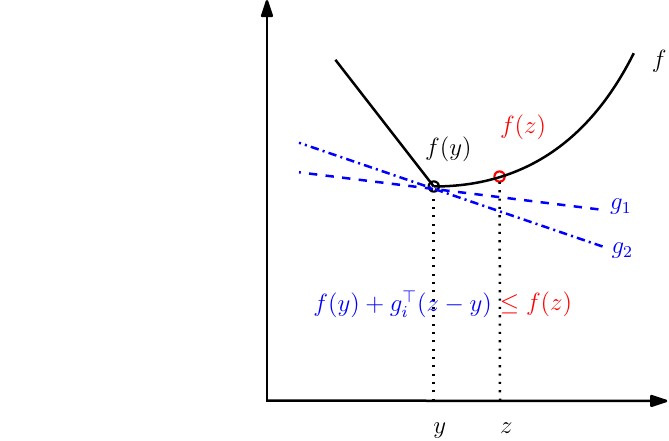}
 \caption{Illustration of sub-gradients for differentiable convex (left) and non-differentiable convex (right) functions.  For continuously differentiable and convex $f$, the sub-gradient is unique, i.e.~$\partial f(x) = \{\nabla f(x)\}$.} \label{fig:subgradient_convex}
\end{figure}

\begin{example}\label{ex:subG_max}
Let $f_i:\R^d\to\R$ be a family of convex and differentiable functions $i = 1,\dots,N$ for some $N\in\N$.  We define $F(x) = \max_{i=1,\dots,N}\ f_i(x)$ and for given $x\in\R^d$ we consider $j \in\arg\max_{i=1,\dots,N}\ f_i(x)$. Then we can compute a sub-gradient of $F$ in $x$ through $\nabla f_j(x)$, i.e.~it holds true that $\nabla f_j(x)\in \partial F(x)$. To prove this, we observe that by convexity we have 
\[f_j(y)\ge f_j(x) + \nabla f_j(x)^\top (y-x) \]
for all $y\in\R^d$. This implies
\[F(y) \ge f_j(y) \ge f_j(x) + \nabla f_j(x)^\top(y-x) = F(x) + \nabla f_j(x)^\top (y-x),\]
where we have used that $F(x) = f_j(x)$ and $F(y)\ge f_j(y)$ for $y\neq x$. This proves that $\nabla f_j(x) \in \partial F(x)$.
\begin{figure}[!htb]
  \centering \hspace{-3cm} \includegraphics[width=0.5\textwidth]{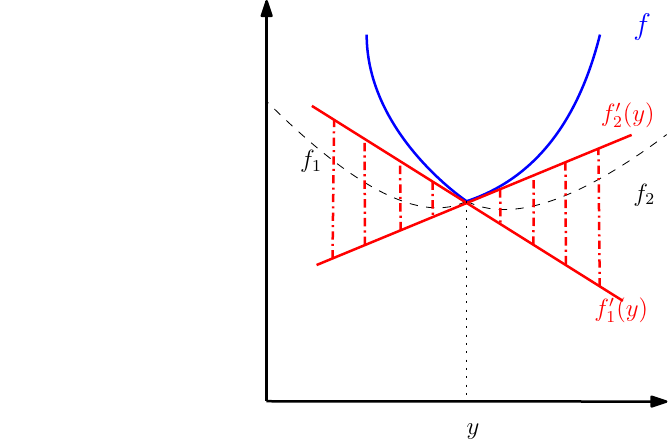}
 \caption{Illustration of Example~\ref{ex:subG_max}. We consider $N=2$ and define $f(x) = \max(f_1(x),f_2(x))$ for two convex and differentiable functions $f_1,f_2$. In the above Situation, we find unique sub-gradients for $f_1(x)>f_2(x)$ given by $\partial f(x) = \{f_1'(x)\}$, and similarly for $f_1(x)<f_2(x)$ given by $\partial f(x) = \{f_2'(x)\}$. In the case of $f_1(x)=f_2(x)$, the sub-differential is given by $\partial f(x) = [f_2'(x), f_1'(x)]$.} \label{fig:subgradient_maximum}
\end{figure} 
\end{example}

There are similar rules for the computation of sub-gradients, which are left as an exercise:

\begin{exercise}
Let $f,f_1,f_2:\R^d\to\R$ be a convex functions. Then the following holds true:
\begin{itemize}
\item Prove that $\partial f(x)$ is a convex set for all $x\in\R^d$.
\item Prove for $a>0$ that $\partial (af)(x) = a\partial f(x)$.
\item Prove that $\partial f_1(x) + \partial f_2(x) \subset \partial (f_1+f_2)(x)$ for any $x\in\R^d$.
\item Let $h(x) = f(Ax+b)$ for $A\in\R^{d\times d}$, $b\in\R^d$. Prove that $A^\top\partial f(Ax+b)\subset \partial h(x)$.  Prove equality for invertible $A$.
\end{itemize}
\end{exercise}

In case of continuously differentiable functions, the sub-gradient is unique and corresponds to the gradient. 
\begin{prop}
Let $f:\R^d\to\R$ be continuously differentiable and convex in $x\in\R^d$. Then the sub-differential is a one-point set $\partial f(x) = \{\nabla f(x)\}$.
\end{prop}
\begin{proof}
Firstly, it is obvious to see that $\nabla f(x)\in\partial f(x)$, since $f$ is convex and it holds
\[f(y)\ge f(x) + \nabla f(x)^\top(y-x)\]
for all $y\in\R^d$.
Let us consider any $g_x\in\partial f(x)$. We will prove that it necessarily follows that $g_x=\nabla f(x)$.  Let $y = x+\lambda z$ for $\lambda>0$, such that 
\[f(x+\lambda z) \ge f(x)+g_x^\top (\lambda z)\]
or rewritten
\[\frac{f(x+\lambda z) - f(x)}{\lambda} \ge g_x^\top z. \]
Since $f$ is assumed to be continuously differentiable, we obtain taking the limit $\lambda\to0$
\[\lim_{\lambda\to0} \frac{f(x+\lambda z) - f(x)}{\lambda} = \nabla f(x)^\top z \ge g_x^\top z, \]
which implies that
\[(\nabla f(x) - g_x)^\top z \ge 0. \]
Since $z\in\R^d$ is arbitrary, we can choose $z=-(\nabla f(x)-g_x)$ in order to prove that
\[-(\nabla f(x)-g_x)^\top (\nabla f(x) - g_x) = - \|\nabla f(x)-g_x\|^2\ge0 \]
which proves that $\nabla f(x) = g_x$. 
\end{proof}

We now formulate an additional optimality condition for non-differentiable but convex objective functions, which can be characterized through the sub-differential.
\begin{prop}\label{prop:opti_nondiff}
Let $f:\R^d\to\R$ be convex. Then $x_\ast\in\R^d$ is a global minimum of $f$ if and only if $0\in\partial f(x_\ast)$.
\end{prop}
\begin{proof}
We start with $x_\ast\in\R^d$ being a global minimum of $f$, i.e.~for all $y\in\R^d$ we have $f(x_\ast)\le f(y)$. Therefore, we directly obtain $0\in\partial f(x_\ast)$, since 
\[f(y) \ge f(x_\ast)=f(x_\ast) + 0^\top (y-x_\ast). \]
Now let $0\in\partial f(x_\ast)$ for some $x_\ast\in\R^d$, then by definition of the sub-differential we have
\[f(y) \ge f(x_\ast) + 0^\top (y-x_\ast) = f(x_\ast) \]
for all $y\in\R^d$ and it follows that $x_\ast$ is a global minimum of $f$.
\end{proof}

We can apply the previous stated optimality condition for solving the next exercise:
\begin{exercise}
Let
\[
f(x)=\frac12\|x-y\|^2+\lambda\|x\|_1,\qquad x\in\R^d,
\]
with $\lambda>0$. This is a simple Lasso-type objective.  Note that $\|\cdot\|_1$ denotes the $1$-norm defined as $\|x\|_1 = \sum_{i=1}^d|x_i|$ for $x=(x_1,\dots,x_d)^\top\in\R^d$. 
\begin{enumerate}
\item Compute a sub-gradient of $f$.
\item Prove that $f$ is a convex function.
\item Apply Proposition~\ref{prop:opti_nondiff} to find a global minimum of $f$.
\end{enumerate}
\end{exercise}

We now formulate a gradient-descent-type method for non-differentiable convex objective functions, called the sub-gradient descent method. Instead of moving in the direction of the negative gradient, the method chooses an arbitrary sub-gradient $g_{x_k}\in\partial f(x_k)$ and moves in the direction $-g_{x_k}$, see Algorithm~\ref{alg:SubGD}. Note that, despite its name, sub-gradient descent is not necessarily a descent method in the sense of Definition~\eqref{defi:descentdirection}.

\begin{algorithm}[htb!]
\begin{algorithmic}[1]
\State \textbf{Input:} \begin{itemize}
 \item objective function $f:\R^d\to\R$
 \item initial $x_0\in\R^d$
 \item sequence of step sizes $(\alpha_k)_{k\in\N}$, $\alpha_k>0$
 \end{itemize}
 \State set $k=0$
\While{"convergence/stopping criterion not met"}
	\State choose a sub-gradient $g_{x_k}\in\partial f(x_k)$
	\State set $x_{k+1} = x_k -\alpha_k g_{x_k}$, $k \mapsto k+1$
\EndWhile
\end{algorithmic}
 \caption{Sub-gradient descent method}\label{alg:SubGD}
\end{algorithm}

\subsection{Convergence for convex and non-smooth objective functions}
Since we do not assume differentiability of $f$ and in particular, we do not assume $L$-smoothness of $f$, we are not able to directly apply the descent Lemma~\ref{lem:descentlemma}.

Before going into details of the proof of convergence for sub-gradient descent methods, we derive the following useful property. 
\begin{lemma}\label{lem:bounded_subG}
Let $f:\R^d\to\R$ be convex and $M$-Lipschitz continuous, i.e.~
\[ |f(x)-f(y)|\le M\|x-y\|\]
for all $x,y\in\R^d$. Then for all $x\in\R^d$ every sub-gradient $g_x\in\partial f(x)$ is uniformly bounded by
$\|g_x\|\le M$.
\end{lemma}
\begin{proof}
Let $g_x\in\partial f(x)\setminus\{0\}$ for any $x\in\R^d$. Then by definition of the sub-gradient it follows that 
\[f(x+z) \ge f(x)+g_x^\top z\]
for any $z\in\R^d$. We can reformulate the inequality such that
\[g_x^\top z \le f(x+z)-f(x) \le |f(x+z)-f(x)|\le M\|z\|. \]
Since $z\in\R^d$ is arbitrary, we set $z=g_x$ implying that
\[g_x^\top g_x = \|g_x\|^2 \le M\|g_x\|\]
and therefore $\|g_x\|\le M$.
\end{proof}

We are now ready to formulate the convergence of the sub-gradient descent method. 
\begin{thm}[Convergence of sub-gradient descent]\label{thm:subGD}
Let $f:\R^d\to\R$ be convex and $M$-Lipschitz continuous, i.e.~~
\[|f(x)-f(y)|\le M\|x-y\| \]
for all $x,y\in\R^d$, and assume that $f$ has a global minimizer $x^\ast\in\R^d$.  Moreover, let $(x_k)_{k\in\N}$ be generated by sub-gradient descent
\[x_{k+1} = x_k - \alpha_k g_{x_k}, \]
with $\alpha_k>0$ and an arbitrary sub-gradient $g_{x_k}\in\partial f(x_k)$. Then it holds
\[e(\bar x_K) = f(\bar x_K)-f(x_\ast)\le \frac{\|x_0-x_\ast\|^2+M^2 \sum_{k=0}^K\alpha_k^2}{2\sum_{k=0}^K\alpha_k}, \]
where $\bar x_K:=\sum_{k=0}^K w_k x_k$ is the weighted average over all iterations with weights \[w_k = \frac{\alpha_k}{\sum_{s=0}^K\alpha_s},\quad  k=0,\dots,K\,.\]
\end{thm}
\begin{proof}
Following the iteration of $(x_k)_{k\in\N}$, we have
\begin{align*}
\|x_{k+1}-x_\ast\|^2 = \|x_k-\alpha_kg_{x_k}-x_\ast\|^2 &= \|x_k-x_\ast\|^2 -2\alpha_k\langle g_{x_k},x_k-x_\ast\rangle + \alpha_k^2 \|g_{x_k}\|^2\\
&\le  \|x_k-x_\ast\|^2 - 2\alpha_k (f(x_k)-f(x_\ast)) + \alpha_k^2 M^2,
\end{align*}
where we have used that $g_{x_k}$ is a sub-gradient of $f$ and $\|g_{x_k}\|^2\le M^2$ by Lemma~\ref{lem:bounded_subG}. Rearranging the above inequality and summing over $k=0,\dots,K$ gives
\begin{align*}
2\sum_{k=0}^K\alpha_k (f(x_k)-f(x_\ast)) &\le \sum_{k=0}^K \left(\|x_k-x_\ast\|^2-\|x_{k+1}-x_\ast\|^2 + \alpha_k^2 M^2\right)\\ 
&= \|x_0-x_\ast\|^2 - \|x_{K+1}-x_\ast\|^2 + M^2\sum_{k=0}^K \alpha_k^2\\
&\le \|x_0-x_\ast\|^2+ M^2\sum_{k=0}^K \alpha_k^2,
\end{align*}
where we have used that the first two terms are a telescoping sum. Using convexity of $f$ and Jensen's inequality, Proposition~\ref{prop:Jensen}, we obtain
\begin{align*}
f(\bar x_K)-f(x_\ast) \le \sum_{k=0}^N w_k \left(f(x_K)-f(x_\ast)\right)\le \frac{\|x_0-x_\ast\|^2 + M^2\sum_{k=0}^N \alpha_k^2 }{2\sum_{k=0}^K\alpha_k}\,,
\end{align*}
where $\bar x_K:=\sum_{k=0}^K w_k x_k$ with $w_k=\frac{\alpha_k}{\sum_{s=0}^K\alpha_k}\in(0,1)$ and $\sum_{k=0}^K w_k=1$.
\end{proof}

\begin{remark}
If the step sizes satisfy
\[
\sum_{k=0}^{\infty}\alpha_k=\infty,
\qquad
\sum_{k=0}^{\infty}\alpha_k^2<\infty,
\]
then the bound in Theorem~\ref{thm:subGD} implies $f(\bar x_K)-f(x^\ast)\to 0$. Thus, convergence is obtained for the weighted averaged iterates. It is left as an exercise to quantify the speed of convergence for different choices of step sizes $(\alpha_k)_{k\in\N}$. A specific example will be considered in the convergence analysis of stochastic gradient descent, in particular, in the proof of Corollary~\ref{cor:SGDnonconvex}. If we use a constant step size $\alpha_k=\alpha$ for $k=0,\dots,K$, then Theorem~\ref{thm:subGD} gives
\[
f(\bar x_K)-f(x^\ast)
\leq
\frac{\|x_0-x^\ast\|^2}{2(K+1)\alpha}
+\frac{M^2\alpha}{2}\,.
\]
Choosing $\alpha=\frac{\|x_0-x^\ast\|}{M\sqrt{K+1}}$
yields
\[
f(\bar x_K)-f(x^\ast)
\leq
\frac{M\|x_0-x^\ast\|}{\sqrt{K+1}}.
\]
Thus, sub-gradient descent achieves a $\mathcal O(1/\sqrt{N})$ convergence rate for averaged iterates.
\end{remark}

\chapter{Accelerated gradient descent methods } \label{ch:accmethods}
In this chapter, we consider first-order optimization schemes that are designed to accelerate the
convergence of gradient descent. The basic idea is to incorporate information from previous
iterations into the update rule. This additional information is called \textit{momentum}. Instead
of moving only in the direction of the current negative gradient, momentum methods combine the
current gradient direction with information accumulated along the previous trajectory. We will motivate the effects of momentum using the example of minimizing a quadratic objective function presented in \cite{QIAN1999145}. As discussed in Section~\ref{sec:GD_strongconvex}, in particular in Theorem~\ref{thm:GD_strongconvex2},  the convergence rate of gradient descent on strongly convex and smooth objective functions, given by
\[c = \frac{\kappa-1}{\kappa+1} = 1-\frac{2}{\kappa+1}\,,\]
scales poorly when the condition number $\kappa = L/\mu$ is large. Here, $L$ denotes the smoothness parameter and $\mu$ the strong convexity parameter. In case of quadratic objective functions of the form $f(x) = \frac12x^\top Q x$ with positive definite matrix $Q\in\R^{d\times d}$, the ratio $\kappa$ corresponds to the condition number of $Q$. We make this more precise in the following example.

\begin{example}[Quadratic objective function]\label{ex:quadratic}
Let $Q\in\R^{d\times d}$ be a symmetric and positive definite matrix with eigenvalues $\lambda_{\max}=\lambda_1\ge \dots\ge \lambda_d=\lambda_{\min} >0$. We aim to solve
\[\min_{x\in\R^d}\ f(x),\quad f(x) = \frac12x^\top Q x \]
using the gradient descent method. Let us consider an eigendecomposition of $Q$ such that
$Q=UDU^\top$ where $U\in\R^{d\times d}$ is an orthogonal matrix and $D={\mathrm{diag}}(\lambda_1,\dots,\lambda_d)\in\R^{d\times d}$ a diagonal matrix with the eigenvalues along the diagonal. The inner product scaled by $Q$ can be rewritten as 
\[\frac12x^\top Q x = \frac12 x^\top (UDU^\top) x = \frac12 (U^\top x)^\top D (U^\top x) = \frac12 z^\top D z \]
with $z=U^\top x$. We can then consider the equivalent optimization problem (since all eigenvalues are positive) of the form
\[\min_{x\in\R^d}\ f(x),\quad f(x) = \frac12 x^\top Dx. \]
The gradient and Hessian compute as
\[\nabla f(x) = Dx \quad \text{and}\quad \nabla^2f(x) = D\]
and the unique global minimum is given by $x_\ast=0\in\R^d$. Let $x_0\in\R^d$, $x_0\neq0$ be the initialization and $\alpha_k = \alpha \in(0,\frac{2}{\lambda_{\max}})$ (smoothness parameter $L=\lambda_{\max}$) a fixed step size. Recall that the gradient descent method is written as
\[x_{k+1} = x_k -\alpha \nabla f(x_k) = x_k - \alpha D x_k. \]
The distance to the global minimum is given by
\begin{align*}
\|x_{k+1}-x_\ast\| = \|x_{k+1} \| = \|x_k-\alpha Dx_k\| = \|(I-\alpha D) x_k\| \le \max(|1-\alpha\lambda_{\min}|,|1-\alpha\lambda_{\max}|) \|x_k\|.
\end{align*}
We can now compute the step size such that the derived upper bound is minimized in the sense that
\[\min_{\alpha\in(0,\frac{2}{\lambda_{\max}})}\  \max(|1-\alpha\lambda_{\min}|,|1-\alpha\lambda_{\max}|). \]
It is an exercise to prove that the resulting optimal step size is given by $\alpha_\ast = \frac{2}{\lambda_{\min}+\lambda_{\max}}$ (note that this step size coincides with the one derived for general $\mu$-strongly convex and $L$-smooth objective functions in Remark~\ref{rem:opt_step}). Let $\kappa = \frac{\lambda_{\max}}{\lambda_{\min}}$ be the condition number of $Q$ and $D$ respectively. Then we have
\begin{align*}
\max(|1-\alpha_\ast \lambda_{\min}|,|1-\alpha_\ast\lambda_{\max}|) &= |1-\alpha_\ast\lambda_{\min}| = |1-\alpha_\ast\lambda_{\max}| \\ &= \left|1-\frac{2\lambda_{\max}}{\lambda_{\min}+\lambda_{\max}}\right| = \frac{\kappa-1}{\kappa+1} = 1-\frac{2}{\kappa+1}.
\end{align*}
Therefore, the gradient descent method with fixed step size $\alpha_\ast$ converges linearly with
\[\|x_k-x_\ast\|\le \left(\frac{\kappa-1}{\kappa+1}\right)^k \|x_0-x_\ast\|. \]
In order to achieve an error of tolerance $\varepsilon>0$, we need to iterate a certain amount of steps:
\begin{align*}
\left(\frac{\kappa-1}{\kappa+1}\right)^k < \varepsilon\quad &\Leftrightarrow \quad \log\left(\frac{1}{\varepsilon}\right) < k\log\left(\frac{\kappa+1}{\kappa-1}\right)\\
&\Leftrightarrow \quad k> \log\left(\frac{1}{\varepsilon}\right) \log\left(1+\frac{2}{\kappa-1}\right)^{-1},
\end{align*}
which increases with increasing condition number $\kappa$. Hence, gradient descent may perform poorly for quadratic functions with high condition number $\kappa$. 
\end{example}
\begin{figure}[!htb]
  \centering  \includegraphics[width=1\textwidth]{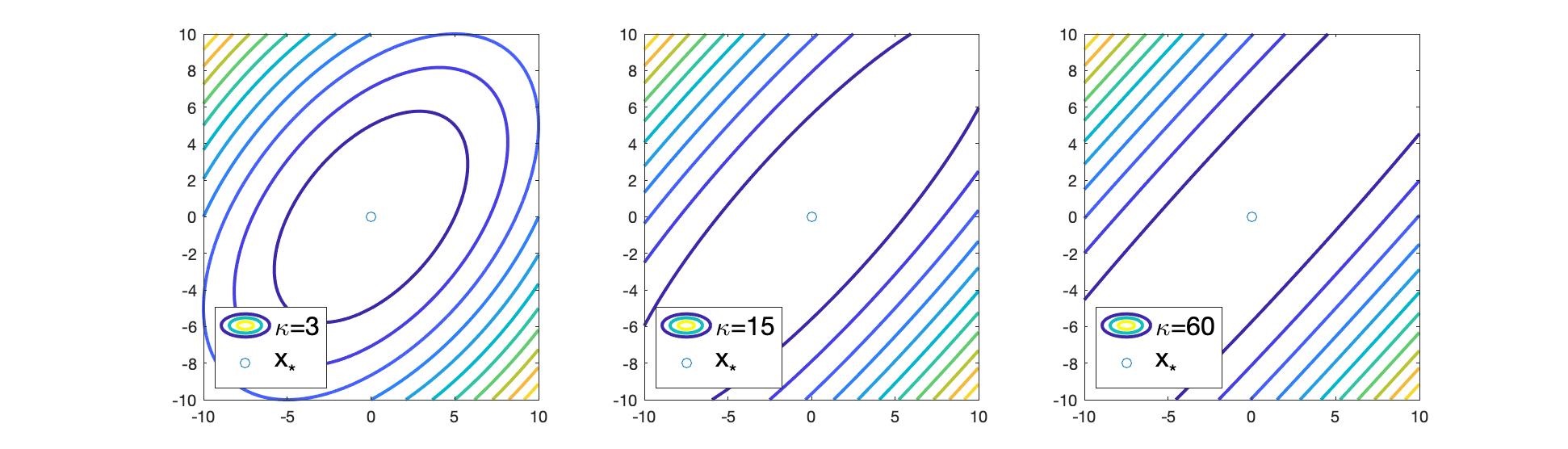}
 \caption{Contour lines of a quadratic function with increasing condition number $\kappa$.} \label{fig:condition_number}
\end{figure} 

\section{Polyak's heavy ball method}
The idea of incorporating momentum into iterative optimization schemes goes back to Polyak's
heavy ball method~\cite{POLYAK19641}. The name is motivated by the physical picture of a heavy
ball rolling down a hill towards a valley. A light ball is strongly affected by local curvature and may oscillate in narrow valleys. A heavier
ball, in contrast, carries momentum from previous positions and can move more steadily through
such regions. In optimization algorithms, this idea is modeled by adding a multiple of the previous
step to the current gradient step. Mathematically, momentum is incorporated through an inertial term depending on the previous displacement of the iterates. We formulate Polyak's heavy ball method (HBM) in the following algorithm:

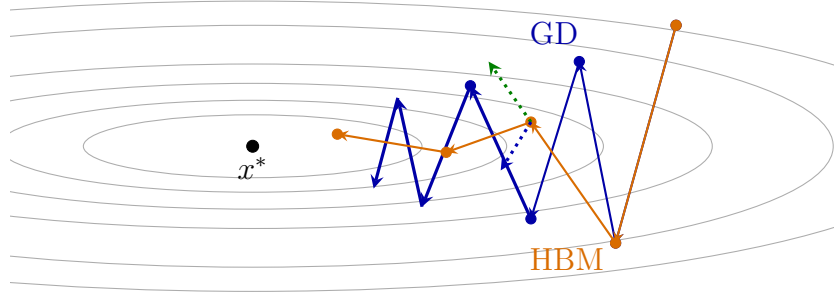
\begin{figure}[!htb]
\centering
\vspace{-1.5cm}
\begin{tikzpicture}[scale=1.6, >=stealth]
\clip (-2.0,-1.5) rectangle (4.8,2.6);

\foreach \a/\b in {1.4/0.26,2.1/0.38,2.9/0.52,3.8/0.68,4.8/1.0,5.8/1.2}
{
    \draw[gray!65] (0,0) ellipse ({\a} and {\b});
}

\filldraw[black] (0,0) circle (1.4pt);
\node[below] at (0,0) {$x^\ast$};

\coordinate (g0) at (3.5,1.0);
\coordinate (g1) at (3.0,-0.8);
\coordinate (g2) at (2.7,0.7);
\coordinate (g3) at (2.3,-0.6);
\coordinate (g4) at (1.8,0.5);
\coordinate (g5) at (1.4,-0.5);
\coordinate (g6) at (1.2,0.4);
\coordinate (g7) at (1.0,-0.35);

\draw[->, thick, blue!65!black] (g0) -- (g1);
\draw[->, thick, blue!65!black] (g1) -- (g2);
\draw[->, thick, blue!65!black] (g2) -- (g3);
\draw[->, very thick, blue!65!black] (g3) -- (g4);
\draw[->, very thick, blue!65!black] (g4) -- (g5);
\draw[->, very thick, blue!65!black] (g5) -- (g6);
\draw[->, very thick, blue!65!black] (g6) -- (g7);

\foreach \p in {g0,g1,g2,g3,g4}
    \filldraw[blue!65!black] (\p) circle (1.2pt);

\coordinate (h0)  at (3.5,1.0);
\coordinate (h1)  at (3.0,-0.8);
\coordinate (h2)  at (2.3,0.2);
\coordinate (h3)  at (1.6,-0.05);   
\coordinate (h4)  at (0.7,0.10);   

\draw[->, thick, orange!85!black] (h0) -- (h1);
\draw[->, thick, orange!85!black] (h1) -- (h2);
\draw[->, thick, orange!85!black] (h2) -- (h3);
\draw[->, thick, orange!85!black] (h3) -- (h4);

\foreach \p in {h0,h1,h2,h3,h4}
    \filldraw[orange!85!black] (\p) circle (1.2pt);

\node[blue!65!black, above right] at (2.2,0.75) {GD};
\node[orange!85!black, below right] at (2.2,-0.75) {HBM};


\coordinate (momtip) at (1.95,0.7);
\draw[->, dotted, very thick, green!50!black] (h2) -- (momtip);

\coordinate (gradtip) at (2.05,-0.20);
\draw[->, dotted, very thick, blue!70!black] (h2) -- (gradtip);

\end{tikzpicture}
\caption{Comparison of gradient descent and the heavy-ball method in a flat valley.
The blue path illustrates the typical zigzag behavior of gradient descent.
The orange path illustrates the smoother trajectory induced by momentum.
At the current iterate $x_k$, the heavy-ball update combines a gradient force
(blue dotted arrow) with a momentum force (green dotted arrow), resulting in the new step
(orange arrow) towards $x_{k+1}$.}
\label{fig:heavy-ball-vs-gd}
\end{figure}

\begin{algorithm}[htb!]
\begin{algorithmic}[1]
\State \textbf{Input:} \begin{itemize}
 \item objective function $f:\R^d\to\R$
 \item initial $x_0\in\R^d$
 \item sequence of step sizes $(\alpha_k)_{k\in\N}$, $\alpha_k>0$, and sequence of momentum parameters $(\beta_k)_{k\in\N}$, $\beta_k\ge0$.
 \end{itemize}
 \State set $x_1=x_0-\alpha_0\nabla f(x_0)$, and $k=1$
\While{"convergence/stopping criterion not met"}
	\State set $x_{k+1} = x_k -\alpha_k \nabla f(x_k)+\beta_k (x_k-x_{k-1})$, $k \mapsto k+1$
\EndWhile
\end{algorithmic}
 \caption{Heavy ball method}\label{alg:HBM}
\end{algorithm}

Let us take a closer look into the iterative update of HBM: 
\[x_{k+1} \quad = \quad \underbrace{x_k - \alpha_k \nabla f(x_k)}_{\text{gradient descent}} \quad + \quad \underbrace{\beta_k(x_k-x_{k-1})}_{\text{Heavy ball momentum}}. \]
This means, we incorporate the update from the previous iteration through
\[x_k-x_{k-1} = -\alpha_{k-1}\nabla f(x_{k-1}) + \beta_{k-1} (x_{k-1}-x_{k-2}) \]
into the next iteration. For example, in the second iteration, the HBM update is given by
\[x_2 = x_1 - \alpha_1\nabla f(x_1) - \beta_1 \alpha_0\nabla f(x_0). \]
The hyperparameter $\beta_k\ge0$ controls the strength of the influence of momentum and can be seen as damping parameter. Moreover, we can choose $\beta_k =0$ to recover the gradient descent method (without momentum).  Of course, the performance of the HBM method highly depends on a good choice the parameter $\beta_k$. In general, HBM is no descent method and therefore, we do not expect a monotonic decrease of the objective function along the iterations.

We will continue with Example~\ref{ex:quadratic} and derive an optimal choice of step size $\alpha_k$ and momentum parameter $\beta_k$ for quadratic objective function.
\begin{example}[Continuation of Example~\ref{ex:quadratic}] \label{ex:quadratic_HBM}
Let us come back to minimizing the quadratic objective function $f(x) = \frac{1}{2}x^\top D x$. Consider HBM with fixed step size $\alpha_k=\alpha>0$ and fixed momentum parameter $\beta_k=\beta>0$. The iterative scheme is given by
\[x_{k+1} = x_k - \alpha D x_k + \beta (x_k-x_{k-1}). \]
In this case, we consider the joint update of the vector 
\[\begin{pmatrix}
x_{k+1}-x_\ast \\ x_k-x_\ast
\end{pmatrix} = \begin{pmatrix}
x_{k+1} \\ x_k
\end{pmatrix}\in\R^{2d} .\]
The iteration can be written as
\begin{align*}
 \begin{pmatrix}
x_{k+1} \\ x_k
\end{pmatrix} = \begin{pmatrix}
x_k-\alpha Dx_k + \beta (x_k-x_{k-1})\\ x_k
\end{pmatrix} &= \begin{pmatrix}
(1+\beta) I x_k -\alpha D x_k -\beta I x_{k-1}\\
I x_k
\end{pmatrix} \\ &= \begin{pmatrix}
(1+\beta)I -\alpha D & -\beta I\\ I & 0 
\end{pmatrix} \begin{pmatrix}
x_k\\x_{k-1}
\end{pmatrix} =: T \begin{pmatrix}
x_k\\x_{k-1}
\end{pmatrix}. 
\end{align*}
Since the matrix $T\in\R^{2d\times 2d}$ is independent of the iteration $k\in\N$, we obtain
\[ \begin{pmatrix}
x_{k+1} \\ x_k
\end{pmatrix} = T^k \begin{pmatrix}
x_1\\ x_0
\end{pmatrix}, \]
and therefore,  the error can be written as
\[\left\|\begin{pmatrix}
x_{k+1}-x_\ast \\ x_k-x_\ast
\end{pmatrix}\right\| =\left\| T^k \begin{pmatrix}
x_1-x_\ast\\ x_0-x_\ast
\end{pmatrix}\right\| \le \|T^k\| \left\| \begin{pmatrix}
x_1-x_\ast\\ x_0-x_\ast
\end{pmatrix} \right\|,\]
for some matrix norm which is consistent with the euclidean 2-norm. Let $\rho(T)=\max_{i=1,\dots,2d} |\lambda_i(T)|$ be the largest eigenvalue $\lambda_i(T)$ (in absolute value) of $T$, also called the spectral radius of $T$. We will make use of the Gelfand formula which states that $\rho(T) = \lim_{k\to\infty} \|T^k\|^{1/k}$ for any matrix norm,  and in particular there exists a sequence $(\varepsilon_k)_{k\in\N}$ converging to $0$, such that
\[\|T^k\| \le (\rho(T) + \varepsilon_k)^k. \]
Before going into details, we transform $T$ to a block diagonal matrix without changing the corresponding eigenvalues (note that the eigenvectors change). From Exercise~\ref{ex:selfsim_matrix} we observe that $T$ is self-similar to a block diagonal matrix 
\[\widehat T = \begin{pmatrix}
T_1 & \ & \ \\
\ & \ddots & \ \\
\ & \ & T_d
\end{pmatrix}\qquad \text{with blocks}\quad 
 T_i = \begin{pmatrix}
1+\beta - \alpha\lambda_i & -\beta \\ 1 & 0
\end{pmatrix}\in\R^{2\times 2}. \]
Let $\mu_i$ be an eigenvalue of $\widehat T$ with corresponding eigenvector $v_i\in\R^{2d}$, then we can compute
\[\widehat T v_i = \mu_i v_i = S^{-1}T S v_i \quad \Leftrightarrow\quad  \mu_i Sv_i = TSv_i,\]
which means that $\mu_i$ is also an eigenvalue of $T$ with corresponding eigenvector $Sv_i$. Without loss of generality we will compute the eigenvalues of $\widehat T$ instead of $T$.  Due to the block structure, we can deduce the computation of the eigenvalues from the computation of the eigenvalues of each block $T_i\in\R^{2\times 2}$. To do so, we aim to find $\mu_i\in\mathbb C$ satisfying 
\[\det(T_i - \mu_i I) = \det\left(\begin{pmatrix}
1+\beta-\alpha\lambda_i -\mu_i & -\beta \\ 1 & -\mu_i
\end{pmatrix}\right) = \mu_i^2 - \mu_i(1+\beta -\alpha \lambda_i)+\beta \overset{!}{=} 0.\]
For each block we obtain two eigenvalues given as
\begin{align*}
\mu_i^{(1)} = \frac{1+\beta-\alpha\lambda_i}{2} - \sqrt{\left(\frac{1+\beta-\alpha\lambda_i}{2}\right)^2 - \beta}\\
\mu_i^{(2)} = \frac{1+\beta-\alpha\lambda_i}{2} + \sqrt{\left(\frac{1+\beta-\alpha\lambda_i}{2}\right)^2 - \beta}.
\end{align*}
We restrict our self to $\beta>0$ such that $\left(\frac{1+\beta-\alpha\lambda_i}{2}\right)^2 - \beta\le 0$ and therefore, the eigenvalues are complex-valued. It then holds true that the absolute value is given by
\begin{align*} 
|\mu_i^{(1)}| = |\mu_i^{(2)}| &= \frac{1}{2}\sqrt{(1-\alpha\lambda_i+\beta)^2 + |\underbrace{(1-\alpha\lambda_i+\beta)^2-4\beta}_{\le0}|}\\ &= \frac{1}{2}\sqrt{(1-\alpha\lambda_i+\beta)^2 - (1-\alpha\lambda_i+\beta)^2+4\beta} = \sqrt{\beta}. 
\end{align*}
Motivated by this observation, we aim to satisfy $\left(\frac{1+\beta-\alpha\lambda_i}{2}\right)^2 - \beta\le 0$ for all $i$, such that the spectral radius is given by $\rho(T)=\sqrt{\beta}$. Let us consider 
\[\Delta(\beta):= \left(\frac{1+\beta-\alpha\lambda_i}{2}\right)^2 - \beta = \frac{1}{4}(\beta^2 - 2(1+\alpha\lambda_i)\beta + (1-\alpha\lambda_i)^2). \]
The mapping $\beta\mapsto\Delta(\beta)$, $\beta\in\R$, describes a parabola, such that $\Delta(\beta)<0$ between two points $\beta^{(1)},\beta^{(2)}$ satisfying $\Delta(\beta^{(1)}) = \Delta(\beta^{(2)}) =0$ (if two exist). Therefore, we solve $\Delta(\beta)\overset{!}{=}0$. We focus on the cases where $\beta<1$ and $\alpha<\frac{4}{\lambda_{\max}}$ such that it holds
\[|1-\sqrt{\alpha \lambda_{\min}}|,|1-\sqrt{\alpha\lambda_{\max}}|\in(0,1). \]
Moreover, we observe that
\[(1-\sqrt{\alpha\lambda_i})^2 \le \max\left((1-\sqrt{\alpha\lambda_{\min}})^2,(1-\sqrt{\alpha\lambda_{\max}})^2\right) \]
such that it is sufficient to choose
\[1>\beta \ge  \max\left((1-\sqrt{\alpha\lambda_{\min}})^2,(1-\sqrt{\alpha\lambda_{\max}})^2\right). \] 
in order to force $\rho(T) = \sqrt{\beta}$. With the specific choice 
\begin{align*}
\alpha = \frac{4}{(\sqrt{\lambda_{\min}}+\sqrt{\lambda_{\max}})^2} \quad \text{and}\quad 
\beta =  \max\left((1-\sqrt{\alpha\lambda_{\min}})^2,(1-\sqrt{\alpha\lambda_{\max}})^2\right)
\end{align*}
we deduce that
\[ \beta = (1-\sqrt{\alpha\lambda_{\max}})^2 = (1-\sqrt{\alpha\lambda_{\min}})^2 = \left(\frac{\sqrt{\lambda_{\max}}-\sqrt{\lambda_{\min}}}{\sqrt{\lambda_{\max}}+\sqrt{\lambda_{\min}}} \right)^2 = \left(\frac{\sqrt{\kappa}-1}{\sqrt{\kappa}+1} \right)^2 \]
and therefore, $\rho(T)= \frac{\sqrt{\kappa}-1}{\sqrt{\kappa}+1}$. Finally,  using the Gelfand formula we obtain an improved upper bound (for sufficiently large $k$) on the error compared to gradient descent method
\[\left\|\begin{pmatrix}
x_{k+1}-x_\ast \\ x_k-x_\ast
\end{pmatrix}\right\|  \le \left(\frac{\sqrt{\kappa}-1}{\sqrt{\kappa}+1}+\varepsilon_k\right)^k \left\| \begin{pmatrix}
x_1-x_\ast\\ x_0-x_\ast
\end{pmatrix} \right\|\, .\]
Similarly as before,  in order to achieve an error of tolerance $\varepsilon>0$, we need to iterate a certain amount of steps: 
\begin{align*}
\left(\frac{\sqrt{\kappa}-1}{\sqrt{\kappa}+1}\right)^k <\varepsilon \Leftrightarrow k> \log\left(\frac{1}{\varepsilon}\right) \log\left(1+\frac{2}{\sqrt{\kappa}-1}\right)^{-1},
\end{align*}
where $ \log\left(1+\frac{2}{\sqrt{\kappa}-1}\right)^{-1}\le  \log\left(1+\frac{2}{\kappa-1}\right)^{-1}$, since $\kappa\ge1$ and therefore $\sqrt{\kappa}\le \kappa$. 
\end{example}

\begin{exercise}\label{ex:selfsim_matrix}
Let $T\in\R^{2d\times 2d}$ be defined as 
\[T=\begin{pmatrix}
(1+\beta)I -\alpha D & -\beta I\\ I & 0 
\end{pmatrix}, \]
with diagonal matrix $D={\mathrm{diag}}(\lambda_1,\dots,\lambda_d)$, $\alpha,\beta>0$. Prove that there exists a regular matrix $S\in\R^{2d\times 2d}$ such that
\[S^{-1}T S = \widehat T = \begin{pmatrix}
T_1 & \ & \ \\
\ & \ddots & \ \\
\ & \ & T_d
\end{pmatrix}, \]
where $\widehat T$ is a block diagonal matrix with 
\[ T_i = \begin{pmatrix}
1+\beta - \alpha\lambda_i & -\beta \\ 1 & 0
\end{pmatrix}\in\R^{2\times 2}. \]
\end{exercise}

\medskip
We observed that in the case of quadratic objective function we are able to improve the rate of convergence compared to gradient descent. This suggests that the convergence behavior of gradient descent, as derived in Section~\ref{sec:GD_convex} and Section~\ref{sec:GD_strongconvex}, might be sub-optimal. 

\section{Discussion about optimality of first-order methods}

For quadratic objective functions we have seen that the rate of convergence of the gradient descent method can be improved through the incorporation of momentum in the form of HBM. This raises the question about optimality of gradient descent as a first-order method. Or the other way around, what is the best possible convergence behavior we can expect using only first-order information.  We consider the following class of first-order iterative methods.
\begin{ass}\label{ass:firstorder}
The sequence $(x_k)_{k\in\N}$ (generated by some iterative scheme) satisfies the condition 
\[ x_k \in x_0 + {\mathrm{span}}\{\nabla f(x_0),\dots,\nabla f(x_{k-1})\}\]
for all $k\ge1$.
\end{ass}
Assumption~\ref{ass:firstorder} means that each iteration $x_k$ can be expressed as a linear combination of the initialization $x_0$ and all previous gradients $\nabla f(x_0),\dots,\nabla f(x_{k-1})$.  Both gradient descent and HBM are examples that satisfy Assumption~\ref{ass:firstorder}. We recall that for objective functions which are $\mu$-strongly convex and $L$-smooth, we obtain linear convergence of gradient descent with fixed step size $\bar \alpha = \frac{2}{\mu+L}$ of the form
\[\|x_k-x_\ast\|^2 \le \left(\frac{\kappa-1}{\kappa+1}\right)^{2k}\|x_0-x_\ast\|^2, \]
see Section~\ref{sec:GD_strongconvex}. For the specific case of quadratic objective function, the upper bound can be improved through HBM to 
\[ \|x_k-x_\ast\|^2 \le \left(\frac{\sqrt{\kappa}-1}{\sqrt{\kappa}+1}\right)^{2k}\|x_0-x_\ast\|^2.\]
The following lower bound from Nesterov, see e.g.~\cite{N2018}, shows that we cannot expect more than this improvement as long as we do not include more than first-order information, i.e.~as long as our iterative scheme satisfies Assumption~\ref{ass:firstorder}. However, the lower bound is derived for a specifically constructed function with high- or even infinite-dimensional domain, to be more precise, for a function $f:\ell^2(\R)\to\R$, where 
\[\ell^2(\R) := \{(z_i)_{i\in\N}\mid z_i\in\R,\ \sum_{i=1}^\infty |z_i|^2<\infty\}. \]
Note that $\ell^2(\R)$ can be equipped with a norm as well as an inner product, such that it forms a Banach and even a Hilbert space. The proof of Theorem~\ref{thm:lowerbound_strongconvex} in \cite{N2018} is constructive, i.e.~one can construct a certain $\mu$-strongly convex and $L$-smooth function $f:\ell^2(\R)\to\R$ satisfying the lower bound. 
\begin{thm}[Lower bound strong convex and smooth, Theorem~2.1.13 in\cite{N2018}]\label{thm:lowerbound_strongconvex}
For each $x_0\in\ell^2(\R)$, $\mu,L>0$ with $\kappa = \frac{L}{\mu}>1$, there exists a $\mu$-strongly convex and $L$-smooth function $f:\ell^2(\R)\to\R$ such that every iterative scheme $(x_k)_{k\in\N}$ satisfying Assumption~\ref{ass:firstorder} satisfies a lower bound on the error given by
\[e(x_k) := \|x_k-x_\ast\|^2\ge \left(\frac{\sqrt{\kappa}-1}{\sqrt{\kappa}+1}\right)^{2k} \|x_0-x_\ast\|^2, \]
where $x_\ast\in\ell^2(\R)$ denotes the unique global minimum of $f$.
\end{thm}

\begin{remark}
We note that the lower bound in Theorem~\ref{thm:lowerbound_strongconvex} covers a more general setting than the one  consider in this lecture course so far. Up to now, we have only considered objective functions $f:\R^d\to\R$ with finite dimensional domain $\R^d$.  Strictly speaking, we would need to go back to the beginning of this course in order to move from finite dimensional to infinite dimensional domains.  One needs to re-define derivatives and consider more general optimality conditions. This would lead to the so-called Fréchet derivative, which are required to formulate gradient descent methods in Hilbert spaces.  However, this is beyond the scope of this lecture course.
\end{remark}

We now return to the setting of Section~\ref{sec:GD_convex}, where we have assumed general convex and $L$-smooth functions (without a strong convexity assumption).  Under these properties, gradient descent with a fixed step size $\bar\alpha\le\frac{1}{L}$ converges with upper bound of the form
\[f(x_k)-f^\ast \le \frac{c}{k+1},\quad k\in\N,c>0,\]
where $f^\ast=\min_{x\in\R^d}\ f(x)$.  Indeed, also in this scenario, it is possible to derive a lower bound on iterative schemes satisfying Assumption~\ref{ass:firstorder}, which suggest a gap between upper and lower bound. 

\begin{thm}[Lower bound convex and smooth, Theorem~2.1.7 in\cite{N2018}]\label{thm:lowerbound_convex}
For every $k\in\N$ with $1\le k \le \frac{1}{2}(d-1)$, $L>0$ and every $x_0\in\R^d$ ($d$ denotes the dimension of the domain), there exists a convex and $L$-smooth function $f:\R^d\to\R$ such that every iterative scheme $(x_k)_{k\in\N}$ satisfying Assumption~\ref{ass:firstorder} satisfies  a lower bound on the error given by
\[e(x_k) := f(x_k)-f^\ast \ge \frac{3L\|x_0-x_\ast\|^2}{32(k+1)^2}, \]
where $f^\ast = \min_{x\in\R^d}\ f(x)>-\infty$ exists.
\end{thm}

\begin{remark}
The considered lower bound in Theorem~\ref{thm:lowerbound_convex} is only satisfied for $k\le \frac{1}{2}(d-1)$, which again, particularly for high dimensional ($d\gg 1$) optimization tasks, suggests a gap between lower and upper bound.  The proof in \cite{N2018} is again via construction.
\end{remark}

\section{Nesterov's acceleration method}
We ask our self if we can improve the upper bounds derived for gradient descent methods (both for convex and strong-convex setting) through momentum methods. This will be part of the next section. Recall that the iteration of HBM is given by
\[ x_{k+1} = x_k-\alpha_k\nabla f(x_k) +\beta_k (x_k-x_{k-1}).\]
We have seen that we can obtain linear convergence for quadratic objective function $f(x) = \frac12 x^\top Qx$ with rate $c=\frac{\sqrt{\kappa}-1}{\sqrt{\kappa}+1}$. To do so, we have derived 
\begin{equation}\label{eq:optHBM_step}
\alpha_k=\bar\alpha = \frac{4}{(\sqrt{L}+\sqrt{\mu})^2}\quad \text{and}\quad \beta_k= \bar\beta = \left( \frac{\sqrt{\kappa}-1}{\sqrt{\kappa} + 1}\right)^2, 
\end{equation}
where $L>0$ denotes the largest eigenvalue and $\mu>0$ the smallest eigenvalue of $Q$, and $\kappa=\frac{L}{\mu}$ is the condition number. We wonder if it is possible to extend this result to general $L$-smooth and $\mu$-strongly convex functions. If we similarly set $\alpha$, $\beta$ fixed as in \eqref{eq:optHBM_step}, do we still obtain linear convergence with rate $c=\frac{\sqrt{\kappa}-1}{\sqrt{\kappa}+1}$? Unfortunately, this is not true.  We consider the following counter example presented in \cite{Lessard2016}, where it turns out that one can construct a one-dimensional $L$-smooth and $\mu$-strongly convex function, for which HBM runs into a circle and does not converge. This example is formulated as exercise:
\begin{exercise}
\begin{enumerate}
\item Find a continuous function $f:\R\to\R$ such that
\[f'(x) = \begin{cases} 25x, &x\le 1\\
x+24, &1< x<2\\
25x-24, &2\le x 
\end{cases}\, . \]
Prove that $f$ is $\mu$-strongly convex with $\mu=1$, $L$-smooth with $L=25$ and has a unique global minimum in $x_\ast=0$.
\item Implement HBM with the optimal step size $\alpha$ and momentum parameter $\beta$ following \eqref{eq:optHBM_step}.
\item Prove that the application of HBM on $f$ with the parameters in \eqref{eq:optHBM_step} result in the recursion
\[x_{k+1} = \frac{13}{9} x_k - \frac{4}{9} x_{k-1} - \frac{1}{9}\nabla f(x_k). \]
\item Find a cycle of points $p\to q\to r\to p$, such that for $x_0=p$ we have 
\[x_{3k} = p,\ x_{3k+1} = q, \ x_{3k+2} = r \]
for all $k\in\N$. To do so, assume $p,q<1$ and $r>2$, apply the heavy ball recursion to create a linear equation for $p,q,r$ and solve it. What does it mean for the convergence behavior?
\end{enumerate}
\end{exercise}

\textbf{Motivation:} We motivate a different approach of incorporating momentum known as Nesterov's acceleration method (NAM).  Recall that the iteration of HBM first computes the gradient at the current location $\nabla f(x_k)$ and then moves into direction of a weighted sum of all previous gradients
\[d_k = -\alpha_k \nabla f(x_k) + \beta_k (x_k-x_{k-1}) = -\alpha_k \nabla f(x_k) -\beta_k \alpha_{k-1}\nabla f(x_{k-1}) + \beta_k\beta_{k-1}(x_{k-1}-x_{k-2})=\dots\, . \]
In NAM, the momentum step and the gradient evaluation are separated. The method first forms an extrapolated point using information from the previous step and then
computes the gradient at this extrapolated point. We can describe this method through a coupled system of two vectors $[p_k,q_k]\in\R^d\times\R^d$:
\begin{enumerate}
\item Assume that we are in location $p_k\in\R^d$, such that we can obtain information from the previous iterations through the computation
\[q_k = p_k + \beta (p_k-p_{k-1}), \]
for some momentum parameter $\beta>0$.
\item In location $q_k$ we compute the next gradient information in order to correct the previously gained information
\[p_{k+1} = q_k - \alpha\nabla f(q_k), \]
where $\alpha>0$ denotes a step size.
\item Compute the iterated weighted information for the next iteration \[q_{k+1} = p_{k+1} + \beta (p_{k+1} - p_k).\]
\end{enumerate}

We summarize NAM in Algorithm~\ref{alg:NAM}.
 \begin{algorithm}[htb!]
\begin{algorithmic}[1]
\State \textbf{Input:} \begin{itemize}
 \item objective function $f:\R^d\to\R$
 \item initial $q_0,p_0\in\R^d$
 \item sequence of step sizes $(\alpha_k)_{k\in\N}$, $\alpha_k>0$, and sequence of momentum parameters $(\beta_k)_{k\in\N}$, $\beta_k\ge0$.
 \end{itemize}
 \State set $p_1 = q_0 - \alpha_0 \nabla f(q_0)$
 \State set $q_1 = p_1 + \beta_0 (p_1-p_0)$
 \State set $k=1$
\While{"convergence/stopping criterion not met"}
	\State set $p_{k+1} = q_k - \alpha_k \nabla f(q_k)$
	\State set $q_{k+1} = p_{k+1} + \beta_k (p_{k+1}-p_k)$ 
	\State set $k \mapsto k+1$
\EndWhile
\end{algorithmic}
 \caption{Nesterov's accelerated gradient descent method}\label{alg:NAM}
\end{algorithm}

\subsection{Convergence for convex and smooth objective functions}\label{sec:NAM_convex}
We start the convergence analysis of Nesterov's method under the assumption that the objective
function is convex and $L$-smooth. In order to analyze NAM, we use a slightly more general formulation of the iteration in terms of three sequences $(x_k,y_k,z_k)_{k\in\mathbb N}$.
The convergence analysis is based on a Lyapunov argument for accelerated optimization schemes; see~\cite{Wilson2021}.
The analysis presented in \cite{Wilson2021} covers a wide range of acceleration algorithms in continuous and discrete-time setting. In terms of the scope of this lecture course, we will focus on a very simplified setting, where we write NAM as system of the form
\begin{equation}\label{eq:NAMconvex}
\begin{split}
x_k &= \tau_k z_k + (1-\tau_k) y_k,\\
y_{k+1} &= x_k - \alpha_k \nabla f(x_k),\\
z_{k+1} &= z_k - \gamma_k \nabla f(x_k),
\end{split}
\end{equation}
with parameters $\alpha_k, \gamma_k>0$ and $\tau_k\in(0,1)$. The sequence $(y_k)$ represents the points obtained after gradient steps, while $(z_k)$ accumulates the momentum information. The point $x_k$ combines both components.
Consider the following example to see that this system can be seen as NAM.
\begin{example}\label{ex:connectionthreevar}  
We now show how the formulation~\eqref{eq:NAMconvex} can be related to Algorithm~\ref{alg:NAM}. Indeed, one can choose the parameters $\gamma_k,\tau_k>0$ such that the system reduces into the form of Algorithm~\ref{alg:NAM}. Therefore, we rewrite the update 
\begin{align*}
z_{k+1}& = z_k + \frac{1}{\tau_k}(x_k-x_k) - \gamma_k\nabla f(x_k)=y_k + \frac{1}{\tau_k}(x_k-y_k) -\gamma_k \nabla f(x_k)\\
&=y_k + \frac{1}{\tau_k} (x_k-\gamma_k \tau_k \nabla f(x_k) - y_k)= y_k+ \frac{1}{\tau_k} (y_{k+1}-y_k),
\end{align*}
where we have chosen $(\alpha_k,\gamma_k,\tau_k)$ such that $\gamma_k \tau_k = \alpha_k$. We can plug this into the update of $x_{k+1}$ in order to eliminate $z_{k+1}$ from \eqref{eq:NAMconvex}: 
\begin{align*}
x_{k+1} = \tau_{k+1} z_{k+1} + (1-\tau_{k+1}) y_{k+1} = y_{k+1} + \frac{\tau_{k+1}(1-\tau_k)}{\tau_k} (y_{k+1}-y_k).
\end{align*}
Finally, we have written the system \eqref{eq:NAMconvex} as update of two variables $(x_k,y_k)_{k\in\N}$ described through
\begin{align*}
y_{k+1} &= x_k - \alpha_k \nabla f(x_k),\\
x_{k+1} &= y_{k+1} + \beta_k (y_{k+1}-y_k),
\end{align*}
where $\beta_k := \frac{\tau_{k+1}(1-\tau_k)}{\tau_k}>0$.
\end{example} 

As mentioned earlier, we will consider a simplified analysis based on Lyapunov methods as presented in \cite{Wilson2021}. There are many different ways of applying Lyapunov methods for analyzing optimization methods. See also Appendix~\ref{app:Lyapunov} for a brief motivation of Lyapunov methods in optimization. We will follow a specific strategy for proving convergence of an iterative scheme $(x_k)_{k\in\N}$ based on Lyapunov theory:
\begin{enumerate}
\item We firstly choose an error function $e:\R^d\to\R_+$ for which we want to prove convergence towards $0$, i.e.~$\lim_{k\to\infty}e(x_k)=0$. 
\item Construct a \textit{Lyapunov function} of the form
\[E_k :=E(x_k) = r(x_k) + A_k e(x_k), \]
where $r:\R^d\to\R_+$ is some auxiliary function, and $(A_k)_{k\in\N}$ is a monotonically increasing sequence with $A_0\ge0$ devoted to describe the speed of convergence. 
\item We aim to bound the increments of the sequence $(E_k)_{k\in\N}$ by
\[E_{k+1} - E_k \le \varepsilon_{k+1}, \]
where $(\varepsilon_{k})_{k\in\N}$ is a real-valued sequence with $\limsup_{k\to\infty}\varepsilon_k<+\infty$.  In our specific case, we aim to prove that $\varepsilon_{k+1}\le 0$ for all $k\in\N$ such that $(E_k)_{k\in\N}$ is non-increasing and particularly bounded by $E_k\le E_0$. It then follows, that 
\[A_k e(x_k) \le r(x_k) + A_k e(x_k) = E_k\le E_0\]
and therefore, we obtain
\[ e(x_k) \le \frac{E_0}{A_k}\]
which illustrates why $(A_k)_{k\in\N}$ describes the speed of convergence.
\end{enumerate}

\medskip

Motivated by \cite{Wilson2021}, in order to analyze the convergence of the system \eqref{eq:NAMconvex}, we will construct the Lyapunov function of the form
\begin{equation}\label{eq:Lyapunov_convex}
E_k = \frac{1}{2}\|z_k-x_\ast\|^2 + A_k (f(y_k)-f^\ast),
\end{equation}
where $f$ is assumed to satisfy $f^\ast =\min_{x\in\R^d}>-\infty$, $x_\ast\in\R^d$ is some global minimum of $f$ and $(A_k)_{k\in\N}$ is a monotonically increasing sequence with $A_0>0$.  We will firstly derive the following upper bound on the increments of $(E_k)_{k\in\N}$.
\begin{lemma}\label{lem:NAMconvex_aux}
Let $f:\R^d\to\R$ be $L$-smooth and convex with $\min_{x\in\R^d} f>-\infty$, and assume there exists at least one global minimum $x_\ast\in\R^d$ of $f$.  Moreover, let $(x_k,y_k,z_k)_{k\in\N}$ be generated by \eqref{eq:NAMconvex} with parameters $\alpha_k = \frac1L,$ $\gamma_k = A_{k+1}-A_k$ and $\tau_k = \frac{\gamma_k}{A_{k+1}}=\frac{A_{k+1}-A_k}{A_{k+1}} \in(0,1)$. i.e. 
\begin{equation*}
\begin{split}
x_k &= y_k + \frac{A_{k+1}-A_k}{A_{k+1}}(z_k-y_k),\\
y_{k+1} &= x_k - \frac{1}{L} \nabla f(x_k),\\
z_{k+1} &=z_k -(A_{k+1}-A_k)\nabla f(x_k),
\end{split}
\end{equation*}
initialized with $(y_0,z_0)\in\R^{d}\times\R^d$. Then the increments of the sequence $(E_k)_{k\in\N}$ defined in \eqref{eq:Lyapunov_convex} satisfy
\[E_{k+1} - E_k \le \varepsilon_{k+1}:= \frac12 (A_{k+1}-A_k)^2 \|\nabla f(x_k)\|^2 + A_{k+1} (f(y_{k+1})-f(x_k))\]
for all $k\in\N$.
\end{lemma}
The proof of Lemma~\ref{lem:NAMconvex_aux} is quite technical and therefore deferred to Appendix~\ref{app:omittedproofs}.
\begin{remark}
In order to show monotonic behavior of the error $(E_k)_{k\in\N}$, i.e.~$\varepsilon_{k+1}\le0$, we will later apply convexity and $L$-smoothness of $f$ to derive 
\[f(y_{k+1})-f(x_k)\propto -\|\nabla f(x_k)\|^2. \]
It will turn out, that the choice of $(A_k)_{k\in\N}$ is the key to prove convergence of the error $e_k = f(y_k) -f^\ast$.
\end{remark}

The previous Lemma guarantees an upper bound on the increments of the form
\[E_{k+1}-E_k \le \frac12(A_{k+1}-A_k)^2 \|\nabla f(x_k)\|^2 + A_{k+1} (f(y_{k+1})-f(x_k)). \]
Next, we want to apply $L$-smoothness in order to derive
\[f(y_{k+1})-f(x_k)\propto -\|\nabla f(x_k)\|^2 \]
and therefore, imply the decrease of $(E_k)_{k\in\N}$. In particular, we will then obtain convergence of the error $e_k = f(y_k)-f^\ast$ of the order $\mathcal O(\frac{1}{(k+1)k})$. Note that for gradient descent under similar assumptions on $f$ we did only prove convergence of order $\mathcal O(\frac{1}{k+1})$.

\begin{thm}[NAM for convex and smooth objective function]
Let $f:\R^d\to\R$ be $L$-smooth and convex with $\min_{x\in\R^d} f>-\infty$, and assume there exists at least one global minimum $x_\ast\in\R^d$ of $f$.  Moreover, let $(x_k,y_k,z_k)_{k\in\N}$ be generated by \eqref{eq:NAMconvex} with parameters $\alpha_k = \frac1L,$ $\gamma_k = A_{k+1}-A_k$ and $\tau_k = \frac{\gamma_k}{A_{k+1}}=\frac{A_{k+1}-A_k}{A_{k+1}} \in(0,1)$. i.e. 
\begin{equation*}
\begin{split}
x_k &= y_k + \frac{A_{k+1}-A_k}{A_{k+1}}(z_k-y_k),\\
y_{k+1} &= x_k - \frac{1}{L} \nabla f(x_k),\\
z_{k+1} &=z_k -(A_{k+1}-A_k)\nabla f(x_k),
\end{split}
\end{equation*}
initialized with $(y_0,z_0)\in\R^{d}\times\R^d$. Then the increments of the sequence $(E_k)_{k\in\N}$ defined in \eqref{eq:Lyapunov_convex} satisfy
\[E_{k+1} - E_k \le \left(\frac12 (A_{k+1}-A_k)^2 -\frac{1}{2L}A_{k+1}\right)\|\nabla f(x_k)\|^2\]
for all $k\in\N$. For the particular choice $A_k = \frac{1}{4L}(k+1)k$, $k\ge1$, and $A_0 = A_1$, we obtain
\[e_k = f(y_k) -f^\ast\le \frac{4LE_0}{(k+1)k},\quad k\ge1. \]
\end{thm}
\begin{proof}
We define $G_L(x_k) = x_k - \frac1L \nabla f(x_k)$ and apply $L$-smoothness of $f$ to deduce
\begin{align*}
f(G_L(x_k)) &\le f(x_k) + \langle \nabla f(x_k), G_L(x_k) - x_k\rangle + \frac{L}2 \|G_L(x_k)-x_k\|^2 \\
&=f(x_k) - \frac1{2L} \|\nabla f(x_k)\|^2,
\end{align*} 
implying that
\[f(y_{k+1})-f(x_k) \le -\frac{1}{2L}\|\nabla f(x_k)\|^2. \]
With the previous Lemma~\ref{lem:NAMconvex_aux} we obtain
\[E_{k+1} - E_k \le \left(\frac12 (A_{k+1}-A_k)^2 -\frac{1}{2L}A_{k+1}\right)\|\nabla f(x_k)\|^2.\]
With $A_k = \frac1{4L}(k+1)k$ it follows that
\[\frac{1}{2}(A_{k+1}-A_k)^2 - \frac{1}{2L}A_{k+1}\le 0, \]
since \[\frac{(A_{k+1}-A_k)^2}{A_{k+1}}=\frac{1}{L} \frac{(k+1)^2}{(k+2)(k+1)}\le \frac1L.\]
It follows that $E_{k+1}-E_k \le 0$ and therefore, we have
\[\frac12 \|z_k-x_\ast\|^2 + A_k (f(y_k)-f^\ast)\le E_0 \]
which implies convergence
\[f(y_k)-f^\ast \le \frac{4LE_0}{(k+1)k},\quad k\ge1.\]
\end{proof}
\begin{remark}
To draw the connection to Algorithm~\ref{alg:NAM} we have to choose
\[ \beta_k = \frac{\tau_{k+1}(1-\tau_k)}{\tau_k}, \]
where \[\tau_k = \frac{A_{k+1}-A_k}{A_{k+1}} = \frac{2k+2}{(k+1)(k+2)} = \frac{2}{k+2}\]
and hence,
\[\beta_k = \frac{\frac{2}{k+3} (1-\frac{2}{k+2})}{\frac2{k+2}} = \frac{k}{k+3}.\]
Note that 
\[\gamma_k\tau_k = \frac{(A_{k+1}-A_k)^2}{A_{k+1}} = \frac{1}{4L}\frac{(2k+2)^2}{(k+1)(k+2)} \to \frac1L\,,\]
which is in minor contrast to the choice $\gamma_k\tau_k = \alpha_k = \frac1L$ in Example~\ref{ex:connectionthreevar}.
\end{remark}

\subsection{Convergence for strongly convex and smooth objective functions}

We have seen that NAM leads to an improvement of the convergence for convex and smooth functions. We now want to consider the strongly convex and smooth setting, where we again show improvement compared to the optimal convergence behavior of the gradient descent method discussed in Remark~\ref{rem:opt_step}. Let $f:\R^d\to\R$ be a $\mu$-strongly convex and $L$-smooth objective function. Motivated from \cite{Wilson2021} we again consider a slightly different system for NAM of three variable $(x_k,y_k,z_k)_{k\in\N}$ generated by
\begin{equation}\label{eq:NAM_strongconvex}
\begin{split}
x_k &= \frac{\tau}{1+\tau} z_k + \frac{1}{1+\tau} y_k\\
y_{k+1} &= x_k - \frac{1}{L}\nabla f(x_k)\\
z_{k+1} &= z_k +\tau(x_k-z_k)-\frac{\tau}{\mu}\nabla f(x_k)
\end{split}\, 
\end{equation}
with $\tau>0$. Similarly as before, the authors in \cite{Wilson2021} consider a more general family of accelerated gradient methods, but due to the scope of this lecture course we again consider only the simplified formulation. In the following example, we make the connection to NAM formulated in Algorithm~\ref{alg:NAM}. We show that the system \eqref{eq:NAM_strongconvex} with fixed choice $\tau = \sqrt{\frac{\mu}{L}}$ can be viewed as special case of Algorithm~\ref{alg:NAM} with fixed $\alpha = \frac{1}{L}$ and $\beta = \frac{\sqrt{L}-\sqrt{\mu}}{\sqrt{L}+\sqrt{\mu}}$.

\begin{example}
We formulate system \eqref{eq:NAM_strongconvex} as update of two variables $(x_k,y_k)_{k\in\N}$ written as
\begin{equation*}
\begin{split}
y_{k+1} &= x_k - \alpha_k\nabla f(x_k)\\
x_{k+1} &= y_{k+1} + \beta_k (y_{k+1}-y_k)
\end{split}\, .
\end{equation*}
Firstly, we observe that
\[(1-\tau) z_k = (1-\tau)\left(\frac{1+\tau}{\tau} x_k -\frac{1}{\tau} y_k \right) = \left(\frac{1}{\tau}-\tau\right) x_k + \frac{\tau-1}{\tau}y_k, \]
such that we can write the update of $z_{k+1}$ through
\begin{align*}
z_{k+1} = z_k + \tau (x_k-z_k) - \frac{\tau}{\mu} \nabla f(x_k)
&=(1-\tau) z_k + \tau x_k-\frac{\tau}{\mu} \nabla f(x_k)\\
&=\left(\frac{1}{\tau} x_k -\tau+\tau\right) x_k +\left(1-\frac{1}{\tau}\right)y_k - \frac{\tau}{\tau}\frac{\tau}{\mu}\nabla f(x_k).
\end{align*}
We set $\tau = \sqrt{\frac{\mu}{L}}$ such that $\frac{\tau^2}{\mu} = \frac{1}L$ and therefore,
\begin{align*}
z_{k+1} = \frac1\tau x_k - \frac1\tau \frac1L\nabla f(x_k) + \left(1-\frac1\tau\right)y_k &= y_k + \frac1\tau \left(x_k-\frac1L\nabla f(x_k) -y_k \right)\\
&= y_k + \frac1\tau (y_{k+1}-y_k).
\end{align*}
We plug this into the update of $x_{k+1}$ to obtain
\begin{align*}
x_{k+1} &= \frac{\tau}{\tau+1}z_{k+1} +\frac{1}{\tau+1} y_{k+1} = \frac{\tau}{\tau+1} \left(y_k+\frac1\tau (y_{k+1}-y_k)\right) + \frac1{\tau+1} y_{k+1}\\
&=y_{k+1} + \frac{\tau-1}{\tau+1} y_k + \frac{1-\tau}{1+\tau}= y_{k+1} + \frac{1-\tau}{1+\tau} (y_{k+1}-y_k).
\end{align*}
Finally, with the choice $\tau = \sqrt{\frac{\mu}{L}}$ we can write the iterative update scheme as
\begin{align*}
y_{k+1} &= x_k - \frac1L \nabla f(x_k)\,,\\
x_{k+1} &=  y_{k+1} + \frac{\sqrt{L}-\sqrt{\mu}}{\sqrt{L}+\sqrt{\mu}} (y_{k+1}-y_k),
\end{align*}
such that we recover Algorithm~\ref{alg:NAM} with fixed $\alpha_k=\alpha = \frac{1}{L}$ and $\beta_k = \beta = \frac{\sqrt{L}-\sqrt{\mu}}{\sqrt{L}+\sqrt{\mu}}$.
\end{example}

We are now ready to prove linear convergence of Algorithm~\ref{alg:NAM} through system \eqref{eq:NAM_strongconvex} in the strongly convex and smooth setting.

\begin{thm}[NAM for strongly convex and smooth objective function] 
Let $f:\R^d\to\R$ be $\mu$-strongly convex and $L$-smooth with $L>\mu$, and let $x_\ast\in\R^d$ be the corresponding unique global minimum of $f$. Moreover, let $(x_k,y_k,z_k)_{k\in\N}$ be generated by \eqref{eq:NAM_strongconvex} with $\tau = \sqrt{\frac{\mu}{L}}\in(0,1)$ and initialized by $(y_0,z_0)\in\R^d\times\R^d$. Then {NAM} converges linearly in the sense that
\[e_k := f(y_k)-f(x_\ast) + \frac{\mu}{2}\|z_k-x_\ast\|^2 \le \left(1-\sqrt{\frac{\mu}{L}}\right)^k \left(f(y_0)-f(x_\ast) +\frac{\mu}{2}\|z_0-x_\ast\|^2\right). \]
\end{thm}

\begin{proof}
We define $e_k := f(y_k)-f(x_\ast) + \frac{\mu}{2}\|z_k-x_\ast\|^2$ and aim to prove 
\[e_{k+1}\le (1-\tau) e_k. \]
Firstly, applying $L$-smoothness of $f$ gives
\begin{equation}\label{eq:NAM_strongconvex_aux1}
f(y_{k+1}) \le f(x_k) + \langle \nabla f(x_x), y_{k+1}-x_k\rangle + \frac{L}{2}\|x_k-y_{k+1}\|^2 = f(x_k) -\frac{1}{2L}\|\nabla f(x_k)\|^2.
\end{equation}
Next, we consider the update of $e_k^{(1)} = \|z_k-x_\ast\|^2$:
\begin{align*}
e_{k+1}^{(1)} &= \|z_k +\tau(x_k-z_k) -\frac{\tau}{\mu}\nabla f(x_k) - x_\ast\|^2\\
&= \|z_k-x_\ast\|^2 + 2\langle \tau (x_k-z_k)-\frac{\tau}{\mu}\nabla f(x_k), z_k - x_\ast\rangle + \|\underbrace{\tau(x_k-z_k) - \frac{\tau}{\mu}\nabla f(x_k)}_{=z_{k+1}-z_k} \|^2\\
&= e_k^{(1)} + 2\tau \langle x_k-z_k,\underbrace{z_k-x_\ast}_{=z_k-x_k+x_k-x_\ast}\rangle - 2\frac{\tau}{\mu} \langle \nabla f(x_k),\underbrace{z_k-x_\ast}_{=z_k-x_k+x_k-x_\ast}\rangle +\|z_{k+1}-z_k\|^2\\
&= e_k^{(1)}+ \|z_{k+1}-z_k\|^2 + 2\tau \langle x_k-z_k,z_k-x_k\rangle - 2\frac{\tau}{\mu}\langle \nabla f(x_k),z_k-x_k\rangle \\
&\quad+2\tau \langle x_k-z_k,x_k-x_\ast\rangle -2\frac{\tau}{\mu}\langle \nabla f(x_k), x_k-x_\ast\rangle 
\end{align*}
Recall that by strong convexity it holds true that
\[f(x_\ast)-f(x_k)\ge \langle x_\ast-x_k,\nabla f(x_k)\rangle + \frac{\mu}2 \|x_k-x_\ast\|^2, \]
such that
\begin{align*}
e_{k+1}^{(1)} &\le e_k^{(1)}+ \|z_{k+1}-z_k\|^2 + 2\tau \langle x_k-z_k,z_k-x_k\rangle - 2\frac{\tau}{\mu}\langle \nabla f(x_k),z_k-x_k\rangle \\
&\quad + 2\frac{\tau}{\mu} (f(x_\ast)-f(x_k)-\frac{\mu}2\|x_k-x_\ast\|^2) + 2\tau \langle x_k-z_k,x_k-x_\ast\rangle.
\end{align*}
We observe that
\[\tau(x_k-z_k) = \tau x_k -((1+\tau) x_k -y_k) = y_k-x_k \]
and obtain
\begin{align*} e_{k+1}^{(1)} &\le e_k^{(1)}+ \|z_{k+1}-z_k\|^2-2\frac{\tau}{\mu} (f(x_k)-f(x_\ast))+ \frac{2}{\mu}\langle \nabla f(x_k),y_k-x_k\rangle\\
&\quad+ 2\tau\langle x_k-z_k,x_k-x_\ast\rangle -2\tau\|x_k-z_k\|^2 -\tau\|x_k-x_\ast\|^2.
\end{align*}
Note that
\begin{align*}2\tau\langle x_k-z_k,x_k-x_\ast\rangle -2\tau\|x_k-z_k\|^2 -\tau\|x_k-x_\ast\|^2 &= -\tau \|x_k-x_\ast-(x_k-z_k)\|^2 - \tau \|x_k-z_k\|^2\\ &= -\tau e_k^{(1)} -\tau\|x_k-z_k\|^2,
\end{align*}
which yields
\begin{align*} 
e_{k+1}^{(1)} &\le (1-\tau) e_k^{(1)}+ \|z_{k+1}-z_k\|^2-2\frac{\tau}{\mu} (f(x_k)-f(x_\ast))+ \frac{2}{\mu}\langle \nabla f(x_k),y_k-x_k\rangle-\tau\|x_k-z_k\|^2.
\end{align*}
Finally, with $e_k^{(2)}=f(y_k)-f(x_\ast)$ we consider the evolution of $e_k$ through
\begin{align*}
e_{k+1} &= \frac{\mu}{2}e_{k+1}^{(1)}+ e_{k+1}^{(2)}\le (1-\tau)\frac{\mu}{2}e_k^{(1)}+\frac{\mu}{2}\|z_{k+1}-z_k\|^2-\tau (f(x_k)-f(x_\ast))+\langle \nabla f(x_k),y_k-x_k\rangle\\ &\quad-\tau\frac{\mu}{2}\|x_k-z_k\|^2 + f(y_{k+1})-f(x_\ast)\\
&\le  (1-\tau)\frac{\mu}{2}e_k^{(1)}+\frac{\mu}{2}\|z_{k+1}-z_k\|^2-\tau (f(x_k)-f(x_\ast))+\langle \nabla f(x_k),y_k-x_k\rangle\\ &\quad-\tau\frac{\mu}{2}\|x_k-z_k\|^2 + f(x_{k})-f(x_\ast)-\frac{1}{2L}\|\nabla f(x_k)\|^2+\{(1-\tau)e_k^{(2)}-(1-\tau)e_k^{(2)}\}  \}\\
&= (1-\tau) e_k -(1-\tau) (f(y_k)-f(x_\ast))-\tau (f(x_k)-f(x_\ast)) + f(x_k)-f(x_\ast)-\frac{1}{2L}\|\nabla f(x_k)\|^2\\
&\quad +\frac{\mu}{2}\|z_{k+1}-z_k\|^2-\tau\frac{\mu}{2}\|x_k-z_k\|^2+\langle \nabla f(x_k),y_k-x_k\rangle\\
&=:  (1-\tau) e_k  + \mathcal R,
\end{align*}
where we have used \eqref{eq:NAM_strongconvex_aux1} in the first inequality and added a zero $(1-\tau) (f(y_k)-f(x_\ast))-(1-\tau) (f(y_k)-f(x_\ast))=0$. It is left to prove that $\mathcal R\le 0$. First note, that $\mathcal R$ simplifies to
\[ \mathcal R = (1-\tau)(f(x_k)-f(y_k)) - \frac1{2L}\|\nabla f(x_k)\|^2 +\frac{\mu}{2}\|z_{k+1}-z_k\|^2-\tau\frac{\mu}{2}\|x_k-z_k\|^2+(1-\tau+\tau)\langle \nabla f(x_k),y_k-x_k\rangle.\]
We apply strong convexity in the form of 
\[\langle \nabla f(x_k),y_k-x_k\rangle \le f(y_k) - f(x_k) -\frac{\mu}{2}\|y_k-x_k\|^2 \]
such that
\begin{align*}
\mathcal R &\le (1-\tau)(f(x_k)-f(y_k))+(1-\tau)(f(y_k)-f(x_k))-\frac{\mu}{2}\|y_k-z_k\|^2\\
&\quad -\frac{1}{2L}\|\nabla f(x_k)\|^2 +\frac{\mu}{2}\|z_{k+1}-z_k\|^2-\tau\frac{\mu}{2}\|x_k-z_k\|^2+\tau\langle \nabla f(x_k),y_k-x_k\rangle \\
&=\tau\langle \nabla f(x_k),y_k-x_k\rangle-(1-\tau)\frac{\mu}{2}\|y_k-x_k\|^2-\frac{1}{2L}\|\nabla f(x_k)\|^2\\
&\quad+\frac{\mu}{2}\|z_{k+1}-z_k\|^2-\tau\frac{\mu}{2}\|x_k-z_k\|^2.
\end{align*}
Recall that we have used $z_{k+1}-z_k = \tau(x_k-z_k)-\frac{\tau}{\mu}\nabla f(x_k)$, which with $\tau(x_k-z_k) = y_k-x_k$ gives
\[\frac{\mu}{2} \|z_{k+1}-z_k\|^2 = \frac{\mu}2\|y_k-x_k\|^2 -\tau \langle \nabla f(x_k),y_k-x_k\rangle + \frac{\tau^2}{2\mu}\|\nabla f(x_k)\|^2, \]
and therefore, we obtain
\begin{align*}
\mathcal R &\le -(1-\tau)\frac{\mu}{2} \|y_k-x_k\|^2 +\frac{\mu}2\|y_k-x_k\|^2 + \tau\langle \nabla f(x_k),y_k-x_k\rangle -\tau \langle \nabla f(x_k),y_k-x_k\rangle\\
&\quad +\frac{\tau^2}{2\mu}\|\nabla f(x_k)\|^2 - \frac{1}{2L}\|\nabla f(x_k)\|^2 - \tau\frac{\mu}{2}\frac{\tau}{\tau}\|x_k-z_k\|^2\\
&=\tau\frac{\mu}{2}\|y_k-x_k\|^2 - \frac{\mu}{2\tau}\|\tau(x_k-z_k)\|^2 + \left(\frac{\tau^2}{2\mu}-\frac{1}{2L} \right)\|\nabla f(x_k)\|^2\\
&= \left(\frac{\tau\mu}{2}-\frac{\mu}{2\tau}\right) \|y_k-x_k\|^2 +\left(\frac{\tau^2}{2\mu}-\frac{1}{2L} \right)\|\nabla f(x_k)\|^2,
\end{align*}
where we have used again that $\tau(x_k-z_k) = y_k-x_k$. With the choice $\tau=\sqrt{\frac{\mu}{L}}\in(0,1)$ we observe that
\[ \frac{\tau^2}{2\mu}-\frac{1}{2L} = \frac{1}{2L}-\frac{1}{2L} = 0\]
and
\[ \frac{\tau\mu}{2}-\frac{\mu}{2\tau} = \frac{\mu}2(\tau-\frac1\tau) \le 0.\]
This proves that $\mathcal R\le 0$ and the assertion follows:
\[e_{k+1} \le (1-\tau)e_k = \left(1-\sqrt{\frac{\mu}{L}}\right)e_k. \]
\end{proof}

\begin{remark}
We can rewrite the error bound for NAM of the previous theorem through
\[e_{k} =\frac{\mu}2\|z_k-x_\ast\|^2+f(y_k)-f(x_\ast) \le (1-\tau)^k e_0 = \left(1-\frac{1}{\sqrt{\kappa}}\right)^k e_0 = \left(\frac{\sqrt{\kappa}-1}{\sqrt{\kappa}}\right)^k e_0.  \]
The corresponding optimal convergence rate of the simple gradient descent scheme was given by
\[ e_{k}^{\mathrm{GD}} :=\|x_k-x_\ast\|^2\le \left(\frac{\kappa-1}{\kappa+1}\right)^{2k} e_0^{\mathrm{GD}}.\]
In Figure~\ref{fig:comp_NAM_GD} we compare both rates of convergence for increasing condition number $\kappa = \frac{L}{\mu}$ illustrating the improvement through NAM.
\end{remark}

\begin{figure}[!htb]
  \centering  
  \includegraphics[width=1\textwidth]{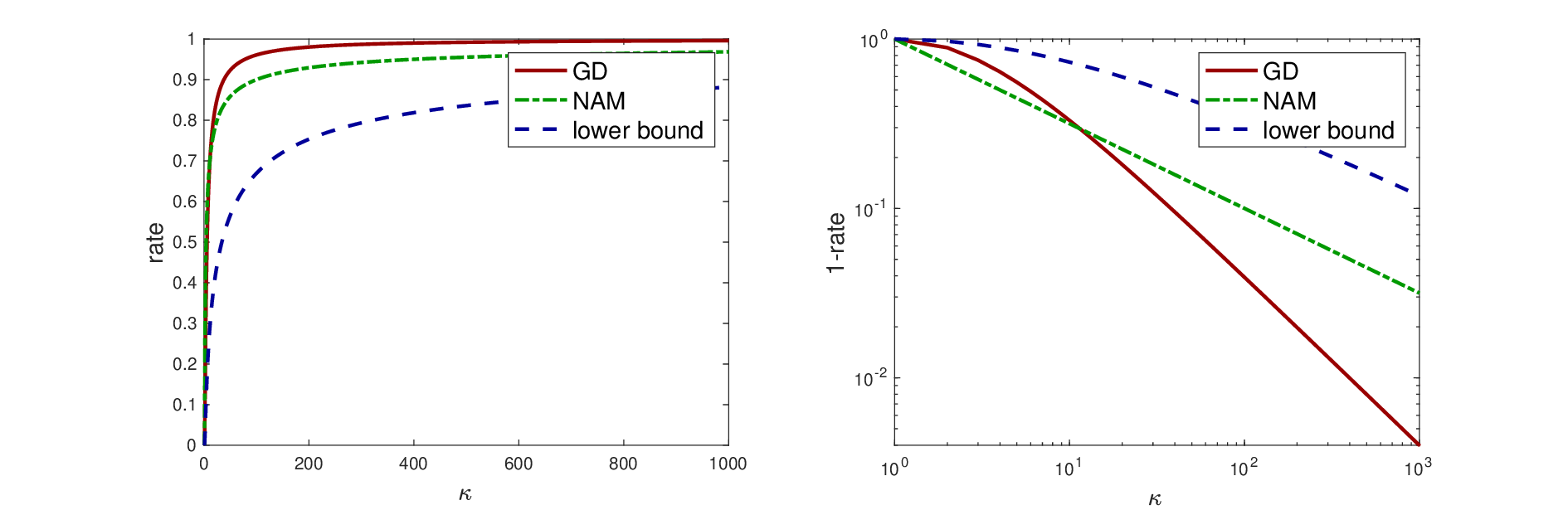}
 \caption{Illustration of the linear convergence rate depending on the condition number $\kappa =\frac{\mu}{L}$ for GD and NAM. The left plot shows the convergence rate $c^{\mathrm{GD}}(\kappa) =\left(\frac{\kappa-1}{\kappa+1}\right)^{2}$ and $c^{\mathrm{NAM}}(\kappa)= \left(\frac{\sqrt{\kappa}-1}{\sqrt{\kappa}}\right)$, whereas the right plot shows the difference to $1$, i.e.~$1-c(\kappa)$, in logarithmic scale. } \label{fig:comp_NAM_GD}
\end{figure}

\chapter{Stochastic gradient methods}\label{ch:sgd}
In this chapter, we introduce stochastic variants of gradient descent. The basic example is stochastic gradient descent (SGD). While deterministic gradient descent uses the full gradient of
the objective function in each iteration, SGD replaces this gradient by a stochastic approximation.

\section{Stochastic gradient descent method}
\subsection{Motivation}
We will start with a motivating example to introduce the problem of minimizing expected and empirical risk. This will motivate the consideration of stochastic variants of gradient descent methods.

\begin{example}
We revisit the regression problem discussed in Chapter~\ref{ch:Intro}, which arises in supervised learning. Recall that we aim to approximate an unknown model 
\[z\mapsto \varphi(z) = y,\quad \varphi:\R^{d_z}\to\R^{d_y} \]
through a parametrized family of functions $g_\theta:\R^{d_z}\to\R^{d_y}$, $\theta\in\Theta$. Given a training data set $\{(z^{(i)},y^{(i)})\}_{i=1}^N$, we have described the training task as optimization problem
\[\min_{\theta\in\Theta}\ f_N(\theta,\{(z^{(i)},y^{(i)})\}_{i=1}^N), \]
where $f_N:\Theta\times (\times_{i=1}^N (\R^{d_z}\times \R^{d_y})\to\R$ denotes the corresponding objective function. In the example of regression, we have considered 
\begin{equation}\label{eq:regression_loss}
f_N(\theta,\{(z^{(i)},y^{(i)})\}_{i=1}^N) = \frac{1}{N}\sum_{i=1}^N \|g_\theta(z^{(i)})-y^{(i)}\|^2 + \mathcal R(\theta),
\end{equation}
where $\mathcal R:\Theta\to\R$ is some regularization function. We aim to incorporate a probabilistic framework in order to introduce (empirical) risk minimization. Let $(\Omega,\mathcal A,\P)$ be our underlying probability space. We model the input and output variable as jointly distributed random variables $(Z,Y)$ on $(\Omega,\mathcal A,\PP)$ with state space $(\R^{d_z}\times \R^{d_y}, \mathcal B(\R^{d_z})\otimes\mathcal B(\R^{d_y}))$. The goal is to find $\theta\in\Theta$ such that $g_\theta$ represents the stochastic model
\[Y = \varphi(Z) [+ \xi],\quad Z\sim\mu_Z,\]
where $\xi$ denotes possible noise. 

\textbf{Challenge:} We assume that distribution $\mu_Z$ of $Z$ and the joint distribution $\mu_{(Z,Y)}$ respectively are unknown. Instead, we assume that we have access to i.i.d.~samples $(Z^{(i)},Y^{(i)})\sim\mu_{(Z,Y)}$.\medskip

We are now interested in how to choose an optimal approximation $g_\theta$. The natural population version of~\eqref{eq:regression_loss} is the expected risk 
\[F(\theta) = \E_{(Z,Y)\sim\mu_{(Z,Y)}}[\|g_\theta(Z)-Y\|^2] + \mathcal R(\theta), \]
where $\E_{(Z,Y)\sim\mu_{(Z,Y)}}$ denotes the expectation w.r.t.~$(Z,Y)$. Given a training data set of i.i.d.~random variables $\{(Z^{(i)},Y^{(i)})\}_{i=1}^N$ distributed according to $(Z^{(1)},Y^{(1)})\sim\mu_{(Z,Y)}$, we can apply a Monte Carlo approximation of the above expectation
\[ \E_{(Z,Y)\sim\mu_{(Z,Y)}}[\|g_\theta(Z)-Y\|^2] \approx \frac{1}{N}\sum_{i=1}^N \|g_\theta(Z^{(i)})-Y^{(i)}\|^2 \]
to construct the empirical objective function
\[F_N(\theta) = \frac{1}{N}\sum_{i=1}^N \|g_\theta(Z^{(i)})-Y^{(i)}\|^2+\mathcal R(\theta), \]
which coincides with \eqref{eq:regression_loss}.
\end{example}

Motivated by this example we introduce the definition of (empirical) risk minimization problems.

\begin{defi}
Let $f:\R^{d}\times\R^p\to \R$ be $\mathcal B(\R^d)\otimes\mathcal B(\R^p)/\mathcal B(\R)$ measurable and $Z:\Omega\to\R^p$ a random variable with distribution $\mu$ such that $\E[|f(x,Z)|]<\infty$ for all $x\in\R^d$.
\begin{enumerate}
\item We define the \textit{expected risk} $F:\R^d\to\R$ as
\[F(x) = \E_{Z\sim\mu}[f(x,Z)]=:\int_{\R^p}f(x,z)\,\mu(\dd z),\quad x\in\R^d. \]
We call the minimization problem
\[\min_{x\in\R^d}\ F(x),\quad F(x) = \E_{Z\sim\mu}[f(x,Z)] \]
\textit{risk minimization problem}.
\item Let $Z_1,\dots,Z_N$ be i.i.d.~random variables with $Z_1\sim\mu$. We define the \textit{empirical risk} $F_N:\R^d\to\R$ by
\[F_N(x) = \frac{1}{N}\sum_{i=1}^N f(x,Z^{(i)}). \]
We call the minimization problem 
\[\min_{x\in\R^d}\ F_N(x),\quad F_N(x) = \frac{1}{N}\sum_{i=1}^N f(x,Z^{(i)}) \]
\textit{empirical risk minimization problem}.
\end{enumerate}
\end{defi}

This chapter will focus on analyzing stochastic optimization methods for solving (empirical and expected) risk minimization problems.  It is important to note that this discussion only addresses a subset of the typical challenges that arise in machine/supervised learning.

\begin{outlook}[Inverse problem perspective]
For a fixed number of data points $N\in\N$ the empirical risk minimization problem is typically ill-posed and it is necessary to include regularization. This is the topic of the research area \textit{inverse problems}. As motivation, we will treat the training task of supervised learning as an inverse problem. Recall, that we are interested to approximate a model $\varphi:\R^{d_z}\to\R^{d_y}$ by a parametrized family of functions $g_\theta:\R^{d_z}\to\R^{d_y}$, $\theta\in\Theta$. Given a training data set $\{(Z^{(i)},Y^{(i)})\}$ we want to minimize the empirical risk
\[\min_{\theta\in\Theta}\ \frac1N\|g_\theta(Z^{(i)})-Y^{(i)}\|^2. \]
An alternative persepective/interpretation is the following. Define a \textit{forward map} $H:\Theta\to\R^{N\cdot d_y}$,
\[H(\theta) := (g_\theta(Z^{(1)},\dots,g_\theta(Z^{(N)}))^\top\in\R^{N\cdot d_y}. \]
With observations $\widehat Y = (Y^{(1)},\dots,Y^{(N)})^\top\in\R^{N\cdot d_y}$, we aim to solve the inverse problem of recovering the parameter $\theta\in\Theta$ such that
\begin{equation}\label{eq:IP}
\widehat Y = H(\theta).
\end{equation}
This problem is typically ill-posed (in the sense of a well-posed problem following Hadamard \cite{H1902}), due to the following reasons:
\begin{enumerate}
\item There might not exist any $\theta\in\Theta$ solving \eqref{eq:IP}.
\item The solution of \eqref{eq:IP} is not necessarily unique, i.e.~there might be $\theta_1,\theta_2\in\Theta$ with $H(\theta_1) = H(\theta_2) = \widehat Y$.
\item The solution of \eqref{eq:IP} might be unstable w.r.t.~changes in $\widehat Y$ (e.g.~due to measurement noise).
\end{enumerate}
Therefore, it is not the best idea to simply solve $\min_{\theta\in\Theta}\ \|H(\theta)-\widehat Y\|^2$. We illustrate the resulting issues in Figure~\ref{fig:ip_existence}--\ref{fig:ip_stability}. In inverse problems a large focus lies in the study of regularization methods, which can, for example, be incorporated as a penalty function. Instead of simply minimizing the data misfit functional, one considers solving the regularized optimization problem 
\[\min_{\theta\in\Theta}\ \|H(\theta)-\widehat Y\|^2 + \mathcal R(\theta),\]
where $\mathcal R: \Theta \to \R$ is a regularization function.
\end{outlook}

\begin{figure}[!htb]
  \centering \includegraphics[width=0.45\textwidth]{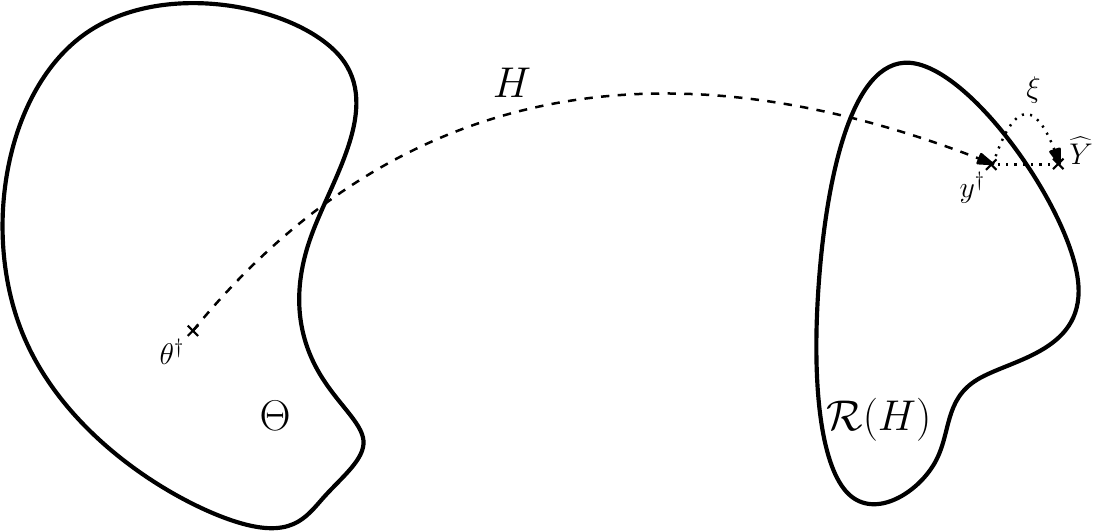}
 \caption{Illustration of ill-posedness through observational noise. The occurrence of noise might shift the observed data outside of the range of the forward map $H$. } \label{fig:ip_existence}
\end{figure} 
\begin{figure}[!htb]
  \centering \includegraphics[width=0.45\textwidth]{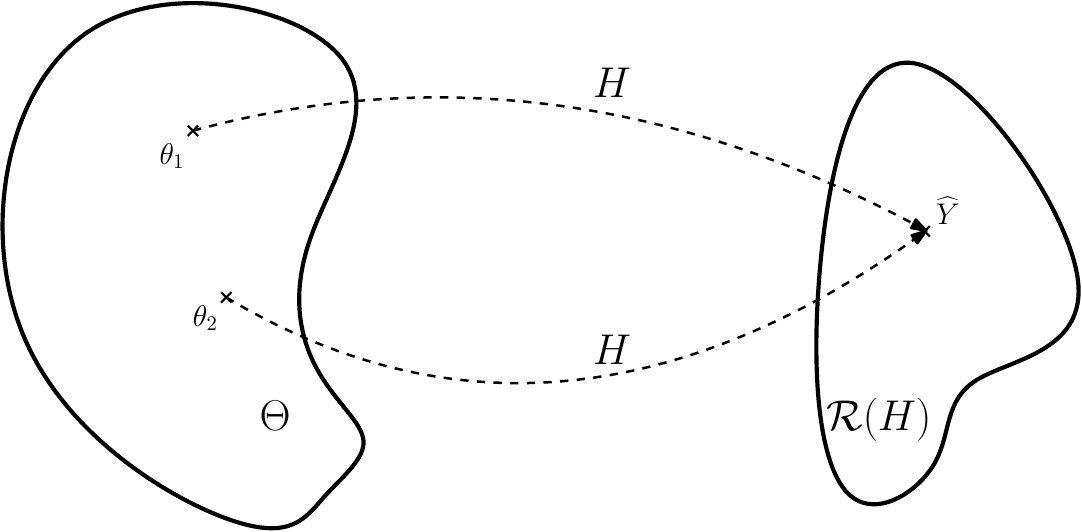}
 \caption{Illustration of ill-posedness through multiple solutions. Two different parameters $\theta_1,\theta_2\in\Theta$ might map onto the observed data $\widehat Y$.} \label{fig:ip_uniqueness}
\end{figure} 
\begin{figure}[!htb]
  \centering \includegraphics[width=0.45\textwidth]{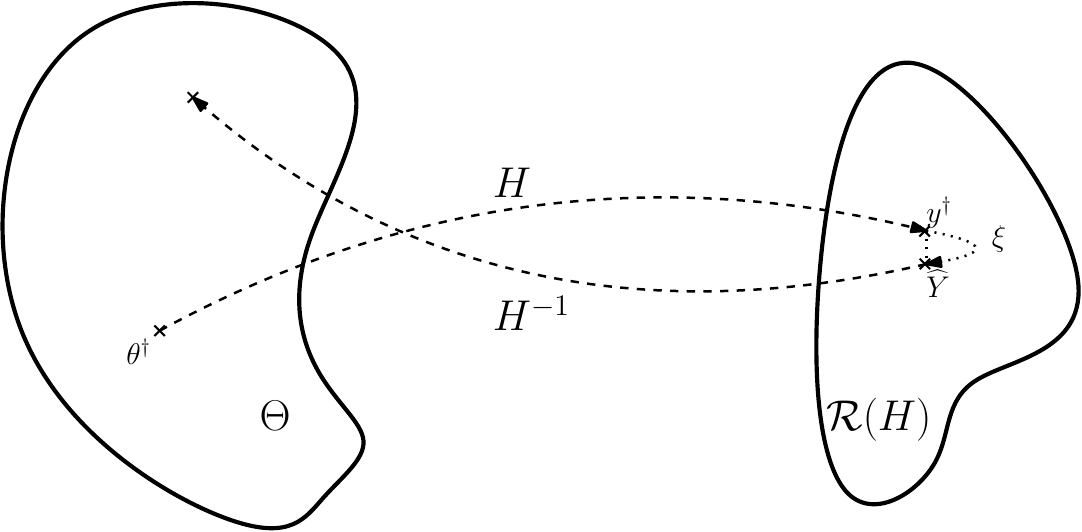}
 \caption{Illustration of ill-posedness through discontinuity. Even if $H$ is invertible, the instability may occur in the solution of the inverse problem resulting from possible discontinuity of the inverse operator.} \label{fig:ip_stability}
\end{figure} 

\begin{outlook}[Statistical learning perspective]
For example, in the area of statistical learning among other questions one is interested in the consistency of solutions of the empirical risk minimization problem. Let $\widehat X_N$ be the minimizer of the empirical risk and $x_\ast$ be the minimizer of the corresponding expected risk (provided both exist), then we can decompose the error
\[F(\widehat X_N) - F(x_\ast) = F(\widehat X_N) -F_N(\widehat X_N) + F_N(\widehat X_N)- F_N(x_\ast) + F_N(x_\ast) - F(x_\ast) \le 2 \sup_{x\in\R^d}\ |F_N(x)-F(x)|. \]
We emphasize that $F_N$ as function depends on the random variables $Z^{(1)},\dots, Z^{(N)}$ and is therefore random. Hence, the minimizer $\widehat X_N$ itself is a random variable. In statistical learning theory one is concerned with questions such as the consistency of $\widehat X_N$ for number of data points $N$ approaching infinity. 
\end{outlook}

In the following chapter we are going to study optimization methods for solving the expected and empirical risk minimization problem. We do not consider questions around generalization and regularization, which are beyond the scope of this lecture course.

\subsection{The stochastic gradient descent algorithm}
We now want to introduce and analyze the stochastic variant of the gradient descent method for solving the expected and empirical risk minimization problem. For a comprehensive overview of stochastic gradient methods, the interested reader may refer to \cite{Bottou2018,garrigos2023handbook,ruder2017overview}.

In the following, let $(\Omega,\mathcal A,\P)$ be the underlying probability space, $Z:\Omega\to\R^p$ be a random variable on $(\Omega,\mathcal A,\P)$ with distribution $\mu_Z$. We are interested in solving the optimization problem
\[\min_{x\in\R^d}\ F(x), \]
where the objective function $F:\R^d\to\R$ is defined as the expectation function in form
\[F(x) = \E_{Z\sim\mu_Z}[f(x,Z)] = \int_{\R^p} f(x,z)\mu_Z(\dd z),\quad x\in\R^d, \]
for a function $f:\R^d\times \R^p \to\R$. Throughout this lecture course we make the following assumption:
\begin{ass}\label{ass:stochapprx}
\begin{enumerate}
\item The function $f:\R^{d}\times\R^p\to\R$ is $\mathcal B(\R^d)\otimes\mathcal B(\R^p)/\mathcal B(\R)$-measurable.
\item For every $z\in\R^p$ the function $x\mapsto f(x,z)$ is continuously differentiable.
\item For every $x\in\R^d$ we have \[\E[|f(x,Z)| + \|\nabla_x f(x,Z)\|]<\infty\]
and \[\E[\|\nabla_x f(x,Z)-\E[\nabla_x f(x,Z)]\|^q]\le b(1+\|x\|^q)  \]
for some $q\ge1$ and $b>0$.
\end{enumerate}
\end{ass}

In order to apply the gradient methods introduced in the previous sections, we run into the question: How do we compute the derivative of $F$? More fundamentally, we can ask under which conditions $F$ is differentiable?

The following result from \cite{Jentzen2020} addresses the question about differentiability of $F$ and illustrates that we can use the random variable $Z$ to construct an unbiased estimator of $\nabla_x F(x)$ for every fixed $x\in\R^d$. The result provides confirmation that under Assumption~\ref{ass:stochapprx} we are allowed to interchange derivative and expectation for the computation of $\nabla_x F(x)$.

\begin{lemma}[Lemma 4.8 in \cite{Jentzen2020}]\label{lem:condexp}
Suppose Assumption~\ref{ass:stochapprx} is satisfied, then it holds true that
\begin{enumerate}
\item the function $F(x) =\E[f(x,Z)]$ is continuously differentiable,
\item $\nabla_x f(x,Z)$ is an unbiased estimator of $\nabla_x F(x)$ for every $x\in\R^d$, i.e.~it holds true that \[\nabla_x F(x) = \E[\nabla_x f(x,Z)]. \]
\end{enumerate}
\end{lemma}

The main consequence of Lemma~\ref{lem:condexp} for the analysis of SGD is the following:
although $\nabla_x f(X_k,Z_{k+1})$ is random, it is an unbiased estimator of the full gradient conditionally on the past,
\[
\E[\nabla_x f(X_k,Z_{k+1})\mid \mathcal F_k] =
\nabla_x F(X_k).
\]
This is the key property that allows us to interpret SGD as gradient descent perturbed by a martingale difference noise.
Each update guides the current iteration in the direction of a (stochastic) approximation of the negative gradient. Since $\nabla_x f(x,Z)$ is an unbiased estimator of $\nabla_x F(x)$, we expect the scheme to perform well on average. As part of this lecture course, we will verify this expectation. The algorithm is formulated in Algorithm~\ref{alg:SGD}. Throughout this chapter, we assume the following scenario.
\begin{ass}\label{ass:iid_sample}
We assume that we have access to a sequence $(Z_k)_{k\in\N}$ of i.i.d. random variables with common distribution $\mu_Z$.
\end{ass}

\begin{algorithm}[htb!]
\begin{algorithmic}[1]
\State \textbf{Input:} \begin{itemize}
 \item loss function $f:\R^d\times \R^p\to\R$
 \item initial random variable $X_0:\Omega\to\R^d$
 \item sequence of step sizes positive $(\alpha_k)_{k\in\N}$ (deterministic or $\mathcal F$-adapted)
 \item sequence of i.i.d.~random variables $(Z_k)_{k\in\N}$ with $Z_1\sim\mu_Z$. 
 \end{itemize}
 \State set $k=0$
\While{"convergence/stopping criterion not met"}
	\State approximate the gradient $\nabla_x F(X_k)$ through
	\[G_k = \nabla_x f(X_k,Z_{k+1}) \]
	\State set $X_{k+1} = X_k -\alpha_k G_k$, $k \mapsto k+1$
\EndWhile
\end{algorithmic}
 \caption{Stochastic gradient descent method (SGD)}\label{alg:SGD}
\end{algorithm}

\begin{remark}
In Algorithm~\ref{alg:SGD}, the random variable $Z_{k+1}$ is independent of the past iterates $X_0,\dots,X_k$. In the definition of $G_k$ we have used the index $k+1$ for $Z_{k+1}$ such that $X_k$ remains measurable w.r.t.~$\sigma(Z_m,m\le k)$ for all $k\ge1$. Indeed, we can then consider one of the filtrations $\mathcal F_k^X = \sigma(X_m,0\le m\le k)$ or $ \mathcal F_k^{Z}= \sigma(X_0,Z_m,m\le k)$.
\end{remark}

\begin{remark}
We note that Algorithm~\ref{alg:SGD} in practice is often used to minimize objective functions in the form of empirical risks,
\[F_N(x) = \frac1N\sum_{i=1}^N f(x,z^{(i)}) = \E_{Z\sim\widehat \mu_N}[f(x,Z)], \]
where $\widehat \mu_N=\frac1N\sum_{i=1}^N \delta_{z^{(i)}}$ denotes the empirical measure over the fixed data set $\{z^{(i)},i=1,\dots,N\}$. To apply SGD to this finite-sum objective, one samples in each iteration an index $\mathfrak i_k$ uniformly from $\{1,\dots,N\}$, independently of the past, and uses
\[
G_k=\nabla_x f(x_k,z^{(\mathfrak i_k)})
\]
as an unbiased estimator of $\nabla F_N(x_k)$. Mini-batch variants sample a random subset $\mathfrak I_k\subset\{1,\dots,N\}$ and average the corresponding gradients. We describe this scheme in Algorithm~\ref{alg:SGD2}.
\end{remark}

\begin{algorithm}[htb!]
\begin{algorithmic}[1]
\State \textbf{Input:} \begin{itemize}
 \item loss function $f:\R^d\times \R^p\to\R$
 \item initial random variable $X_0:\Omega\to\R^d$
 \item sequence of step sizes $(\alpha_k)_{k\in\N}$, $\alpha_k>0$ (deterministic or $\mathcal F$-adapted)
 \item fixed realization of fixed deterministic data set $\{z^{(i)}\}_{i=1}^N$ with $z^{(i)}\in\R^p$. 
 \end{itemize}
 \State set $k=0$
\While{"convergence/stopping criterion not met"}
	\State generate independently $\mathfrak i_{k+1}\sim\mathcal U(\{1,\dots,N\})$
	\State approximate the gradient $\nabla_x F_N(X_k)$ through
	\[G_k = \nabla_x f(X_k,z^{\mathfrak i_{k+1}}) \]
	\State set $X_{k+1} = X_k -\alpha_k G_k$, $k \mapsto k+1$
\EndWhile
\end{algorithmic}
 \caption{Stochastic gradient descent method with finite data}\label{alg:SGD2}
\end{algorithm}

While Algorithm~\ref{alg:SGD} generates a new independent realization of $Z$ in each iteration, Algorithm~\ref{alg:SGD2} first fixes the number of realized independent samples of $Z$ and then randomly iterates through this data set during SGD. The randomness in Algorithm~\ref{alg:SGD2} occurs through the realization of the random indices $(\mathfrak i_k)_{k\in\N}$. For both algorithms, the resulting iteration $(X_k)_{k\in\N}$ is a stochastic process, which is path-wise constructed via $(Z_k)_{k\in\N}$ (Algorithm~\ref{alg:SGD}) and $(\mathfrak i_k)_{k\in\N}$ (Algorithm~\ref{alg:SGD2}) respectively. In both cases, we can view the stochastic process as an adapted process with respect to the natural filtration
\[\mathcal F_k^X = \sigma(X_m,m\le k)\quad \text{or}\quad \mathcal F_k^Z = \sigma(X_0,\ Z_m,m\le k) \]
and 
\[\mathcal F_k^X = \sigma(X_m,m\le k)\quad\text{or}\quad \mathcal F_k^{\mathfrak i} = \sigma(X_0,\ \mathfrak i_m,m\le k)\,.\]
This filtration will be relevant when analyzing the convergence behavior of SGD, where we take the expectation conditioned on the information from the past. 

We will first discuss SGD from an stochastic approximation perspective motivated by the Robbins \& Monro algorithm \cite{RM1951}. 
\begin{outlook}(Robbins \& Monro Algorithm)  
We consider a brief outlook to the original stochastic approximation method introduced in \cite{RM1951} aiming to root-finding. The authors considered the following question: Given a family of real-valued random variables $(Y_x)_{x\in\R}$ one can define the expectation function (in $x$)
\[M(x) = \E[Y_x] = \int_R y\, \mu(dy;x) \]
where $\mu(\cdot; x)$ denotes the distribution of $Y_x$. Given $z\in\R$, the aim is construct an algorithm to find the (unique) solution of the equation
\[M(x) = z.\]
The challenging aspect in this question is the unknown expectation function $M$. However, the authors assume that one is able to sample independently from $(Y_x)_{x\in\R}$ (or from  $\mu(\cdot; x)$ respectively). The Robbins \& Monro algorithm in its original from iterates through 
\[X_{k+1} = X_k + \alpha_k (z-Y_k), \]
where $Y_k$, $k\in\N$ are independent random variables with distribution $\mu(\cdot; X_k)$ and $(\alpha_k)_{k\in\N}$ is a sequence of step sizes $\alpha_k>0$. Robbins \& Monro proved convergence of $(X_k)_{k\in\N}$ in $L^2$ towards $z$ under certain assumptions on $M$ and $(\mu(\cdot; x))_{x\in\R}$. Slightly later Blum \cite{Blum1954} (1954) proved almost sure convergence under additional condition on the sequence of step sizes
\[\sum_{k=0}^\infty \alpha_k =\infty\quad \text{and}\quad \sum_{k=0}^\infty \alpha_k^2 <\infty. \]
\end{outlook}

We will apply the Robbins \& Siegmund Theorem based on Doob's supermartingale convergence theorem in order to derive an almost sure convergence result for SGD. However, the direct application of this theorem only leads to an asymptotic convergence result without quantification of the convergence speed. A more careful analysis conducted in Section~\ref{sec:as_rates} allows to describe almost sure convergence rates.

In principle, we can view the approximation $G_k$ in Algorithm~\ref{alg:SGD} and \ref{alg:SGD2} as a form of Monte Carlo approximation, which quantifies an error to the (exact/full) gradient descent method. A smaller variance of the estimator $G_k$ suggests a better convergence behavior of SGD. Therefore, in Section~\ref{sec:SGD_varred}, we will consider variance reduced versions of the SGD method.

In order to analyze SGD, we rewrite the iterative scheme as follows
\begin{align*}
X_{k+1} = X_k -\alpha_k \nabla_x f(X_k,Z_{k+1}) = X_k - \alpha_k\nabla_x F(X_k) + \alpha_k \left(\nabla_x F(X_k) - \nabla_x f(X_k,Z_{k+1})\right)\\
=:  X_k - \alpha_k\nabla_x F(X_k) + \alpha_k M_{k+1}.
\end{align*}
Recall, that we consider $(X_k)_{k\in\N}$ as an adapted process with respect to its natural filtration
$\mathcal F_k = \sigma(X_0,\ Z_m,m\le k)$. In the next section, it will turn out, that the  process $(M_k)_{k\in\N}$ satisfies
\begin{equation}\label{eq:unbiased_filtration}
\E[M_{k+1}\mid \mathcal F_k] = \E[\nabla_x F(X_k) - \nabla_x f(X_k,Z_{k+1})\mid \mathcal F_k] = 0,
\end{equation}
since we always assume that we are able to apply Lemma~\ref{lem:condexp} to derive $\nabla_x F(x) = \E[ \nabla_x f(x,Z)]$. Note that this property holds for fixed $x\in\R^d$ and we will extend this behavior to the conditional expectation $\E[ \cdot \mid \mathcal F]$, where it will be particularly relevant that $(X_k)_{k\in\N}$ is $\mathcal F$-adapted and $Z_{k+1}$ is independent of $\mathcal F_k$.

\subsection{Technical detail: Factorization of conditional expectation}
In this section, we will discuss one important property which is needed for the verification of $\nabla_x f(X_k,Z_{k+1})$ as an unbiased estimator of $\nabla_x F(X_k)$ conditioned on the iterations of SGD represented through the natural filtration $(\mathcal F_k)_{k\in\N}$. In the literature of SGD it is frequently used that $\E[\nabla_x F(X_k)-\E[\nabla_x f(X_k,Z_{k+1})\mid \mathcal F_k]=0$, since it is assumed that $\nabla_x f(x,Z_{k+1})$ is an unbiased estimator of $\nabla_x F(x)$ for every $x\in\R^d$. Although this identity is intuitive, in general this implication is non-trivial and we will need to investigate some additional work. The verification will require technical tools from measure theory such as the monotone class theorem in order to derive some form of factorization of the conditional expectation. We will prove the following lemma which can be found in \cite[Proposition~1.12]{da1992stochastic} and \cite[Corollary~2.9]{Jentzen2020}.

\begin{lemma}\label{lem:condexp_factorization}
Let
\begin{itemize}
\item $(\Omega,\mathcal A,\mathbb P)$ be a probability space, $\mathcal F\subset \mathcal A$ be some sub-$\sigma$-algebra of $\mathcal A$ on $\Omega$,
\item $(\mathbb Y_1,\mathcal Y_1)$ and $(\mathbb Y_2,\mathcal Y)_2$ be measurable spaces, $Y_1:\Omega\to\mathbb Y_1$ be $\mathcal A$/$\mathcal Y_1$-measurable and independent of $\mathcal F$, and $Y_2:\Omega\to\mathbb Y_2$ be $\mathcal F$/$\mathcal Y_2$-measurable.
\item $\Phi:\mathbb Y_1\times \mathbb Y_2\to\R$ be $(\mathcal Y_1\otimes \mathcal Y_2)/\mathcal B(\R)$-measurable with $\E[|\Phi(Y_1,Y_2)|]<\infty$, and $\E[|\Phi(Y_1,y_2)|]<\infty$ for all $y_2\in\mathbb Y_2$.
\end{itemize}
With $\varphi:\mathbb Y_2\to\R$ defined by $\varphi(y_2) = \E[\Phi(Y_1,y_2)]$, $y_2\in\mathbb Y_2$  we have
\begin{enumerate}
\item $\varphi$ is $\mathcal Y_2$/$\mathcal B(\R)$-measurable,
\item for all $A\in\mathcal F$ it holds
\(\E[\Phi(Y_1,Y_2)\mathds{1}_A] = \E[\varphi(Y_2)\mathds{1}_A]. \)
\end{enumerate}
\end{lemma}

\begin{remark}
We consider Lemma~\ref{lem:condexp_factorization} as a generalization of the following rule for conditional expectation. Let $Y_1,Y_2$ be real-valued random variables on $(\Omega,\mathcal A,\mathbb P)$, $\mathcal F\subset \mathcal A$ be some sub-$\sigma$-algebra and $Y_1$ be independent of $\mathcal F$. Then we can compute the conditional expectation
\[\E[Y_1\cdot Y_2\mid\mathcal F] = \E[Y_1]\cdot \E[Y_2\mid \mathcal F].\] 
For $Y_2$ being $\mathcal F$-measurable, we deduce the assertion of Lemma~\ref{lem:condexp_factorization} with $\Phi(y_1,y_2) = y_1\cdot y_2$.

We can apply Lemma~\ref{lem:condexp} and Lemma~\ref{lem:condexp_factorization} to verify \eqref{eq:unbiased_filtration}. When applying Lemma~\ref{lem:condexp_factorization} to SGD, the random variables $Z_k$ will take the role of $Y_1$ and the random variables $X_k$ will take the role of $Y_2$. Note that for simplicity, we work with a real-valued mapping $\Phi$, but the result can be extended to $\R^d$-valued mappings.
\end{remark}

\begin{proof}[Proof of Lemma~\ref{lem:condexp_factorization}]
Firstly, note that it follows from Fubini's theorem that $\varphi$ is $\cY_2/\cB(\R)$-measurable. Let us start with a brief outline of the proof for the second assertion:
\begin{enumerate}[label = {\textbf{Step \arabic*.}},align=left]
\item We prove that the assertion holds for $\Phi(y_1,y_2) = \mathds{1}_B(y_1,y_2)$ with arbitrary $B\in\mathcal Y_1\otimes \mathcal Y_2$.
\item We use \textbf{step 1} in order to prove the assertion for step functions \[\Phi_N(y_1,y_2) = \sum_{k=1}^N d_k\mathds{1}_{D_k}(y_1,y_2)\] for $d_k\ge0$ and $D_k\in\mathcal Y_1\otimes\cY_2$.
\item We prove the assertion for positive functions $\Phi$, which can be expressed as limit of monotonically increasing step functions.
\item We finish the proof by splitting $\Phi$ into positive and negative part. 
\end{enumerate}

We go through all of the steps.

\textbf{Step 1:} Let $B\in \cX\otimes \cY$ be arbitrary and consider $\Phi(y_1,y_2) = \mathds{1}_B(y_1,y_2)$ as well as $\varphi(y_2) = \E[\mathds{1}_B(Y_1,y_2)]$. We will apply the strategy discussed in Remark~\ref{rem:monotoneclass} in order to verify that the property
\[\E[\mathds{1}_B(Y_1,Y_2)\mathds{1}_A] = \E[\varphi(Y_2)\mathds{1}_A],\quad \text{for all}\ A\in\cF \]
is satisfied for any $B\in\cY_1\otimes \cY_2$. Following the strategy described in \Cref{app:measuretheory}, we define the set 
\[\cM = \{D\in\cY_1\otimes \cY_2\mid \E[\mathds{1}_D(Y_1,Y_2)\mathds{1}_A] = \E[\left(\E[\mathds{1}_D(Y_1,y_2)]\right)|_{y_2=Y_2} \mathds{1}_A],\quad \text{for all}\ A\in\cF\}, \]
which will be the candidate for the Dynkin system. Moreover, we consider the $\cap$-stable generator
\[ \cE = \{S\in \cY_1\otimes \cY_2\mid S=E_1\times E_2,\ E_1\in\cY_1,\ E_2\in\cY_2\} \]
with $\sigma(\cE) = \cY_1\otimes \cY_2$. We need to prove that $\cM$ is a Dynkin system, and $\cE\subset \cM$. We begin with the latter property. Let $E_1\in\cY_1$ and $E_2\in\cY_2$, then we have for all $A\in \cF$ that
\begin{align*}
\E[\mathds{1}_{E_1\times E_2}(Y_1,Y_2) \mathds{1}_A] &= \P(\{Y_1\in E_1\} \cap \{Y_2\in E_2\} \cap A)\\ &\overset{Y_1\ \text{indep.}\ \cF}= \P(\{Y_1\in E_1\}) \P(\{Y_2\in E_2\} \cap A)\\
&=\E[ \E[\mathds{1}_{E_1}(Y_1)] \mathds{1}_{E_2}(Y_2)\mathds{1}_A]\\
&=\E[ \left(\E[\mathds{1}_{E_1}(Y_1)\mathds{1}_{E_2}(y_2)]\right)|_{y_2=Y_2} \mathds{1}_A]\\
&=\E[ \left(\E[\mathds{1}_{E_1\times E_2}(Y_1,y_2)]\right)|_{y_2=Y_2} \mathds{1}_A],
\end{align*}
which proves that $E_1\times E_2\in \cM$, i.e.~$\cE\subset \cM$. Next, it is easy to verify, that $\cM$ is a Dynkin system (we will skip the details here). 
Therefore, using Theorem~\ref{thm:monclass}, we imply that
\[\cY_1\otimes \cY_2 = \sigma(\cE) \overset{\cE\ \cap-\text{stable}}{=} d(\cE) \overset{\cE\subset\cM}{\subset} d(\cM) \overset{\cM\ \text{Dynkin system}}= \cM \subset \cY_1\otimes \cY_2, \]
which means that for all $B\in\cY_1\otimes \cY_2$ we have
\[\E[\mathds{1}_B(Y_1,Y_2)\mathds{1}_A] = \E[\varphi(Y_2)\mathds{1}_A],\quad \text{for all}\ A\in\cF, \]
which finishes \textbf{step 1}.

\textbf{Step 2:} Let $\Phi(y_1,y_2) = \sum_{k=1}^N d_k \mathds{1}_{D_k}(y_1,y_2)$ for $d_k\ge0$ and $D_k\in\cY_1\otimes \cY_2$. We apply linearity of the expectation and \textbf{step 1} to obtain
\begin{align*}
\E[\Phi(Y_1,Y_2)\mathds{1}_A] &= \sum_{k=1}^N d_k \E[\mathds{1}_{D_k}(Y_1,Y_2)\mathds{1}_A]\\ &\overset{\textbf{step 1}}=\sum_{k=1}^N d_k \E[\left(\E[\mathds{1}_{D_k}(Y_1,y_2) \right)|_{y_2=Y_2}\mathds{1}_A]\\
&=\E[\left(\E[\sum_{k=1}^N d_k\mathds{1}_{D_k}(Y_1,y_2) \right)|_{y_2=Y_2}\mathds{1}_A] = \E[\left(\E[\Phi(Y_1,y_2)]\right)|_{y_2=Y_2} \mathds{1}_A] = \E[\varphi(Y_2)\mathds{1}_A].
\end{align*}

\textbf{Step 3:} Let $\Phi:\mathbb Y_1\times \mathbb Y_2\to[0,\infty)$ be $(\cY_1\otimes\cY_2)/\cB([0,\infty))$-measurable, then we can find a monotonically increasing sequence of step functions $(\Phi_n)_{n\in\N}$,
\[\Phi_n(y_1,y_2) = \sum_{k=1}^{N_n} d_k^{(n)} \mathds{1}_{D_k^{(n)}}(y_1,y_2), \]
such that $\lim_{n\to\infty} \Phi_n(y_1,y_2) = \Phi(y_1,y_2)$ point-wise. Monotonicity is to understand in the sense of
$\Phi_n(y_1,y_2) \le \Phi_{n+1}(y_1,y_2)$
for all $(y_1,y_2)\in\mathbb Y_1\times \mathbb Y_2$ and all $n\in\N$. We apply monotone convergence and the findings of \textbf{step 2} to imply
\begin{align*}
\E[\Phi(Y_1,Y_2)\mathds{1}_A] \E[ \lim_{n\to\infty} \Phi_n(Y_1,Y_2) \mathds{1}_A] &= \lim_{n\to\infty} \E[\Phi_n(Y_1,Y_2)\mathds{1}_A]\\ &=  \lim_{n\to\infty} \E[\left(\E[\Phi_n(Y_1,y_2)]\right)|_{y_2=Y_2}\mathds{1}_A] \\ &= \E[\left(\E[\Phi(Y_1,y_2)]\right)|_{y_2=Y_2}\mathds{1}_A] = \E[\varphi(Y_2)\mathds{1}_A],
\end{align*}
for all $A\in\cF$.

\textbf{Step 4:} Let $\Phi:\mathbb Y_1\times\mathbb Y_2\to\R$ be $(\cY_1\otimes\cY_2)/\cB(\R)$-measurable and consider the decomposition 
\[\Phi(y_1,y_2) = \Phi^+(y_1,y_2) - \Phi^-(y_1,y_2) \]
for positive and $(\cY_1\otimes\cY_2)/\cB(\R)$-measurable functions $\Phi^+,\Phi^-$. Moreover, we define
\[\varphi^+(y_2) = \E[\Phi^+(Y_1,y_2)] \le \E[|\Phi(Y_1,y_2)|]<\infty,\quad \varphi^-(y_2) = \E[\Phi^-(Y_1,y_2)] \le \E[|\Phi(Y_1,y_2)|]<\infty \]
where we can write 
\[ \varphi(y_2)= \E[\Phi(Y_1,y_2)] = \E[\Phi^+(Y_1,y_2) - \Phi^-(Y_1,y_2)] = \varphi^+(y_2)-\varphi^-(y_2).\]
We apply linearity of the expectation and the findings of \textbf{step 3} to deduce
\begin{align*}
\E[\Phi(Y_1,Y_2)\mathds{1}_A]= \E[\Phi^+(Y_1,Y_2)\mathds{1}_A] - \E[\Phi^-(Y_1,Y_2)\mathds{1}_A] = \E[\varphi^+(Y_2)\mathds{1}_A]-\E[\varphi^-(Y_2)\mathds{1}_A] = \E[\varphi(Y_2)\mathds{1}_A].
\end{align*}
\end{proof}

\section{Almost sure convergence for non-convex objective functions}\label{sec:SGDas}

In the following section, we will analyze the almost sure convergence behavior of SGD. We will make use of an almost sure convergence theorem of Robbins \& Siegmund \cite{RS1971} which is based on Doob's supermartingale convergence theorem. We refer to Appendix~\ref{app:martingales} for a brief summary on martingales. In the analysis of SGD, the objective value does not decrease deterministically in every iteration. However, it often satisfies a descent inequality after taking conditional expectations with respect
to the past. The Robbins--Siegmund theorem is designed precisely for this situation: it allows us to conclude almost sure convergence from such conditional descent estimates, provided that the
error terms are summable.

\begin{thm}[Robbins \& Siegmund]\label{thm:RS}
Let $(\Omega,\mathcal A,\mathcal F,\mathbb P)$ be a filtered probability space, $(Z_k)_{k\in\N}$, $(A_k)_{k\in\N}$, $(B_k)_{k\in\N}$ and $(C_k)_{k\in\N}$ be non-negative and $\mathcal F$-adapted stochastic processes, such that
\[\sum_{k=0}^\infty A_k<\infty \quad \text{and}\quad \sum_{k=0}^\infty B_k<\infty\]
almost surely. Moreover, suppose
\begin{equation} \label{eq:RS_recursion}
\E[Z_{k+1}\mid \mathcal F_k] \le Z_k (1+A_k) + B_k - C_k.
\end{equation}
Then 
\begin{enumerate}
\item there exists an almost surely finite random variable $Z_\infty$ such that $Z_k\to Z_\infty$ almost surely for $k\to\infty$,
\item it holds true that $\sum_{k=0}^\infty C_k<\infty$ almost surely. 
\end{enumerate}
\end{thm}

\begin{proof}
We want to apply Doob's martingale convergence theorem, Theorem~\ref{thm:martingaleconv}, in order to prove the assertion. Therefore, we are going to construct a supermartingale based on the stated stochastic processes.

\textbf{Step 1 (construction of a supermartingale):} We define the auxiliary random variables
\[\widehat Z_k = \frac{Z_k}{\prod_{i=0}^{k-1} (1+ A_i)}, \quad \widehat B_k = \frac{B_k}{\prod_{i=0}^{k} (1+ A_i)},\quad  \widehat C_k = \frac{C_k}{\prod_{i=0}^{k} (1+ A_i)} \]
and observe that
\begin{equation}\label{eq:aux_RS1}
\begin{split}
\mathbb E[\widehat Z_{k+1}\mid \mathcal F_k] = \left(\prod_{i=0}^{k} (1+A_i)^{-1}\right) \E[Z_{k+1}\mid \mathcal F_k] &\leq \left(\prod_{i=0}^{k}(1+A_i)^{-1}\right)\left(Z_k(1+A_k) + B_k - C_k \right)\\ & = \widehat Z_k + \widehat B_k - \widehat C_k.
\end{split}
\end{equation}
Our candidate for the supermartingale is
\[M_k = \widehat Z_k - \sum_{i=0}^{k-1}(\widehat B_i - \widehat C_i),\]
for which we observe
\begin{align*}
\E[M_{k+1}\mid \mathcal F_k]  = \E[\widehat Z_{k+1}\mid\mathcal F_k] - \sum_{i=0}^k \left(\E[\widehat B_i\mid \mathcal F_k] - \E[\widehat C_i\mid \mathcal F_k] \right)
&\le \widehat Z_k + \widehat B_k- \widehat C_k - \sum_{i=0}^k (\widehat B_i - \widehat C_i)\\
&= \widehat Z_k - \sum_{i=0}^{k-1} (\widehat B_i- \widehat C_i) = M_k,
\end{align*}
where we have used \eqref{eq:aux_RS1} and that $\widehat B_i$, $\widehat C_i$ are $\mathcal F_k$-measurable for $i\le k$. In order to apply Doob's martingale convergence theorem, we need to verify $\sup_{k\in\N}\ \E[M_k^{-}]<\infty$. Since in general, it is not obvious that this property will hold, we introduce a localization

\textbf{Step 2 (localization):} We define the stopping time $\tau_\varepsilon = \inf\{k\ge 1: \ \sum_{i=0}^k \widehat B_i >\varepsilon\}$ for $\varepsilon>0$. Since $(B_k)_{k\in\N}$ is $\mathcal F$-adapted, $\tau_\varepsilon$ is a stopping time with respect to $\mathcal F$. Moreover, $(M_{k\wedge \tau_\varepsilon})_{k\in\N}$ is still a supermartingale, and additionally satisfies 
\[M_{k\wedge \tau_\varepsilon} = \widehat Z_{k\wedge \tau_\varepsilon} - \sum_{i=0}^{(k\wedge \tau_\varepsilon)-1}\widehat B_i + \sum_{i=0}^{(k\wedge\tau_\varepsilon)-1} \widehat C_i\ge -\sum_{i=0}^{(k\wedge\tau_\varepsilon)-1}\widehat B_i \ge -\varepsilon, \]
since $\sum_{i=0}^{(k\wedge\tau_\varepsilon)-1}\widehat B_i\le \varepsilon$ by construction of the stopping time $\tau_\varepsilon$. Since $(M_{k\wedge \tau_\varepsilon})_{k\in\N}$ is uniformly bounded from below (and due to the monotonic decrease of the expectation for supermartingales) we obtain
\[\sup_{k\in\N}\ \E[|M_{k\wedge\tau_\varepsilon}|]<\infty. \]
We are now ready to apply Theorem~\ref{thm:martingaleconv} to find a integrable random variable $M^\varepsilon_\infty$ with $\lim_{k\to\infty} M_{k\wedge\tau_{\varepsilon}} = M_\infty^\varepsilon$ almost surely. Next, we have to remove the stopping time.

\textbf{Step 3 (remove localization):} Let $(\varepsilon_n)_{n\in\N}$ be an increasing sequence with $\lim_{n\to\infty} \varepsilon_n = \infty$. First note, that for each $n\in\N$ we have
\[\lim_{k\to\infty} M_{k\wedge\tau_{\varepsilon_n}}(\omega) = M_\infty^{\varepsilon_n}(\omega) \]
for almost all $\omega\in\Omega$. We observe that for each $\omega\in\Omega$ with 
$\sum_{i=0}^\infty\widehat B_i(\omega) <\infty$ there exists $N\in\N$ such that $\omega\in\{\tau_{\varepsilon_N}=\infty\}$, i.e.~for this $\omega$ it holds
\[M_{k\wedge \tau_{\varepsilon_N}}(\omega) = M_k(\omega) \]
for all $k\in\N$, but similarly
\[\lim_{k\to\infty} M_k(\omega) = \lim_{k\to\infty} M_{k\wedge\tau_{\varepsilon_N}}(\omega) = M_\infty^{\varepsilon_N}(\omega) <\infty, \]
where the last inequality $<\infty$ holds since $\E[|M_\infty^{\varepsilon_N}|] <\infty$.

\textbf{Step 4 (conclusion):} Finally, we move back to the assertion regarding $(Z_k)_{k\in\N}$ and $(C_k)_{k\in\N}$. Observe that
\[-\infty < - \sum_{i=0}^\infty \widehat B_i(\omega) \le \lim_{k\to\infty} M_k(\omega) = \lim_{k\to\infty} \widehat Z_k(\omega) - \sum_{i=0}^{k-1}(\widehat B_i(\omega)- \widehat C_i(\omega)) <\infty, \]
where $\widehat Z_k(\omega), \widehat B_i(\omega), \widehat C_i(\omega) \ge 0$ implying that
\[\lim_{k\to\infty} \widehat Z_k(\omega) <\infty \quad \text{and}\quad \sum_{i=0}^\infty \widehat C_i(\omega) <\infty \]
for almost all $\omega\in\Omega$.
Moreover, it holds true that
\[Z_k(\omega) = \widehat Z_k(\omega) \prod_{i=0}^{k-1} (1+A_i(\omega)),\]
where both $\widehat Z_k(\omega)$ and $\prod_{i=0}^{k-1} (1+A_i(\omega))$ converge for almost all $\omega\in\Omega$. The latter one follows by monotonicity and
\[ 0 \le \prod_{i=0}^{k-1} (1+A_i(\omega)) \le \exp(\sum_{i=0}^{k-1} A_i(\omega)), \]
where the upper bound converges by assumption. Therefore, $\lim_{k\to\infty} Z_k(\omega) = Z_\infty(\omega)$ exists for almost all $\omega\in\Omega$. Similarly, we have
\[\sum_{i=0}^k C_i(\omega) = \sum_{i=0}^k \widehat C_i(\omega) \prod_{j=0}^i (1+A_j(\omega)) \le \left(\prod_{j=0}^\infty (1+A_j(\omega))\right) \sum_{i=0}^k \widehat C_i(\omega) \]
which implies
\[\sum_{i=0}^\infty C_i(\omega) <\infty \]
for almost all $\omega\in\Omega$. 
\end{proof}

The conclusion of the Robbins--Siegmund theorem is qualitative: it yields almost sure convergence and the summability of the accumulated descent terms. This summability property will later be
used once more in Section~\ref{sec:as_rates}, where we refine the argument and derive almost sure convergence rates from it. For the moment, we only use it to deduce convergence to zero of suitable non-negative quantities. The following corollary is an easy, but very useful, extension of Theorem~\ref{thm:RS}.
\begin{cor}\label{cor:RS}
Let $(\Omega,\mathcal A,\mathcal F,\mathbb P)$ be a filtered probability space, $(Z_k)_{k\in\N}$, $(A_k)_{k\in\N}$, $(B_k)_{k\in\N}$ and $(D_k)_{k\in\N}$ be non-negative and $\mathcal F$-adapted stochastic processes, such that
\[\sum_{k=0}^\infty A_k<\infty, \quad \sum_{k=0}^\infty B_k<\infty \quad \text{and}\quad \sum_{k=0}^\infty D_k =\infty \]
almost surely. Moreover, suppose
\[\E[Z_{k+1}\mid \mathcal F_k] \le Z_k (1+A_k-D_k) + B_k. \]
Then $Z_k$ converges almost surely to $0$ for $k\to\infty$.
\end{cor}
\begin{proof}
In the setting of Theorem~\ref{thm:RS}, we have $Z_k = Z_k$, $A_k = A_k$, $B_k = B_k$ and $C_k = D_kZ_k$. By Theorem~\ref{thm:RS} we there exists $Z_\infty$ which is almost surely finite, integrable and satisfies $Z_k\to Z_\infty$ almost surely. In addition, we also have that $\sum_{k=1}^\infty C_k = \sum D_k Z_k<\infty$ implying that $\liminf_{k\to\infty} Z_k = 0$ due to the assumption $\sum_{k=1}^\infty D_k = \infty$ almost surely. Consequently, as the limit inferior and limit coincide for converging sequences, we have
     \[ Z_\infty = \lim_{k\to\infty} Z_k = \liminf_{k\to\infty} Z_k = 0\quad \text{almost surely}\,.\]
\end{proof}
We are now ready to prove almost sure convergence of SGD in the non-convex setting. We will assume that the objective function $F$ is $L$-smooth and lower bounded, i.e.~$\inf_{x\in\R^d}\ F(x)>-\infty$. Similar to the case of gradient descent, we do not expect more than convergence to stationary points. In particular, we are able to extend Theorem~\ref{thm:GD_diminishingstep} to the stochastic version. 
The proof is based on a one-step conditional descent estimate, \eqref{eq:smooth_condexpest}. This estimate is the stochastic analogue of the deterministic descent condition~\eqref{eq:descent_condition}: although the objective value does not decrease path-wise in every iteration, it decreases after taking conditional expectations, up to an additional
error term caused by the stochastic gradient noise. This yields a recursion of the form
\eqref{eq:RS_recursion} required to apply the Robbins--Siegmund theorem. The estimate~\eqref{eq:smooth_condexpest} is the central technical step in the analysis of SGD. It will reappear throughout the rest of this chapter.
\begin{thm}[SGD almost sure convergence]\label{thm:SGD_as}
Let $F:\R^d\to\R$ be $L$-smooth and bounded from below by $F^\ast = \inf_{x\in\R^d}\ F(x)>-\infty$, let $(\alpha_k)_{k\in\N}$ (deterministic or $\mathcal F$-adapted) satisfy 
\[\alpha_k>0,\quad \sum_{k=0}^\infty \alpha_k =\infty\quad \text{and}\quad \sum_{k=0}^\infty \alpha_k^2<\infty \]
(almost surely). We assume that the assumptions of Lemma~\ref{lem:condexp} are satisfied, and 
\[\E[\|\nabla_x f(x,Z)-\E[\nabla_x f(x,Z)]\|^2] \le c(1+(F(x)-F^\ast)) \] 
for some constant $c>0$ and all $x\in\R^d$. Moreover, let $X_0$ be a random variable such that $\E[F(X_0)]<\infty$ and $(X_k)_{k\in\N}$ be the sequence of random variables generated by Algorithm~\ref{alg:SGD}. Then it holds true that the sequence of random variables $(F(X_k))_{k\in\N}$ converges almost surely to some random variable $F_\infty$, almost surely finite, and 
\[\liminf_{k\to\infty} \|\nabla_x F(X_k)\|^2 = 0 \]
almost surely.
\end{thm}
\begin{proof}
We define the natural filtration $\cF = (\cF_k)_{k\in\N}$ through $\cF_k=\sigma(X_m,m\le k)= \sigma(X_0,\ Z_m,m\le k)$ and note that $(\alpha_k)_{k\in\N}$ is $\cF$-adapted per construction. Using the $L$-smoothness of $F$ we obtain (path-wise) that
\begin{align*}
F(X_{k+1}) &= F(X_k-\alpha_k \nabla_x f(X_k,Z_{k+1})) \\ &\le F(X_k) -\alpha_k \langle \nabla_x F(X_k),\nabla_x f(X_k,Z_{k+1})\rangle + \alpha_k^2 \frac{L}2 \|\nabla_x f(X_k,Z_{k+1})\|^2\\
& = F(X_k) -\alpha_k \|\nabla_x F(X_k)\|^2+\alpha_k \langle \nabla_x F(X_k),M_{k+1}\rangle\\&\quad + \alpha_k^2 \frac{L}2(\|\nabla_x F(X_k)\|^2-2\langle \nabla_x F(X_k), M_{k+1}\rangle + \|M_{k+1}\|^2),
\end{align*}
where $M_{k+1} := \nabla_x F(X_k) - \nabla_x f(X_k,Z_{k+1})$. By Lemma~\ref{lem:condexp} and Lemma~\ref{lem:condexp_factorization} we obtain \[\E[M_{k+1}\mid \cF_k] = 0\] and
\[ \E[\|M_{k+1}\|^2\mid \cF_k]\le c (1+ (F(X_k)-F^\ast)). \]
This yields
\begin{align}\label{eq:smooth_condexpest}
\E[F(X_{k+1})-F^\ast\mid \cF_k] &\le (F(X_k) - F^\ast) + (\frac{L}{2}\alpha_k^2-\alpha_k) \|\nabla_x F(X_k)\|^2 + \frac{L}{2}\alpha_k^2 c (1 + (F(X_k)-F^\ast)) \notag\\
&=(1+c\frac{L}{2}\alpha_k^2) (F(X_k)-F^\ast) + c\frac{L}{2}\alpha_k^2 - \alpha_k (1-\frac{L}{2}\alpha_k)\|\nabla_x F(X_k)\|^2.
\end{align}
W.l.o.g.~we assume that $\alpha_k \le (1-\varepsilon)\frac{2}{L}$ for some $\varepsilon\in(0,1)$ (else let $k$ be sufficiently large), such that
$(1-\frac{L}{2}\alpha_k)\ge\varepsilon> 0$. We can now apply Theorem~\ref{thm:RS} to imply that $\lim_{k\to\infty} F(X_k)-F^\ast$ exists almost surely and is finite, as well as
\[\varepsilon \sum_{k=0}^\infty \alpha_k\|\nabla_x F(X_k)\|^2\le \sum_{k=0}^\infty \alpha_k (1-\frac{L}{2}\alpha_k)\|\nabla_x F(X_k)\|^2<\infty  \]
almost surely. Since we have assumed $\sum_{k=0}^\infty \alpha_k =\infty$ almost surely, 
we obtain
\[\liminf_{k\to\infty} \|\nabla_x F(X_k)\|^2 = 0 \]
almost surely.
\end{proof}

\begin{remark}
    In the literature of SGD, the so-called ABC-condition has been established as a popular assumption on the variance of the stochastic gradient estimators. The assumption is an important relaxation of bounded noise and can be verified in many applications \cite{gower2021sgd,khaled2022better}. Here, we assume that there exist constants $A,B,C\ge0$ such that
    \[\E[\|\nabla_x f(x,Z)-\E[\nabla_x f(x,Z)]\|^2] \le A(F(x)-F^\ast) + B\|\nabla_x F(x)\|^2+C\,. \]
    All the results, presented in these lecture notes based on bounded variance assumptions can be extended to the ABC-condition.
\end{remark}

Before delving into the derivation of convergence rates for SGD, we formulate the following Corollary which states almost sure convergence under same assumptions of Theorem~\ref{thm:SGD_as}, but with the additional property of strong convexity of $F$. 

\begin{cor}\label{cor:SGD_as}
Suppose that the assumptions of Theorem~\ref{thm:SGD_as} are satisfied and additionally, assume that $F$ is $\mu$-strongly convex. Then the sequence of random variables $(X_k)_{k\in\N}$ converges almost surely to the unique global minimum $x_\ast\in\R^d$ of $F$.
\end{cor}
\begin{proof}
    Let $F$ be $\mu$-strongly convex and denote $x_\ast= \arg\min_{x\in\R^d}\ F(x)$. Following Proposition~\ref{prop:strongconvexPL}, we have that
    \[\frac12\|\nabla F(x)\|^2 \ge \mu(F(x)-F(x_\ast))\,.\]
    W.l.o.g.~we assume this time that that $\alpha_k\le \frac{1}{L}$ so that from \eqref{eq:smooth_condexpest} we have that
    \[\E[F(X_{k+1}) - F(x_\ast)\mid \cF_k] \le (1+c\frac{L}2\alpha_k^2 - \alpha_k\mu) (F(X_k)-F(x_\ast)) + c\frac{L}2\alpha_k^2\,,\]
    where we have used that $(1-\frac{L}2\alpha_k)\ge \frac12$. By Robbins-Siegmund (Corollary~\ref{cor:RS}) we deduce that $F(X_k)-F(x_\ast) \to 0$ almost surely as $k\to\infty$. Again by strong convexity we recall that $\frac{\mu}2\|X_k-x_\ast\|^2 \le F(X_k)-F(x_\ast)$ (almost surely) which yields the claim. 
\end{proof}

\section{Convergence in expectation}\label{sec:SGDexp}
Similarly as in the setting of the deterministic gradient descent scheme, the previous result states convergence to a stationary point without explicit rate of convergence, but under rather mild assumptions on the objective function $F$. In order to obtain a speed of convergence, additional properties such as convexity must be assumed for $F$. We will observe that the convergence behavior is worse than in the setting of deterministic gradient descent methods. However, in order to make a fair comparison between deterministic and stochastic gradient descent schemes, one must consider the complexity of both algorithms, including the computational cost of implementation.
In the following section, we consider a row of convergence results in expectation.
\subsection{Convergence for non-convex and smooth objective functions}
We begin with the convergence of the expected gradient norm under a non-convex but smooth setting. The proof builds up on the estimates derived in the proof of Theorem~\ref{thm:SGD_as}.

\begin{thm}\label{thm:SGDnonconvex}
    Suppose that the assumptions of Theorem~\ref{thm:SGD_as} and the additional assumption that $\alpha_k\in(0,1/L]$ almost surely are in place. Then for all $K>0$ it holds true that
    \[ \min_{0\le k\le K}\E[\|\nabla F(X_k)\|^2] \le  \frac{2(\E[F(X_0)]- F^\ast)}{\sum_{j=0}^K\alpha_j} + \frac{c L  \sum_{k=0}^K\alpha_k^2}{\sum_{j=0}^K\alpha_j} \,.\]
\end{thm}
\begin{proof}
    Let $\cF_k = \sigma(X_m,m\le k)$ be the natural filtration and consider the one-step conditional descent condition \eqref{eq:smooth_condexpest}. We take another expectation on both sides which yields
    \begin{align*}
    \E[F(X_{k+1})] = \E[\E[F(X_{k+1})\mid \cF_k]] \le \E[F(X_k)] - \alpha_k(1-\frac{L\alpha_k}2) \E[\|\nabla_x F(X_k)\|^2] + c\frac{L}2 \alpha_k^2.
    \end{align*}
    Reordering the inequality and iterating over $k=0,\dots,K$ gives
    \begin{align*}
        \sum_{k=0}^K\alpha_k(1-\frac{L\alpha_k}2) \E[\|\nabla_x F(X_k)\|^2] &\le \sum_{k=0}^K\Big(\E[F(X_k)] - \E[F(X_{k+1})]\Big) +c\frac{L}2  \sum_{k=0}^K\alpha_k^2 \\
        &= (\E[F(X_0)]- \E[F(X_{K+1})]) + c\frac{L}2  \sum_{k=0}^K\alpha_k^2\\
        &\le (\E[F(X_0)]- F^\ast) + c\frac{L}2  \sum_{k=0}^K\alpha_k^2
    \end{align*}
    Under the assumption $\alpha_k\le 1/L$ for all $k\ge0$, we have that
    \begin{align*}
        \frac12\sum_{k=0}^K\alpha_k \E[\|\nabla_x F(X_k)\|^2] &\le  (\E[F(X_0)]- F^\ast) + c\frac{L}2  \sum_{k=0}^K\alpha_k^2\,.
    \end{align*}
    Similar to the proof of Corollary~\ref{cor:gradient_convergence}, we have
    \[ \min_{0\le k\le K}\E[\|\nabla F(X_k)\|^2] \le \frac1{\sum_{j=0}^K\alpha_j}\sum_{k=0}^K\alpha_k \E[\|\nabla_x F(X_k)\|^2] \le  \frac{2(\E[F(X_0)]- F^\ast)}{\sum_{j=0}^K\alpha_j} + \frac{cL  \sum_{k=0}^K\alpha_k^2}{\sum_{j=0}^K\alpha_j} \,.\]
\end{proof}

From the derived upper bound we obtain convergence under the sufficient condition that
\[ \sum_{k=0}^{\infty} \alpha_k = \infty \quad \text{and}\quad \sum_{k=0}^\infty \alpha_k^2 <\infty.  \]
In order to quantify the speed of convergence, we provide a specific choice of the step sizes. Note that in the case below, we have $\sum_{k=0}^\infty \alpha_k^2 = \infty$ but
\[\frac{\sum_{k=0}^K \alpha_k^2}{\sum_{k=0}^K \alpha_k} \to 0 \,.\]

\begin{cor}\label{cor:SGDnonconvex}
Suppose that the same conditions of Theorem~\ref{thm:SGDnonconvex} are satisfied. Moreover, let $\alpha_k := \frac{1}{L\sqrt{k+1}}$. Then it holds true that
\[ \min_{0\le k\le K}\E[\|\nabla F(X_k)\|^2] \in \cO\left(\frac{\log(K)}{\sqrt{K}}\right). \]
\end{cor}
\begin{proof}
Firstly, we observe that 
\[\sum_{j=0}^{K} \alpha_j = \frac1L \sum_{j=0}^{K} \frac{1}{\sqrt{j+1}} \ge \frac1L \int_{1}^{K} \frac{1}{\sqrt{t+1}}\, \dd t = \frac2L (\sqrt{K+1}-\sqrt{2}) \ge \frac1L\sqrt{K+1}, \]
for sufficiently large $K$ ($K\ge 7$). On the other side, we have
\[\sum_{j=0}^{K}\alpha_j^2 = \frac{1}{L^2} \sum_{j=0}^{K}\frac{1}{j+1} \le \frac1{L^2}\left(1+\int_0^{K} \frac{1}{t+1}\,\dd t\right) = \frac{1}{L^2}(1+\log(K+1)). \]
By the upper bound derived in Theorem~\ref{thm:SGDnonconvex}, we obtain
\[\min_{0\le k\le K}\E[\|\nabla F(X_k)\|^2] \le \frac{2L(\E[F(X_0)]-F^\ast) + \frac{c}{L} (1+\alpha_0 L)}{\sqrt{K+1}} + \frac{c}{L}(1+\alpha_0 L) \frac{\log(K+1)}{\sqrt{K+1}}\in \cO\left(\frac{\log(K)}{\sqrt{K}}\right)\, . \]
\end{proof}

\subsection{Convergence for convex and smooth objective functions}
In the following, we will prove convergence of SGD under the assumption that the objective function is convex. Compared to the convergence result for gradient descent schemes, we will prove a slower convergence behavior. The following Theorem presents the resulting error bound in expectation.

\begin{thm}[SGD for convex and smooth objective function] \label{thm:SGD_convex}
Let $F:\R^d\to\R$ be convex and $L$-smooth, and assume that the set of global minima of $F$ is non-empty. We assume that the assumptions of Lemma~\ref{lem:condexp} are satisfied and that there exists $c>0$ such that
\[\E[\|\nabla_x f(x,Z)-\E[\nabla_x f(x,Z)]\|^2]\le c \]
for all $x\in\R^d$. Let $X_0$ be a random variable such that $\E[|F(X_0)|+\|X_0-x_\ast\|^2]<\infty$ for some $x_\ast\in\arg\min_{x\in\R^d}\ F(x)$. Moreover, let $(X_k)_{k\in\N}$ be generated by Algorithm~\ref{alg:SGD} with deterministic and decreasing sequence of step sizes $(\alpha_k)_{k\in\N}$ such that $\alpha_k\in(0,\frac1L]$. Then for
\[\bar X_K := \sum_{k=0}^{K-1}w_k^K X_{k+1},\quad w_k^K:= \frac{\alpha_k}{\sum_{j=0}^{K-1}\alpha_j}, \quad K\ge2, \]
it holds true that
\[\E[F(\bar X_K)-F(x_\ast)] \le \frac{\E[\|X_0-x_\ast\|^2]}{2\sum_{j=0}^{K-1}\alpha_j} + \frac{c(1+\alpha_0 L)\sum_{k=0}^{K-1} \alpha_k^2 }{2\sum_{j=0}^{K-1}\alpha_j}.\]
\end{thm}
\begin{proof}
Let $x_\ast\in\arg\min_{x\in\R^d}\ F(x)$ and $\cF_k=\sigma(X_m,m\le k)$ be the natural filtration. Similar as in the proof of Theorem~\ref{thm:SGD_as} we have
\begin{align*}
\E[F(X_{k+1})] = \E[\E[F(X_{k+1})\mid \cF_k]] \le \E[F(X_k)] - \alpha_k(1-\frac{L\alpha_k}2) \E[\|\nabla_x F(X_k)\|^2] + c\frac{L}2 \alpha_k^2.
\end{align*}
We can also derive the following,
\begin{align*}
\|X_{k+1}-x_\ast\|^2 &= \|X_k-x_\ast\|^2 - 2\alpha_k\langle \nabla_x f(X_k,Z_{k+1}), X_k-x_\ast\rangle + \alpha_k^2\|\nabla_x f(X_k,Z_{k+1})\|^2\\
&=\|X_k-x_\ast\|^2 - 2\alpha_k \langle \nabla_x F(X_k), X_k-x_\ast\rangle + \alpha_k^2\|\nabla_x F(X_k)\|^2\\
&\quad + 2\alpha_k\langle M_{k+1} , X_k-x_\ast\rangle + \alpha_k^2 \|M_{k+1}\|^2 - 2\alpha_k^2 \langle M_{k+1},\nabla_x F(X_k)\rangle,
\end{align*}
where again $M_{k+1}=\nabla_x F(X_k)-\nabla_x f(X_k,Z_{k+1}) $. Taking the conditional expectation wrt. $\cF_k$ results in the bound
\begin{align*}
\E[\|X_{k+1}-x_\ast\|^2\mid \cF_k] &= \|X_k-x_\ast\|^2 - 2\alpha_k \langle \nabla_x F(X_k), X_k-x_\ast\rangle + \alpha_k^2\|\nabla_x F(X_k)\|^2+\alpha_k^2 \E[\|M_{k+1}\|^2\mid \cF_k]\\
&\le \|X_k-x_\ast\|^2 - 2\alpha_k \langle \nabla_x F(X_k), X_k-x_\ast\rangle + \alpha_k^2\|\nabla_x F(X_k)\|^2+\alpha_k^2 c\, .
\end{align*}
We take again expectation and rewrite the derived inequality in form
\begin{align*}
2\alpha_k \E[\langle \nabla_x F(X_k), X_k-x_\ast\rangle] \le \E[\|X_{k}-x_\ast\|^2]-\E[\|X_{k+1}-x_\ast\|^2] + \alpha_k^2\E[\|\nabla_x F(X_k)\|^2] + \alpha_k^2c \, .
\end{align*}
By convexity of $F$ we have almost surely that
\[F(X_k) \le F(x_\ast) + \langle X_k-x_\ast, \nabla_x F(X_k)\rangle,\]
such that
\begin{align*}
\E[F(X_{k+1})] &\le F(x_\ast) + \E[\langle X_k-x_\ast,\nabla_x F(X_k)\rangle]-\alpha_k(1-\frac{L\alpha_k}2)\E[\|\nabla_x F(X_k)\|^2] + c\frac{L}2\alpha_k^2\\
&\le F(x_\ast) + \frac{1}{2\alpha_k} \left( \E[\|X_k-x_\ast\|^2]-\E[\|X_{k+1}-x_\ast\|^2]\right)\\ &\quad -\alpha_k\left(\frac12-\frac{L\alpha_k}2\right)\E[\|\nabla_x F(X_k)\|^2] + (\frac{\alpha_k}2+\frac{L\alpha_k^2}2)c\\
&\le  F(x_\ast) + \frac{1}{2\alpha_k} \left( \E[\|X_k-x_\ast\|^2]-\E[\|X_{k+1}-x_\ast\|^2]\right)+ \left(\frac{\alpha_k}2+\frac{L\alpha_k^2}2\right)c,
\end{align*}
where we have used that $(\frac12-\frac{L\alpha_k}2)\ge0$. Note that $\sum_{k=0}^{K-1} w_k^K = 1$, such that by Jensen's inequality it follows that
\begin{align*}
\E[F(\bar X_K)-F(x_\ast)] &\le \frac{1}{\sum_{j=0}^{K-1}\alpha_j}\sum_{k=0}^{K-1}\alpha_k \E[F(X_{k+1})-F(x_\ast)]\\
&\le \frac{1}{2\sum_{j=0}^{K-1}\alpha_j}\sum_{k=0}^{K-1} \left(\E[\|X_k-x_\ast\|^2]-\E[\|X_{k+1}-x_\ast\|^2] \right)\\
&\quad + \frac{1}{\sum_{j=0}^{K-1}\alpha_j}\sum_{k=0}^{K-1}\left(\frac{\alpha_k^2}{2}+\frac{\alpha_k^3L}{2}\right) c\\
&\le \frac{\E[\|X_0-x_\ast\|^2]}{2\sum_{j=0}^{K-1}\alpha_j} + \frac{c(1+\alpha_0 L)\sum_{k=0}^{K-1} \alpha_k^2}{2\sum_{j=0}^{K-1}\alpha_j},
\end{align*}
where we have used that $\alpha_k$ is decreasing and therefore, $\alpha_k\le\alpha_0$. 
\end{proof}

\begin{remark}
The convergence statement in Theorem~\ref{thm:SGD_convex} is formulated for the weighted average
\(\bar X_K\)
rather than for the last iterate $X_K$. This is one of the main differences compared to the
deterministic gradient descent proof of Theorem~\ref{thm:GD_convex}. In the deterministic convex setting with a
fixed step size, the descent estimate leads to a simple telescoping sum and, since the objective values decrease monotonically, one directly obtains a last-iterate rate. For SGD, the objective values do not decrease monotonically along individual trajectories because of the stochastic gradient noise. After taking expectations, the one-step estimate still gives a telescoping bound, but only for the weighted sum
\[
\sum_{k=0}^N \alpha_k \E\big[F(X_k)-F(x^\ast)\big].
\]
To convert this bound into a convergence rate for objective values, we therefore introduced the weighted average $\bar X_K$ and used convexity of $F$ together with Jensen's inequality.
\end{remark}

Similarly as in the previous non-convex setting, we obtain the following rate of convergence.

\begin{cor}
Suppose that the same conditions of Theorem~\ref{thm:SGD_convex} are satisfied. Moreover, let $\alpha_k := \frac{1}{L\sqrt{k+1}}$. Then it holds true that
\[ \E[F(\bar X_K)-F(x_\ast)] \in \cO\left(\frac{\log(K)}{\sqrt{K}}\right). \]
\end{cor}
\begin{proof}
Recall from Corollary~\ref{cor:SGDnonconvex} that 
\[\sum_{j=0}^{K-1} \alpha_j \ge \frac1L\sqrt{K}, \]
for sufficiently large $K$ ($K\ge 8$) and 
\[\sum_{j=0}^{K-1}\alpha_j^2 \le \frac{1}{L^2}(1+\log(K)). \]
By the upper bound derived in Theorem~\ref{thm:SGD_convex}, we obtain
\[\E[F(\bar X_K)-F(x_\ast)] \le \frac{L\E[\|X_0-x_\ast\|^2] + \frac{c}{2L} (1+\alpha_0 L)}{\sqrt{K}} + \frac{c}{2L}(1+\alpha_0 L) \frac{\log(K)}{\sqrt{K}}\in \cO\left(\frac{\log(K)}{\sqrt{K}}\right)\, . \]
\end{proof}

\subsection{Convergence for strongly convex and smooth objective functions}\label{sec:SGD_strcnvx}

If we additionally assume strong convexity, we can further improve the derived upper bound. Moreover, we can even prove convergence to a unique global minimum of $F$, as we also did for the deterministic gradient descent scheme. However, due to the stochastic approximation of the gradient, we lose the behavior of linear convergence. 

\begin{thm}[SGD for strong convex and smooth objective function]\label{thm:SGD_strngconvex}
Let $F:\R^d\to\R$ be $\mu$-strongly convex and $L$-smooth. We assume that the assumptions of Lemma~\ref{lem:condexp} are satisfied and that there exists $c>0$ such that
\[\E[\|\nabla_x f(x,Z) - \E[\nabla_x f(x,Z)] \|^2] \le c \]
for all $x\in\R^d$. Let $X_0$ be a random variable such that $\E[|F(X_0)| + \|X_0-x_\ast\|^2]<\infty$, where $x_\ast\in\R^d$ is the unique global minimum of $F$. Moreover, let $(X_k)_{k\in\N}$ be generated by Algorithm~\ref{alg:SGD} with deterministic and decreasing sequence of step sizes $(\alpha_k)_{k\in\N}$ such that $\alpha_k\in(0,\frac1L]$. Then it holds true that
\[\E[\|X_{k+1}-x_\ast\|^2] \le (1-\alpha_k \mu) \E[\|X_k-x_\ast\|^2] + c\alpha_k^2 \]
for all $k\ge0$.
\end{thm}
\begin{proof}
Let $\cF_k=\sigma(X_m,m\le k)$ be again the natural filtration and recall that we have derived in the proof of Theorem~\ref{thm:SGD_convex} that
\[\E[\|X_{k+1}-x_\ast\|^2] \le \E[\|X_k-x_\ast\|^2] - 2\alpha_k \E[\langle \nabla_x F(X_k), X_k-x_\ast\rangle ] + \alpha_k^2\E[\|\nabla_x F(X_k)\|^2] + \alpha_k^2 c. \]
By $\mu$-strong convexity we have that
\[F(x_\ast) - F(X_k) \ge \langle x_\ast-X_k,\nabla_x F(X_k)\rangle + \frac{\mu}2\|X_k-x_\ast\|^2, \]
which can be rewritten as
\[-\langle X_k-x_\ast, \nabla_x F(X_k)\rangle\le -\left(F(X_k)-F(x_\ast)\right) - \frac{\mu}2\|X_{k}-x_\ast\|^2 \quad \text{almost surely}\, .  \]
Combining both inequalities, we obtain that
\begin{align*}
\E[\|X_{k+1}-x_\ast\|^2] \le (1-\alpha_k\mu)\E[\|X_k-x_\ast\|^2] - 2\alpha_k\E[F(X_k)-F(x_\ast)] + \alpha_k^2\E[\|\nabla_x F(X_k)\|^2] + \alpha_k^2 c\, .
\end{align*}
The assumption of $L$-smoothness implies that
\[-(F(X_k)-F(x_\ast)) \le -\frac{1}{2L}\|\nabla_x F(X_k)\|^2\quad \text{almost surely}, \]
such that 
\begin{align*}
\E[\|X_{k+1}-x_\ast\|^2] &\le (1-\alpha_k\mu)\E[\|X_k-x_\ast\|^2] + \alpha_k (\alpha_k-\frac{1}{L})\E[\|\nabla_x F(X_k)\|^2] + \alpha_k^2 c\\
&\le (1-\alpha_k\mu) \E[\|X_k-x_\ast\|^2] + \alpha_k^2 c,
\end{align*}
where we have used that $\alpha_k\le \frac1L$.
\end{proof}

From the above derived error bound we observe that the iterated error decomposes into the error arising due to the optimization error from the exact gradient descent scheme applied to strongly convex and smooth objective functions, and into an error arising from the variance of the stochastic approximation of the gradients. In order to deduce a convergence rate along the iterations, we need to balance both errors by either decreasing the step size $\alpha_k$ sufficiently or by decreasing the variance term. The latter one will be the topic of Section~\ref{sec:SGD_varred}, where we consider methods of variance reduction. In the following, we will present the former approach of decreasing the step size to $0$. 

\begin{cor}\label{cor:SGD_strngconvex}
Suppose that the same conditions as in Theorem~\ref{thm:SGD_strngconvex} are satisfied. Moreover, let $\alpha_k = \frac{\tau}{\mu(k+s)}$ for some $\tau\ge 2$ and $s:= \kappa\tau = \frac{L}{\mu}\tau\ge1$. Then it holds true that $\alpha_0\le \frac{1}{L}$ and there exists $\gamma\ge 2(s+1)\max(\E[\|X_0-x_\ast\|^2], \frac{\tau^2c}{s\mu^2})$ such that
\[\E[\|X_k-x_\ast\|^2] \le \frac{\gamma}{k+s}\]
for all $k\ge1$.
\end{cor}
\begin{proof}
Firstly, we observe that by definition
\[ \alpha_0 = \frac{\tau}{\mu \cdot s} = \frac{\tau \mu}{\mu L \tau} = \frac{1}{L}\, , \]
such that $\alpha_k\le\frac1L$ for all $k\ge0$. We define $\Delta_k = \E[\|X_k-x_\ast\|^2]$ and prove the second claim via induction. For $k=1$ it holds true that
\begin{align*}
\Delta_1 \le (1-\alpha_0\mu)\Delta_0 + \alpha_0^2 c &\le (1-\frac{1}{\kappa}) \Delta_0 + \frac{\tau^2 c}{\mu^2 s}\\ &\le \frac{(s+1)2\max(\Delta_0, \frac{\tau^2 c}{\mu^2 s})}{s+1} \le \frac{\gamma}{s+1}\, .
\end{align*}
Now, suppose that the upper bound $\Delta_k\le\frac{\gamma}{k+s}$ is satisfied for some $k\ge1$, then it follows that
\begin{align*}
\Delta_{k+1} \le (1-\alpha_k \mu)\Delta_k + \alpha_k^2 c&\le \left(1-\frac{\tau}{k+s}\right)\frac{\gamma}{k+s} + \frac{\tau^2 c}{\mu^2} \frac{1}{(k+s)^2}\\
&= \frac{\gamma}{k+1+s} + \frac{\gamma}{(k+s)(k+1+s)}-\frac{\gamma\tau}{(k+s)^2} + \frac{\tau^2c}{\mu^2}\frac1{(k+s)^2}\\
&\le \frac\gamma{k+1+s} + \frac{\gamma-\tau\gamma +\frac{\tau^2c}{\mu^2}}{(k+s)^2}\\
&\le \frac\gamma{k+1+s} + \frac{\frac{\tau^2c}{\mu^2}-\gamma}{(k+s)^2}\\
&\le \frac\gamma{k+1+s},
\end{align*}
where we have used that $\gamma(1-\tau) \le -\gamma$ and $\gamma\ge 2(s+1) \max(\Delta_0, \frac{\tau^2c}{\mu^2s})\ge \frac{\tau^2 c}{\mu^2}$.
\end{proof}

\subsection{Convergence under PL-condition and smooth objective functions}

Similar to the deterministic gradient descent method, we are able to prove convergence under the PL-condition. We obtain the same type of convergence behavior as in the strong convex setting, with the main difference being the error discrepancy in the objective function evaluation.

\begin{thm}[SGD under PL-condition]\label{thm:SGD_PL}
Let $F:\R^d\to\R$ be $L$-smooth and assume that $F$ satisfies the PL-condition
\[\|\nabla_x F(x)\|^2 \ge 2r(F(x)-F^\ast)\]
for some $r\in(0,L)$ and all $x\in\R^d$, where $F^\ast = \inf_{x\in\R^d}\ F(x) >-\infty$. We assume that the assumptions of Lemma~\ref{lem:condexp} are satisfied and that there exists $c>0$ such that
\[\E[\|\nabla_x f(x,Z)- \E[\nabla_x f(x,Z)] \|^2]\le c \]
for all $x\in\R^d$. Let $X_0$ be a random variable such that $\E[|F(X_0)|+\|X_0\|^2]<\infty$, and let $(X_k)_{k\in\N}$ be generated by Algorithm~\ref{alg:SGD} with deterministic and decreasing sequence of step sizes $(\alpha_k)_{k\in\N}$ such that $\alpha_k\in(0,\frac1L]$. Then it holds true that
\[\E[F(X_k)-F^\ast] \le (1-\alpha_k r)\E[F(X_k)-F^\ast] + c\frac{L}2 \alpha_k^2\,. \]
\end{thm}

\begin{proof}
We have already seen that under $L$-smoothness we obtain the following bound
\[\E[F(X_{k+1}) - F^\ast] \le \E[F(X_k)-F^\ast] -\alpha_k(1-\frac{L}2\alpha_k) \E[\|\nabla_x F(X_k)\|^2] + c\frac{L}2\alpha_k^2\, . \]
Under the PL-condition and the fact that $1-\frac{L}2\alpha_k\ge \frac12>0$, we improve the upper bound to
\begin{align*}
\E[F(X_{k+1}) - F^\ast] &\le\E[F(X_k)-F^\ast] -\alpha_k(1-\frac{L}2\alpha_k) 2r \E[F(X_k)-F^\ast]+ c\frac{L}2\alpha_k^2\\
&=(1-\alpha_k(1-\frac{L}2\alpha_k)2r)\E[F(X_k)-F^\ast] + c\frac{L}2\alpha_k^2\\
&\le (1-\alpha_k r)\E[F(X_k)-F^\ast] + c\frac{L}2\alpha_k^2\, .
\end{align*}
\end{proof}

We can apply the same step size strategy as in the strongly convex setting to derive convergence.
\begin{cor}
Suppose that the same conditions as in Theorem~\ref{thm:SGD_PL} are satisfied. Moreover, let $\alpha_k = \frac{\tau}{r(k+s)}$ for some $\tau\ge 2$ and $s := \frac{L}{r}\tau\ge1$. Then it holds true that $\alpha_0\le \frac{1}{L}$ and there exists $\gamma\ge (s+1)2\max(\E[F(X_0)-F^\ast], \frac{\tau^2c}{r^2})$ such that
\[\E[F(X_k)-F^\ast] \le \frac{\gamma}{k+s}\]
for all $k\ge1$.
\end{cor}
\begin{proof}
The proof proceeds line by line as the proof of Corollary~\ref{cor:SGD_strngconvex}.
\end{proof}

\subsection{Discussion about the complexity of SGD}

In the previous sections, we derived convergence rates of SGD under non-convexity, convexity, strong convexity and the PL-condition. We obtained similar results for the deterministic GD scheme. Comparing both GD and SGD, we observe that the derived results for SGD are significantly worse. 

\begin{center}
\begin{tabular}{|l|c|c|c|c|}
\hline
\ & convex & strongly convex & PL & non-convex \\
\hline
QoI & $F(\bar X_k) - F^\ast$ & $\|X_k-X_\ast\|$ & $F(X_k)-F^\ast$ & $\min_{0\le n\le k}\ \E[\|\nabla F(X_n)\|^2]$ \\
\hline \hline
GD & $\cO\left(\frac{1}{k}\right)$ & $\cO(\rho^k),\ \rho\in(0,1)$ & $\cO(\rho^k),\ \rho\in(0,1)$ & $\cO\left(\frac{1}{k}\right)$\\
\hline
SGD & $\cO\left(\frac{\log(k)}{\sqrt{k}}\right)$ & $\cO(\frac1{k})$ & $\cO(\frac1{k})$ & $\cO\left(\frac{1}{\sqrt{k}}\right)$ \\
\hline
\end{tabular}
\end{center}
However, the implementation of GD might be expensive (e.g.~for large data sets in empirical risk minimization) or even impossible. In this case, it is infeasible to implement GD and there is no other choice to work with SGD. 

The comparison of GD and SGD gets more involved in cases where we are able to compute the exact gradient. Let us consider the empirical risk minimization problem 
\[\min_{x\in\R^d}\ F_N(x),\quad F_n(x) = \frac{1}{N}\sum_{i=1}^N f(x,z^{(i)}),\ z^{(i)}\in\R^p,\ i=1,\dots,N,\]
where we assume that $x\mapsto f(x,z^{(i)})$ are $\mu$-strongly convex and $L$-smooth for all $i$. We have seen that GD with a fixed step size $\bar\alpha\le \frac1L$ converges linearly with rate $\rho \in(0,1)$ such that
\[\|x_{k,N}^{\mathrm{GD}}-x_N\|^2\le \rho^{k}\|x_0-x_N\|^2, \]
where $(x_{k,N}^{\mathrm{GD}})_{k\in\N}$ denotes the iteration generated by GD and $x_N=\arg\min_{x\in\R^d}\ F_N(x)$. In comparison, SGD (with decreasing step size) converges sub-linear with
\[\E[\|X_{k,N}^{\mathrm{SGD}}-x_N\|^2] \le \frac{\gamma}{k+s}. \]
In order to achieve an error of a certain tolerance $\varepsilon>0$ we need to iterate 
\begin{itemize}
\item[(i)] $k^{\mathrm{GD}}\ge \cO(\log(\varepsilon^{-1}))$, such that $\|x_{k,N}^{\mathrm{GD}}-x_N\|^2\le\varepsilon$,
\item[(ii)] $k^{\mathrm{SGD}}\ge \cO(\varepsilon^{-1})$, such that $\E[\|X_{k,N}^{\mathrm{SGD}}-x_N\|^2]\le \varepsilon$.
\end{itemize}
The first guess is, that as long we are able to compute the full gradient, there is no reason to implement SGD over GD. However, this train of thought is too naive. The reason is, that up to now we have ignored the empirical error which occurs through solving $x_N=\arg\min\ F_N(x)$. Indeed, $x_N$ should be treated as random variable $X_N$, which depends on $Z^{(1)},\dots,Z^{(N)}$. We want to quantify the error of both GD and SGD to $x_\ast = \arg\min_{x\in\R^d}\ F(x)$, where $F(x) = \E[f(x,Z)]$ is the expected risk. For SGD, implemented through Algorithm~\ref{alg:SGD}, we again obtain convergence (with decreasing step size) in the form of 
\[\E[\|X_{k,N}^{\mathrm{SGD}}-x_\ast\|^2] \le \frac{\gamma}{k+s}.\]
However, the situation changes for GD. Assuming that we are not able to compute the exact gradient of $F$, we firstly have to approximate $F$ through $F_N$ and then apply GD to find $X_N = \arg\min_{x\in\R^d}\ F_N(x)$. The final error decomposes to
\[\frac12\E[ \|X_{k,N}^{\mathrm{GD}} - x_\ast\|^2] \le \E[\|X_{k,N}^{\mathrm{GD}}-X_N\|^2] + \E[\|X_N-x_\ast\|^2] \le \rho^k \E[\|X_{0,N}^{\mathrm{GD}}-X_N\|^2] + \E[\|X_N-x_\ast\|^2], \]
where $X_N$ denotes the random minimum of $F_N$ given $Z^{(1)},\dots,Z^{(N)}$. In the following, we study the error $\E[\|X_N-x_\ast\|^2]$ for strongly convex objective functions depending on the number of data points.

\begin{thm}
Let $F:\R^d\to\R$ be $\mu$-strongly convex and let $f:\R^d\times\R^p\to\R$ such that $x\mapsto f(x,z)$ is $\mu$-strongly convex for all $z\in\R^p$. Moreover, let $x_\ast = \arg\min_{x\in\R^d}\ F(x)$, and assume that $\E[\nabla_x f(x_\ast,Z)] = \nabla_x F(x_\ast)$ and
\[\E[\|\nabla_x f(x_\ast,Z) - \E[\nabla_x f(x_\ast,Z)] \|^2]\le B \]
for some $B>0$. Let $Z^{(1)},\dots,Z^{(N)}$ be a family of i.i.d.~random variables with distribution $\mu_Z$, then it holds true that
\[\E[\|X_N-x_\ast\|^2]\le \frac{B}{\mu^2}\frac{1}{N}, \]
where $X_N = \arg\min_{x\in\R^d} F_N(x)$.
\end{thm}
\begin{proof}
Let 
\begin{align*}
 x_\ast &= \arg\min_{x\in\R^d}\ F(x),\quad F(x) = \E[f(x,Z)]\\
 X_N&= \arg\min_{x\in\R^d}\ F_N(x),\quad F_N(x) = \frac1N\sum_{i=1}^N f(x,Z^{(i)}),
 \end{align*}
which means that $\nabla_x F(x_\ast)=0$ and $\nabla_x F_N(X_N)=0$ almost surely. By strong convexity we can apply Lemma~\ref{lem:auxiliary_bound} to imply
\begin{align*}
\|X_N-x_\ast\|^2\le \frac1\mu \langle X_n-x_\ast, \nabla_x F_N(X_N)-\nabla_x F_N(x_\ast)\rangle &= \frac1\mu\langle X_n-x_\ast,\nabla_x F(x_\ast) - \nabla_x F_N(x_\ast)\rangle\\
&\le \frac1\mu\|X_N-x_\ast\| \|\nabla_x F(x_\ast)-\nabla_x F_N(x_\ast)\|,
\end{align*}
almost surely, where we have used Cauchy-Schwarz inequality in the last line. Reordering the above inequality leads to
\begin{equation}\label{eq:sec4_aux}
\|X_N-x_\ast\| \le \frac{1}{\mu}\|\nabla_x F(x_\ast)-\nabla_x F_N(x_\ast)\| 
\end{equation}
almost surely. Note that
\[\E[\nabla_x F_N(x_\ast)] = \E[\frac1N\sum_{i=1}^N \nabla_x f(x_\ast,Z^{(i)})] = \nabla_x F(x_\ast)\]
such that we obtain
\begin{align*}
\E[\|\nabla_x F_N(x_\ast)&- \nabla_x F(x_\ast)\|^2] = \E[\|\frac1N\sum_{i=1}^N\left(\nabla_x f(x_\ast,Z^{(i)})-\E[\nabla_x f(x_\ast,Z^{(i)})]\right)\|^2]\\
&=\frac1{N^2}\sum_{i,j=1}^N\E[\langle \nabla_x f(x_\ast,Z^{(i)})-\E[\nabla_x f(x_\ast,Z^{(i)})],\nabla_x f(x_\ast,Z^{(j)})-\E[\nabla_x f(x_\ast,Z^{(j)})]\rangle]\\
&=\frac1{N^2} \sum_{i=1}^N \E[\|\nabla_x f(x_\ast,Z^{(i)})-\E[\nabla_x f(x_\ast,Z^{(i)})]\|^2]\le \frac{B}{N},
\end{align*}
where we have used that the $Z^{(1)},\dots,Z^{(N)}$ are i.i.d.~random variables. Together with inequality~\eqref{eq:sec4_aux} we close the proof with
\[\E[\|X_N-x_\ast\|^2] \le \frac{B}{\mu^2}\frac{1}{N}. \]
\end{proof}
The overall error of GD is then given by
\[\frac12\E[ \|X_{k,N}^{\mathrm{GD}} - x_\ast\|^2]\le c(\frac1N + \rho^{k}) \]
for some constant $c>0$. Therefore, it is sufficient to choose $N\ge \cO(\varepsilon^{-1})$ and $k\ge\cO(\log(\varepsilon^{-1})$, such that the computational cost of GD are given by
\[{\mathrm{cost}}_{\mathrm{GD}} (\varepsilon) = N\cdot k \simeq  \varepsilon^{-1}\log(\varepsilon^{-1}),\] 
whereas the computational cost of SGD are given by
\[{\mathrm{cost}}_{\mathrm{SGD}}(\varepsilon) = 1\cdot k \simeq \varepsilon^{-1}. \]
Note that $\simeq$ hides constants independent of $\varepsilon$.

\subsection{Lower bound of SGD}

In the following, we will observe that we do not expect to improve the derived upper bound on SGD in the strongly convex setting. Therefore, we consider the example of a simple quadratic function. The example to be considered has been studied in detail in \cite{JENTZEN2020101438}.

\begin{example}[SGD lower error bound]
Let $(Z_k)_{k\in\N}$ be a sequence of i.i.d.~random variables in $\R^d$ with distribution $\mu_Z$ and $\E[\|Z_1\|^2]<\infty$. We define
\[f(x,z) = \frac12\|x-z\|^2,\quad x,z\in\R^d \]
and the corresponding expectation function
\[F(x) = \E[f(x,Z)],\quad x\in\R^d,\quad Z\sim\mu_Z\,. \]
Firstly, observe that this expectation computes as
\[F(x) = \frac{1}{2}\E[\|x-Z\|^2] = \frac12\E[\|x-\E[Z]\|^2] + \frac12\E[\|Z-\E[Z|\|^2] \]
such that we identify $x_\ast = \E[Z]\in\R^d$ as the global minimum of $F$. In this formulation we are also able to compute exact derivatives of $F$ and obtain
\begin{align*}
\E[\|\nabla_x f(x,Z) - \nabla_x F(x)\|^2] = \E[\|(x-Z)-(x-\E[Z])\|^2] = \E[\|Z-\E[Z]\|^2]=:\sigma^2\,.
\end{align*}
Next, consider the iteration generated by SGD with sequence of step sizes $\alpha_k=\frac{\tau}{k^{\nu}}$, $k\in\N$ for some $\nu>0$ and $\tau>0$, which can be written as
\[X_{k+1} = X_k - \frac{\tau}{k^{\nu}}(X_k-Z_{k+1}) = (1-\frac{\tau}{k^{\nu}})X_k + \frac{\tau}{k^{\nu}} Z_{k+1}\,. \]
In this example, the mean squared error can be computed explicitly,
\begin{align*}
\E[\|X_{k+1}-x_\ast \|^2] &= \E[\|X_{k+1}-\E[Z]\|^2]\\ &= \left(1-\frac{\tau}{k^{\nu}}\right)^2\E[\|X_k-\E[Z]\|^2] + 2\left(1-\frac{\tau}{k^{\nu}}\right) \frac{\tau}{k^{\nu}}\E[\langle X_k-\E[Z],Z_{k+1}-\E[Z]\rangle] \\ &\quad+ \left(\frac{\tau}{k^{\nu}}\right)^2 \E[ \|Z_{k+1}-\E[Z]\|^2]\\
&= \left(1-\frac{\tau}{k^{\nu}}\right)^2\E[\|X_k-\E[Z]\|^2]+ \left(\frac{\tau}{k^{\nu}}\right)^2\sigma^2\,.
\end{align*}
It follows that
\begin{align*}
\E[\|X_k-\E[Z]\|^2] &= \prod_{j=0}^{k-1} \left(1-\frac{\tau}{j^{\nu}}\right)^2 \E[\|X_0-\E[Z]\|^2] + \sigma^2 \sum_{j=0}^{k-1}\left(\frac{\tau}{k^{\nu}}\right)^2 \prod_{i=j+1}^{k-1}\left(1-\frac{\tau}{i^\nu}\right)^2\\
&\ge \sigma^2 \sum_{j=0}^{k-1}\left(\frac{\tau}{k^{\nu}}\right)^2 \prod_{i=j+1}^{k-1}\left(1-\frac{\tau}{i^\nu}\right)^2\,.
\end{align*}
The last expression is bounded from below by $Ck^{-\nu}$ for all sufficiently large $k\in\N$ and some constant $C>0$; see~\cite{JENTZEN2020101438} for details. 
For the choice $\nu=1$, this lower bound matches the upper rate derived in Section~\ref{sec:SGD_strcnvx} up to
constants.
\end{example}

\section{Almost sure convergence rates}\label{sec:as_rates}
In \Cref{sec:SGDas} we have proved almost sure convergence of SGD in the sense that the objective values converge almost surely to some limit and that
\[ \liminf_{k\to\infty} \|\nabla_x F(X_k)\|^2=0\quad \text{almost surely}.\]
The aim of the present section is to strengthen this statement by deriving almost sure convergence rates. The main idea is to combine the one-step descent estimate \eqref{eq:smooth_condexpest} with a rate version of the Robbins--Siegmund theorem. This approach follows the analysis of \cite{LY2022,pmlr-v134-sebbouh21a}, where almost sure convergence rates for SGD, stochastic heavy ball and stochastic Nesterov acceleration are derived under smoothness assumption. In this lecture, we restrict ourselves to the smooth non-convex and strongly convex case. 

\subsection{From supermartingales to almost sure convergence rates}
We first collect two elementary rate lemmas for stochastic processes. The first one is useful in the strongly convex case, where the descent inequality contains a contraction term. The second one is used in the non-convex case, where only summability of the weighted gradient norm is available.

\begin{lemma}[Almost sure rates for contracting supermartingales]\label{lem:asrates}
Let $(Z_k)_{k\in\mathbb{N}}$ be a sequence of non-negative random variables adapted to $(\mathcal{F}_k)_{k\in\mathbb{N}}$. Let \(\alpha_k=\frac{c_\alpha}{(k+1)^{\theta}}\), $k\in\N$ for some $\theta\in(\frac12,1)$. Suppose that there exist constants $c_1,c_2>0$ such that
\[
\mathbb{E}[Z_{k+1}\mid \mathcal{F}_k]
\leq
(1-c_1\alpha_k)Z_k+c_2\alpha_k^2
\]
Then, for every $\eta\in(2-2\theta,1)$,  
\[
\lim_{k\to\infty} (k+1)^{1-\eta}Z_k=0
\qquad \text{almost surely}.
\]
\end{lemma}

\begin{proof}
The proof follows the same idea as in the proof of Corollary~\ref{cor:RS}, but applied to a rescaled process. Define the adapted process $\widetilde Z_k:=(k+1)^{1-\eta}Z_k$.
Using the elementary estimate
\[
(k+2)^{1-\eta}
\leq
(k+1)^{1-\eta}+(1-\eta)(k+1)^{-\eta},
\]
we obtain, for all $k\in\N$,
\[
\begin{aligned}
\mathbb{E}[\widetilde Z_{k+1}\mid \mathcal{F}_k]
&=
(k+2)^{1-\eta}
\mathbb{E}[Z_{k+1}\mid \mathcal{F}_k]\leq
(k+2)^{1-\eta}
(1-c_1\alpha_k)Z_k
+
c_2(k+2)^{1-\eta}\alpha_k^2 \\
&\le (k+1)^{1-\eta} (1-\frac{c_1 c_\alpha}{(k+1)^{\theta}})Z_k + c_2(k+1)^{1-\eta}\frac{c_1 c_\alpha}{(k+1)^{2\theta}}\\ &\quad + (1-\eta)(k+1)^{-\eta} (1-\frac{c_1 c_\alpha}{(k+1)^{\theta}})Z_k + (1-\eta)c_2(k+1)^{-\eta} \frac{c_1 c_\alpha}{(k+1)^{2\theta}} \\
&=(1+\frac{1-\eta}{k+1}-\frac{c_1 c_\alpha}{(k+1)^{\theta}}-\frac{(1-\eta)c_1c_\alpha}{(k+1)^{\theta+1}})(k+1)^{1-\eta}Z_k\\
&+\quad \frac{c_1c_2c_\alpha}{(k+1)^{2\theta+\eta-1}} + \frac{(1-\eta)c_1c_2c_\alpha}{(k+1)^{2\theta+\eta}}\,.
\end{aligned}
\]
Since $\theta< 1$, there exists $\bar k\ge1$ sufficiently large and $c_3>0$ such that
\[ \frac{1-\eta}{k+1}-\frac{c_1c_\alpha}{(k+1)^\theta}\le - \frac{c_3}{(k+1)^\theta}\]
for all $k\ge\bar k$. Hence, for all $k\ge\bar k$,
\[
\begin{aligned}
\mathbb{E}[\widetilde Z_{k+1}\mid \mathcal{F}_k]\le (1-\frac{c_3}{(k+1)^\theta})\tilde Z_k + \frac{c_1c_2c_\alpha}{(k+1)^{2\theta+\eta-1}} + \frac{(1-\eta)c_1c_2c_\alpha}{(k+1)^{2\theta+\eta}}\,.
\end{aligned}
\]
Moreover, by assumption, we have
\[ \sum_{k=1}^\infty\frac{c_3}{(k+1)^\theta}=\infty,\quad \sum_{k=0}^\infty \frac{c_1c_2c_\alpha}{(k+1)^{2\theta+\eta-1}}<\infty\quad \text{and}\quad \sum_{k=0}^\infty \frac{(1-\eta)c_1c_2c_\alpha}{(k+1)^{2\theta+\eta}}<\infty\,.\]
Finally, Corollary~\ref{cor:RS} implies that $(\widetilde Z_k)_{k\in\mathbb{N}}$ converges almost surely to zero.
\end{proof}

In the non-convex regime, we require the following adjusted statement.
\begin{lemma}[From summability to rates]\label{lem:sum_rates}
Let $(a_k)_{k\in\mathbb{N}}$ be a decreasing sequence of strictly positive reals such that
\[
\sum_{k=0}^{\infty}a_k=\infty
\quad 
\text{and}
\quad
\sum_{k=1}^{\infty}\frac{a_k}{\sum_{i=0}^{k-1}a_i}=\infty.
\]
Suppose that $(Y_k)_{k\in\mathbb{N}}$ is a sequence of non-negative random variables such that
\[
\sum_{k=0}^{\infty}a_kY_k<\infty
\qquad \text{almost surely}.
\]
Then we have,
\[
\lim_{k\to\infty}\ \min_{0\leq i\leq k}Y_i\, \sum_{i=0}^{k-1} a_i
=
0
\qquad \text{almost surely}.
\]
\end{lemma}

\begin{proof}
We argue path-wise on the event $\{\sum_{k=0}^{\infty}a_kY_k<\infty\}$. Let
$
S_k:=\sum_{i=0}^{k}a_i.
$
Assume that the assertion is false. Then there exist $\delta>0$ and infinitely many $k\in\mathbb{N}$ (represented as a subsequence $(k_n)_{n\in\N}$) such that
\[
\min_{0\leq i\leq k}Y_i\, S_k \geq \delta.
\]
Since $(S_k)_{k\in\mathbb{N}}$ is increasing, this implies for the corresponding indices that
\[
a_iY_i\geq a_i\min_{0\le j\le k} Y_i\geq a_i\frac{\delta}{S_k}
\]
for all $i=0,\dots, k_n$ (and $n\in\N$). Summing over the indices gives, 
\[ \delta \lim_{n\to\infty} \sum_{i=1}^{k_n}\frac{a_i}{\sum_{j=0}^{k_n-1}a_j} \le \lim_{n\to\infty} \sum_{i=1}^{k_n} a_i Y_i\,,\]
which yields a contradiction to the finiteness of $\sum_{i=0}^{\infty}a_iY_i$.
\end{proof}

\subsection{Non-convex and smooth objective functions}
We first consider the smooth non-convex case. As already discussed, in the absence of convexity one can generally only expect convergence towards stationary points. Therefore, the natural quantity is again the gradient norm. This result is an almost sure counterpart to the classical convergence-rate statement in expectation from Theorem~\ref{thm:SGDnonconvex}. The essential ingredient is the summability property
\[
\sum_{k=0}^{\infty}
\alpha_k\|\nabla_xF(X_k)\|^2<\infty
\qquad \text{almost surely}.
\]
\begin{thm}[Almost sure rate for the minimal gradient norm]
Suppose that the assumptions of Theorem~\ref{thm:SGD_as}, and the additional assumptions that $\alpha_k\in(0,1/L]$ and 
\[ \sum_{k=1}^{\infty}
 \frac{\alpha_k}{\sum_{i=0}^{k-1}\alpha_i}
 =
 \infty
\]
are in place. Then it holds that
\[
\lim_{k\to\infty}\ \Big(\sum_{i=0}^{k}\alpha_i\Big)\min_{0\leq i\leq k}\|\nabla_xF(X_i)\|^2
=
0
\qquad \text{almost surely}.
\]
In particular, for the choice
$
\alpha_k=\frac{\alpha}{(k+1)^{\frac12+\eta}},
\
\alpha\in(0,1/L],\ \eta\in(0,\tfrac12),
$
we obtain
\[
\lim_{k\to\infty} (k+1)^{\frac12-\eta}\min_{0\leq i\leq k}\|\nabla_xF(X_i)\|^2
=0
\qquad \text{almost surely}.
\]
\end{thm}

\begin{proof}
Recall from \eqref{eq:smooth_condexpest}, that for all $k\in\N$, 
\begin{align*} 
\E[F(X_{k+1})-F^\ast\mid\cF_k] &\le (1+c\frac{L}2\alpha_k^2) (F(X_k)-F^\ast)+c\frac{L}{2}\alpha_k^2 - \alpha_k(1-\frac{L\alpha_k}{2})\|\nabla_xF(X_k)\|^2\\
&\le (1+c\frac{L}2\alpha_k^2) (F(X_k)-F^\ast)+c\frac{L}{2}\alpha_k^2 - \frac{\alpha_k}2\|\nabla_xF(X_k)\|^2\,.
\end{align*}
By Theorem~\ref{thm:RS},
\[
\sum_{k=0}^{\infty}
\alpha_k\|\nabla_xF(X_k)\|^2
<\infty
\qquad \text{almost surely}.
\]
Next, we apply Lemma~\ref{lem:sum_rates} with
\[
Y_k:=\|\nabla_xF(X_k)\|^2
\quad\text{and}\quad
a_k:=\alpha_k\,,
\]
which yields
\[
\lim_{k\to\infty}\ \Big(\sum_{i=0}^{k}\alpha_i\Big)\min_{0\leq i\leq k}\|\nabla_xF(X_i)\|^2
=
0
\qquad \text{almost surely}.
\]
For the polynomial step size
\[
\alpha_k=\frac{\alpha}{(k+1)^{\frac12+\eta}},
\]
we can compute the growth of the accumulated step sizes explicitly. Set
\[
S_k:=\sum_{i=0}^{k}\alpha_i
=
\alpha \sum_{i=0}^{k}\frac{1}{(i+1)^{\frac12+\eta}}
=
\alpha \sum_{j=1}^{k+1}\frac{1}{j^{\frac12+\eta}}
\]
and lower bound this sum via an integral similar as in the proof of Corollary~\ref{cor:SGDnonconvex}.
Since the function $x\mapsto x^{-\frac12-\eta}$ is positive and decreasing on $[1,\infty)$, one has:
\[
\sum_{j=1}^{k+1}\frac{1}{j^{\frac12+\eta}}
\geq
\int_{1}^{k+2} x^{-\frac12-\eta}\,{\mathrm d}x\,.
\]
Due to $\eta\in(0,\tfrac12)$, we have $\frac12-\eta>0$, and hence
\[
S_k \ge \alpha\,\int_{1}^{k+2} x^{-\frac12-\eta}\,{\mathrm d}x
=
\frac{\alpha}{\frac12-\eta}
\left((k+2)^{\frac12-\eta}-1\right)\,.
\] 
This implies that there exists a constant $B>0$ and an index $k_0\in\mathbb{N}$ such that
\[S_k=\sum_{i=0}^{k}\alpha_i\geq B k^{\frac12-\eta} \qquad \text{for all } k\geq k_0\,. \]
In summary, this verifies that
\[\lim_{k\to\infty}\ k^{\frac12-\eta}\min_{0\leq i\leq k}\|\nabla_xF(X_i)\|^2
\le \frac{1}{B}\lim_{k\to\infty} S_k \min_{0\leq i\leq k}\|\nabla_xF(X_i)\|^2 = 0
\qquad \text{almost surely}\,.
\]
\end{proof}
\subsection{Strongly convex and smooth objective functions}
We now turn to the strongly convex case. Recall that if $F$ is $\mu$-strongly convex, then there exists a unique global minimizer $x_\ast\in\mathbb{R}^d$, and by Proposition~\ref{prop:strongconvex_PL} it satisfies the PL condition
\[
\|\nabla_xF(x)\|^2\geq 2\mu(F(x)-F(x_\ast)).
\]
In contrast to the proof of Theorem~\ref{thm:SGD_strngconvex} and Theorem~\ref{thm:SGD_PL}, we work with the recursion under conditional expectation rather than full expectation. Note that the following result is directly based on the proof of Theorem~\ref{thm:SGD_PL}, and a similar statement can be derived in the setting of Theorem~\ref{thm:SGD_strngconvex}.
\begin{thm}[Almost sure rate under strong convexity]\label{thm:SGD_asPL}
Suppose that the assumptions of Theorem~\ref{thm:SGD_PL} are in place, and the additional requirement that 
\[
\alpha_k=\frac{\alpha}{(k+1)^{1-\theta}}\
\]
for some $\alpha\in(0,1/L]$ and $\theta\in(0,\frac12)$ holds.

Then, for every $\eta\in(2\theta,1)$, it holds that
\[
\lim_{k\to\infty}\ k^{1-\eta}(F(X_k)-F(x_\ast))
=0
\qquad \text{almost surely}.
\]
\end{thm}

\begin{proof}
We define
\[
Y_k:=F(X_k)-F(x_\ast).
\]
Similarly to the proof of Theorem~\ref{thm:SGD_PL}, combining \eqref{eq:smooth_condexpest} and the PL condition, we obtain
\[
\begin{aligned}
\mathbb{E}[Y_{k+1}\mid \mathcal{F}_k]
&\leq
Y_k
-
\frac{\alpha_k}2
\|\nabla_xF(X_k)\|^2
+
\frac{LC}{2}\alpha_k^2\\
&\leq Y_k
-
\mu\alpha_kY_k
+
\frac{LC}{2}\alpha_k^2
\,.
\end{aligned}
\]
Note that we have used $\alpha_k\le 1/L$. Therefore, Lemma~\ref{lem:asrates} yields
\[
\lim_{k\to\infty}\ k^{1-\eta}Y_k
= 0\qquad \text{almost surely}
\]
for every $\eta\in(2\theta,1)$. This proves the first assertion.
\end{proof}

The parameter $\theta\in(0,\frac12)$ can be chosen arbitrarily small, and then $\eta>2\theta$ can also be chosen arbitrarily small. Hence, the rate in Theorem~\ref{thm:SGD_asPL} can get arbitrarily close to $1/k$, the derived rate in Theorem~\ref{thm:SGD_PL} for the convergence in expectation. 

\section{Variance reduction}\label{sec:SGD_varred}
In the derived convergence results for SGD we have always assumed that $\nabla_x f(x,Z)$ is an unbiased estimator of $\nabla_x F(x)$ with uniformly bounded variance 
\[ \E[\|\nabla_x f(x,Z)-\nabla_x F(x)\|^2]\le {\mathrm{var}}\,.\]
This upper bound ${\mathrm{var}}>0$ occurs in all derived error bounds in a similar way:

\begin{center}
\begin{tabular}{|l|c|}
\hline 
Assumption & error bound \\
\hline
convex & $\frac{C_1}{\sqrt{k}} + {\mathrm{var}}\cdot \frac{\log(k)}{\sqrt{k}}$\\
strong convex & $(1-\alpha_k\mu) e_k + {\mathrm{var}} \cdot \alpha_k^2$ \\
PL-condition & $(1-\alpha_k r) e_k + {\mathrm{var}}\cdot \alpha_k^2$\\
\hline
\end{tabular}
\end{center}

In order to push the total error towards zero, we had to choose $\alpha_k\to0$. In the specific cases of strong convexity or under PL-condition, we lose the behavior of linear convergence which we obtained for the exact (deterministic) gradient descent scheme. In the deterministic setting, we were able to choose $\alpha_k=\bar\alpha>0$ such that
\[e_{k+1} \le \rho(\bar\alpha) e_k,\quad \rho(\bar\alpha) \in(0,1)\,\]
By controlling the variance error term through $\alpha_k\to0$, we obtain an error bound for SGD of the form
\[e_{k+1} \le \rho_k e_k + {\mathrm{var}}_k, \]
where $\rho_k\to1$ for $k\to\infty$. This behavior makes the analysis challenging and in particular, we have seen that SGD does not converge linearly. 

In the following section, we consider different types of variance reduction methods, which control the variance error term in a different way. We are no longer forced to consider $\alpha_k\to0$ and will choose a fixed step size $\alpha_k=\bar\alpha >0$ for all $k\in\N$.

\subsection{Dynamic Sampling}
Our first method to be considered is called \textit{dynamic sampling}, where we control the variance error term through a dynamical batch-sampling strategy. The unbiased estimator of the descent direction $\nabla_x F(X_k)$ is estimated through a batch of samples $(Z_k^{(m)})_{k\in\N,\ m=1,\dots, B_{k-1}}$, where the random variables are assumed to be independent in $k$ and in $m$ with an identical distribution $\mu_Z$.

\begin{algorithm}[htb!]
\begin{algorithmic}[1]
\State \textbf{Input:} \begin{itemize}
 \item loss function $f:\R^d\times \R^p\to\R$
 \item initial random variable $X_0:\Omega\to\R^d$
 \item sequence of step sizes $(\alpha_k)_{k\in\N}$, $\alpha_k>0$ (deterministic or $\mathcal F$-adapted)
 \item sequence of batch sizes $(B_k)_{k\in\N}$
 \item sequence of i.i.d.~random variables $(Z_k^{(m)})_{k\in\N,\ m=1,\dots,B_{k-1}}$ with $Z_1^{(1)}\sim\mu_Z$. 
 \end{itemize}
 \State set $k=0$
\While{"convergence/stopping criterion not met"}
	\State approximate the gradient $\nabla_x F(X_k)$ through
	\[G_k = \frac{1}{B_{k}}\sum_{m=1}^{B_k}\nabla_x f(X_k,Z_{k+1}^{(m)}) \]
	\State set $X_{k+1} = X_k -\alpha_k G_k$, $k \mapsto k+1$
\EndWhile
\end{algorithmic}
 \caption{SGD with dynamic sampling}\label{alg:SGD_DS}
\end{algorithm}

In the following, we will analyze SGD with dynamical batch-sampling. The focus will be placed on the strongly convex setting and the method will be compared to a fixed batch size $\bar B>0$ across all iterations. For both schemes, we will consider a fixed step size $\bar \alpha>0$.

Firstly, we discuss how the assumed uniform upper bound on the variance is effected through the incorporation of batch-sampling. 
\begin{lemma}\label{lem:batching}
Let the assumptions of Lemma~\ref{lem:condexp} be satisfied and assume that
\[\E[\|\nabla_x f(x,Z) - \nabla_x F(x)\|^2]\le c \]
for some $c>0$ and all $x\in\R^d$. Moreover, let $Z^{(1)},\dots, Z^{(B)}$ be i.i.d.~random variables with distribution $\mu_Z$. Then for $G:=\frac{1}{B}\sum_{m=1}^B\nabla_x f(x,Z^{(m)})$ it holds true that
\[ \E[\|G-\nabla_x F(x)\|^2 \le \frac{c}{B} \]
for all $x\in\R^d$. 
\end{lemma}
\begin{proof}
By $\E[\nabla_x f(x,Z^{(1)})] = \nabla_x F(x)$ and the independence of $Z^{(1)},\dots,Z^{(m)}$ we have
\begin{align*}
\E[\|G-\nabla_x F(x) \|^2] &= \frac1{B^2}\sum_{m,n=1}^B\E[\langle \nabla_x f(x,Z^{(m)}) - \nabla_x F(x) , \nabla_x f(x,Z^{(n)}) - \nabla_x F(x)\rangle ]\\ &= \frac{1}{B^2} \sum_{m=1}^B \E[\|\nabla_x f(x,Z^{(m)})- \nabla_x F(x)\|^2] \le \frac{c}{B}\,.
\end{align*}
\end{proof}

With the previous observations we are now able to extend Theorem~\ref{thm:SGD_strngconvex} to dynamical batch-sampling. 
\begin{thm}[SGD with dynamic sampling]\label{thm:SGD_DS}
Let $F:\R^d\to\R$ be $\mu$-strongly convex and $L$-smooth. We assume that the assumptions of Lemma~\ref{lem:condexp} are satisfied and that there exists $c>0$ such that
\[\E[\|\nabla_x f(x,Z) - \E[\nabla_x f(x,Z)] \|^2] \le c \]
for all $x\in\R^d$. Let $X_0$ be a random variable such that $\E[|F(X_0)| + \|X_0-x_\ast\|^2]<\infty$, where $x_\ast\in\R^d$ is the unique global minimum of $F$. Moreover, let $(X_k)_{k\in\N}$ be generated by Algorithm~\ref{alg:SGD_DS} with sequence of batch sizes $(B_k)_{k\in\N}$, $B_k\ge1$ and deterministic, decreasing sequence of step sizes $(\alpha_k)_{k\in\N}$ such that $\alpha_k\in(0,\frac1L]$. Then for the error $e_k:=\E[\|X_k-x_\ast\|^2]$ it holds true that
\[e_{k+1} \le (1-\alpha_k \mu) e_k + \frac{c\alpha_k^2}{B_k}\]
for all $k\ge0$. Furthermore, for a fixed step size $\alpha_k = \bar\alpha = \frac{\tau}{\mu}$ with $\tau\le \frac{1}{\kappa} = \frac{\mu}{L}$, it holds true that
\begin{equation}\label{eq:SGD_iterativeerror}
e_{k+1} \le \rho e_k + \frac{c\bar{\alpha}^2}{B_k}
\end{equation}
where $\rho = (1-\tau)\in(0,1)$.
\end{thm}
\begin{proof}
The proof follows by Theorem~\ref{thm:SGD_strngconvex} combined with Lemma~\ref{lem:batching}.
\end{proof}

In the following, we will keep the step size $\bar \alpha = \frac{\tau}{\mu}\le \frac1L$ fixed and aim to derive an optimal sequence of batch sizes. Firstly, we need to formulate what we mean by an \textit{optimal} batch size. Therefore, we will assume that the computation of each iteration of SGD occurs with cost which are determined through the generation of the samples $(Z_k^{(m)})_{m=1,\dots,B_{k-1}}$. We will assume that these cost are normalized.
\begin{ass}
The generation of the state $X_k$ by Algorithm~\ref{alg:SGD_DS} with sequence of batch sizes $(B_k)_{k\in\N}$ occurs with computational cost
\[{\mathrm{cost}}(X_k) = \sum_{j=0}^{k-1} B_j\,. \]
\end{ass}

We want to choose the batch sizes such that we minimize the computational cost under the constraint that the final error bound is below a specified tolerance level $\varepsilon>0$, i.e.~we want to solve the constrained minimization problem
\[\min_{B_0,\dots,B_{K-1}}\ \sum_{j=0}^{K-1} B_j\quad \text{s.t.}\ e_K\le \varepsilon\,. \]
We refer the interested reader to \cite{pmlr-v178-weissmann22a}, where this approach has been considered in a more general framework. The total error after iteration $K\ge1$ can be upper bounded by iterating the error bound \eqref{eq:SGD_iterativeerror}:
\[e_K \le \rho^K e_0 + c\bar\alpha^2 \sum_{j=0}^{K-1} \rho^{K-1-j} \frac{1}{B_j}\,.\]
For simplicity, we assume that $K\ge\ceil{\log(\rho)^{-1}\log(\frac{\varepsilon}2e_0^{-1})}$ such that
\[e_K\le \frac{\varepsilon}{2}.\]
For this given $K$ we want to determine $B_0,\dots,B_{K-1}$ under the constrain
\[ c\bar\alpha^2 \sum_{j=0}^{K-1} \rho^{K-1-j} \frac{1}{B_j}\le \frac{\varepsilon}2\,.\]
We start with the following auxiliary result:
\begin{lemma}\label{lem:optimalbatch}
Let $\varepsilon>0$, $\gamma>0$ and $a_j>0$, $j\in\{0,\dots,K-1\}$. Then the choice
\[B_j = C(\varepsilon,K)\cdot a_j^{\frac{1}{1+\gamma}},\quad C(\varepsilon,K) = \varepsilon^{-\frac1\gamma}\left(\sum_{s=0}^{K-1} a_s^{\frac1{1+\gamma}}\right)^{\frac1\gamma} \]
solves the constrained optimization problem
\[\min_{B_0,\dots,B_{K-1}}\ \sum_{j=0}^{K-1} B_j,\quad \text{s.t.}\ \sum_{j=0}^{K-1} a_j B_j^{-\gamma}\le \varepsilon\, . \]
\end{lemma}
\begin{proof}
We only derive a stationary point of the considered constrained optimization problem. 
The Lagrange function is given by
\[\cL(B_0,\dots,B_{K-1},\lambda) = \sum_{j=0}^{K-1} B_j + \lambda\left(\sum_{j=0}^{K-1} a_j B_j^{-\gamma} - \varepsilon\right) \]
and the corresponding optimality conditions are
\begin{align*}
({\mathrm I})& \quad 1-\lambda -\lambda\gamma a_j B_j^{-(1+\gamma)} = 0,\ j=0,\dots,K-1,\\
({\mathrm{II}})& \quad \sum_{j=0}^{K-1} a_j B_j^{-\gamma} -\varepsilon =0\,.
\end{align*}
We solve $({\mathrm{I}})$ to derive
\[B_j = (\lambda \gamma)^{\frac{1}{1+\gamma}} a_j^{\frac{1}{1+\gamma}}, \]
which together with $({\mathrm{II}})$ gives
\[(\lambda\gamma)^{-\frac{\gamma}{1+\gamma}} \sum_{j=0}^{K-1} a_j\cdot a_j^{-\frac{\gamma}{\gamma+1}} = (\lambda\gamma)^{-\frac{\gamma}{1+\gamma}} \sum_{j=0}^{K-1}  a_j^{\frac{1}{\gamma+1}} = \varepsilon\]
and therefore, 
\[\lambda\gamma = \varepsilon^{-\frac{1+\gamma}{\gamma}} \left( \sum_{j=0}^{K-1} a_j^{\frac{1}{1+\gamma}}\right)^{\frac{1+\gamma}{\gamma}}\,. \]
This results in 
\[B_j = C(\varepsilon,K)\cdot a_j^{\frac{1}{1+\gamma}},\quad C(\varepsilon,K) = \varepsilon^{-\frac1\gamma}\left(\sum_{s=0}^{K-1} a_s^{\frac1{1+\gamma}}\right)^{\frac1\gamma}\,. \]
\end{proof}

We are now ready to choose the optimal batch size for Algorithm~\ref{alg:SGD_DS} under strong convexity assumption: 
\[\min_{B_0,\dots,B_{K-1}}\ \sum_{j=0}^{K-1} B_j,\quad \text{s.t.}\ \sum_{j=0}^{K-1}c\bar\alpha^2 \rho^{K-1-j} B_j^{-1} \le \varepsilon/2\, , \]
which by Lemma~\ref{lem:optimalbatch} leads to
\[C(\varepsilon,K) = 2\varepsilon^{-1} \sqrt{c}\bar\alpha \sum_{j=0}^{K-1}\rho^{\frac{K-1-j}2} = 2\varepsilon^{-1} \sqrt{c}\bar\alpha \sum_{j=0}^{K-1}\rho^{\frac{j}2} = 2\varepsilon^{-1} \sqrt{c}\bar\alpha\left(\frac{1-\rho^{\frac{K}2}}{1-\rho^{\frac12}}\right) \]
and therefore, to an optimal dynamical batch size
\begin{equation}\label{eq:optimalbatch}
B_j = 2\varepsilon^{-1} c\bar\alpha^2\left(\frac{1-\rho^{\frac{K}2}}{1-\rho^{\frac12}}\right)\rho^{\frac{K-1-j}2}\,.
\end{equation}
The corresponding computational cost are given by
\[\sum_{j=0}^{K-1} B_j = 2\varepsilon^{-1} c\bar\alpha^2\left(\frac{1-\rho^{\frac{K}2}}{1-\rho^{\frac12}}\right)\sum_{j=0}^{K-1}\rho^{\frac{K-1-j}2} = 2\varepsilon^{-1} c\bar\alpha^2\left(\frac{1-\rho^{\frac{K}2}}{1-\rho^{\frac12}}\right)^2 \simeq \varepsilon^{-1}, \]
where $\left(\frac{1-\rho^{\frac{K}2}}{1-\rho^{\frac12}}\right)^2\in\left(1,\left(\frac{1}{1-\sqrt{\rho}}\right)^2\right)$, independent of $K$. We compare the derived dynamical batch-sampling strategy to a fixed batch size $\bar B\ge 1$ for all $k = 0,\dots,K-1$. This fixed batch size has again to be chosen such that $e_K\le \varepsilon$. For simplicity, let again $K\ge\ceil{\log(\rho^{-1})\log(\frac{\varepsilon}2e_0^{-1})}$ such that $\rho^Ke_0\le\frac{\varepsilon}2$ and therefore,
\[e_K \le \frac{\varepsilon}2 + c\bar\alpha^2 \frac{1}{\bar B} \sum_{j=0}^{K-1} \rho^{K-1-j} = \frac{\varepsilon}2 + c\bar\alpha^2\frac{1-\rho^{K}}{1-\rho} \frac1{\bar B},  \]
where $\frac{1-\rho^{K}}{1-\rho}\le \frac{1}{1-\rho}$. The fixed batch size $\bar B$ needs to be chosen such that
\begin{equation}\label{eq:fixedbatch}
\bar B\ge 2\varepsilon^{-1} c\bar\alpha^2 (1-\rho)^{-1} \simeq \varepsilon^{-1}
\end{equation}
and the corresponding computational cost are given by
\[\sum_{j=0}^{K-1} B_j = K\cdot \bar B \simeq |\log(\varepsilon^{-1})|\varepsilon^{-1}\,. \]
We summarize the derived batch-sampling strategies in the following theorem.
\begin{thm}
Suppose that the conditions of Theorem~\ref{thm:SGD_DS} are satisfied and define $e_0 = \E[\|X_0-x_\ast\|^2]$. Let $\varepsilon>0$ and $K\ge\ceil{\log(\rho^{-1})\log(\frac{\varepsilon}2e_0^{-1})}$. Let $(X_k^{\mathrm{DS}})_{k=0,\dots,K}$ be generated by Algorithm~\ref{alg:SGD_DS} with the dynamical batch sizes
$(B_k)_{k=0,\dots,K-1}$ defined in~\eqref{eq:optimalbatch}. Moreover, let
$(X_k^{\mathrm{FB}})_{k=0,\dots,K}$ be generated by Algorithm~\ref{alg:SGD_DS} with the fixed batch size $B_k=\bar B$ from~\eqref{eq:fixedbatch}.
Then
\begin{align*}
e_K^{\mathrm{DS}} := \E[\|X_K^{\mathrm{DS}}-x_\ast\|^2] &\le \varepsilon\quad \text{and}\quad 
e_K^{\mathrm{FB}} := \E[\|X_K^{\mathrm{FB}}-x_\ast\|^2] \le \varepsilon\,,
\end{align*}
while the computational cost is given by
\begin{align*}
{\mathrm{cost}}^{\mathrm{DS}} := {\mathrm{cost}}(X_K^{\mathrm{DS}}) = \sum_{j=0}^{K-1}B_j &\simeq \varepsilon^{-1},\\
{\mathrm{cost}}^{\mathrm{FB}} := {\mathrm{cost}}(X_K^{\mathrm{FB}}) = K\cdot \bar B &\simeq |\log(\varepsilon^{-1})|\varepsilon^{-1}\,.
\end{align*}
\end{thm}

\subsection{Stochastic average gradient method (SAG)}

In the following three subsections, we consider variance reduction methods for finite-sum optimization problems. We fix the realization of the data set \(\{z^{(1)},\dots,z^{(N)}\}\subset\R^p\) and focus on the empirical objective
\[
F_N(x)=\frac1N\sum_{i=1}^N f(x,z^{(i)})
=:\frac1N\sum_{i=1}^N f_i(x),
\qquad x\in\R^d,
\]
where $f_i(x):=f(x,z^{(i)})$. In this section, we do not study the statistical error caused by replacing the expected risk by the empirical risk.

If $F_N$ is $\mu$-strongly convex and $L$-smooth, the deterministic gradient descent scheme applied to $F_N$ converges linearly to its unique minimizer. However, each iteration requires the evaluation of the full gradient
\[\nabla F_N(x_k)=\frac1N\sum_{i=1}^N\nabla f_i(x_k),\]
and therefore one pass over the entire data set. In contrast, finite-sum SGD evaluates only one component gradient $\nabla f_{i_k}(x_k)$ per iteration. This reduces the cost of a single iteration, but the persistent variance of the stochastic
gradient estimator prevents the same linear convergence behavior with a fixed step size. Variance reduction methods are designed to combine the advantages of both approaches: cheap stochastic updates and linear convergence under strong convexity.
In each iteration, the goal is to evaluate only one new gradient across the family of functions $\{f_i\}_{i=1,\dots,N}$ such as it is the case in Algorithm~\ref{alg:SGD2}. Note that all of these algorithms essentially assume that the objective function is in the form of a finite sum $F_N$. 

In the following it is clear out of context, that the gradient $\nabla_x f_i(\cdot)$ is computed wrt.~$x$, such that from now on we will omit the dependence on $x$ and simply write $\nabla f_i(\cdot)$.

\paragraph{Motivation:}
Suppose that we want to estimate an unknown parameter $\theta\in\R$ and $G$ be an unbiased estimator of $\theta$, i.e.~$\E[G]=\theta$. Moreover, let $\xi$ be a random variable with mean close to zero, $\E[\xi]\approx 0$, such that the modified random variable $G_\xi := G-\xi$ is \textit{nearly} unbiased, i.e.~$\E[G_\xi] = \E[G]-\E[\xi]\approx \theta$. (In case $\E[\xi]=0$, $G_\xi$ even remains unbiased). The modification becomes interesting when considering the resulting variance:
\[\V(G_\xi) = \V(G-\xi) = \V(G)+\V(\xi)-2{\mathrm{Cov}}(G,\xi)\,.\]
So in case we find a high (positive) correlation between $G$ and $\xi$, we hope for a significant reduction of the variance without introducing a large bias. This concept can be viewed as motivation for the following three algorithms to be considered.\medskip

We consider the first algorithm which forms the basis for introducing variance reduction in SGD. In \cite{Schmidt2017} the authors propose the \textit{stochastic average gradient} (SAG) method, which achieves linear convergence for strongly convex objective functions while having same complexity characteristics as SGD. The idea is to reuse the gradient information obtained from the past.

\begin{algorithm}[htb!]
\begin{algorithmic}[1]
\State \textbf{Input:} \begin{itemize}
 \item objective function $F_N:\R^d\to\R$, $F_N(x) = \frac1N\sum_{i=1}^N f_i(x)$
 \item initial random variable $X_0:\Omega\to\R^d$
 \item sequence of step sizes $(\alpha_k)_{k\in\N}$, $\alpha_k>0$ (deterministic or $\mathcal F$-adapted)
 \end{itemize}
 \State set $k=0$, initialize $G_0^{(i)}=0$, $i=1,\dots,N$
 \State compute $\bar G_0 = \frac{1}{N} \sum_{i=1}^N G_0^{(i)}$
\While{"convergence/stopping criterion not met"}
	\State generate independently $\mathfrak i_{k+1} \sim \mathcal U(\{1,\dots,N\})$
	\State set $G_k^{(i)} = \begin{cases} \nabla f_i(X_k), & i = \mathfrak i_{k+1} \\ G_{k-1}^{(i)}, & i \neq \mathfrak i_{k+1} \end{cases}$
	\State approximate the gradient $\nabla F_N(X_k)$ through
	\[\bar G_k = \frac1N\sum_{i=1}^N G_k^{(i)} = \bar G_{k-1} - \frac1N G_{k-1}^{(\mathfrak i_{k+1})} + \frac1N \nabla f_{\mathfrak i_{k+1}}(X_k) \]
	\State set $X_{k+1} = X_k -\alpha_k \bar G_k$, $k \mapsto k+1$
\EndWhile
\end{algorithmic}
 \caption{Stochastic average gradient method (SAG)}\label{alg:SAG}
\end{algorithm}

The algorithm stores all gradient approximations across the entire data set $i=1,\dots,N$, and in each iteration it updates the approximation of the gradient $f_{\mathfrak i}$ for only one randomly picked index $\mathfrak i$. Similarly to SGD in the form of Algorithm~\ref{alg:SGD2}, only one new gradient needs to be evaluated per iteration. However, one needs to have capacity for storing the gradient approximation for each index $i=1,\dots,N$, which can be seen as the main disadvantage of SAG. In contrast to the basic SGD estimator, the SAG estimator is generally biased. As mentioned above, the algorithm achieves linear convergence toward the global optimum of $F_N$ as shown in \cite{Schmidt2017}.

\begin{thm}[Theorem~1 in \cite{Schmidt2017}]
Let $F_N,\ f_i:\R^d\to\R$, $i=1,\dots,N$, be $\mu$-strongly convex and $L$-smooth. Moreover, let $(X_k)_{k\in\N}$ be generated by Algorithm~\ref{alg:SAG} with fixed step size $\alpha_k = \bar\alpha = \frac{1}{16L}$. Then it holds true that
\[\frac\mu2\E[\|X_k-x_\ast\|^2]\le \E[F_N(X_k)-F_N(x_\ast)] \le \left(1-\min\left\{\frac{\mu}{16L},\frac{1}{8N}\right\}\right)^k C_0, \]
where $C_0 = \E[F_N(X_0)-F_N(x_\ast)] + \frac{4L}{N} \E[\|X_0-x_\ast\|^2] + \frac{\sigma^2}{16L}$ with $\sigma^2 = \frac1N\sum_{i=1}^N \|\nabla f_i(x_\ast)\|^2$. 
\end{thm}

\begin{remark}
Ignoring the cost of storing $\{G_k^{(i)}\}_{i=1,\dots,N}$ we can run $N$ iterations of Algorithm~\ref{alg:SAG} to achieve a similar complexity as the deterministic full gradient descent scheme. To compare both SAG and deterministic GD we can view SAG as linearly converging scheme with rate \[\rho^{\mathrm{SAG}}=\left(1-\min\left\{\frac{\mu}{16L},\frac{1}{8N}\right\}\right)^N\,.\] For small $N$ ($N\le \frac{2L}{\mu}$) the rate is dominated by 
\[\rho^{\mathrm{SAG}}=\left(1-\frac{\mu}{16L}\right)^N\]
which corresponds to $N$ steps of GD with step size $\bar\alpha=\frac{1}{16L}$. In comparison, for large $N$ ($N\ge\frac{2L}{\mu} = 2\kappa$) the rate is dominated by
\[\rho^{\mathrm{SAG}}=\left(1-\frac{1}{8N}\right)^N\le \exp(-\frac18)\,,\]
such that the rate can be uniformly controlled in $N$.
\end{remark}

\subsection{SAGA}

The convergence analysis of SAG is challenging due to the biased estimation of the gradients $\nabla F_N(X_k)$ in each iteration:
\[ \E[\bar G_{k}\mid\cF_k] = \underbrace{\E[\bar G_{k-1} - \frac{1}{N}G_{k-1}^{(\mathfrak i_{k+1})}]}_{\neq0} + \frac1N \nabla F_N(X_k) = (1-\frac1N) \bar G_{k-1} + \frac1N \nabla F_N(X_k)\,.\]
This problem can be solved by replacing the approximation of $\nabla F_N(X_k)$ through an unbiased estimator of form
\[\bar G_k = \frac1N\sum_{i=1}^N G_{k-1}^{(i)} - G_{k-1}^{(\mathfrak i_{k+1})} +\nabla f_{\mathfrak i_{k+1}}(X_k)\,.\]
This estimator is unbiased since 
\[\E[\frac1N\sum_{i=1}^N G_{k-1}^{(i)} - G_{k-1}^{(\mathfrak i_{k+1})}\mid \cF_k] = 0\quad\text{and}\quad \E[\nabla f_{\mathfrak i_{k+1}}(X_k)\mid\cF_k] = \nabla F_N(X_k)\,.\]
Note that $\bar G_k$ does not correspond to the mean over all stored gradients $\bar G_k \neq \frac{1}{N}\sum_{i=1}^N G_{k}^{(i)}$ anymore. 

This observation led to a modified algorithm called \textit{SAGA} which has been introduced in \cite{NIPS2014_ede7e2b6}.

\begin{algorithm}[htb!]
\begin{algorithmic}[1]
\State \textbf{Input:} \begin{itemize}
 \item objective function $F_N:\R^d\to\R$, $F_N(x) = \frac1N\sum_{i=1}^N f_i(x)$
 \item initial random variable $X_0:\Omega\to\R^d$
 \item sequence of step sizes $(\alpha_k)_{k\in\N}$, $\alpha_k>0$ (deterministic or $\mathcal F$-adapted)
 \end{itemize}
 \State set $k=0$, initialize $G_0^{(i)}=0$, $i=1,\dots,N$
 \State compute $\bar G_0 = \frac{1}{N} \sum_{i=1}^N G_0^{(i)}$
\While{"convergence/stopping criterion not met"}
	\State generate independently $\mathfrak i_{k+1} \sim \mathcal U(\{1,\dots,N\})$
	\State set $G_k^{(i)} = \begin{cases} \nabla f_i(X_k), & i = \mathfrak i_{k+1} \\ G_{k-1}^{(i)}, & i \neq \mathfrak i_{k+1} \end{cases}$
	\State approximate the gradient $\nabla F_N(X_k)$ through 
	\[\bar G_k = \bar G_{k-1} - G_{k-1}^{(\mathfrak i_{k+1})} + \nabla f_{\mathfrak i_{k+1}}(X_k) 
    \]
	\State set $X_{k+1} = X_k -\alpha_k \bar G_k$, $k \mapsto k+1$
\EndWhile
\end{algorithmic}
 \caption{SAGA}\label{alg:SAGA}
\end{algorithm}

We follow the proof of linear convergence under strong convexity for Algorithm~\ref{alg:SAGA} presented in \cite{NIPS2014_ede7e2b6}. Let us assume that $F_N$, $f_i$, $i=1,\dots,N$ are $\mu$-strongly convex and $L$-smooth. For $X_0(\omega) = x_0\in\R^d$ we define point-wise
\[ \phi_0^{(i)}(\omega) = x_0,\quad \phi_{k+1}^{(i)}(\omega) = \begin{cases} \phi_k^{(i)}(\omega), & i\neq \mathfrak i_{k+1}(\omega)\\ X_k(\omega), &i=\mathfrak i_{k+1}(\omega) \end{cases} \]
and consider the error function of form
\begin{align*}
E_k = \underbrace{c\|X_k-x_\ast\|^2}_{=:E_k^{(2)}} + \underbrace{\frac1N\sum_{i=1}^N \left( f_i(\phi_k^{(i)})-f_i(x_\ast) - \langle \nabla f_i(x_\ast), \phi_k^{(i)}-x_\ast\rangle\right)}_{=:E_k^{(1)}}\ge c\|X_k-x_\ast\|^2,
\end{align*}
where $x_\ast = \arg\min_{x\in\R^d}\ F_N(x)$. Note that $E_k^{(1)}\ge0$ by convexity of $f_i$. We observe that
\[\|X_{k+1}-x_\ast\|^2 = \|X_k-x_\ast\|^2 - 2\bar\alpha\langle X_k-x_\ast,\bar G_k\rangle +\bar\alpha^2 \|\bar G_k\|^2 \]
and with $\cF_k:=\sigma(X_0,\mathfrak i_m,m\le k)$ and $\E[\bar G_k\mid\cF_k] = \nabla F_N(X_k)$ we have that
\[\E[\|X_{k+1}-x_\ast\|^2\mid \cF_k] = \|X_k-x_\ast\|^2 -2\bar\alpha \langle X_k-x_\ast, \nabla F_N(X_k)\rangle + \bar\alpha^2 \E[\|\bar G_k\|^2\mid \cF_k]\,.\]
In order to obtain an improved convergence result compared to SGD, we need to control $\E[\|\bar G_k\|^2\mid \cF_k]$ sufficiently well.
\begin{lemma}\label{lem:varred}
Let $F_N,\ f_i:\R^d\to\R$, $i=1,\dots,N$, be $\mu$-strongly convex and $L$-smooth. Then for any $\beta>0$ it holds true that 
\begin{align*}
\E[\|\bar G_k\|^2\mid \cF_k] &\le (1+\beta^{-1}) \frac1N\sum_{i=1}^N \|\nabla f_i(\phi_k^{(i)})-\nabla f_i(x_\ast)\|^2\\&\quad + (1+\beta) \frac{1}{N}\sum_{i=1}^N \|\nabla f_i(X_k)-\nabla f_i(x_\ast)\|^2 - \beta\|\nabla F_N(X_k)\|^2\,,
\end{align*}
where $x_\ast =\arg\min_{x\in\R^d}\ F_N(x)$.
\end{lemma}
\begin{proof}
We will apply multiple times the following equality
\begin{equation}\label{eq:variance_formula}
\E[\|Q-\E[Q]\|^2] = \E[\|Q\|^2]-\|\E[Q]\|^2 
\end{equation}
for any random vector $Q$ with $\E[\|Q\|^2]<\infty$. By construction of $\bar G_k$ we can write
\begin{align*}
\E[\|\bar G_k\|^2 \mid \cF_k] &= \E[\|\frac1N\sum_{i=1}^N \nabla f_i(\phi_k^{(i)})- \nabla f_{\mathfrak i_{k+1}}(\phi_k^{(\mathfrak i_{k+1})})+\nabla f_{\mathfrak i_{k+1}}(X_k)\|^2\mid \cF_k]\\
&=: \E[\|Q\|^2\mid \cF_k] = \E[\|Q-\E[Q\mid\cF_k]\|^2\mid \cF_k] + \|\E[Q\mid\cF_k]\|^2\\
&=\E[\|\frac1N\sum_{i=1}^N \nabla f_i(\phi_k^{(i)})- \nabla f_{\mathfrak i_{k+1}}(\phi_k^{(\mathfrak i_{k+1})})+\nabla f_{\mathfrak i_{k+1}}(X_k) - \nabla F_N(X_k)\|^2\mid \cF_k]\\ &\quad+\|\nabla F_N(X_k)\|^2
\end{align*}
Let us consider the first term
\begin{align*}
&\E[\|\frac1N\sum_{i=1}^N \nabla f_i(\phi_k^{(i)})- \nabla f_{\mathfrak i_{k+1}}(\phi_k^{(\mathfrak i_{k+1})})+\nabla f_{\mathfrak i_{k+1}}(X_k) - \nabla F_N(X_k)\|^2\mid \cF_k]\\
&=\E[\|\left\{\nabla f_{\mathfrak i_{k+1}}(\phi_k^{(\mathfrak i_{k+1})}) - \nabla f_{\mathfrak i_{k+1}}(x_\ast)-\frac1N\sum_{i=1}^N \nabla f_i(\phi_k^{(i)})\right\}\\&\quad\quad\quad- \left\{\nabla f_{\mathfrak i_{k+1}}(X_k) - \nabla f_{\mathfrak i_{k+1}}(x_\ast) - \nabla F_N(X_k)\right\}\|^2\mid \cF_k]\\
&\le (1+\beta^{-1}) \E[\|\nabla f_{\mathfrak i_{k+1}}(\phi_k^{(\mathfrak i_{k+1})}) - \nabla f_{\mathfrak i_{k+1}}(x_\ast)-\frac1N\sum_{i=1}^N \nabla f_i(\phi_k^{(i)})\|^2\mid\cF_k]\\ &\quad+(1+\beta)\E[\|\nabla f_{\mathfrak i_{k+1}}(X_k) - \nabla f_{\mathfrak i_{k+1}}(x_\ast) - \nabla F_N(X_k)\|^2\mid\cF_k]\,,
\end{align*}
where we have used that $\|x+y\|^2 \le (1+\beta^{-1})\|x\|^2+(1+\beta)\|y\|^2$ for any $\beta>0$. We define $Q_1 = \nabla f_{\mathfrak i_{k+1}}(\phi_k^{(\mathfrak i_{k+1})}) - \nabla f_{\mathfrak i_{k+1}}(x_\ast)$ with
\[\E[Q_1\mid\cF_k] = \frac1N\sum_{i=1}^N \nabla f_i(\phi_k^{(i)})-\nabla F_N(x_\ast)=\frac1N\sum_{i=1}^N \nabla f_i(\phi_k^{(i)})\,, \]
and similarly
\[Q_2 = \nabla f_{\mathfrak i_{k+1}}(X_k) - \nabla f_{\mathfrak i_{k+1}}(x_\ast) \quad \text{with}\quad \E[Q_2\mid\cF_k] = \nabla F_N(X_k)\,.\]
Finally, we obtain
\begin{align*}
&(1+\beta^{-1}) \E[\|\nabla f_{\mathfrak i_{k+1}}(\phi_k^{(\mathfrak i_{k+1})}) - \nabla f_{\mathfrak i_{k+1}}(x_\ast)-\frac1N\sum_{i=1}^N \nabla f_i(\phi_k^{(i)})\|^2\mid\cF_k]\\ &\quad+(1+\beta)\E[\|\nabla f_{\mathfrak i_{k+1}}(X_k) - \nabla f_{\mathfrak i_{k+1}}(x_\ast) - \nabla F_N(X_k)\|^2\mid\cF_k]\\
&\le (1+\beta^{-1})\Big\{\E[\|\nabla f_{\mathfrak i_{k+1}}(\phi_k^{(\mathfrak i_{k+1})}) - \nabla f_{\mathfrak i_{k+1}}(x_\ast)\|^2\mid\cF_k]-\|\frac1N\sum_{i=1}^N \nabla f_i(\phi_k^{(i)})\|^2 \Big\}\\
&\quad+ (1+\beta)\left\{\E[\|\nabla f_{\mathfrak i_{k+1}}(X_k) - \nabla f_{\mathfrak i_{k+1}}(x_\ast)\|^2 \mid\cF_k]- \|\nabla F_N(X_k)\|^2 \right\}
\end{align*}
and all together
\begin{align*}
\E[\|\bar G_k\|^2\mid \cF_k] &\le (1+\beta^{-1}) \frac1N\sum_{i=1}^N \|\nabla f_i(\phi_k^{(i)})-\nabla f_i(x_\ast)\|^2\\&\quad + (1+\beta) \frac{1}{N}\sum_{i=1}^N \|\nabla f_i(X_k)-\nabla f_i(x_\ast)\|^2 - \beta\|\nabla F_N(X_k)\|^2\,.
\end{align*}
\end{proof}

Applying the upper bound in Lemma~\ref{lem:varred} results in
\begin{equation}\label{eq:SAGA_errorbound1}
\begin{split}
\E[\|X_{k+1}-x_\ast\|^2\mid \cF_k]&\le \|X_k-x_\ast\|^2 -2\bar\alpha \langle X_k-x_\ast, \nabla F_N(X_k)\rangle -\bar\alpha^2\beta \|\nabla F_N(X_k)\|^2\\
&\quad +\bar\alpha^2 (1+\beta^{-1})\frac1N\sum_{i=1}^N\|\nabla f_i(\phi_k^{(i)})-\nabla f_i(x_\ast)\|^2\\ &\quad +\bar\alpha^2 (1+\beta) \frac1N\sum_{i=1}^N\|\nabla f_i(X_k)-\nabla f_i(x_\ast)\|^2\,.
\end{split}
\end{equation}
With the following Lemma, we are able to control $\frac1N\sum_{i=1}^N\|\nabla f_i(X_k)-\nabla f_i(x_\ast)\|^2$ by setting it in relation to $-\langle X_k-x_\ast, \nabla F_N(X_k)\rangle$.
\begin{lemma}\label{lem:finitesum_convex_smooth}
Let $F_N,\ f_i:\R^d\to\R$, $i=1,\dots,N$, be $\mu$-strongly convex and $L$-smooth. Then for all $x\in\R^d$ it holds true that 
\begin{align*}
\langle \nabla F_N(x), x_\ast-x\rangle &\le -\frac{L-\mu}{L} (F_N(x)-F_N(x_\ast)) - \frac{\mu}2\|x-x_\ast\|^2\\&\quad - \frac{1}{2L}\frac1N\sum_{i=1}^N\|\nabla f_i(x)-\nabla f_i(x_\ast)\|^2\,,
\end{align*}
where $x_\ast =\arg\min_{x\in\R^d}\ F_N(x)$.
\end{lemma}
The proof of Lemma~\ref{lem:finitesum_convex_smooth} is left as an exercise for the interested reader. We can now combine the upper bound in Lemma~\ref{lem:finitesum_convex_smooth} with \eqref{eq:SAGA_errorbound1} to obtain
\begin{align*}
\E[\|X_{k+1}-x_\ast\|^2\mid \cF_k] &\le (1-\bar\alpha\mu)\|X_k-x_\ast\|^2-2\bar\alpha \frac{L-\mu}{L} \left(F_N(X_k)-F_N(x_\ast)\right)-\bar\alpha^2 \beta\|\nabla F_N(X_k)\|^2\\
&\quad + (\bar\alpha^2(1+\beta)-\frac{\bar\alpha}L) \frac1N\sum_{i=1}^N\|\nabla f_i(X_k)-\nabla f_i(x_\ast)\|^2\\
&\quad + \bar\alpha^2(1+\beta^{-1})\frac1N\sum_{i=1}^N\|\nabla f_i(\phi_k^{(i)})-\nabla f_i(x_\ast)\|^2
\end{align*}
It remains to control 
\[\frac1N\sum_{i=1}^N\|\nabla f_i(\phi_k^{(i)})-\nabla f_i(x_\ast)\|^2\,.\]
Since all $f_i$, $i=1,\dots,N$, are assumed to be $L$-smooth with the same $L>0$, we have that
\[\|\nabla f_i(\phi_k^{(i)})-\nabla f_i(x_\ast)\|^2\le 2L(f_i(\phi_k^{(i)})-f_i(x_\ast)-\langle \nabla f_i(x_\ast),\phi_k^{(i)}-x_\ast\rangle)\]
(see also Lemma~\ref{lem:SVRG_aux} below) implying that
\[\frac1N\sum_{i=1}^N\|\nabla f_i(\phi_k^{(i)})-\nabla f_i(x_\ast)\|^2\le \frac1N\sum_{i=1}^N f_i(\phi_k^{(i)})-F_N(x_\ast) -\frac1N\sum_{i=1}^N \langle \nabla f_i(x_\ast),\phi_k^{(i)}-x_\ast\rangle =E_k^{(1)},\]
which also explains the origin of the error function $E_k = E_k^{(2)}+E_k^{(1)}$. We are now ready to prove the linear convergence of SAGA.
\begin{thm}[Theorem~1 in \cite{NIPS2014_ede7e2b6}]\label{thm:SAGA}
Let $F_N,\ f_i:\R^d\to\R$, $i=1,\dots,N$, be $\mu$-strongly convex and $L$-smooth, and let $X_0$ be a random variable such that $\E[E_0] = c\E[\|X_0-x_\ast\|^2] + \E[F_N(X_0)-F_N(x_\ast)]<\infty$. Moreover, let $(X_k)_{k\in\N}$ be generated by Algorithm~\ref{alg:SAGA} with fixed step size $\alpha_k =\bar\alpha = \frac{1}{2(\mu N+L)}$. Then for $c=\frac{1}{2\bar\alpha (1-\bar\alpha\mu)N}$ it holds true that
\[\E[E_{k+1}] \le (1-\bar\alpha\mu)\E[E_k]\,. \]
\end{thm}
\begin{proof}
Firstly, we observe that each $\phi_{k+1}^{(i)}$ given $\cF_k$ is distributed according to
\[\phi_{k+1}^{(i)}\mid\cF_k \sim \frac{1}{N}\delta_{X_k} + (1-\frac1N)\delta_{\phi_k^{(i)}}\,, \]
such that
\begin{align*}
\frac1N\sum_{i=1}^N \E[ f_i(\phi_{k+1}^{(i)})\mid \cF_k] &= \frac1N\sum_{i=1}^N \left(\frac1N f_i(X_k) + (1-\frac1N) f_i(\phi_{k}^{(i)}) \right)
= \frac1N F_N(X_k) + (1-\frac1N) \frac1N\sum_{i=1}^N f_i(\phi_{k}^{(i)})\,.
\end{align*}
With a similar computation we also obtain 
\begin{align*} 
\E[-\frac1N\sum_{i=1}^N \langle \nabla f_i(x_\ast),\phi_{k+1}^{(i)}-x_\ast\rangle] &= -\frac1N\langle \nabla F_N(x_\ast),X_k-x_\ast\rangle - (1-\frac1N)\sum_{i=1}^N \langle \nabla f_i(x_\ast),\phi_k^{(i)}-x_\ast\rangle\\ &= - (1-\frac1N)\sum_{i=1}^N \langle \nabla f_i(x_\ast),\phi_k^{(i)}-x_\ast\rangle\,.
\end{align*}
Finally, we obtain the iterative error bound
\begin{align*}
\E[E_{k+1}\mid \cF_k] &\le \left(\frac1N - c2\bar\alpha\frac{L-\mu}{L}\right) \left(F_N(X_k)-F_N(x_\ast)\right) - c\bar\alpha^2 \beta\|\nabla F_N(X_K)\|^2\\
&\quad +\left(1-\bar\alpha\mu\right)c\|X_k-x_\ast\|^2 +\left(1-\frac1N+2c(1+\beta^{-1})\bar\alpha^2L-\bar\alpha\mu+\bar\alpha\mu\right) E_k^{(1)}\\
&\quad + \left(c\bar\alpha(\bar\alpha(1+\beta)-\frac1L\right) \frac1N\sum_{i=1}^N\|\nabla f_i(X_k)-\nabla f_i(x_\ast)\|^2\,,
\end{align*}
and using $-\|\nabla F_N(X_K)\|^2\le -2\mu(F_N(X_k)-F_N(x_\ast))$ by $\mu$-strong convexity of $F_N$ yields
\begin{align*}
   \E[E_{k+1}\mid \cF_k]  &\le (1-\bar\alpha\mu) E_k +\left(\frac1N - 2c\bar\alpha \frac{L-\mu}{L}-2\mu\bar\alpha^2\beta c\right) \cdot \left(F_N(X_k)-F_N(x_\ast)\right)\\
&\quad+\left(\bar\alpha\mu-\frac1N+2c(1+\beta^{-1})\bar\alpha^2 L \right) E_k^{(1)}\\
&\quad+c\bar\alpha(\bar\alpha(1+\beta)-\frac1L) \frac1N\sum_{i=1}^N\|\nabla f_i(X_k)-\nabla f_i(x_\ast)\|^2\,,
\end{align*}
With the choice $\bar\alpha= \frac{1}{2(\mu N+L)}$, $\beta = \frac{2\mu N+L}{L}$ and $c=\frac{1}{2\bar\alpha (1-\bar\alpha\mu)N}$ one can verify that
\begin{itemize}
\item $\bar\alpha (1+\beta)-\frac1L = 0$,
\item $\bar\alpha\mu-\frac1N+2c(1+\beta^{-1})\bar\alpha^2 L\le 0$,
\item $\frac1N - 2c\bar\alpha \frac{L-\mu}{L}-2\mu\bar\alpha^2\beta c=0$,
\end{itemize}
such that we conclude the proof with
\[ \E[E_{k+1}] = \E[\E[E_{k+1}\mid \cF_k]] \le (1-\bar\alpha\mu) \E[E_k]\,.\]
\end{proof}

We observe that $c\E[\|X_k-x_\ast\|^2]\le E_k$ such that we obtain the following corollary:

\begin{cor}[Corollary~1 in \cite{NIPS2014_ede7e2b6}]
Under the same assumptions as in Theorem~\ref{thm:SAGA} it holds true that
\[\E[\|X_k-x_\ast\|^2] \le \left(1-\frac{\mu}{2(\mu N+ L)}\right)^k\left(\E[\|X_0-x_\ast\|^2 + \frac{N}{\mu N + L}\E[F_N(X_k)-F_N(x_\ast)] \right)\,, \]
where $x_\ast =\arg\min_{x\in\R^d}\ F_N(x)$.
\end{cor}

\begin{remark}
We emphasize again, that for both SAG and SAGA it is necessary to store $\{G_k^{(i)}\}_{i=1,\dots,N}$ which may occur with additional cost. 
\end{remark}

In the following, we take a look at Table~1 of \cite{Schmidt2017} (extended by the values for SAGA), where the theoretically derived rates of convergence of SAG (and SAGA) are compared to various deterministic first-order methods, which in each iteration need to evaluate the gradients across the entire data set $i=1,\dots,N$. We observe a significant improvement through SAG and SAGA.

\begin{table}[htb!]
\begin{center}
\begin{tabular}{|l|c|c|l|l|}
\hline
Algorithm & $\bar\alpha$ & rate & $\mu=0.01$ & $\mu = 0.0001$\\
\hline 
\parbox[0pt][3em][c]{0cm}{}GD & $\frac1L$ & $(1-\frac\mu{L})$ & $\sim 0.9999$ & $\sim 1$\\
\parbox[0pt][3em][c]{0cm}{}GD & $\frac{2}{\mu+L}$ & $(1-\frac{2\mu}{L+\mu})^2$ & $\sim0.9996$ & $\sim 1$\\
\parbox[0pt][3em][c]{0cm}{}NAM & $\frac1L$ & $(1-\sqrt{\frac{\mu}{L}})$ & $\sim 0.99$ & $\sim 0.999$\\
\hline
\parbox[0pt][3em][c]{0cm}{}lower bound & -- & $(1-\frac{2\sqrt{\mu}}{\sqrt{L}+\sqrt{\mu}})^2$ & $\sim 0.9608$ & $\sim 0.996$\\
\hline 
\parbox[0pt][3em][c]{0cm}{}SAG & $\frac1{16L}$ & $(1-\min\{\frac{\mu}{16L},\frac{1}{8N}\})^N$ & $\sim0.8825$ & $\sim0.9938$\\
\parbox[0pt][3em][c]{0cm}{}SAGA & $\frac{1}{2(\mu N + L)}$ & $(1-\frac{\mu}{2(\mu N+ L)})^N$ & $\sim0.635$ & $\sim0.956$\\
\hline
\end{tabular}
\caption{For both scenarios we assume that $L=100$ and $N=10^5$.}\label{tab:SAG}
\end{center}
\end{table}

\subsection{Stochastic variance reduced gradient (SVRG) }

In the following, we consider a variance reduction method for SGD that avoids storing gradients across the entire data set. The \textit{stochastic variance reduced gradient} (SVRG) method has been introduced in \cite{NIPS2013_ac1dd209}. The algorithm operates cyclic by a loop of SGD followed by an exact gradient update. Every $M$ iterations, the scheme updates the gradient information across the entire data set. For each cycle of SGD followed by the exact gradient iteration, we need to evaluate $2M + N$ gradients. In order to compare the algorithm to GD, SAG and SAGA, we will need to pay attention to this observation and rescale the effective rate of convergence accordingly. 

\begin{algorithm}[htb!]
\begin{algorithmic}[1]
\State \textbf{Input:} \begin{itemize}
 \item objective function $F_N:\R^d\to\R$, $F_N(x) = \frac1N\sum_{i=1}^N f_i(x)$
 \item initial random variable $X_0:\Omega\to\R^d$
 \item length of cycle $M\ge1$
 \item sequence of step sizes $(\alpha_k^{(m)})_{k\in\N,\ m=0,\dots,M-1}$, $\alpha_k^{(m)}>0$ (deterministic or $\mathcal F$-adapted)
 \end{itemize}
 \State set $k=0$
 \State set $\tilde X_0 = X_0^{(0)}$, $\tilde G_0 = \nabla F_N(\tilde X_0)$ 
 \State set $X_0^{(0)} = X_0$ $\bar G_0^{(0)} = \tilde G_0$ 
\While{"convergence/stopping criterion not met"}
	\For{$m=0,\dots,M-1$ }
		\State generate independently $\mathfrak i_{k+1}^{(m+1)} \sim \mathcal U(\{1,\dots,N\})$
		\State approximate the gradient $\nabla F_N(X_k^{(m)})$ through
		\[\bar G_k^{(m)} = \nabla f_{\mathfrak i_{k+1}^{(m+1)}}(X_k^{(m)}) - \nabla f_{\mathfrak i_{k+1}^{(m+1)}}(\tilde X_k) + \tilde G_k  \]
		\State set $X_{k}^{(m+1)} = X_k^{(m)} -\alpha_k^{(m)} \bar G_k^{(m)}$, 
	\EndFor
	\State generate independently $\mathfrak m_{k+1}\sim\cU(\{0,\dots,M-1\})$,
	\State set $\tilde X_{k+1} = X_k^{(\mathfrak m_{k+1})}$,
	\State set $\tilde G_{k+1} = \nabla F_N(\tilde X_{k+1}) = \frac1N \sum_{i=1}^N \nabla f_i(\tilde X_{k+1})$
	\State set $X_{k+1}^{(0)} = \tilde X_{k+1}$, $\bar G_{k+1}^{(0)} = \tilde G_{k+1}$
	\State set $k \mapsto k+1$
\EndWhile
\end{algorithmic}
 \caption{Stochastic variance reduced gradient method (SVRG)}\label{alg:SVRG}
\end{algorithm}
We will follow the proof for linear convergence of SVRG presented in \cite{NIPS2013_ac1dd209}. Let us start with the following auxiliary bound on the gradients $\nabla f_i$, $i=1,\dots,N$. Note that this result can be viewed as an extension of the inequality
\[\frac1{2L}\|\nabla F_N(x)- \nabla F_N(x_\ast)\|^2\le F_N(x)-F_N(x_\ast),\]
which we have derived under $L$-smoothness and convexity of $F_N$.

\begin{lemma}\label{lem:SVRG_aux}
Let $F_N,\ f_i:\R^d\to\R$, $i=1,\dots,N$, be convex and $L$-smooth. Then for $x_\ast =\arg\min_{x\in\R^d}\ F_N(x)$ and all $x\in\R^d$ it holds
\[\frac1N\sum_{i=1}^N \|\nabla f_i(x)-\nabla f_i(x_\ast)\|^2 \le 2L \left(F_N(x)-F_N(x_\ast)\right)\,. \]
\end{lemma}
\begin{proof}
For each $i=1,\dots,N$, define the function
\[\varphi_i(x) = f_i(x)-f_i(x_\ast)-\langle \nabla f_i(x_\ast),x-x_\ast\rangle\ge0 \]
such that $\nabla \varphi_i(x) = \nabla f_i(x)-\nabla f_i(x_\ast)$ and $\nabla \varphi_i(x_\ast) =0$. Since $\varphi_i(x_\ast)=0$ and $\varphi_i(x)\ge0$ by convexity of $f_i$, we observe that $\varphi_i(x_\ast) =\min_{x\in\R^d}\ \varphi_i(x)$. For any $x\in\R^d$ it holds that
\begin{align*}
0=\varphi_i(x_\ast) \le \min_{\alpha\ge0}\ \varphi_i(x-\alpha \nabla f_i(x)) &\le \min_{\alpha\ge0}\ \left(\varphi_i(x)-\alpha (1-\frac{L}2\alpha) \|\nabla \varphi_i(x)\|^2\right)\\
&=\varphi_i(x)-\frac1{2L}\|\nabla \varphi_i(x)\|^2,
\end{align*}
where we have used that $\varphi_i$ is $L$-smooth (since $\nabla \varphi_i(x) = \nabla f_i(x) - \nabla f_i(x_\ast)$ remains $L$-Lipschitz continuous) and that $\max_\alpha \alpha(1-\frac{L}2\alpha) = \frac{1}{2L}$. Hence, it follows that
\[\|\nabla\varphi_i(x)\|^2 = \|\nabla f_i(x)-\nabla f_i(x_\ast)\|^2 \le 2L \varphi_i(x) = 2L\left(f_i(x)-f_i(x_\ast)-\langle \nabla f_i(x_\ast),x-x_\ast\rangle\right).\]
Taking the average over all $i=1,\dots,N$, finishes the proof
\begin{align*}
\frac1N\sum_{i=1}^N \|\nabla f_i(x)-\nabla f_i(x_\ast)\|^2 &\le 2L\left(\frac1N\sum_{i=1}^Nf_i(x)-\frac1N\sum_{i=1}^Nf_i(x_\ast)-\langle\frac1N\sum_{i=1}^N \nabla f_i(x_\ast),x-x_\ast\rangle\right)\\ &= 2L\left(F_N(x)-F_N(x_\ast)\right)\,.
\end{align*}
\end{proof}

Let $(X_k^{(m)})_{k\in\N,\ m=0,\dots,M}$ be generated by Algorithm~\ref{alg:SVRG} with fixed step size $\bar\alpha>0$ and consider the following natural filtration
\[\cF_k^{(m)} = \begin{cases}\sigma(\{X_0\},\{\mathfrak i_\ell^{(s)}, \ell \le k,s\le M\}, \{\mathfrak i_k^{(s)},s\le m\},\{\mathfrak m_\ell,\ell\le k\}), &k\ge1 \\ \sigma(\{X_0\},\{\mathfrak i_0^{(s)},s\le m\}), & k=0 \end{cases}\,. \]
Similarly as in the case of SAGA, we have that
\[\E[\|X_{k}^{(m+1)} - x_\ast\|^2\mid \cF_k^{(m)}] = \|X_k^{(m)}-x_\ast\|^2 - 2\bar\alpha\langle X_k^{(m)}-x_\ast,\nabla F(X_k^{(m)})\rangle +\bar\alpha^2\E[\|\bar G_k^{(m)}\|^2\mid\cF_k^{(m)}], \]
since $\bar G_k^{(m)}$ is an unbiased estimator of $\nabla F_N(X_k^{(m)})$:
\[\E[\bar G_k^{(m)}\mid \cF_k^{(m)}] = \E[\nabla f_{\mathfrak i_{k+1}^{(m+1)}}(X_k^{(m)})\mid \cF_k^{(m)}] - \underbrace{\E[\nabla f_{\mathfrak i_{k+1}^{(m+1)}}(\tilde X_k) - \tilde G_k\mid \cF_k^{(m)}]}_{=0}  = \nabla F_N(X_k^{(m)})\,. \]

We again aim to control $\E[\|\bar G_k^{(m)}\|^2\mid\cF_k^{(m)}]$ in order to achieve significant variance reduction.

\begin{lemma}\label{lem:SVRG_aux2}
Let $F_N,\ f_i:\R^d\to\R$, $i=1,\dots,N$, be $\mu$-strongly convex and $L$-smooth, and let $X_0$ be a random variable such that $\E[\|X_0\|^2] + \E[|F_N(X_0)|]<\infty$. Moreover, let $(X_k^{(m)})_{k\in\N, m=0,\dots,M}$ be generated by Algorithm~\ref{alg:SVRG} with fixed step size $\alpha_k^{(m)} =\bar\alpha>0$. Then for all $k\ge0$ and $m\ge1$ it holds true that
\[ \E[\|\bar G_k^{(m)}\|^2\mid \cF_k^{(m)}] \le 4L \left(F_N(X_k^{(m)})-F_N(x_\ast) + F_N(\tilde X_k)-F_N(x_\ast) \right)\,,\]
where $x_\ast =\arg\min_{x\in\R^d}\ F_N(x)$.
\end{lemma}
\begin{proof}
The proof follows by similar argumentation as the proof of Lemma~\ref{lem:varred}. We again apply \eqref{eq:variance_formula} together with $\|x+y\|^2 \le 2\|x\|^2 + 2\|y\|^2$:
\begin{align*}
\E[\|\bar G_k^{(m)}\|^2\mid \cF_k^{(m)}]&\le 2\E[\|\nabla f_{\mathfrak i_{k+1}^{(m+1)}}(X_k^{(m)})- \nabla f_{\mathfrak i_{k+1}^{(m+1)}}(x_\ast)\|^2\mid \cF_k^{(m)}]\\&\quad + 2\E[\|\nabla f_{\mathfrak i_{k+1}^{(m+1)}}(\tilde X_k)-\nabla f_{\mathfrak i_{k+1}^{(m+1)}}(x_\ast) - \nabla F_N(\tilde X_k)\|^2\mid \cF_k^{(m)}]\\
&=2\E[\|\nabla f_{\mathfrak i_{k+1}^{(m+1)}}(X_k^{(m)})- \nabla f_{\mathfrak i_{k+1}^{(m+1)}}(x_\ast)\|^2\mid \cF_k^{(m)}]\\
&\quad+2\E[\|\{\nabla f_{\mathfrak i_{k+1}^{(m+1)}}(\tilde X_k)-\nabla f_{\mathfrak i_{k+1}^{(m+1)}}(x_\ast)\} - \{\nabla F_N(\tilde X_k)-\nabla F_N(x_\ast)\}\|^2\mid \cF_k^{(m)}]\\
&\le 2\E[\|\nabla f_{\mathfrak i_{k+1}^{(m+1)}}(X_k^{(m)})- \nabla f_{\mathfrak i_{k+1}^{(m+1)}}(x_\ast)\|^2\mid \cF_k^{(m)}]\\
&\quad+2\E[\|\nabla f_{\mathfrak i_{k+1}^{(m+1)}}(\tilde X_k)-\nabla f_{\mathfrak i_{k+1}^{(m+1)}}(x_\ast)\|^2\mid \cF_k^{(m)}]\\
& = \frac{2}{N}\sum_{i=1}^N\left(\|\nabla f_i(X_k^{(m)})- \nabla f_i(x_\ast)\|^2+\|\nabla f_i(\tilde X_k)-\nabla f_i(x_\ast)\|^2\right)\,.
\end{align*}
The assertion follows by application of Lemma~\ref{lem:SVRG_aux}.
\end{proof}

We are now ready to prove linear convergence of SVRG under strong convexity.
\begin{thm}[Theorem~1 in \cite{NIPS2013_ac1dd209}]\label{thm:SVRG}
Let $F_N,\ f_i:\R^d\to\R$, $i=1,\dots,N$, be $\mu$-strongly convex and $L$-smooth, and let $X_0$ be a random variable such that $\E[\|X_0\|^2] + \E[|F_N(X_0)|]<\infty$. Moreover, let $(X_k^{(m)})_{k\in\N,\ m=0,\dots,M}$ be generated by Algorithm~\ref{alg:SVRG} with fixed step size $\alpha_k^{(m)} =\bar\alpha>0$ and $M\ge1$ such that
\[0<\rho := \frac{1}{\mu(1-2\bar\alpha L)\bar\alpha M} + \frac{2\bar\alpha L}{1-2\bar\alpha L}<1\,. \] 
Then it holds true that 
\[\E[F_N(\tilde X_k)-F_N(x_\ast)] \le \rho^k \E[F_N(X_0) - F_N(x_\ast)]\,. \]
\end{thm}
\begin{proof}
Using the update rule, the conditional unbiasedness of the SVRG estimator, and convexity of
$F_N$, we obtain
\begin{align*}
\E[\|X_k^{(m+1)}-x_\ast\|^2\mid \cF_k^{(m)}] &\le \|X_k^{(m)}-x_\ast\|^2 -2\bar\alpha \left(F_N(X_k^{(m)})-F_N(x_\ast)\right) + \bar\alpha^2 \E[\|\bar G_k^{(m)}\|^2\mid\cF_k^{(m)}]\\
&\le\|X_k^{(m)}-x_\ast\|^2 -2\bar\alpha \left(F_N(X_k^{(m)})-F_N(x_\ast)\right)\\
&\quad+ 4L\bar\alpha^2\left(F_N(X_k^{(m)})-F_N(x_\ast)\right)+ 4L\bar\alpha^2\left(F_N(\tilde X_k)-F_N(x_\ast)\right)\\
&=\|X_k^{(m)}-x_\ast\|^2 -2\bar\alpha(1-2L\bar\alpha) \left(F_N(X_k^{(m)})-F_N(x_\ast)\right)\\
&\quad+4L\bar\alpha^2\left(F_N(\tilde X_k)-F_N(x_\ast)\right)\,
\end{align*}
where we have applied Lemma~\ref{lem:SVRG_aux2}. By construction of the Algorithm~\ref{alg:SVRG} we have
\[\E[F_N(\tilde X_{k+1})-F_N(x_\ast)\mid \cF_k^{(M)}] = \frac1M\sum_{m=0}^{M-1}\left(F_N(X_k^{(m)})-F_N(x_\ast)\right),\]
such that
\begin{align*}
&\E[\|X_k^{(M)}-x_\ast\|^2] + 2\bar\alpha(1-2L\bar\alpha) M \E[F_N(\tilde X_{k+1})-F_N(x_\ast)]\\ &\le \E[\|X_k^{(0)}-x_\ast\|^2] -2\bar\alpha(1-2L\bar\alpha)\sum_{m=0}^{M-1} \E[F_N(X_k^{(m)})-F_N(x_\ast)]\\
&\quad+2\bar\alpha(1-2L\bar\alpha)M\frac1M\sum_{m=0}^{M-1} \E[F_N(X_k^{(m)})-F_N(x_\ast)]\\
&\quad+4L\bar\alpha^2 M \E[F_N(\tilde X_k) - F_N(x_\ast)]\\
&\le \frac{2}{\mu}\E[F_N(X_k^{(0)})-F_N(x_\ast)] + 4L\bar\alpha^2 M \E[F_N(\tilde X_k) - F_N(x_\ast)]\\
&=2(\frac1{\mu}+2L\bar\alpha^2 M)\E[F_N(\tilde X_k) - F_N(x_\ast)]\,,
\end{align*}
where we have used that $\frac{\mu}2\|x-x_\ast\|^2\le F_N(x)-F_N(x_\ast)$ by strong convexity of $F_N$. Finally, we have
\[\E[F_N(\tilde X_{k+1})-F_N(x_\ast)]\le\left(\frac{1}{\mu(1-2\bar\alpha L)\bar\alpha M} + \frac{2\bar\alpha L}{1-2\bar\alpha L}\right)\E[F_N(\tilde X_k)-F_N(x_\ast)]\,.  \]
\end{proof}
\begin{remark}
Returning to the setting of Table~\ref{tab:SAG}. For $L=100$, $\mu=0.01$ and $N=10^5=10\kappa$, we want to choose $\bar\alpha=\frac{\tau}{L}$, $\tau\le\frac12$ such that $\rho$ defined in Theorem~\ref{thm:SVRG} is given by
\[\rho = \frac{1}{1-2\tau}\left(\frac{\kappa}{\tau M} + 2\tau\right) = \frac{1}{1-2\tau}\left(\frac{N}{10 \tau M} + 2\tau\right)\,. \]
We set $\tau = \frac1{10}$, $M=4N$ such that $\rho=\frac9{16}$. Since one SVRG outer iteration requires $2M+N$ component gradient evaluations, while one full
gradient descent step requires $N$ component gradient evaluations, the effective contraction factor
per cost of one full gradient evaluation is
\[\rho^{\mathrm{SVRG}} = \rho^{\frac{N}{2M+N}} = \left(\frac{9}{16}\right)^{\frac{1}{9}} \approx 0.94, \]
which still achieves a better convergence rate than GD. Note that the above choice of step size has been chosen heuristically and is not optimized. 
\end{remark}

The three finite-sum methods illustrate different ways of reducing variance. SAG and SAGA store past component gradients and therefore require additional memory of order $Nd$. SVRG avoids this storage cost by using occasional full-gradient computations at snapshot points. In all cases, the goal is to recover linear convergence in strongly convex finite-sum problems while using mostly stochastic updates.

\section{Further reading and outlook}
Stochastic gradient methods date back to the seminal work of Robbins and Monro~\cite{RM1951} and
are now among the standard algorithms for large-scale optimization and machine learning. As mentioned earlier, for a broad introduction to stochastic approximation and stochastic gradient methods, we refer to
Bottou et al.~\cite{Bottou2018} and to the recent handbook of convergence proofs by Garrigos and Gower~\cite{garrigos2023handbook}.

The convergence results in Section~\ref{sec:SGDas} and \ref{sec:SGDexp} are classical in spirit. Proofs of the convergence rates in expectation for SGD under smoothness and convexity assumptions can be found, for example, in \cite{moulines11,Nguyen2018,poliak1987introduction}. For smooth but non-convex objectives, the standard goal is to prove convergence of the (minimal, averaged or randomly selected) gradient norm; see \cite{ghadimi-lan13}. For the non-smooth but convex setting we refer to \cite{Lan2021}. 

Almost sure convergence has a long history in stochastic approximation. One fundamental tool is the Robbins--Siegmund supermartingale convergence theorem \cite{RS1971}, which we discussed in Section~\ref{sec:SGDas}. Other approaches to verify almost sure convergence guarantees for SGD are described in \cite{Bertsekas_2000,Blum1954,chouzenoux2023kurdykalojasiewicz,dereich2021convergence,Nguyen2018,mertikopoulos2020sure}. Almost sure convergence rates have received renewed attention in recent years. In particular, Sebbouh et al.~\cite{pmlr-v134-sebbouh21a} and Liu and Yuan~\cite{LY2022} derive almost sure rates for stochastic gradient methods in smooth (convex and non-convex) settings. They also treat stochastic momentum such as heavy ball and Nesterov acceleration.

Another important line of work replaces convexity by local or global gradient domination conditions, such as the PL condition \eqref{eq:PL} or the weaker condition \eqref{eq:wPL}. These conditions are weaker than strong convexity but can still imply convergence to global minimizers. For convergence results in expectation under gradient domination, we refer to
\cite{fatkhullin2022,pmlr-v134-fontaine21a,10.1007/978-3-319-46128-1_50,khaled2022better}. 
Almost sure (last iterate) convergence rates for SGD under global and local gradient domination are studied in \cite{weissmann2025almost}.

Several important topics are only briefly touched upon in these notes. These include stochastic
momentum methods, adaptive step size rules such as AdaGrad \cite{JMLR:v12:duchi11a}, AdaDelta \cite{zeiler2012adadeltaadaptivelearningrate}, RMSProp \cite{tieleman2012rmsprop} and Adam \cite{kingma2014adam}, high-probability bounds \cite{liu23},
and local convergence theory near non-isolated minima or saddle points \cite{mertikopoulos2020sure}. For stochastic momentum
methods, we refer to recent analyzes such as \cite{gadat16,gess2023convergence,LY2022,mai2020convergence,pmlr-v134-sebbouh21a}.

\bibliographystyle{plain}
\bibliography{mybib.bib}

\appendix

\chapter{Appendix}
\section{Convex sets and functions}\label{app:convex}
We give a brief overview of convex functions and properties that are important for optimization. This appendix covers selected material from \cite[Appendix~B]{DPB15}, adapted to
the notation and scope of these notes. We start with the following basic definition.
\begin{defi}[convex function]
A subset $C\subset \R^d$ is called \textit{convex}, if 
\[\lambda x + (1-\lambda) y \in C \]
for all $x,y\in C$ and $\lambda\in[0,1]$. 
Let $C\subset\R^d$ be a convex set. A function $f:C\to\R$ is called \textit{convex function}, if 
\[f(\lambda x + (1-\lambda) y) \le \lambda f(x) + (1-\lambda) f(y)\]
for all $x,y\in C$ and $\lambda\in[0,1]$. A function $f$ is called concave, if $-f$ is convex. Moreover, a function $f$ is called strictly convex, if
\[f(\lambda x+ (1-\lambda)y)< \lambda f(x) + (1-\lambda) f(y) \]
for all $x,y\in C$, $x\neq y$ and $\lambda\in(0,1)$.
\end{defi}

\begin{example}
Consider the following examples.
\begin{enumerate}
\item Let $\{C_i,\ i\in I\}$ be a family of convex sets. Then $\cap_{i\in I} C_i$ is a convex set. 
\item Let $C_1,C_2\subset \R^d$ be two convex sets, then the set
\[C = \{ x\in\R^d\mid x = x_1+x_2,\ x_1\in C_1, x_2\in C_2\} \]
is a convex set.
\item The image of a convex set under linear transformation is again a convex set. 
\item Let $C\subset\R^d$ be a convex set and let $f:C\to\R$ be a convex function. Then the level sets 
\[ A_\alpha = \{x\in C\mid f(x)\le \alpha\} \quad \text{and}\quad B_\alpha = \{x\in C\mid f(x)<\alpha\} \]
are convex sets for all $\alpha \in \R$.
\end{enumerate}
\end{example}

\begin{defi}
Let $f:C\to\R$ be a function and $C\subset\R^d$ be a convex set. We define the \textit{epigraph} of $f$ as
\[\mathcal E(f) := \{(x,w)\in C\times \R \mid f(x)\le w \},  \]
the set of all points above the \textit{graph} of $f$ defined as
\[\mathcal G(f) := \{(x,w)\in C\times \R \mid f(x)= w \}. \]
\end{defi}
We can characterize convex functions via convexity of its epigraph.
\begin{prop}[Fact]
Let $f:C\to\R$ be a function and $C\subset \R^d$ be a convex set. The $f$ is convex if and only if its epigraph $\mathcal E(f)$ is a convex set.
\end{prop}

\begin{figure}[!htb]
  \centering \includegraphics[width=0.5\textwidth]{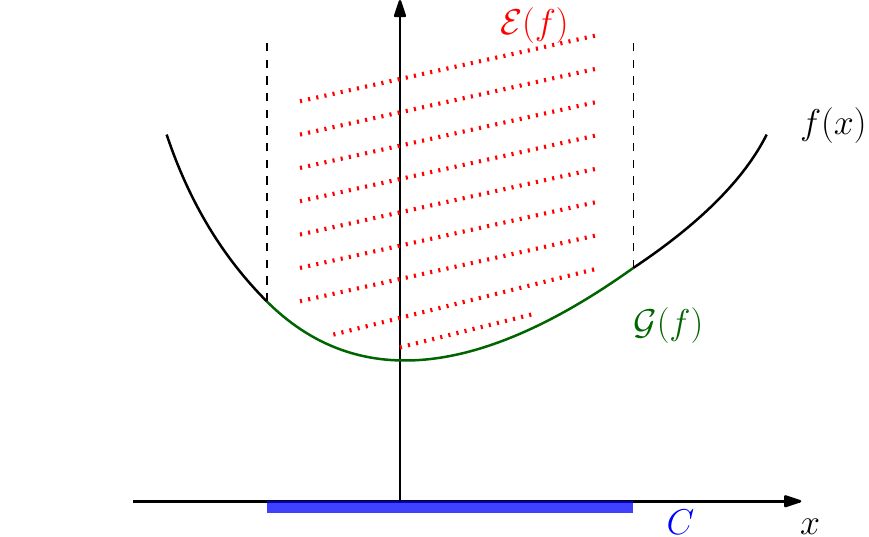}
 \caption{Illustration of the connection between the graph $\mathcal G(f)$, the epigraph $\mathcal E(f)$ and the convexity of $f$. } \label{fig:epigraph}
\end{figure} 

Next, we describe an inequality for convex functions which will be used often times in this lecture.
\begin{prop}[Jensen's inequality]\label{prop:Jensen}
Let $f:C\to\R$ be a convex function, $C\subset \R^d$ be a convex set and $\lambda_1,\dots,\lambda_n\in(0,1)$ with $\sum_{i=1}^n\lambda_i = 1$. Then it holds true that
\[f(\sum_{i=1}^n \lambda_i x_i) \le \sum_{i=1}^n \lambda_i f(x_i). \]
\end{prop}

\begin{exercise}
Prove Proposition~\ref{prop:Jensen}.
\end{exercise}

One typical example for the application of Jensen's inequality is
\[\left(\frac{1}{n}\sum_{i=1}^n x_i\right)^2 \le \frac{1}{n} \sum_{i=1}^n x_i^2. \]
\begin{example}\label{ex:convex_functions}
In the following we present a row of useful properties in the context of convex functions. 
\begin{enumerate}
\item Every linear function is convex.
\item Every norm in $\R^d$ is convex.
\item Let $f_1,\dots,f_n$ be convex functions and $c_1,\dots,c_n\ge0$.
\begin{enumerate}
\item[a)] Then $x\mapsto \sum_{i=1}^n f_i(x)$ is a convex function.
\item[b)] Then $x\mapsto \sup_{i=1,\dots,n}\ f_i(x)$ is a convex function. 
\end{enumerate}
\end{enumerate}
\end{example}

\begin{exercise}
Prove the properties presented in Example~\ref{ex:convex_functions}.
\end{exercise}

\begin{prop}\label{prop:app_convexity1}
Let $C\subset \R^d$ be a convex set and let $f:C\to\R$ be differentiable over $C$. Then the following holds true:
\begin{enumerate}
\item The function $f$ is convex over $C$ if and only if 
\[f(z)\ge f(x) + (z-x)^\top \nabla f(x), \]
for all $x,z\in C$.
\item If 
\[f(z) > f(x) + (z-x)^\top \nabla f(x),\quad x\neq z, \] 
for all $x,z\in C$, then $f$ is strictly convex.
\end{enumerate} 
\end{prop}
\begin{proof}
We only prove the first assertion, the second one follows by similar argumentation. Firstly, assume that $f$ is convex and let $x,z\in C$. Since $C$ is a convex set, $x+a(z-x) \in C$ for all $a\in(0,1)$.  The convexity of $f$ implies that
\[f(x+a(z-x)) \le af(z) + (1-a) f(x),\]
which we can rewrite to
\[\frac{f(x+a(z-x))-f(x)}{a}\le f(z)-f(x). \]
On the other side, since $f$ is differentiable, we have
\[(z-x)^\top \nabla f(x)=\lim_{a\leftharpoondown 0} \frac{f(x+a(z-x))-f(x)}{a}\le f(z)-f(x). \]
This proves the first direction $"\Rightarrow"$.
Secondly, for the direction $"\Leftarrow"$ suppose that
\begin{equation}\label{eq:ass_direction2}
f(z)\ge f(x)+ (z-x)^\top \nabla f(x)
\end{equation}
for all $x,z\in C$.  We fix arbitrary $x,y\in C$ and $\lambda\in(0,1)$, and define $z:= \lambda x+ (1-\lambda)y$. Then by \eqref{eq:ass_direction2} we have both 
\begin{align*}
f(x)&\ge f(z) + (x-z)^\top \nabla f(z)\\
f(y)&\ge f(z) + (y-z)^\top \nabla f(z).
\end{align*}
Combining both inequalities yields
\begin{align*}
\lambda f(x) + (1-\lambda) f(y) &\ge \lambda f(z) + \lambda (x-z)^\top \nabla f(z) + (1-\lambda) f(z)+ (1-\lambda) (y-z)^\top \nabla f(z)\\
&=f(z)+ (\lambda x+(1-\lambda)y-z)\nabla f(z)\\ &= f(z) =f(\lambda x+(1-\lambda)y).
\end{align*}
Since $x,y\in C$ are arbitrary, we have proven convexity of $f$. 
\end{proof}

\begin{prop}
Let $C\subset\R^d$ be a convex set, $f:\R^d\to \R$ twice continuous differentiable over $C$ and let $Q\in\R^{d\times d}$ be a symmetric matrix. 
\begin{enumerate}
\item If $\nabla^2 f(x)$ is positive semi-definite for all $x\in C$, then $f$ is convex over $C$.
\item If $\nabla^2 f(x)$ is positive definite for all $x\in C$, then $f$ is strictly convex over $C$.
\item If $C=\R^d$ and $f$ is convex, then $\nabla^2 f(x)$ is positive semi-definite for all $x\in C$.
\item The quadratic function $x\mapsto f(x) = x^\top Q x$ is convex if and only if $Q$ is positive semi-definite. Moreover, $f$ is strictly convex if and only if $Q$ is positive definite.  
\end{enumerate}
\end{prop}

\begin{proof}
\begin{enumerate}
\item Let $x,y\in C$, then by the mean-value theorem there exists $\alpha\in[0,1]$ such that
\begin{align*}
f(y) &= f(x) + (y-x)^\top \nabla f(x) + \frac{1}{2} (y-x)^\top \nabla^2 f(x+\alpha (y-x))(y-x)\\ &\ge f(x) + (y-x)^\top f(x),
\end{align*}
where we have used that $ \nabla^2 f(x+\alpha (y-x))$ is positive semi-definite, since $x+\alpha(y-x)\in C$. Then $f$ is convex by Proposition~\ref{prop:app_convexity1}.
\item Follows similarly as the proof of the first claim. 
\item Let $f$ be convex and suppose that there exists $x\in\R^d$ such that $z^\top \nabla^2 f(x) z <0$ for some $z\in\R^d$. By assumption we have that $x\mapsto \nabla^2 f(x)$ is continuous, such that for $\bar\alpha>0$ small enough it holds true that
\[z^\top \nabla^2f(x) z<0,\quad \text{for}\ \alpha \in [0,\bar\alpha). \]
For an arbitrary $\tilde \alpha\in[0,\bar\alpha)$ we set $\tilde z = \tilde \alpha z$. By the mean-value theorem there exists $\beta\in[0,1]$ such that
\begin{align*}
f(x+\tilde z) &= f(x) + \tilde z^\top \nabla f(x)+ \frac{1}{2} \tilde z^\top \nabla^2 f(x+ \beta \tilde z) \tilde z\\
&= f(x) + \tilde z^\top \nabla f(x)+\frac12 \tilde\alpha^2 z^\top \nabla f(x+\underbrace{\beta\tilde\alpha}_{<\bar\alpha} z)z\\
&< f(x) + \tilde z^\top \nabla f(x),
\end{align*}
which is in contradiction to the convexity of $f$ (see Proposition~\ref{prop:app_convexity1}).
\item The hessian of $f$ is given by $\nabla^2 f(x) = 2Q$ for all $x\in\R^d$. With 1. and 3. it follows that $f$ is convex if and only if $Q$ is positive semi-definite.  If $Q$ is positive definite, then convexity of $f$ follows by the 3. assertion.  It is left to prove that strict convexity of $f$ implies that $Q$ is positive definite. Assume that $f$ is strictly convex, then by 3. we know that $Q$ is positive semi-definite.  If a matrix $A\in\R^{d\times d}$ is positive semi-definite but not positive definite, then there exists at least one $x\in\R^d$, $x\neq 0$ such that $x^\top A x = 0$, i.e.~$Ax = 0 = 0\cdot x$. Hence, to prove that $Q$ is positive definite we will show that $\lambda = 0$ is no eigenvalue of $Q$. Suppose $\lambda = 0$ is an eigenvalue of $Q$, then there exists $x\neq 0$ such that $Qx = \lambda x = 0$. This implies that
\[\frac12 f(x) + \frac12 f(-x) = \frac12(x^\top Q x + (-x)^\top Q (-x) = x^\top Q x = 0 = f(\frac12 x-\frac12 x), \]
which is in contradiction to strict convexity of $f$.  Hence, $Q$ is positive definite. 
\end{enumerate}
\end{proof}

We will often consider the class of strong convex functions.
\begin{defi}[strongly convex function]\label{def:strong_convexity}
Let $C\subset \R^d$ be a convex set.  A function $f:C\to \R$ is called \textit{$\mu$-strongly convex} over $C$, if 
\[f(y) \ge f(x) + (y-x)^\top \nabla f(x) + \frac{\mu}{2} \|y-x\|^2 \]
for all $x,y\in C$.
\end{defi}
\begin{prop}\label{prop:strong_convexity}
\begin{enumerate}
\item Every $\mu$-strongly convex function $f$ is also strictly convex.
\item A function $f:C\to\R$ is $\mu$-strongly convex if and only if its hessian $\nabla^2 f(x)$ is uniformly positive definite with
\[z^\top \nabla^2f(x) z \ge \mu \|z\|^2 \]
for all $x\in C$ and all $z\in\R^d$.
\end{enumerate}
\end{prop}
\begin{exercise}
Prove Proposition~\ref{prop:strong_convexity}.
\end{exercise}

\section{Optimality conditions}\label{app:op-condition}
We recall a few basic optimality conditions that are used throughout the notes. The selection and
presentation are partly guided by \cite[Section 1.1]{DPB15}, with the notation adapted to the rest of these notes. We start with the necessary optimality conditions of first-order, which only needs differentiability of the objective function.  

\begin{thm}\label{thm:fo_opti}
Let $f:\R^d\to\R$ be continuously differentiable over the open set $S\subset \R^d$ and suppose $x_\ast\in S$ is a local minimum of $f$. Then it holds true that $\nabla f(x_\ast) = 0$.
\end{thm}
\begin{proof}
Let $x_\ast\in S$ be a local minimum of $f$ and consider the corresponding ball $\mathcal B_\varepsilon(x_\ast)$ such that $f(x_\ast)\le f(x)$ for all $x\in\mathcal B_\varepsilon(x_\ast)$. Let $z\in \R^d$ be arbitrary but fixed, such that for $\bar\alpha>0$ small enough $x_\ast+\alpha z\in \mathcal B_\varepsilon(x_\ast)$ and hence $f(x_\ast + \alpha z) \ge f(x_\ast)$ for all $\alpha\in [0,\bar\alpha)$. We define $\alpha \mapsto g(\alpha) = f(x_\ast+ \alpha z)$, $\alpha>0$ and observe 
\[ \frac{{\mathrm d} g(0)}{d \alpha} = \lim_{\alpha\to 0} \frac{f(x_\ast + \alpha z) - f(x_\ast)}{\alpha} \ge 0.\]
On the other side by chain rule we also have
\[\frac{{\mathrm d} g(0)}{{\mathrm d} \alpha} = z^\top \nabla f(x_\ast + \alpha z).\]
In particular, this implies $0\le z^\top \nabla f(x_\ast)$. Since $z\in\R^d$ is arbitrary, with $y=-z\in\R^d$ we also obtain $0\le y^\top \nabla f(x_\ast)=-z^\top \nabla f(x_\ast)$. Therefore, for all $z\in\R^d$ we have $z^\top \nabla f(x_\ast) = 0$, which yields $\nabla f(x_\ast) = 0$.
\end{proof}

Under the additional assumption that $f$ is twice continuously differentiable there is also a necessary optimality condition of second-order.

\begin{thm}\label{thm:so_opti}
Let $f:\R^d\to \R$ be twice continuously differentiable over the open set $S\subset \R^d$ and suppose $x_\ast \in S$ is a local minimum of $f$. Then it holds true that $\nabla^2 f(x_\ast)$ is positive semi-definite. 
\end{thm}
\begin{proof}
Let again $x_\ast\in S$ be the local minimum of $f$ with the corresponding ball $\mathcal B_\varepsilon(x_\ast)$ such that $f(x_\ast)\le f(x)$ for all $x\in\mathcal B_\varepsilon(x_\ast)$ and consider an arbitrary $z\in\R^d$. With the Taylor expansion it holds
\[f(x_\ast+\alpha z)-f(x_\ast) = \alpha \underbrace{\nabla f(x_\ast)^\top}_{=0} z + \frac{\alpha^2}{2} z^\top \nabla^2 f(x_\ast) z + o(\alpha^2) =   \frac{\alpha^2}{2} z^\top \nabla^2 f(x_\ast) z + o(\alpha^2)\, . \]
Let $\bar\alpha>0$ be small enough such that $f(x_\ast + \alpha z) \ge f(x_\ast)$ for all $\alpha\in [0,\bar\alpha)$. With the above equation we obtain
\[0\le \frac{f(x_\ast+\alpha z)-f(x_\ast)}{\alpha^2} = \frac{1}{2} z^\top \nabla^2 f(x_\ast) z + \frac{o(\alpha^2)}{\alpha^2}\to \frac{1}{2} z^\top \nabla^2 f(x_\ast) z \]
for $\alpha\to 0$. Since $z\in\R^d$ is chosen arbitrary, this leads to $z^\top \nabla^2 f(x_\ast) z\ge 0$ for all $z\in\R^d$ and the assertion follows.
\end{proof}

\begin{remark}
Theorem~\ref{thm:fo_opti} and \ref{thm:so_opti} characterize necessary, but not sufficient, conditions for optimality. Firstly, consider for example $f(x) = -x^2$ with stationary point $x_\ast = 0$, which is no local minimum.  Secondly, consider $f(x) = x_1^2-x_2^4$, $x=(x_1,x_2)^\top\in\R^2$ with stationary point $x_\ast = (0,0)^\top$ and positive semi-definite Hessian $\nabla^2 f(x_\ast) = \begin{pmatrix} 2 & 0 \\ 0 & 0 \end{pmatrix}$, which is no local minimum. 
\end{remark}

In the following theorem, we formulate second-order sufficient conditions for optimality.
\begin{thm}\label{thm:so_suff_opti}
Let $f:\R^d\to \R$ be twice continuously differentiable over the open set $S\subset \R^d$. Let $x_\ast\in S$ with 
\begin{enumerate}
\item $\nabla f(x_\ast) = 0$
\item $\nabla^2 f(x_\ast)$ (strictly) positive definite.
\end{enumerate}
Then $x_\ast$ is a strict local minimum of $f$ and there exist $\gamma>0$, $\varepsilon>0$ such that
\[f(x)\ge f(x_\ast) + \frac{\gamma}{2}\|x-x_\ast\|^2 \]
for all $x\in\mathcal B_{\varepsilon}(x_\ast)$.
\end{thm}

Before proving Theorem~\ref{thm:so_suff_opti}, we will need to prove the following auxiliary result.

\begin{lemma}\label{lem:eigenvalues_definiteness}
Let $A\in\R^{d\times d}$ be a symmetric matrix with real valued eigenvalues $\lambda_1\le \dots\le \lambda_d$ and corresponding eigenvectors $v_1,\dots,v_d$. Then it holds true that:
\begin{enumerate}
\item $\lambda_1\|z\|^2\le z^\top A z\le \lambda_d \|z\|^2$ for all $z\in\R^d$.
\item The matrix $A$ is (strict) positive definite if and only if all eigenvalues are (strictly) positive.
\end{enumerate}
\end{lemma}
\begin{proof}
We start with the first assertion. Let $z\in\R^d$ and write
\[z = \sum_{i=1}^d \xi_i v_i\]
with coefficients $\xi_i\in\R$. Then we can write
\begin{align*}
z^\top A z = \sum_{i=1}^d \xi_i^2 \langle v_i, \underbrace{A v_i}_{=\lambda_i v_i}\rangle = \sum_{i=1}^d \lambda_i \xi_i^2 \|v_i\|^2\begin{cases}\ge & \lambda_1\sum_{i=1}^d \xi_i^2 \|v_i\|^2 = \lambda_i \|z\|^2,\\ \le & \lambda_d \sum_{i=1}^d \xi_i^2\|v_i\|^2 = \lambda_d \|z\|^2., \end{cases}
\end{align*}
which finishes the proof of the first claim. 
For the second assertion we start with "$\Rightarrow$". Let $(\lambda_i,v_i)$ be eigenvalue and corresponding eigenvector of the (strictly) positive definite matrix $A$.  By definition of a  positive definite matrix we have
\[0\le (<) v_i^\top A v_i = v_i^\top (\lambda_i v_i) = \lambda_i \|v_i\|^2. \]
Since $\|v_i\|^2>0$ for $v_i \neq 0$, we obtain $\lambda_i \ge (>) 0$ for all $i=1,\dots,d$.
For the other way around "$\Leftarrow$", we assume that $0\le (<) \lambda_1\le \dots \le \lambda_d$ are eigenvalues of $A$ with corresponding eigenvectors $v_1\dots, v_d$.  Then it follows with the first assertion 
\[z^\top A z \ge \lambda_1 \|z\|^2 \ge (>) 0\]
for all $z\in\R^d$ with $z\neq 0$.
\end{proof}
We are now ready to prove Theorem~\ref{thm:so_suff_opti}.
\begin{proof}[Proof of Theorem~\ref{thm:so_suff_opti}]
Let $\lambda>0$ be the smallest eigenvalue of the positive definite matrix $\nabla^2 f(x_\ast)$. By Lemma~\ref{lem:eigenvalues_definiteness} it holds true that
\[z^\top \nabla^2 f(x_\ast) z \ge \lambda \|z\|^2.\]
Application of Taylor's expansion around $x_\ast$ together with $\nabla f(x_\ast)=0$ yields 
\begin{align*}
f(x_\ast + d) - f(x_\ast) &= \nabla f(x_\ast)^\top d + \frac12 d^\top \nabla^2 f(x_\ast) d + o(\|d\|^2)\\
&= \frac12 d^\top \nabla^2 f(x_\ast)d + o(\|d\|^2)\\
&\ge\frac{ \lambda}2 \|d\|^2 + o(\|d\|^2)\\
&= \left(\frac{\lambda}2+ \frac{o(\|d\|^2)}{\|d\|^2}\right) \|d\|^2.
\end{align*}
Let $\varepsilon>0$ be sufficiently small such that it holds $\frac{o(\|d\|^2)}{\|d\|^2}\in(-\frac\lambda4,\frac\lambda4)$ for $\|d\|^2<\varepsilon$. For $x\in\R^d$ with $\|x-x_\ast\|<\varepsilon$ it follows
\[f(x) \ge f(x_\ast) + \left(\frac{\lambda}{2} + \frac{o(\|x-x_\ast\|^2)}{\|x-x_\ast\|^2}\right)\|x-x_\ast\|^2\ge f(x_\ast) + \underbrace{(\frac{\lambda}{2}- \frac{\lambda}{4})}_{=: \frac{\gamma}{2}} \|x-x_\ast\|^2>f(x_\ast). \]
\end{proof}

Under the additional assumption of a convex objective function $f$, we can further characterize sufficient optimality conditions. The proof of the following proposition is left as an exercise for the interested reader.
\begin{prop}\label{prop:opt_cond_convex}
Let $f:\R^d\to \R$ be continuously differentiable and convex over the convex set $S\subset \R^d$. Then it holds true that:
\begin{enumerate}
\item Every local minimum of $f$ over $S$ is also a global minimum over $S$.
\item If $f$ is even strictly convex, then there exists at most one global minimum. 
\item Let $S$ be an open set. Then the condition $\nabla f(x_\ast)=0$ is a sufficient and necessary condition for $x_\ast\in S$ to be a global minimum of $f$ over $S$.
\end{enumerate}
\end{prop}

\section{Lyapunov methods for optimization}\label{app:Lyapunov}
We will motivate the application of Lyapunov theory to optimization methods based on stability analysis of ordinary differential equations (ODEs). Lyapunov theory can be applied to analyze the behavior of dynamical systems without solving them analytically.  Let 
\begin{equation}\label{eq:ODE_system}
\frac{\dd z(t)}{\dd t} = g(z(t)),\quad z(0) = z_0\in\R^d
\end{equation}
be a continuous-time dynamical system described as ODE with $g:\R^d\to\R^d$ locally Lipschitz-continuous.  
\begin{defi}
We call a point $\bar z\in\R^d$ \textit{equilibrium point} of \eqref{eq:ODE_system} if $g(\bar z)=0$.
\end{defi}
Lyapunov methods can be applied to describe stability of equilibrium points of \eqref{eq:ODE_system}. Without loss of generality we  assume that $\bar z=0\in\R^d$ is an equilibrium point of $\eqref{eq:ODE_system}$.
\begin{defi}
The equilibrium point $\bar z=0$ is called
\begin{enumerate}
\item \textit{stable}, if for all $\varepsilon>0$ there exists a $\delta=\delta(\varepsilon)>0$ such that:
\[\|z(0)\|<\delta \quad \implies \quad \|z(t)\|<\varepsilon \quad\text{for all}\ t\ge0. \]
\item \textit{unstable}, if $\bar z$ is not stable. 
\item \textit{locally asymptotically stable}, if $\bar z$ is stable and $\delta>0$ can be chosen such that:
\[\|z(0)\|<\delta \quad \implies \quad \lim_{t\to\infty} z(t) = 0. \]
\item \textit{globally stable}, if $\bar z$ is stable and $\lim_{\varepsilon\to\infty}\delta(\varepsilon)=\infty$.
\item \textit{globally asymptotically stable}, if $\lim_{t\to\infty} z(t) = 0$ for all $z_0\in\R^d$.
\end{enumerate}
\end{defi}

We consider a continuously differentiable function $V:\R^d\to\R$ - in the following called candidate of a Lyapunov function - which will be used to be described along the trajectories.  With the help of chain rule, we can describe the dynamical behavior of $V$ along the trajectory of $z(t)$: 
\[\frac{\dd V(z(t))}{\dd t} =\langle \nabla_z V(z(t)), \frac{\dd z(t)}{\dd t}\rangle = \langle \nabla_z V(z(t)), g(z(t))\rangle. \]
Under specific assumptions on the function $V$ one can verify global stability of $\bar z$.
\begin{thm}[Theorem~3.2 in \cite{khalil2002nonlinear}]\label{thm:Lyapunov}
Let $\bar z=0$ be an equilibrium point  of \eqref{eq:ODE_system}. Moreover, let $V:\R^d\to\R$ be continuously differentiable with
\begin{enumerate}
\item $V(\bar z) =0$ and $V(z)>0$ for all $z\in\R^d\setminus \{\bar z\}$,
\item $V(z) \to\infty$ for $\|z\|\to\infty$,
\item $\frac{\dd V(z(t))}{\dd t} = \langle \nabla_z V(z(t)), g(z(t))\rangle \le -W(z(t))$, for some continuous function $W:\R^d\to\R$.
\end{enumerate}
If $W(z)\ge0$ for all $z\in\R^d$, then $\bar z$ is globally stable. Moreover, if $W(z)>0$ for all $z\in\R^d\setminus\{\bar z\}$, then $\bar z$ is even globally asymptotically stable. 
\end{thm}
A function satisfying the above conditions is sometimes also referred to a \textit{Lyapunov function}.

\medskip

We want to apply Lyapunov theory as tool for analyzing convergence of optimization  methods. Let us start with an motivating example in continuous-time. 
\begin{example}
The gradient descent method with fixed step size $\bar\alpha>0$ is written as
\[x_{k+1} = x_k - \bar\alpha \nabla f(x_k), \]
which for $\bar\alpha\to0$ can be interpreted as Euler-scheme of the ODE
\begin{equation}\label{eq:gradientflow}
\frac{\dd x(t)}{\dd t} = -\nabla f(x(t)).
\end{equation}
In general, optimization schemes are often described and/or analysed through its continuous-time formulation. The system is sometimes also called \textit{gradient flow} and intuitively speaking, it describes gradient descent with degenerated step size.  It can be used as indicator of how gradient descent may perform with sufficiently small step size $\bar\alpha>0$. We have analysed gradient descent methods under various settings such as (strong) convexity and smoothness. The convergence analysis of the gradient flow \eqref{eq:gradientflow} can be done in a very similar way using Lyapunov theory.  We want to construct a function describing the convergence behavior of the dynamical system \eqref{eq:gradientflow} without solving it explicitly.  The most straightforward analysis can be done for the error function 
$V(x(t)) = \mathcal E(x(t)) = f(x(t))- f^\ast$, where $f^\ast = \min_{x\in\R^d}\ f(x)>-\infty$. Similar as before, we can describe the dynamical behavior of $V$ through the ODE
\[\frac{\dd V(x(t))}{\dd t} = \langle \nabla_x V(x(t)), \frac{\dd x(t)}{\dd t}\rangle = -\langle \nabla f(x(t)), \nabla f(x(t))\rangle = -\|\nabla f(x(t))\|^2. \]
Under suitable conditions on $f$ one can now apply Theorem~\ref{thm:Lyapunov} in order to quantify stability of a (unique) stationary point $x_\ast\in\R^d$ of $f$ as equilibrium point of the gradient flow \eqref{eq:gradientflow}.
\end{example}

Let us consider a optimization scheme in continuous-time of the form
\begin{equation}\label{eq:cont_time_Opti}
\frac{\dd x(t)}{\dd t} = g(x(t)), 
\end{equation}
and a corresponding error function $\mathcal E:\R^d\to\R$ to be analyzed. We aim to quantify the convergence of \eqref{eq:cont_time_Opti} through the error function along the trajectory
\[\frac{\dd \mathcal E(x(t))}{\dd t} = \langle \nabla_x\mathcal E(x(t)), g(x(t))\rangle\begin{cases}\le 0\\ <0 \end{cases}\, . \]
This can be used to obtain results such as monotonically decreasing error ($\le 0$) or even convergence of the error ($<0$). However, these properties are not sufficient to describe the speed of convergence.  In order to say something about the convergence speed, we can define a time-dependent error function $\widehat{\mathcal E}:[0,\infty)\times\R^d\to\R$ of the form
\[\widehat{\mathcal E}(t,x) = \gamma(t) \mathcal E(x),\]
where $\gamma:[0,\infty)\to\R_+$ (smooth and continuously differentiable) is devoted to describe the speed of convergence.  Let us illustrate the strategy of proving convergence with the help of this error function throught the following example.

\begin{example}
We want to minimize $f$ using \eqref{eq:cont_time_Opti} and prove convergence of the error function $\mathcal E(x) = f(x)-f^\ast$. One possible strategy is to construct the time dependent error function \[\widehat{\mathcal E}(t,x(t)) = \gamma(t)\mathcal E(x(t))+r(x(t)), \]
where $\gamma:[0,\infty)$ with $\frac{\dd \gamma(t)}{\dd t}>0$ describes our guess of convergence rate and $r:\R^d\to\R$ is an auxiliary function with $r(x)\ge0$.  Suppose we are able to prove monotonicity of the error in the form of
\[\frac{\dd \widehat{\mathcal E}(t,x(t))}{\dd t} \le 0 \]
(using again chain rule), we can directly imply by definition of $\widehat{\mathcal E}$ that
\[\widehat{\mathcal E}(t,x(t)) = \gamma(t) (f(x(t))-f^\ast) + r(x(t)) \le \widehat{\mathcal E}(0,x(0)), \]
and therefore, we can imply convergence of the error $\mathcal E$ in the form of
\[\mathcal E(x(t))=f(x(t))-f^\ast \le \frac{\widehat{\mathcal E}(0,x(0))}{\gamma(t)}.\]
\end{example}

A similar strategy for discrete-time optimization schemes is described in Section~\ref{sec:NAM_convex}.

\section{Measure-theoretic background}\label{app:measuretheory}
In the following section, we will briefly recall the definition of Dynkin systems and $\sigma$-algebras, and a useful tool which allows to prove measure theoretical properties on a $\cap$-stable generator of a $\sigma$-algebra. For example, if one wants to verify that two measures $\mu_1$ and $\mu_2$ on the same measurable space $(\Omega,\mathcal A)$ are equal, it is sufficient to verify the equality on an $\cap$-stable generator $\cE$ of $\mathcal A$ (i.e.~$\sigma(\cE)=\cA$). For more details we refer to \cite[Section~1]{klenke}.

\begin{defi}
Let $\cA\subset \cP(\Omega)$ be a non-empty family of subsets of $\Omega$. We call $\cA$ a \textit{$\sigma$-algebra} if
\begin{enumerate}
\item $\Omega\in\cA$,
\item $A\in\cA$ $\implies$ $A^\complement\in\cA$,
\item $A_1,A_2,\dots \in\cA$ $\implies$ $\bigcup\limits_{i=1}^\infty A_i \in\cA$
\end{enumerate}
\end{defi}

\begin{defi}
Let $\cD\subset \cP(\Omega)$ be a non-empty family of subsets of $\Omega$. We call $\cD$ a \textit{Dynkin system} if
\begin{enumerate}
\item $\Omega\in\cD$,
\item $A\in\cD$ $\implies$ $A^\complement\in\cD$,
\item $A_1,A_2,\dots \in\cD$ pairwise disjoint (i.e. $A_i\cap A_j = \emptyset$, $i\neq j$)  $\implies$ $\biguplus\limits_{i=1}^\infty A_i \in\cD$
\end{enumerate}
\end{defi}

\begin{prop}
Let $\cD$ be a Dynkin system, then the following is equivalent:
\begin{itemize}
\item $\cD$ is a $\sigma$-algebra,
\item $\cD$ is $\cap$-stable, i.e.~for $A,B\in\cD$ it follows $A\cap B\in\cD$.
\end{itemize} 
\end{prop}

\begin{defi}
Let $\cE\subset \cP(\Omega)$ be a non-empty family of subsets of $\Omega$. Then we define 
\[\sigma(\cE) = \bigcap\limits_{\cE\subset B,\ B\ \sigma-\text{algebra}} B \]
as the $\sigma$-Algebra generated by $\cE$. Similarly we define
\[ d(\cE) = \bigcap\limits_{\cE\subset B,\ B\ \text{Dynkin system}} B \]
as the Dynkin system generated by $\cE$.
\end{defi}

\begin{thm}\label{thm:monclass}
Let $\cE\subset \cP(\Omega)$ be a non-empty family of subsets of $\Omega$ which is $\cap$-stable, then $d(\cE)=\sigma(\cE)$.
\end{thm}

\begin{remark}\label{rem:monotoneclass}
In order to prove that some abstract condition $\oplus$ is satisfied for all $A\in\cA$, where $\cA$ is a $\sigma$-algebra over $\Omega$, we can follow a specific strategy:
\begin{enumerate}
\item define the set $\cM = \{A\in\cA\mid \text{condition}\ \oplus\ \text{is satisfied for}\ A\}\subset \cA$,
\item prove that $\cM$ is a Dynkin system,
\item find a $\cap$-stable generator $\cE$ of $\cA$ such that $\cE\subset \cM$,
\item imply that
\[\cA = \sigma(\cE) = d(\cE) \overset{\cE\subseteq \cM}{\subseteq} d(\cM) \overset{\cM\ \text{Dynkin system}}{=} \cM \subseteq \cA, \]
which yields that $\cM = \cA$, i.e.~condition $\oplus$ is satisfied for all $A\in\cA$.
\end{enumerate}
\end{remark}

\section{Martingales}\label{app:martingales}
In this section, we briefly recall Doob's martingale convergence theorem, in particular for supermartingales. We refer to \cite[Section~9--11]{klenke} for more details.
\begin{defi}
Let $(\Omega,\mathcal A,\mathcal F,\mathbb P)$ be a filtered probability space. We call stochastic process $X=(X_k)_{k\in\N}$ \textit{martingale} with respect to the filtration $\mathcal F$, if
\begin{enumerate}
\item $X$ is $\mathcal F$-adapted,
\item $\E[|X_k|]<\infty$ for all $k\in\N$,
\item $\E[X_k\mid \mathcal F_l] = X_l$ for all $k,l\in\N$ with $l\le k$. 
\end{enumerate}
If 3.~holds with $\le$ ($\ge$), then we call $X$ a \textit{supermartingale} (\textit{submartingale}).
\end{defi}

\begin{thm}[Doob's supermartingale convergence]\label{thm:martingaleconv}
Let $X=(X_k)_{k\in\N}$ be a supermartingale or submartingale with
\[\sup_{k\in\N}\ \E[|X_k|] <\infty, \]
then $(X_k)_{k\in\N}$ converges almost surely to an $\mathcal F_\infty$-measurable and integrable random variable $X_\infty$.
\end{thm}

\begin{remark}
For a martingale we have the property of constant expectation $\E[X_k]= \E[X_0]$ for all $k\in\N$. For a supermartingale in contrast, we have decreasing expectation $\E[X_k]\le \E[X_l]\le \E[X_0]$ for $k\ge l$. Therefore, it is sufficient to replace the condition $\sup_{k\in\N}\ \E[|X_k|] <\infty$ of Theorem~\ref{thm:martingaleconv} with the expectation of the negative part $X_k^{-} = \max(0,-X_k)$, since we have
\[\E[|X_k|] = \E[X_k^+ + X_k^-] = \E[X_k^+ - X_k^- + X_k^- + X_k^-] = \E[X_k] + 2\E[X_k^-] \le \E[X_0] + 2\E[X_k^-].\]
In case that $X_k$ is bounded from below it follows that $\E[X_k^-]<\infty$. This means it is sufficient to prove a uniform lower bound on $(X_k)_{k\in\N}$ in order to apply Theorem~\ref{thm:martingaleconv}.
\end{remark}

\section{Deferred proofs}\label{app:omittedproofs}
\subsection{Deferred proofs of Chapter~\ref{ch:unOpt}}
\begin{proof}[Proof of Theorem~\ref{thm:armijo}]
We will prove the assertion via contradiction. Firstly, suppose that $\bar x\in\R^d$ is an accumulating point of the sequence $(x_k)_{k\in\N}$ satisfying $\nabla f(\bar x) \neq 0$. Define the corresponding sub-sequence $(x_{k_n})_{n\in\N}$ converging to $\bar x$. Since the objective function $f$ is continuous, the sequence $(f(x_{k_n}))_{n\in\N}$ converges to $f(\bar x)$. Moreover, through the Armijo step size rule it follows that the sequence $(f(x_k))_{k\in\N}$ is monotonically decreasing, i.e.~$f(x_{k+1}) < f(x_k)$, and hence, the whole sequence $(f(x_k))_{k\in\N}$ converges to $f(\bar x)$. In particular, it follows that $(f(x_k))_{k\in\N}$ is a Cauchy-sequence and therefore
\[\lim_{k\to\infty} (f(x_{k+1})-f(x_k)) = 0. \]
Due to the Armijo step size rule,
\[f(x_k) - f(x_{k+1}) \ge \sigma \alpha_k \|\nabla f(x_k)\|\,, \]
implying that 
\[\lim_{k\to\infty} \alpha_k \|\nabla f(x_k)\| = 0. \]

Since we have assumed that $\nabla f(\bar x) \neq 0$ it follows from continuity of $\nabla f$ that $\lim_{n\to\infty}\|\nabla f(x_{k_n})\|\neq 0$, and hence, $\lim_{n\to\infty} \alpha_{k_n} = 0$. By construction of the Armijo step size rule, we can write $\alpha_{k_n} = \rho^{\ell_n}\cdot s_0$, where $\ell_n\in\N$ denotes the first iteration such that condition \eqref{eq:armijo_condition} is satisfied. Since $\lim_{n\to\infty} \alpha_{k_n} = 0$, for $n\in\N$ large enough there exists $\ell_n>0$ such that $f(x_{k_n}+\rho^{\ell_n-1} s_0d_{k_n}) > f(x_{k_n}) + \sigma\rho^{\ell_n-1}s_0 \nabla f(x_{k_n})^\top d_{k_n}$, which means that condition \eqref{eq:armijo_condition} is at least once not satisfied during the application of Algorithm~\ref{alg:armijo}. Thus, we have
\begin{equation} \label{eq:armijo_cond_n}
\frac{f(x_{k_n}+\rho^{\ell_n-1} s_0 d_{k_n})-f(x_{k_n})}{\rho^{\ell_n-1} s_0}> \sigma \nabla f(x_{k_n})^\top d_{k_n}. 
\end{equation}

With the mean-value theorem there exists some $r_n\in[0,\rho^{\ell_n-1}\cdot s_0]$ such that
\begin{align*}
f(x_{k_n} + \rho^{\ell_n-1} s_0 d_{k_n} ) - f(x_{k_n}) &= \nabla f(x_{k_n} + r_n d_{k_n})^\top (x_{k_n} + \rho^{\ell_n-1} s_0 d_{k_n} - x_{k_n})\\
&=\rho^{\ell_n-1} s_0  \nabla f(x_{k_n} + r_n d_{k_n})^\top d_{k_n},
\end{align*}
and with \eqref{eq:armijo_cond_n} it follows
\begin{equation} \label{eq:armijo_cond_n2}
\nabla f(x_{k_n} + r_n d_{k_n})^\top d_{k_n} > \sigma \nabla f(x_{k_n})^\top d_{k_n}.
\end{equation}
Due to continuity of $\nabla f$ and the assumption that $\lim_{n\to\infty} x_{k_n} = \bar x$, we obtain
\[\lim_{n\to\infty} d_{k_n} =\lim_{n\to\infty} - \frac{\nabla f(x_{k_n})}{\|\nabla f(x_{k_n})\|} = -\frac{\nabla f(\bar x)}{\|\nabla f(\bar x)\|}, \]
and, since $\lim_{n\to\infty}\alpha_{k_n} = 0$, we obtain
\[\lim_{n\to\infty} r_n = \rho^{\ell_n-1}s_0 = \frac{\rho^{\ell_n}}{\rho} s_0 = \frac{1}{\rho}\alpha_{k_n} = 0. \]
Both limits together imply that
\[\lim_{n\to\infty} \nabla f(x_{k_n} + r_n d_{k_n})^\top d_{k_n} = -\frac{\nabla f(\bar x)^\top \nabla f(\bar x)}{\|\nabla f(\bar x)\|} = -\|\nabla f(\bar x)\|. \]
and similarly, 
\[\lim_{n\to\infty} \sigma\nabla f(x_{k_n})^\top d_{k_n} = -\sigma \|\nabla f(\bar x)\|. \]
Finally, taking the limit $n\to\infty$ in equation \eqref{eq:armijo_cond_n2} it follows
\( -\|\nabla f(\bar x)\| \ge -\sigma \|\nabla f(\bar x)\|,\)
which contradicts the assumption $\sigma\in(0,1)$ for $\nabla f(\bar x) \neq 0$.  This proves that $\nabla f(\bar x) = 0$. 
\end{proof}

\subsection{Deferred proofs of Chapter~\ref{ch:accmethods}}
\begin{proof}[Proof of Lemma~\ref{lem:NAMconvex_aux}]
Firstly, we write down the increments of $(E_k)_{k\in\N}$ 
\[E_{k+1}-E_k = \frac{1}{2}\|z_{k+1}-x_\ast\|^2 - \frac12 \|z_k-x_\ast\|^2 + A_{k+1}(f(y_{k+1})-f^\ast) - A_k (f(y_k)-f^\ast) \]
and observe that
\begin{align*}
&\frac12\|z_{k+1}-x_\ast\|^2 - \frac12\|z_k-x_\ast\|^2 = \frac12\|z_{k+1}-x_\ast\|^2 -\frac12\|(z_k-z_{k+1})+(z_{k+1}-x_\ast)\|^2 \\
&=\frac12\|z_{k+1}-x_\ast\|^2 - \frac12\|z_k-z_{k+1}\|^2 - \langle z_k-z_{k+1},z_{k+1}-x_\ast\rangle - \frac12\|z_{k+1}-x_\ast\|^2\\
&= -\frac12\|z_k-z_{k+1}\|^2 -\langle z_k-z_{k+1},z_{k+1}-x_\ast\rangle = -\frac12 \|z_k-z_{k+1}\|^2 + \langle x_\ast-z_{k+1},(A_{k+1}-A_k)\nabla f(x_k)\rangle.
\end{align*}
This leads to
\begin{align*}
E_{k+1}-E_k &= \langle x_\ast-z_{k+1},(A_{k+1}-A_k)\nabla f(x_k)\rangle -\frac12 \|z_k-z_{k+1}\|^2\\ &\quad + A_{k+1}(f(y_{k+1})-f^\ast) - A_k (f(y_k)-f^\ast)\\
&=  \langle x_\ast-z_{k},(A_{k+1}-A_k)\nabla f(x_k)\rangle+\langle z_{k}-z_{k+1},(A_{k+1}-A_k)\nabla f(x_k)\rangle  -\frac12 \|z_k-z_{k+1}\|^2\\
&\quad + A_{k+1}(f(y_{k+1})-f^\ast) - A_k (f(y_k)-f^\ast).
\end{align*}
Since $0\le \frac12\|a-b\|^2 = \frac12\|a\|^2 -\langle a,b\rangle +\frac12\|b\|^2$ for $a,b\in\R^d$, we can apply the inequality $\langle a,b\rangle -\frac12\|a\|^2\le \frac12\|b\|^2$ to derive
\[\langle z_{k}-z_{k+1},(A_{k+1}-A_k)\nabla f(x_k)\rangle -\frac12 \|z_k-z_{k+1}\|^2 \le \frac12\|(A_{k+1}-A_k)\nabla f(x_k)\|^2\]
such that
\begin{align*}
E_{k+1}-E_k &=  \frac12(A_{k+1}-A_k)^2\|\nabla f(x_k)\|^2+\langle x_\ast-z_{k},(A_{k+1}-A_k)\nabla f(x_k)\rangle\\
&\quad + A_{k+1}(f(y_{k+1})-f^\ast) - A_k (f(y_k)-f^\ast).
\end{align*}
Moreover, we observe that
\begin{align*}
A_{k+1}(f(y_{k+1})-f^\ast) - A_k (f(y_k)-f^\ast) &= (A_{k+1}-A_k)(f(x_k)-f^\ast) + A_k(f(x_k)-f(y_k))\\&\quad + A_{k+1} (f(y_{k+1}-f(x_k)).
\end{align*}
With $ \varepsilon_{k+1}:= \frac12 (A_{k+1}-A_k)^2 \|\nabla f(x_k)\|^2 + A_{k+1} (f(y_{k+1})-f(x_k))$ it follows that
\begin{align*}
E_{k+1}-E_k &\le \varepsilon_{k+1} +\langle x_\ast-z_{k},(A_{k+1}-A_k)\nabla f(x_k)\rangle\\
&\quad +(A_{k+1}-A_k)(f(x_k)-f^\ast)+A_k(f(x_k)-f(y_k)),
\end{align*}
and it is left to prove that
\[\mathcal R = \langle x_\ast-z_{k},(A_{k+1}-A_k)\nabla f(x_k)\rangle +(A_{k+1}-A_k)(f(x_k)-f^\ast)+A_k(f(x_k)-f(y_k))\le 0.\]
We add $-y_k+y_k=0$ to derive
\begin{align*}
\mathcal R &= \langle x_\ast-y_{k},(A_{k+1}-A_k)\nabla f(x_k)\rangle +\langle y_k-z_k,(A_{k+1}-A_k)\nabla f(x_k)\rangle\\ &\quad +(A_{k+1}-A_k)(f(x_k)-f^\ast)+A_k(f(x_k)-f(y_k)),
\end{align*}
where we now aim to apply convexity of $f$ in the form of 
\[f(z)-f(y) + \langle y-z,\nabla f(z)\rangle \le 0. \]
Therefore,  using $y_k-z_k = \frac{A_{k+1}}{A_{k+1}-A_k}(y_k-x_k)$ we again rewrite $\mathcal R$ in the form of
\begin{align*}
\mathcal R &=  \langle x_\ast-y_{k},(A_{k+1}-A_k)\nabla f(x_k)\rangle + A_{k+1}\langle y_k-x_k,\nabla f(x_k)\rangle\\
&\quad +(A_{k+1}-A_k)(f(x_k)-f^\ast)+A_k(f(x_k)-f(y_k))\\
&=\langle x_\ast-y_{k},(A_{k+1}-A_k)\nabla f(x_k)\rangle + A_{k+1}\langle y_k-x_k,\nabla f(x_k)\rangle\\
&\quad -A_k \langle x_k,\nabla f(x_k)\rangle + A_k \langle x_k,\nabla f(x_k)\rangle\\
&\quad +(A_{k+1}-A_k)(f(x_k)-f^\ast)+A_k(f(x_k)-f(y_k))\\
&= (A_{k+1}-A_k)\langle x_\ast - x_k,\nabla f(x_k)\rangle + A_k \langle y_k-x_k, \nabla f(x_k)\rangle \\
&\quad +(A_{k+1}-A_k)(f(x_k)-f^\ast)+A_k(f(x_k)-f(y_k))\\
&= (A_{k+1}-A_k)\left\{ f(x_k)-f(x_\ast) + \langle x_\ast-x_k,\nabla f(x_k)\rangle \right\}\\
&\quad + A_k \left\{ f(x_k)-f(y_k) + \langle y_k -x_k,\nabla f(x_k)\rangle \right\}\\
&\le 0
\end{align*}
by convexity of $f$. We have proved that
$E_{k+1}-E_k \le \varepsilon_{k+1}$.
\end{proof}

\end{document}